\documentclass[11pt,leqno]{amsart}
\usepackage{hyperref}
\usepackage{epsfig}
\usepackage{enumitem}
\usepackage{amsmath, amssymb}
\usepackage{amscd}
\usepackage[all,cmtip,matrix,arrow]{xy}
\usepackage{mathrsfs}
\usepackage{graphicx}
\usepackage{tikz-cd}
\usepackage{mathabx}
\usepackage{array}
\usepackage{longtable}
\xyoption{rotate}
\xyoption{pdf}
\usepackage{xparse}

\newcommand{\Implies}[2]{$\text{\ref{#1}}\implies\text{\ref{#2}}$}
\newcommand{\Iff}[2]{$\text{\ref{#1}}\iff\text{\ref{#2}}$}

\makeatletter
\DeclareRobustCommand*{\bfseries}{%
  \not@math@alphabet\bfseries\mathbf
  \fontseries\bfdefault\selectfont
  \boldmath
}
\makeatother

\AtBeginEnvironment{tikzcd}{\setlength{\mathsurround}{0pt}}

\newtheorem{theo}{Theorem}[section]
\newtheorem{lemma}[theo]{Lemma}
\newtheorem{defi}[theo]{Definition}
\newtheorem{prop}[theo]{Proposition}

\newtheorem{cor}[theo]{Corollary}
\newtheorem{remark}[theo]{Remark}

\newtheorem{example}[theo]{Example}
\newtheorem{ques}[theo]{Question}
\numberwithin{equation}{section}

\mathchardef\mhyphen="2D

\def\CP{\mathbb{CP}}
\def\A{{\mathbb A}}

\def\N{\mathbb{N}}
\def\bL{\mathbb{L}}

\def\bS{\mathbb{S}}
\def\C{\mathbb{C}}
\def\Z{\mathbb{Z}}
\def\Q{\mathbb{Q}}
\def\bG{\mathbb{G}}

\def\coh{\operatorname{coh}}

\def\wt{\widetilde}

\def\bR{{\mathbf R}}
\def\bL{{\mathbf L}}

\def\PP{{\mathbb P}}

\def\pre-tr{\operatorname{pre-tr}}
\def\h{\operatorname{h}}

\def\Hom{\operatorname{Hom}}
\def\Ob{\operatorname{Ob}}
\def\Map{\operatorname{Map}}
\def\End{\operatorname{End}}

\def\gr{\operatorname{gr}}

\DeclareMathOperator*{\colim}{colim}

\newcommand{\tens}[1]{%
  \mathbin{\mathop{\otimes}\displaylimits_{#1}}%
}

\newcommand{\Ltens}[1]{%
  \mathbin{\mathop{\otimes}\displaylimits^{\bL}_{#1}}%
}

\newcommand{\stens}[1]{%
	\mathbin{\mathop{\otimes}\displaylimits^{#1}}%
}

\newcommand{\ms}[1]{\mathscr{#1}}
\newcommand{\msK}{\ms{K}}

\newcommand{\hy}{\mhyphen}

\newcommand{\indlim}[1][]{\mathop{\varinjlim}\limits_{#1}}
\newcommand{\inddlim}[1][]{{``{\indlim[#1]}"}}
\newcommand{\prolim}[1][]{\mathop{\varprojlim}\limits_{#1}}
\newcommand{\proolim}[1][]{{``{\prolim[#1]}"}}
\newcommand{\biggplus}[1][]{\mathop{\bigoplus}\limits_{#1}}
\newcommand{\bigooplus}[1][]{{``{\biggplus[#1]}"}}
\newcommand{\prodd}[1][]{\mathop{\prod}\limits_{#1}}

\newcommand{\bbar}{\overline}
\newcommand{\hhat}{\widehat}

\newcommand{\xto}{\xrightarrow}

\newcommand{\hto}{\hookrightarrow}

\newcommand{\onto}{\twoheadrightarrow}

\newcommand{\Bbar}{\operatorname{Bar}}
\newcommand{\Cobar}{\operatorname{Cobar}}
\newcommand{\Loc}{\operatorname{Loc}}
\newcommand{\Idl}{\operatorname{Idl}}
\newcommand{\Vcyc}{\operatorname{Vcyc}}

\newcommand{\lex}{\operatorname{lex}}
\newcommand{\dec}{\operatorname{dec}}
\newcommand{\rex}{\operatorname{rex}}
\newcommand{\rig}{\operatorname{rig}}

\newcommand{\lax}{\operatorname{lax}}

\newcommand{\Mot}{\operatorname{Mot}}

\newcommand{\QCoh}{\operatorname{QCoh}}

\newcommand{\nuc}{\operatorname{nuc}}
\newcommand{\Nuc}{\operatorname{Nuc}}

\newcommand{\IndCoh}{\operatorname{IndCoh}}

\newcommand{\Cat}{\operatorname{Cat}}

\newcommand{\Cone}{\operatorname{Cone}}
\newcommand{\Fiber}{\operatorname{Fiber}}

\newcommand{\idem}{\operatorname{idem}}

\newcommand{\loc}{\operatorname{loc}}
\newcommand{\splt}{\operatorname{split}}
\newcommand{\cont}{\operatorname{cont}}
\newcommand{\red}{\operatorname{red}}
\newcommand{\sm}{\operatorname{sm}}
\newcommand{\acc}{\operatorname{acc}}
\newcommand{\ex}{\operatorname{ex}}

\newcommand{\st}{\operatorname{st}}
\newcommand{\cg}{\operatorname{cg}}
\newcommand{\Prr}{\operatorname{Pr}}
\newcommand{\dual}{\operatorname{dual}}
\newcommand{\Hoch}{\operatorname{Hoch}}

\newcommand{\deff}{\operatorname{def}}
\newcommand{\Fil}{\operatorname{Fil}}

\newcommand{\bbD}{{\mathbb D}}

\newcommand{\bE}{{\mathbb E}}

\newcommand{\mk}{\mathrm k}
\newcommand{\ev}{\mathrm ev}

\newcommand{\cF}{{\mathcal F}}
\newcommand{\cG}{{\mathcal G}}
\newcommand{\cO}{{\mathcal O}}

\newcommand{\cD}{{\mathcal D}}
\newcommand{\cV}{{\mathcal V}}
\newcommand{\cA}{{\mathcal A}}
\newcommand{\cB}{{\mathcal B}}
\newcommand{\cI}{{\mathcal I}}
\newcommand{\cC}{{\mathcal C}}
\newcommand{\cE}{{\mathcal E}}
\newcommand{\cW}{{\mathcal W}}
\newcommand{\cY}{{\mathcal Y}}
\newcommand{\cU}{{\mathcal U}}

\newcommand{\cS}{{\mathcal S}}
\newcommand{\cT}{{\mathcal T}}
\newcommand{\cK}{{\mathcal K}}

\newcommand{\veps}{\varepsilon}

\newcommand{\un}{\underline}
\newcommand{\la}{\langle}
\newcommand{\ra}{\rangle}

\newcommand{\incl}{\operatorname{incl}}

\newcommand{\THH}{\operatorname{THH}}
\newcommand{\MU}{\operatorname{MU}}
\newcommand{\TR}{\operatorname{TR}}
\newcommand{\TC}{\operatorname{TC}}
\newcommand{\HH}{\operatorname{HH}}
\newcommand{\HC}{\operatorname{HC}}
\newcommand{\HP}{\operatorname{HP}}

\newcommand{\CAlg}{\operatorname{CAlg}}

\newcommand{\Fun}{\operatorname{Fun}}

\newcommand{\Perf}{\operatorname{Perf}}

\newcommand{\perf}{\operatorname{perf}}

\newcommand{\Kar}{\operatorname{Kar}}

\newcommand{\rat}{\operatorname{rat}}

\newcommand{\Cosh}{\operatorname{Cosh}}

\newcommand{\Shv}{\operatorname{Shv}}

\newcommand{\coker}{\operatorname{coker}}

\newcommand{\im}{\operatorname{Im}}

\newcommand{\Rep}{\operatorname{Rep}}

\newcommand{\Tor}{\operatorname{Tor}}

\newcommand{\Id}{\operatorname{Id}}

\newcommand{\Sp}{\operatorname{Sp}}

\newcommand{\Sym}{\operatorname{Sym}}

\newcommand{\Ind}{\operatorname{Ind}}
\newcommand{\Pro}{\operatorname{Pro}}
\newcommand{\Calk}{\operatorname{Calk}}
\newcommand{\Res}{\operatorname{Res}}
\newcommand{\Spec}{\operatorname{Spec}}

\newcommand{\Spa}{\operatorname{Spa}}

\newcommand{\Stab}{\rm Stab}

\newcommand{\Alg}{\operatorname{Alg}}
\newcommand{\Coalg}{\operatorname{Coalg}}

\newcommand{\id}{\operatorname{id}}

\newcommand{\coev}{\operatorname{coev}}

\newcommand{\tors}{\operatorname{tors}}

\newcommand{\compl}{\operatorname{compl}}

\newcommand{\acycl}{\operatorname{acycl}}

\newcommand{\Set}{\operatorname{Set}}

\newcommand{\Gr}{\operatorname{Gr}}

\newcommand{\pt}{\operatorname{pt}}

\newcommand{\Mod}{\operatorname{Mod}}

\newcommand{\mult}{\operatorname{mult}}

\newcommand{\naive}{\operatorname{naive}}

\newcommand{\CycSp}{\operatorname{CycSp}}
\newcommand{\cyc}{\operatorname{cyc}}
\newcommand{\inv}{\operatorname{inv}}
\newcommand{\triv}{\operatorname{triv}}
\newcommand{\Orb}{\operatorname{Orb}}

\newcommand{\tref}{\operatorname{ref}}
\newcommand{\biggg}{\operatorname{big}}

\usepackage{epsf}
\usepackage{amscd}

\newcommand{\mQ}{\mathcal{Q}}

\newcommand{\mPr}{\mathfrak{Pr}}

\newcommand{\bw}{\mathrm{w}}

\title[Rigidity of the category of localizing motives]
{Rigidity of the category of localizing motives}

\author{Alexander I. Efimov}
\address{Steklov Mathematical Institute of RAS, Gubkin St. 8, GSP-1, Moscow 119991, Russia}
\email{efimov@mccme.ru}

\begin{document}

\begin{abstract} In this paper we study the category of localizing motives $\Mot^{\loc}$ -- the target of the universal finitary localizing invariant of idempotent-complete stable categories as defined by Blumberg-Gepner-Tabuada. We prove that this (presentable stable) category is rigid symmetric monoidal in the sense of Gaitsgory and Rozenblyum. In particular, it is dualizable. More precisely, we prove a more general version of this result for the category $\Mot^{\loc}_{\cE}$ -- the target of the universal finitary localizing invariant of dualizable modules over a rigid symmetric monoidal category $\cE.$
	
We obtain general results on morphisms and internal $\Hom$ in the categories $\Mot^{\loc}_{\cE}.$ As an application we compute the morphisms in multiple non-trivial examples. In particular, we prove the corepresentability statements for $\TR$ (topological restriction) and $\TC$ (topological cyclic homology) when restricted to connective $\bE_1$-rings. As a corollary, for a connective $\bE_{\infty}$-ring $R$ we obtain a $\TR(R)$-module structure on the nil $K$-theory spectrum $NK(R).$

We also apply the rigidity theorem to define refined versions of negative cyclic homology and periodic cyclic homology. This was announced previously in \cite{E24b}, and certain very interesting examples were computed by Meyer and Wagner in \cite{MW24}. Here we do several computations in characteristic $0,$ in particular showing that in seemingly innocuous situations the answer can be given by an interesting algebra of overconvergent functions.
\end{abstract}


\maketitle

\tableofcontents

\section{Introduction}

We recall the following interpretation of non-connective algebraic $K$-theory of stable categories due to Blumberg-Gepner-Tabuada \cite{BGT}. Consider the $\infty$-category $\Cat^{\perf}$ of all small idempotent-complete stable categories. Recall that a functor $F:\Cat^{\perf}\to\cT$ to a stable category $\cT$ is a localizing invariant if $F(0)=0$ and $F$ takes short exact sequences to exact triangles (formally, fiber-cofiber sequences to cofiber sequences). If $\cT$ is moreover presentable, we say that $F$ is {\it finitary} if it commutes with filtered colimits. There is a universal finitary localizing invariant  
\begin{equation}\label{eq:U_loc_for_Cat^perf_intro}
\cU_{\loc}:\Cat^{\perf}\to\Mot^{\loc}.
\end{equation}
Then one has
\begin{equation*}
\Hom_{\Mot^{\loc}}(\cU_{\loc}(\Sp^{\omega}),\cU_{\loc}(\cC))=K(\cC),
\end{equation*}
where $\Hom$ stands for the mapping spectrum, $\Sp$ is the category of spectra, and $\Sp^{\omega}$ is the full subcategory of compact objects, i.e. finite spectra.

We call $\Mot^{\loc}$ the category of localizing motives, it is also known as the category of noncommutative motives. Note that it's definition is analogous to the definition of the Grothendieck group of an abelian (or exact) category.

It is natural to try to understand more about the category $\Mot^{\loc}.$  A priori we know that it is presentable stable and by \cite{BGT} the object $\cU_{\loc}(\Sp^{\omega})$ is compact. It also has a natural symmetric monoidal structure, given by
\begin{equation*}
\cU_{\loc}(\cC)\otimes\cU_{\loc}(\cD)=\cU_{\loc}(\cC\otimes\cD).
\end{equation*} 

One can hope that the category $\Mot^{\loc}$ is at least dualizable, which conceptually is very close to being compactly generated. A much more naive guess would be that $\Mot^{\loc}$ is generated by dualizable objects, which are automatically compact since the unit object is compact. Surprisingly, the latter guess turns to be not too far from the real picture. More precisely, we prove the following result, which is a special case of Theorem \ref{th:dualizability_and_rigidity}.

\begin{theo}\label{th:rigidity_over_Sp_intro}
The symmetric monoidal category $\Mot^{\loc}$ is rigid in the sense of Gaitsgory and Rozenblyum \cite{GaRo17}. In particular, it is dualizable.
\end{theo}

Recall that for small symmetric monoidal categories the classical notion of rigidity means that every object is dualizable. If a presentable stable symmetric monoidal category $\cE$ is compactly generated, then its rigidity in the sense of \cite{GaRo17} exactly means that $\cE$ is equivalent to an ind-completion of a small rigid category. A typical example is given by the derived category $D_{qc}(X)$ for a quasi-compact quasi-separated scheme scheme $X$ \cite{BVdB03}. However, there are many examples of large rigid categories which are not compactly generated, e.g. categories of sheaves on (non-profinite) compact Hausdorff spaces.

We recall that by a recent result of Ramzi-Sosnilo-Winges \cite[Theorem 1.5]{RSW25} the functor \eqref{eq:U_loc_for_Cat^perf_intro} is essentially surjective. Taking this into account, we obtain the following immediate application of Theorem \ref{th:rigidity_over_Sp_intro}.

\begin{cor}
Any finitary localizing invariant $\Cat^{\perf}\to \Sp$ is of the form $K(\cC\otimes-)$ for some $\cC\in\Cat^{\perf}.$
\end{cor} 

To prove Theorem \ref{th:rigidity_over_Sp_intro} we use a certain characterization of rigid categories which is contained in \cite{Ram24b}. Namely, $\cE$ is rigid if and only if the unit object $1_{\cE}$ is compact and $\cE$ is generated (via colimits) by sequential colimits $\indlim[n]x_n,$ such that each map $x_n\to x_{n+1}$ is trace-class in the sense of \cite[Definition 13.11]{CS20}. The latter condition means that the corresponding morphisms $1_{\cE}\to\un{\Hom}_{\cE}(x_n,x_{n+1})$ factor through $\un{\Hom}_{\cE}(x_n,1_{\cE})\otimes x_{n+1}.$ Morally this means that the map $x_n\to x_{n+1}$ ``virtually factors through a dualizable object'', although there is no such factorization in general.

In fact we prove a much more general relative version of Theorem \ref{th:rigidity_over_Sp_intro}. First, we recall from \cite{E24} that we can equivalently consider the localizing invariants of dualizable categories. Namely, we denote by $\Cat_{\st}^{\dual}$ the category of (presentable stable) dualizable categories, where the morphisms are strongly continuous functors, i.e. functors with a colimit-preserving right adjoint. We denote by the same symbol the universal finitary localizing invariant
\begin{equation}\label{eq:U_loc_for_Cat^dual_intro}
\cU_{\loc}:\Cat_{\st}^{\dual}\to \Mot^{\loc},
\end{equation}
whose target is the same as in \eqref{eq:U_loc_for_Cat^perf_intro}. More precisely, the functor \eqref{eq:U_loc_for_Cat^perf_intro} is isomorphic to the composition of \eqref{eq:U_loc_for_Cat^dual_intro} with $\Ind(-):\Cat^{\perf}\to\Cat_{\st}^{\dual}.$ 

Now we consider the following relative version. Let $\cE$ be a (presentable stable) rigid symmetric monoidal category. We denote by $\Cat_{\cE}^{\dual}$ the category of $\cE$-linear dualizable categories, which is the same as the category of $\cE$-modules in $\Cat_{\st}^{\dual}.$ Again, we consider the universal finitary localizing invariant
\begin{equation*}
\cU_{\loc}:\Cat_{\cE}^{\dual}\to\Mot^{\loc}_{\cE},
\end{equation*}
whose target naturally acquires a symmetric monoidal structure so that $\cU_{\loc}$ is refined to a symmetric monoidal functor. We prove the following generalization of Theorem \ref{th:rigidity_over_Sp_intro}.

\begin{theo}\label{th:rigidity_over_E_intro}
Let $\cE$ be a presentable stable rigid symmetric monoidal category. Then the category $\Mot^{\loc}_{\cE}$ is also rigid.
\end{theo}

In fact the statement of Theorem \ref{th:dualizability_and_rigidity} is even more general: we prove rigidity of $\Mot^{\loc}_{\cE}$ assuming that $\cE$ is only $\bE_2$-monoidal. We also formulate and prove a version of this statement when $\cE$ is only $\bE_1$-monoidal. The latter applies in particular to $G$-equivariant localizing motives, which are discussed in Section \ref{sec:G_equivariant_motives}. In particular, we obtain equivalent formulations of the $K$-theoretic Farrell-Jones conjecture in terms of the categories of $G$-equivariant localizing motives, see Proposition \ref{prop:Farrell-Jones_reformulations}.

The proof of Theorem \ref{th:rigidity_over_E_intro} is quite difficult in general, but it simplifies considerably when $\cE$ is compactly generated. The general idea however is the same: we prove that for any countably presented relatively compactly generated $\cE$-module $\cC$ there exists a short exact sequence
\begin{equation*}
0\to\cC_1\to\cC_2\to\cC\to 0,
\end{equation*} 
such that both $\cC_1$ and $\cC_2$ are equivalent to sequential colimits in $\Cat_{\cE}^{\dual}$ with trace-class transition functors. Here the notion of a trace-class functor is the same as above: it is a trace-class morphism in the symmetric monoidal category $\Cat_{\cE}^{\dual}.$

Now let us return to the case $\cE=\Sp,$ and consider another question: how to compute the morphisms in $\Mot^{\loc}$? We first recall the answer in a much simpler situation when instead of localizing motives one deals with splitting motives, also known as additive motives \cite{BGT}. We return to the setting of small categories, and consider the universal finitary splitting invariant
\begin{equation*}
\cU_{\splt}:\Cat^{\perf}\to \Mot^{\splt}.
\end{equation*}
Here a functor $\Cat^{\perf}\to\cT$ is a splitting invariant if it is additive in semi-orthogonal decompositions, which is a much weaker condition than being a localizing invariant. The category $\Mot^{\splt}$ is compactly generated, more precisely the collection of compact generators is given by $\{\cU_{\splt}(\cC),\,\cC\in(\Cat^{\perf})^{\omega}\}$ (recall that the category $\Cat^{\perf}$ is itself compactly generated by \cite{BGT}). The morphisms in $\Mot^{\splt}$ are given by
\begin{equation*}
\Hom_{\Mot^{\splt}}(\cU_{\splt}(\cC),\cU_{\splt}(\cD))=K_{\geq 0}(\Fun^{\ex}(\cC,\cD)),\quad \cC\in(\Cat^{\perf})^{\omega},\,\cD\in\Cat^{\perf}.
\end{equation*}

The situation with $\Mot^{\loc}$ is much more difficult and interesting. In Section \ref{sec:morphisms_in_Mot^loc} we prove two kinds of general results on morphisms in the categories of localizing motives over a rigid base. The first one is given by Theorem \ref{th:morphisms_in_Mot^loc_via_limits} which allows under some assumptions to obtain a description of morphisms in $\Mot^{\loc}_{\cE}$ as an inverse limit of non-connective $K$-theory spectra of categories of functors. As an application, we obtain the commutation of $\cU_{\loc}$ with infinite products, see Corollary \ref{cor:U_loc_commutes_with_products} (for a general rigid base). Instead of formulating the statement of Theorem \ref{th:morphisms_in_Mot^loc_via_limits} here, we mention an interesting example of its application from Section \ref{ssec:corepresentability_TR}, which deals with the case of small categories over the absolute base as in \eqref{eq:U_loc_for_Cat^perf_intro}. 

Consider the $\bE_{\infty}$-ring $\bS[x]=\Sigma^{\infty}(\N_+)$ of functions on the flat affine line over $\bS$ \cite{Lur18}. We put $\cU_{\loc}(A)=\cU_{\loc}(\Perf(A))$ for an $\bE_1$-ring $A,$ and consider the reduced motive $\wt{\cU}_{\loc}(\bS[x]).$ Then Theorem \ref{th:corepresentability_of_TR} states that for any connective $\bE_1$-ring $R$ we have an isomorphism
\begin{equation}\label{eq:corepresentability_of_TR_intro}
\Hom_{\Mot^{\loc}}(\wt{\cU}_{\loc}(\bS[x]),\cU_{\loc}(R))\cong \Omega\prolim[n]K(R[x^{-1}]/x^{-n},(x^{-1}))\cong\TR(R).
\end{equation}
Here $\TR$ is the topological restriction \cite{BHM93, HM97, BM16}, and the second isomorphism in \eqref{eq:corepresentability_of_TR_intro} is given by \cite[Theorem A]{McC23}. If $R$ is a connective $\bE_{\infty}$-ring, then the $\bE_{\infty}$-ring structure on the mapping spectrum comes from the $\bE_{\infty}$-coalgebra structure on $\wt{\cU}_{\loc}(\bS[x]),$ given by the multiplicative monoid structure on the flat affine line. As an immediate application we obtain a structure of a $\TR(R)$-module on the nil $K$-theory spectrum $NK(R),$ see Corollary \ref{cor:nil_K_theory_module_over_TR}.

Note that if $R$ is a usual commutative ring, then we recover the classical big Witt vectors on the level of $\pi_0:$
\begin{equation*}
\pi_0\Hom_{\Mot^{\loc}}(\wt{\cU}_{\loc}(\bS[x]),\cU_{\loc}(R))\cong W_{\biggg}(R).
\end{equation*}
This has a well-known classical analogue, namely a theorem of Almkvist \cite{Alm74}, which in terms of splitting motives states that we have an isomorphism
\begin{equation*}
\pi_0\Hom_{\Mot^{\splt}}(\wt{\cU}_{\splt}(\bS[x]),\cU_{\splt}(R))\cong W_{\rat}(R),
\end{equation*}
see \cite{Tab14}. Here $W_{\rat}(R)\subset W_{\biggg}(R)$ is the subring of rational big Witt vectors, namely
\begin{equation*}
W_{\rat}(R)=\left\{\frac{f}{g}\mid f,g\in 1+x^{-1}R[x^{-1}]\right\}\subset 1+x^{-1}R[[x^{-1}]]=W_{\biggg}(R).
\end{equation*}

We also obtain a corepresentability statement for topological cyclic homology $\TC,$ when restricted to connective $\bE_1$-rings, given by Theorem \ref{th:corepresentability_of_TC}. In this case the corepresenting object is the unit object of the kernel of $\A^1$-localization. The latter is in fact a smashing localization by Theorem \ref{th:A^1_invariant_localizing_motives}.

Another general result on morphisms in $\Mot^{\loc}_{\cE}$ is given by Theorem \ref{th:morphisms_in_Mot^loc_via_internal_Hom} which in particular allows to describe the morphisms in $\Mot^{\loc}_{\cE}$ up to a shift as non-connective $K$-theory of a certain category of functors. We formulate one of the statements of this theorem here, specializing to the case of small categories over the absolute base. As in \cite{E24} we use the notion of a (countable) Calkin category, namely for $\cC\in\Cat^{\perf}$ we put $\Calk_{\omega_1}(\cC)=(\Ind(\cC)^{\omega_1}/\cC)^{\Kar},$ where ``$\Kar$'' means the idempotent completion. We denote by $\Calk_{\omega_1}^2(\cC)$ the second iteration of this construction.

\begin{theo}\label{th:functors_to_Calk^2_intro}(special case of Theorem \ref{th:morphisms_in_Mot^loc_via_internal_Hom})
Let $\cA,\cB\in\Cat^{\perf},$ and suppose that $\cA$ is $\omega_1$-compact (equivalently, the triangulated category $\h\cA$ has at most countably many isomorphism classes of objects, and the morphisms in $\h\cA$ are at most countable abelian groups). Then we have an isomorphism
\begin{equation*}
\Omega^2 K(\Fun^{\ex}(\cA,\Calk_{\omega_1}^2(\cB)))\xto{\sim}\Hom_{\Mot^{\loc}}(\cU_{\loc}(\cA),\cU_{\loc}(\cB)).
\end{equation*}
\end{theo}

We announced this statement earlier, and it was used by Ramzi-Sosnilo-Winges in \cite{RSW25} to prove that the functor $\cU_{\loc}:\Cat^{\perf}\to\Mot^{\loc}$ is also the universal $\omega_1$-finitary localizing invariant, where ``$\kappa$-finitary'' means the commutation with $\kappa$-filtered colimits. In fact their method applies to a general rigid base $\cE,$ once Theorem \ref{th:morphisms_in_Mot^loc_via_internal_Hom} is proven. We explain this in Section \ref{ssec:equivalence_Mot^loc_omega_1_and_Mot^loc}, showing that $\Cat_{\cE}^{\dual}\to\Mot^{\loc}_{\cE}$ is the universal $\omega_1$-finitary localizing invariant.

The proof of Theorem \ref{th:functors_to_Calk^2_intro} and its relative version is quite sophisticated, especially when working over a non-compactly generated rigid base. One of the important ingredients is a very special property of Calkin categories of dualizable categories, which we call {\it formal $\omega_1$-injectivity}. Namely, we say that $\cD\in\Cat_{\st}^{\dual}$
is formally $\omega_1$-injective if the functor $\hat{\cY}:\cD\to\Ind(\cD^{\omega_1})$ commutes with countable products (here $\hat{\cY}$ is the left adjoint to the colimit functor). We prove that for any dualizable category $\cC$ the category $\Calk_{\omega_1}(\cC)=\Ind(\cC^{\omega_1})/\hat{\cY}(\cC)$ is formally $\omega_1$-injective. This allows to deduce certain (internal) injectivity properties of Calkin categories, which eventually lead to the proof of Theorem \ref{th:morphisms_in_Mot^loc_via_internal_Hom} using our result on $K$-theory of inverse limits of categories \cite[Theorem 6.1]{E25}. 

In Sections \ref{sec:K_homology_of_schemes} and \ref{sec:completed_co_sheaves} we give examples of computations which use Theorem \ref{th:morphisms_in_Mot^loc_via_internal_Hom}.

We mention another application of our rigidity theorem, namely the construction of refined negative cyclic homology, which was announced in \cite{E24b}. It was computed in certain interesting examples by Meyer and Wagner in \cite{MW24}, closely related with the $q$-de Rham cohomology \cite{Sch17, Sch25, Wag25a, Wag25b}. We give a brief sketch of the construction here, and we refer to Section \ref{sec:refined_HC^-} for details. 

For simplicity we consider the case when the base ring $\mk$ is a $\Q$-algebra, concentrated in degree $0.$ We denote by $\Cat_{\mk}^{\perf}$ the category of $\mk$-linear idempotent-complete stable categories (equivalently, dg categories), and by $\Mot^{\loc}_{\mk}$ the corresponding category of localizing motives. The classical negative cyclic homology \cite{Con94, FT83, Tsy83} can be considered as a functor
\begin{equation*}
\HC^-(-/\mk):\Cat_{\mk}^{\perf}\to\Mod_u^{\wedge}\hy\mk[[u]],
\end{equation*}
where $u$ is a variable of cohomological degree $2.$ Here the target is the symmetric monoidal category of $u$-complete modules. We tacitly identify the $\bE_{\infty}$-algebra $\mk[[u]]$ with the $\bE_{\infty}$-algebra of singular cochains on $\CP^{\infty}=BS^1,$ which is possible since $\mk$ is a $\Q$-algebra.

As defined, $\HC^-$ is a finitary localizing invariant (it commutes with filtered colimits since we defined the target to be the category of $u$-complete modules). Hence, it induces a symmetric monoidal functor from $\Mot^{\loc}_{\mk}$ to $\Mod_u^{\wedge}\hy\mk[[u]],$ which we denote by the same symbol. Here the source is rigid by our Theorem \ref{th:dualizability_and_rigidity}, but the target is only locally rigid: it is compactly generated, compact objects are dualizable, but the unit object $\mk[[u]]$ is not compact. As in \cite{E25}, we denote by $\Nuc(\mk[[u]])$ the rigidification of $\Mod_u^{\wedge}\hy\mk[[u]].$ More precisely, in loc. cit. this was one of the definitions of our version of the category of nuclear modules for usual formal schemes, but the same principles apply in this situation. By the universal property of the rigidification, we obtain a unique factorization in the category of presentable stable symmetric monoidal categories
\begin{equation*}
\begin{tikzcd}
\Mot^{\loc}_{\mk} \ar{r}{\HC^{-,\tref}(-/\mk)}\ar[swap]{rd}{\HC^-(-/\mk)} & [2em] \Nuc(\mk[[u]])\ar[d]\\
& \Mod_u^{\wedge}\hy\mk[[u]].
\end{tikzcd}
\end{equation*}
The horizontal arrow here is our refined negative cyclic homology. We give examples of computations in Section \ref{sec:refined_HC^-}. We also define the closely related refined periodic cyclic homology and give a non-trivial example of a computation.

Here we mention an example when in a seemingly innocuous situation the answer is given by a quite non-trivial algebra of overconvergent functions:
\begin{equation*}
\HC^{-,\tref}(\Q[x^{\pm 1}]/\Q[x])\cong \cO(\bigcap\limits_{n\geq 1}\{|u|\leq |x|^n\ne 0\})\in\Nuc(\Q[x][[u]]).
\end{equation*}
The precise meaning of this (idempotent) algebra object is given in Proposition \ref{prop:refined_HC^-_Laurent_polynomials}, see also the discussion after its formulation.

We give a brief overview of the structure of the paper.

Section \ref{sec:preliminaries} contains preliminary material on rigid and dualizable categories. This is mostly a compressed version of \cite[Section 1]{E25}, and one of the goals here is to fix the notation and terminology. It seems that the only new statement in this section is Proposition \ref{prop:prod_dual_vs_Cat_cg} \ref{prod_dual_is_not_compactly_generated}. It gives an example showing that for a general (presentable stable) rigid symmetric monoidal category $\cE$ the class of relatively compactly generated $\cE$-modules is not closed under infinite products in $\Cat_{\cE}^{\dual}.$

In Section \ref{sec:duality_between_cats_of_1_morphisms} we consider an abstract category $\cA$ enriched over $\Pr^L_{\st}$ and give a certain sufficient condition for duality between the categories $\cA(X,Y)$ and $\cA(Y,X)$ for a pair of objects $X,Y.$ It is conceptually similar to the criterion of rigidity via trace-class maps mentioned above. We use this statement to prove part of Theorem \ref{th:dualizability_and_rigidity} on the duality between $\Mot^{\loc}_{\cE}$ and $\Mot^{\loc}_{\cE^{mop}}$ when $\cE$ is rigid $\bE_1$-monoidal and $\cE^{mop}$ is the same category with the opposite multiplication. 

In Section \ref{sec:dualizability_and_rigidity} we prove one of our main results, namely the dualizability and rigidity of the categories of localizing motives (Theorem \ref{th:dualizability_and_rigidity}). The main auxiliary notion is that of a nuclear relatively compactly generated left $\cE$-module, where $\cE$ is rigid $\bE_1$-monoidal. We crucially use our previous result \cite[Theorem 1.14]{E25} which in particular states that the category $\Cat_{\cE}^{\cg}$ of relatively compactly generated $\cE$-modules is compactly assembled. We prove certain non-trivial inheritance properties for nuclear $\cE$-modules which allow to construct ``left resolutions of length $1$'' by nuclear $\cE$-modules. The existence of such resolutions is the main ingredient for the proof of Theorem \ref{th:dualizability_and_rigidity}. We also show that proper relatively compactly generated $\cE$-modules are nuclear, which is slightly more subtle in the case when the base category $\cE$ is not compactly generated.

In section \ref{sec:morphisms_in_Mot^loc} we formulate and prove results on morphisms and internal $\Hom$ in the categories of localizing motives. A relatively simple statement is Theorem \ref{th:morphisms_in_Mot^loc_via_limits} which deals with the case when the source is a motive of a nuclear $\cE$-module. The morphisms resp. internal $\Hom$ are described as a (codirected) inverse limit of $K$-theory resp. motives of certain categories. This gives a powerful tool for more general computations because of the existence of the aforementioned resolutions by nuclear $\cE$-modules. We also prove a much more difficult Theorem \ref{th:morphisms_in_Mot^loc_via_internal_Hom} which loosely speaking describes the situations when the internal $\Hom$ in $\Cat_{\cE}^{\dual}$ categorifies the internal $\Hom$ in $\Mot_{\cE}^{\loc}.$ This includes the situation when the target is a Calkin category and the source is $\omega_1$-compact, which is part \ref{internal_Hom_into_Calk} of the theorem. Part \ref{internal_Hom_from_proper} deals with the case when the source is proper and $\omega_1$-compact and the target is arbitrary. To prove Theorem \ref{th:morphisms_in_Mot^loc_via_internal_Hom} we develop the theory of formally $\omega_1$-injective dualizable categories. We prove that the Calkin categories are formally $\omega_1$-injective using a very tricky statement on the permutations of products and filtered colimits \cite[Corollary A.4]{E25}. We also study the dualizable internal $\Hom$ into a formally $\omega_1$-injective $\cE$-module, which turns out to be very nicely behaved. This circle of ideas leads to the proof of Theorem \ref{th:morphisms_in_Mot^loc_via_internal_Hom}. As an application, in Subsection \ref{ssec:equivalence_Mot^loc_omega_1_and_Mot^loc} we obtain an equivalence between finitary and $\omega_1$-finitary localizing motives over a rigid base, generalizing a result of Ramzi-Sosnilo-Winges \cite{RSW25}.

In Section \ref{sec:Mot^loc_E_as_functor_of_E} we study the assignment $\cE\mapsto\Mot^{\loc}_{\cE}$ as a functor. Most of the results here are essentially straightforward, except for Theorem \ref{th:localizing_motives_over_rigidification} on the category $\Mot^{\loc}_{\cE^{\rig}},$ where $\cE^{\rig}$ is a rigidification of a locally rigid symmetric monoidal category $\cE$ whose unit object is $\omega_1$-compact. We prove the fully faithfulness of the natural functor 
$\Mot^{\loc}_{\cE^{\rig}}\to(\Mot^{\loc}_{\cE})^{\rig}.$ It allows to reduce the study of localizing motives over complicated categories like $\Nuc(\Z_p)$ to the case when the base is much simpler, in this case it would be the standard $p$-complete derived category $D_{p\hy\compl}(\Z).$

In Section \ref{sec:G_equivariant_motives} we apply Theorem \ref{th:dualizability_and_rigidity} to deduce the dualizability and self-duality (but not rigidity) for $G$-equivariant localizing motives $\Mot^{\loc}_{BG}.$ We deduce some applications. In particular, we show that the $K$-theoretic Farrell-Jones conjecture for a discrete group $G$ is equivalent to the following: the category $\Mot^{\loc}_{BG}$ is generated via colimits by the images of inductions $\Mot^{\loc}_{BH}\to\Mot^{\loc}_{BG},$ where $H$ runs through virtually cyclic subgroups of $G.$

In Section \ref{sec:motives_not_generated} we prove a negative result: the category $\Mot^{\loc}$ is not generated (as a localizing subcategory) by motives of connective $\bE_1$-rings. In particular, we prove that for the algebra $\Q[\veps]$ of dual numbers the object $\cU_{\loc}(D^b_{\coh}(\Q[\veps]))$ cannot be generated by the objects $\cU_{\loc}(R),$ where $R$ is connective.

In Section \ref{sec:A^1_invariant_motives} we consider the $\A^1$-invariant localizing motives $\Mot^{\loc,\A^1}$ and prove that the localization $\Mot^{\loc}\to\Mot^{\loc,\A^1}$ is smashing. It is important to note that we consider the flat affine line instead of a smooth affine line over $\bS.$ The corresponding idempotent $\bE_{\infty}$-algebra in $\Mot^{\loc}$ is given by the geometric realization $|\cU_{\loc}(\Delta^{\bullet})|,$ where the algebraic simplices $\Delta^n$ are flat affine spaces. The result automatically applies to an arbitrary rigid base category.

In Section \ref{sec:corepresentability_TR_TC} we prove the results on the corepresentability of $\TR$ (topological restriction) and $\TC$ (topological cyclic homology) mentioned above.

In Section \ref{sec:K_homology_of_schemes} we apply our methods to compute the $K$-homology for certain classes of schemes. More precisely, we work over a noetherian ring $\mk,$ and for a separated scheme of finite type $X$ over $\mk$ we consider the spectrum $\Hom_{\Mot^{\loc}_{\mk}}(\cU_{\loc}(X),\cU_{\loc}(\mk)).$ In the case when $X$ is proper and $\mk$ is regular, we prove that this $K$-homology (relative to $\mk$) is identified with $G$-theory of $X,$ i.e. $K$-theory of the category of coherent sheaves (Theorem \ref{th:K_homology_proper}). We also prove a closely related result on $K$-homology of proper connective (associative) dg algebras over $\mk$ (Theorem \ref{th:K_homology_proper_connective_algebras}). On the other hand, if $X$ is smooth with a smooth compactification over $\mk,$ we prove that its $K$-homology is identified with the so-called ``$K$-theory with proper supports'', the precise statement is Theorem \ref{th:K_homology_smooth_with_compactification}. We also give a sketch of the proof of this result when there is no smooth compactification, the details will appear in \cite{E}.

In Section \ref{sec:completed_co_sheaves} we obtain another application, namely the computation of continuous $K$-theory of certain versions of categories of sheaves on locally compact Hausdorff spaces. Recall that in the case of constant coefficients our result \cite[Theorem 6.11]{E24} we have $K^{\cont}(\Shv(X;\cC))\cong \Gamma_c(X;K^{\cont}(\cC))$ for a dualizable category $\cC.$ In Subsections \ref{ssec:completed_cosheaves} and \ref{ssec:completed_sheaves} we apply Theorems \ref{th:morphisms_in_Mot^loc_via_limits} and \ref{th:morphisms_in_Mot^loc_via_internal_Hom} to obtain similar results for the category of the so-called completed cosheaves and another category which can be called the ``category of completed sheaves'' (this has nothing to do with hypercompletion or left completion with respect to a $t$-structure). Instead of compactly supported cohomology, the answers are given by the Borel-Moore homology and by the cohomology respectively.  
 
Finally, in Section \ref{sec:refined_HC^-} we define the refined negative cyclic homology mentioned above and also the refined periodic cyclic homology. We give examples of computations, namely we compute $\HC^{-,\tref}(\Q[x]/\Q),$ $\HC^{-,\tref}(\Q[x,x^{-1}]/\Q[x])$ and $\HP^{\tref}(\Q/\Q[x]).$ The latter computation is quite complicated, and the answer is again given by a certain algebra of over convergent functions. As an application, we prove that the categories $\Mot^{\loc}_{\Q[x]}$ and $\Mot^{\loc,\A^1}_{\Q[x]}$ are not compactly generated, and the functor $\cU_{\loc}^{\A^1}:\Cat_{\Q[x]}^{\perf}\to\Mot^{\loc,\A^1}_{\Q[x]}$ is not a truncating invariant in the sense of \cite{LT19}.  

{\noindent{\bf Acknowledgements}} I am especially grateful to Peter Scholze for many stimulating discussions about localizing motives and refined variants of Hochschild homology. I am also grateful to Ko Aoki, Alexander Beilinson, Dustin Clausen, Adriano C\'ordova Fedeli, Vladimir Drinfeld, Dennis Gaitsgory, Dmitry Kaledin, David Kazhdan, Akhil Mathew, Thomas Nikolaus, Maxime Ramzi, Vladimir Sosnilo and Georg Tamme for useful discussions. Part of this work was done while I was a visitor in the Max Planck Institute for Mathematics in Bonn from December 2022 till February 2024, and I am grateful to the institute for their hospitality and support. I was partially supported by the HSE University Basic Research Program.
 
\section{Preliminaries}
\label{sec:preliminaries}

This section contains preliminary material on rigid monoidal categories and dualizable modules over them. It is mostly covered by \cite[Section 1]{E25} but with some exceptions, including Subsections \ref{ssec:smashing_ideals}, \ref{ssec:right_left_trace_class}, \ref{ssec:infinite_products} and \ref{ssec:deformed_tensor_algebra} below. Probably the only new result in this section is Proposition \ref{prop:prod_dual_vs_Cat_cg} \ref{prod_dual_is_not_compactly_generated}.

\subsection{Notation and terminology}
\label{ssec:notation_and_terminology}

We will freely use the theory of $\infty$-categories and higher algebra, as developed in \cite{Lur09, Lur17}. We will mostly deal with $(\infty,1)$-categories, however we will also consider certain categories enriched over $\Pr^L;$ these can be always assumed to have finitely many objects, so we do not need any version of presentable $(\infty,2)$-categories. 

We will simply say ``category'' instead of ``$\infty$-category'' if the meaning is clear from the context. We denote by $\cS$ the $\infty$-category of spaces, which is freely generated by one object via colimits. We denote by $\Sp$ the category of spectra.

We will use the same notation as in \cite{E24, E25} for various categories of stable categories. In particular, we denote by $\Cat^{\perf}$ the category of idempotent-complete stable categories and exact functors between them. We denote by $\Pr^L_{\st}$ the category of presentable stable categories and continuous (i.e. colimit-preserving) functors between them. We will say that a continuous functor $F:\cC\to\cD$ is strongly continuous if its right adjoint commutes with colimits. Given $\cC,\cD\in\Pr^L_{\st},$ we denote by $\Fun^L(\cC,\cD)$ resp. $\Fun^{LL}(\cC,\cD)$ the category of continuous resp. strongly continuous functors. For a regular cardinal $\kappa,$ we denote by $\Pr^L_{\st,\kappa}\subset \Pr^L_{\st}$ the non-full subcategory of $\kappa$-presentable stable categories, with $1$-morphisms being continuous functors which preserve $\kappa$-compact objects. The assignment $\cC\mapsto \cC^{\kappa}$ defines an equivalence $\Pr^L_{\st,\kappa}\xto{\sim}\Cat^{\kappa\hy\rex}_{\st},$ where $\Cat^{\kappa\hy\rex}_{\st}$ is the category of small stable categories with $\kappa$-small colimits, and exact functors which commute with $\kappa$-small colimits (in the case $\kappa=\omega$ we also require the idempotent-completeness).

For a presentable stable category $\cC,$ a full subcategory $\cD\subset\cC$ is called {\it localizing} if $\cD$ is stable, closed under colimits and is generated via colimits by a small set of objects (the latter condition exactly means that $\cD$ is itself presentable). 

We will say that an exact functor $F:\cA\to\cB$ between idempotent-complete small stable categories is a homological epimorphism if the functor $\Ind(F):\Ind(\cA)\to\Ind(\cB)$ is a quotient functor, or equivalently if its right adjoint is fully faithful. Equivalently, this means that $F$ is an epimorphism in $\Cat^{\perf}.$  

For a functor $F:\cC\to\cD$ we will denote by $F^R:\cD\to\cC$ the right adjoint functor, if it exists. Similarly, we denote by $F^L:\cD\to\cC$ the left adjoint functor.

We denote by $\Cat_{\st}^{\dual}\subset \Pr^L_{\st}$ the (non-full) subcategory of dualizable categories and strongly continuous functors between them. Recall that the ind-completion functor $\Ind(-):\Cat^{\perf}\to \Cat_{\\st}^{\dual}$ is fully faithful. Its essential image is the category $\Cat^{\cg}_{\st}\simeq \Pr^L_{\st,\omega}\subset\Cat_{\st}^{\dual}$ of compactly generated categories. We will consider the duality functor $(-)^{\vee}:\Cat_{\st}^{\dual}\to\Cat_{\st}^{\dual}$ as a (symmetric monoidal) covariant involution. It sends $\cC$ to the dual category $\cC^{\vee},$ and a strongly continuous functor $F:\cC\to\cD$ is sent to $F^{\vee,L}\cong (F^R)^{\vee}.$ 

Recall that any dualizable category is $\omega_1$-presentable \cite[Corollary 1.21]{E24}. For $\cC\in\Cat_{\st}^{\dual}$ we will denote by $\hat{\cY}_{\cC}:\cC\to\Ind(\cC^{\omega_1})$ (or simply $\hat{\cY}$) the left adjoint to the colimit functor. The same applies to more general compactly assembled categories (not necessarily stable). We will sometimes denote by the same symbol the similar functor $\cC\to\Ind(\cC^{\kappa})$ for a possibly larger uncountable regular cardinal $\kappa,$ as well as the functor $\cC\to\Ind(\cC).$ For an uncountable regular cardinal $\kappa$ and a dualizable category $\cC$ we will use the notation $\Calk_{\kappa}(\cC)=\Ind(\Calk_{\kappa}^{\cont}(\cC)),$ where the small category $\Calk_{\kappa}^{\cont}(\cC)$  is introduced in \cite[Section 1.11]{E24}. Namely, we have the short exact sequence
\begin{equation*}
0\to\cC\xto{\hat{\cY}}\Ind(\cC^{\kappa})\to\Calk_{\kappa}(\cC)\to 0,
\end{equation*}
which defines the Calkin categories. Recall that for $\cA\in\Cat^{\perf}$ we also us the notation $\Calk_{\kappa}(\cA)=(\Ind(\cA)^{\kappa}/\cA)^{\Kar}$ (where $\Kar$ stands for Karoubi completion). We have
\begin{equation*}
\Calk_{\kappa}^{\cont}(\Ind(\cA))\simeq \Calk_{\kappa}(\cA),\quad \Calk_{\kappa}(\Ind(\cA))\simeq\Ind(\Calk_{\kappa}(\cA)).
\end{equation*}

In a compactly assembled category $\cC$ a morphism $f:x\to y$ is called compact if in $\Ind(\cC)$ the morphism $\cY(f):\cY(x)\to \cY(y)$ factors through $\hat{\cY}(y).$ If $\cC$ is compactly generated, this means that $f$ factors through a compact object of $\cC.$

We recall from \cite[Section 1.7]{E25} that for an accessible exact functor $F:\cC\to\cD$ between presentable stable categories we denote by $F^{\cont}:\cC\to\cD$ its ``continuous approximation''. It is defined by the universal property
\begin{equation*}
\Map_{\Fun^L(\cC,\cD)}(G,F^{\cont})=\Map_{\Fun^{\acc,\ex}(\cC,\cD)}(G,F),\quad G\in\Fun^L(\cC,\cD). 
\end{equation*}
The same applies to lax linear functors between modules over rigid monoidal categories, as explained in loc. cit. If $F$ is itself continuous, i.e. $F^{\cont}=F,$ then we write $F^{R,\cont}=(F^R)^{\cont}$ for the continuous approximation of its right adjoint. If $\cC$ is dualizable, then $F^{\cont}$ is given by the composition
\begin{equation*}
\cC\xto{\hat{\cY}}\Ind(\cC)\xto{\Ind(F)}\Ind(\cD)\xto{\colim}\cD.
\end{equation*}
In particular, if $\cC$ is compactly generated, then $F^{\cont}$ is the left Kan extension of $F_{\mid \cC^{\omega}}.$

For a dualizable category $\cC$ and an object $x\in\cC$ we denote by $x^{\vee}\in\cC^{\vee}$ the object with the universal property
\begin{equation*}
\Map_{\cC^{\vee}}(y,x^{\vee})\simeq \Map_{\cC^{\vee}\otimes\cC}(y\boxtimes x,\coev_{\cC}(\bS)).
\end{equation*}
If we denote by $x\otimes-:\Sp\to\cC$ the continuous functor sending $\bS$ to $x$ then we have an isomorphism of functors
\begin{equation*}
(x\otimes-)^{R,\cont}\cong \ev_{\cC}(-,x^{\vee}).
\end{equation*}
This gives an equivalent definition of $x^{\vee}.$

We will use the following convention: a functor $p:I\to J$ between small $\infty$-categories is cofinal if for any $j\in J$ the category $I_{j/}=I\times_J J_{j/}$ is weakly contractible. Equivalently, for any $\infty$-category $\cC$ and for any functor $F:J\to\cC$ we have
\begin{equation*}
	\indlim(J\xto{F}\cC)\cong \indlim(I\xto{p}J\xto{F}\cC),
\end{equation*}
assuming that one of the colimits exists.

We will mostly use directed posets instead of general filtered $\infty$-categories. Recall that by \cite[Proposition 5.3.1.18]{Lur09} for any $\kappa$-filtered $\infty$-category $I$ there exists a $\kappa$-directed poset $J$ with a cofinal functor $J\to I.$ For an $\infty$-category $\cA$ we will typically write the objects of $\Ind(\cA)$ resp. $\Pro(\cA)$ as $\inddlim[i\in I]x_i$ resp. $\proolim[i\in I]x_i,$ where $I$ is directed resp. codirected, and we have a functor $I\to\cA,$ $i\mapsto x_i.$

By default, all the functors between stable categories will be assumed to be exact. A stable category $\cA$ is enriched over $\Sp,$ and we will denote by $\Hom_{\cA}(x,y)$ or $\cA(x,y)$ the spectrum of morphisms from $x$ to $y.$ We will frequently use the identification $\Ind(\cA)\simeq\Fun(\cA^{op},\Sp)$ in the case when $\cA$ is small stable. We denote the evaluation functor for $\Ind(\cA)$ by
\begin{equation*}
	-\tens{\cA}-:\Ind(\cA)\otimes\Ind(\cA^{op})\to\Sp,
\end{equation*}
so we have
\begin{equation*}
	(\inddlim[i]x_i)\tens{\cA}(\inddlim[j]y_j^{op})\cong \indlim[i,j]\cA(y_j,x_i).
\end{equation*}
Here $y_j\in\cA$ and we denote by $y_j^{op}$ the corresponding object of $\cA^{op}.$

If $\cE$ is an $\bE_1$-monoidal category, we have in general two different $\bE_1$-monoidal categories $\cE^{op}$ and $\cE^{mop}.$ Namely, $\cE^{op}$ has the opposite underlying category, but the order of multiplication is the same. On the other hand, $\cE^{mop}$ has the same underlying category, but the order of multiplication is reversed. In particular, left $\cE$-modules are the same as right $\cE^{mop}$-modules.

Given a left $\cE$-module $\cC$ and two objects $x,y\in\cC,$ we denote by $\un{\Hom}_{\cC/\cE}(x,y)\in\cE$ the relative internal $\Hom$-object. If it exists, it is uniquely determined by the universal property
\begin{equation*}
	\Map_{\cE}(z,\un{\Hom}_{\cC/\cE}(x,y))\simeq \Map_{\cC}(z\otimes x,y).
\end{equation*}
If $\cD$ is a right $\cE$-modulle and $d,d'\in\cD$ then we write $\un{\Hom}_{\cD/\cE}^r(d,d')$ instead of $\un{\Hom}_{\cD/\cE^{mop}}(d,d').$

For an $\bE_1$-ring $A$ we denote by $\Mod\hy A$ the category of right $A$-modules. If $A$ is connective, we will sometimes write $D(A)$ instead of $\Mod\hy A,$ and if $A$ is an $\bE_{\infty}$-ring, we will also write $\Mod_A.$ 

If $A$ is an algebra in an $\bE_1$-monoidal category $\cE,$ then we also denote by $\Mod\hy A$ the category of right $A$-modules in $\cE.$ This should not lead to confusion, because the corresponding monoidal category will be clear from the context.

We will mostly use homological grading unless we specify otherwise.

Finally, we recall certain Grothendieck's axioms for abelian categories \cite{Gro} which make sense for a presentable $\infty$-category $\cC.$ First, we say that $\cC$ satisfies strong (AB5) if filtered colimits in $\cC$ commute with finite limits. We say that $\cC$ satisfies (AB6) if for any set $I,$ for any collection of directed posets $(J_i)_{i\in I}$ and for any family of functors $J_i\to \cC,$ $j_i\mapsto x_{j_i},$ the following natural map is an isomorphism:
\begin{equation*}
	\indlim[(j_i)_i\in\prod\limits_{i\in I}J_i] \prodd[i]x_{j_i}\xto{\sim} \prodd[i]\indlim[j_i\in J_i]x_{j_i}.
\end{equation*}  
If $\cC$ is stable, then it automatically satisfies strong (AB5), and the axiom (AB6) is equivalent to the dualizability of $\cC$ \cite[Proposition 1.53]{E24}. In general, $\cC$ is compactly assembled if and only if it satisfies both strong (AB5) and (AB6).

\subsection{Rigid monoidal categories}
\label{ssec:rigid_monoidal}

We recall the notion of rigidity for large monoidal categories from \cite[Definition 9.1.2]{GaRo17}. An $\bE_1$-monoidal presentable stable category $\cE\in\Alg_{\bE_1}(\Pr^L_{\st})$ is called rigid if the following conditions hold:
\begin{itemize}
\item The unit object $1_{\cE}$ is compact;
\item The multiplication functor $\mult:\cE\otimes\cE\to \cE$ has a colimit-preserving right adjoint, which is $\cE\hy\cE$-linear. 
\end{itemize} 
For brevity we will simply say ``rigid $\bE_1$-monoidal category'' instead of ``rigid presentable stable $\bE_1$-monoidal category'', and similarly for symmetric monoidal categories and for $\bE_n$-monoidal categories for $n\geq 1.$

One can also define the notion of a rigid $\bE_0$-monoidal category simply as a dualizable category with a distinguished compact object. However, we choose not to use such terminology to avoid confusion.

We will freely use the basic facts about rigid monoidal categories which are collected in \cite[Section 1.2]{E25} (see references therein). In particular, if $\cE$ is compactly generated, then rigidity means that $\cE$ is equivalent as a monoidal category to $\Ind(\cA),$ where $\cA$ is a small stable monoidal category in which every object is (both left and right) dualizable. In general rigidity implies dualizability and self-duality, but there are many examples of non-compactly generated rigid monoidal categories \cite{E25}. If $\cE\to\cE'$ is a continuous monoidal functor between presentable stable $\bE_1$-monoidal categories such that $\cE$ is rigid and the unit $1_{\cE'}$ is compact, then $F$ is strongly continuous. If $F':\cE\to\cE'$ is another such functor, then any morphism $\varphi:F\to F'$ (in the category of monoidal functors) is an isomorphism.

We recall the notions of right and left trace-class morphisms in a presentable stable $\bE_1$-monoidal category $\cE.$ First, for $x,y\in\cE$ the left and right internal $\Hom$ are defined by the following universal properties:
\begin{equation*}
	\Map_{\cE}(z,\un{\Hom}_{\cE}^l(x,y))\simeq\Map_{\cE}(z\otimes x,y),\quad \Map_{\cE}(z,\un{\Hom}_{\cE}^r(x,y))\simeq\Map_{\cE}(x\otimes z,y).
\end{equation*}
For $x\in\cE$ we put $x^{r\vee}=\un{\Hom}_{\cE}^r(x,1_{\cE}),$ $x^{l\vee}=\un{\Hom}_{\cE}^l(x,1_{\cE}).$ A map $x\to y$ in $\cE$ is called right resp. left trace-class if the associated morphism $1_{\cE}\to\un{\Hom}_{\cE}^r(x,y)$ resp. $1_{\cE}\to\un{\Hom}_{\cE}^l(x,y)$ factors through $x^{r\vee}\otimes y$ resp. $y\otimes x^{l\vee}.$ In the symmetric monoidal case such maps are called simply trace-class, and this notion is due to Clausen and Scholze \cite[Definition 13.11]{CS20}.

It is important that in rigid monoidal categories the class of compact morphisms coincides with the class of right trace-class morphisms, and also coincides with the class left trace-class morphisms. We recall the following criterion of rigidity for $\bE_1$-monoidal presentable stable categories which is in \cite{Ram24b} for the symmetric monoidal case, and the proof of the general case is the same.

\begin{prop}\label{prop:rigidity_criterion}\cite[Corollary 4.52]{Ram24b}\cite[Proposition 4.57]{E25}
	Let $\cE$ be a presentable stable $\bE_1$-monoidal category. Then $\cE$ is rigid if and only if the following two conditions hold.
	\begin{enumerate}[label=(\roman*), ref=(\roman*)]
		\item The unit object of $\cE$ is compact.
		\item $\cE$ is generated via colimits by the objects of the form $x=\indlim[n\in\N]x_n,$ where each map $x_n\to x_{n+1}$ is both left and right trace-class. 
	\end{enumerate} 
\end{prop} 

In Section \ref{sec:duality_between_cats_of_1_morphisms} we will prove a closely related analogue of the ``if'' direction, namely a sufficient condition for the duality between $\cA(X,Y)$ and $\cA(Y,X)$ when $\cA$ is enriched over $\Prr^L_{\st}.$ This is Theorem \ref{th:conditions_for_E_0_rigidity} below. We will use it to prove the first part of Theorem \ref{th:dualizability_and_rigidity}, namely the duality between $\Mot^{\loc}_{\cE}$ and $\Mot^{\loc}_{\cE^{mop}}$ when $\cE$ is rigid $\bE_1$-monoidal.

Finally, we recall the notion of a rigidification. We denote by $\Alg_{\bE_1}^{\rig}(\Pr^L_{\st})\subset \Alg_{\bE_1}(\Pr^L_{\st})$ the full subcategory of rigid monoidal categories, and similarly for $\CAlg^{\rig}(\Pr^L_{\st})\subset \CAlg(\Pr^L_{\st}).$ Both inclusions have right adjoints, but we will be only interested in the latter one. It is called the rigidification and denoted by $(-)^{\rig}:\CAlg(\Pr^L_{\st})\to \CAlg^{\rig}(\Pr^L_{\st}).$ By \cite[Construction 4.75, Theorem 4.77]{Ram24b}, if $\cC\in\CAlg(\Pr^L_{\st,\kappa})$ for some regular cardinal $\kappa,$ then $\cC^{\rig}\subset\Ind(\cC^{\kappa})$ is the full localizing subcategory generated by the objects of the form $\inddlim[\Q_{\leq}](\Phi:\Q_{\leq}\to\cC^{\kappa}),$ where for $a<b$ the map $\Phi(a)\to \Phi(b)$ is trace-class in $\cC.$

Suppose that $\kappa$ is uncountable, $\cE\in\CAlg^{\rig}(\Pr^L_{\st}),$ $\cC\in\CAlg(\Pr^L_{\st,\kappa}),$ $F:\cE\to\cC$ is a continuous symmetric monoidal functor and denote by $\wt{F}:\cE\to\cC^{\rig}$ the (automatically strongly continuous) symmetric monoidal functor corresponding to $F$ by adjunction. Then we have a commutative diagram in $\CAlg(\Cat_{\st}^{\dual}):$ 
\begin{equation}\label{eq:how_rigidification_works}
\begin{tikzcd}
\cE\ar[r, "\wt{F}"]\ar{d}{\hat{\cY}} & \cC^{\rig}\ar[d]\\
 \Ind(\cE^{\kappa}) \ar{r}{\Ind(F^{\kappa})} & \Ind(\cC^{\kappa}). 
\end{tikzcd}
\end{equation}

\subsection{Smashing ideals and idempotent $\bE_{\infty}$-algebras}
\label{ssec:smashing_ideals}

Let $\cE\in \CAlg(\Prr^L_{\st})$ be a presentable stable symmetric monoidal category, not necessarily rigid. Recall from \cite[Definition 4.8.2.8]{Lur17} that an $\bE_{\infty}$-algebra $A\in\CAlg(\cE)$ is called idempotent if the multiplication morphism $A\otimes A\to A$ is an isomorphism. We denote by $\CAlg^{\idem}(\cE)\subset\CAlg(\cE)$ the full subcategory of idempotent $\bE_{\infty}$-algebras. By \cite[Proposition 4.8.2.8]{Lur17} The forgetful functor
\begin{equation*}
\CAlg^{\idem}(\cE)\to \Alg_{\bE_0}(\cC)\simeq\cC_{1_{\cE}/}
\end{equation*}
is fully faithful, and its essential image consists of pairs $(A,e:1_{\cE}\to A)$ such that the map $\id_A\otimes e:A\to A\otimes A$ is an isomorphism.

Recall that a localizing ideal $\cI\subset\cE$ is called {\it smashing} if the inclusion functor has a continuous right adjoint, which is also $\cE$-linear, i.e. the projection formula holds. We denote by $\Idl_{\sm}(\cE)$ the poset of smashing ideals in $\cE.$ It is well known that we have an equivalence 
\begin{equation*}
\CAlg^{\idem}(\cC)\xto{\sim}\Idl_{\sm}(\cE),\quad A\mapsto \ker(A\otimes-),
\end{equation*}
see for example \cite[Propositions 2.5 and 2.17]{Aok23}. In particular, the category $\CAlg^{\idem}(\cE)$ is also a poset. Note that it is small: if $\kappa$ is an uncountable regular cardinal such that $\cE$ is $\kappa$-presentable, then any smashing ideal in $\cE$ is generated as a localizing subcategory by a set of $\kappa$-compact objects.

Moreover, the category $\CAlg^{\idem}(\cE)$ is equivalent to the category of coidempotent $\bE_{\infty}$-coalgebras, which is also a poset \cite[Proposition 2.14]{Aok23}. Under this equivalence the algebra $A$ corresponds to the coalgebra $C=\Fiber(1_{\cE}\to A).$ Clearly, $C$ generates the corresponding smashing ideal $\cI.$ Moreover, $\cI$ itself is naturally symmetric monoidal, since it is equivalent to the quotient of $\cE$ by the localizing ideal generated by $A.$ The unit object of $\cI$ is given by $C.$

In the special case when $\cE$ is rigid, a localizing ideal is smashing if and only if the inclusion functor is strongly continuous; the projection formula holds automatically.

\subsection{Locally rigid categories}
\label{ssec:locally_rigid}

We recall the notion of a locally rigid category, which was originally defined in \cite{AGKRV20} as a semi-rigid category. Namely, a symmetric monoidal presentable stable category $\cE\in\CAlg(\Pr^L_{\st})$ is called locally rigid if $\cE$ is dualizable and the multiplication functor $\mult:\cE\otimes\cE\to\cE$ has a colimit-preserving right adjoint $\mult^R$ which is $\cE\hy\cE$-linear. We refer to \cite{Ram24b} for a detailed account, and to \cite[Section 1.6]{E25} for the recollection of some basic facts.

Note that a rigid symmetric monoidal category is in particular locally rigid. On the other hand, for a locally rigid category $\cE$ the rigidity is equivalent to the compactness of the unit object $1_{\cE}.$ If $\cE\in\CAlg(\Pr^L_{\st})$ is compactly generated, then $\cE$ is locally rigid if and only if every compact object of $\cE$ is dualizable.

We refer to \cite{Ram24b, E25} for many examples. We note that if $\cC$ is rigid symmetric monoidal, then any smashing ideal $\cE\subset\cC$ is naturally a locally rigid category. In fact, any locally rigid category $\cE$ can be obtained in this way. The most ``economic'' choice is the so-called one-point rigidification $\cE_+\supset\cE$ as defined in \cite[Definition 1.39]{E25}. Namely, $\cE_+\subset\Ind(\cE)$ is the localizing subcategory generated by $\hat{\cY}(\cE)$ and $\cY(1_{\cE}).$ By \cite[Proposition 1.40]{E25}, $\cE_+$ is indeed rigid and the full subcategory $\hat{\cY}(\cE)\subset\cE_+$ is a smashing ideal. 

If $\cE$ is a locally rigid category such that the unit-object $1_{\cE}$ is $\omega_1$-compact, then the category $\cE^{\rig}$ has very nice properties listed in \cite[Theorem 4.2]{E25}. In particular, if we denote by $\Lambda_{\cE}:\cE\to\cE^{\rig}$ the right adjoint to the natural functor, then the result in loc. cit. in particular states that $\Lambda_{\cE}$ is fully faithful and symmetric monoidal (a priori it is only lax monoidal). The fully faithfulness is easy to see (and it does not require the unit of $\cE$ to be $\omega_1$-compact), but the monoidality is difficult. The latter is closely related with another difficult statement which is also part of the theorem in loc. cit.: any trace-class map in $\cE$ is a composition of two trace-class maps. We refer to loc. cit. for details.  

\subsection{Dualizable modules over rigid monoidal categories}
\label{ssec:dualizable_modules}

Let $\cE$ be a rigid $\bE_1$-monoidal category. We will use the same notation and terminology as in \cite[Section 1.3]{E25} for various categories of left $\cE$-modules. In particular we denote by $\Cat_{\cE}^{\dual}$ the category of dualizable left $\cE$-modules. Recall that it is equivalent to the category of left $\cE$-modules in $\Cat_{\st}^{\dual}.$ We also denote by $\Cat_{\cE}^{\cg}\subset \Cat_{\cE}^{\dual}$ the full subcategory of relatively compactly generated left $\cE$-modules. We denote by $\Pr^L_{\cE}$ the category of presentable stable left $\cE$-modules. For $\cC,\cD\in\Pr^L_{\cE}$ we denote by $\Fun_{\cE}^L(\cC,\cD)$ resp. $\Fun_{\cE}^{LL}(\cC,\cD)$ the category of $\cE$-linear continuous resp. strongly continuous functors. 

We refer to loc.cit. and references therein for the recollection of some basic facts about $\cE$-modules. In particular, we will tacitly use the following statement: a lax $\cE$-linear continuous functor between presentable stable left $\cE$-modules is automatically $\cE$-linear. This implies that for $\cC\in\Pr^L_{\cE},$ an object $x\in\cC$ is compact if and only if it is relatively compact over $\cE.$

If $\cE$ is compactly generated, then we denote by $\Cat_{\cE^{\omega}}^{\perf}$ the category of left $\cE^{\omega}$-modules in $\Cat^{\perf}.$ In this case we have an equivalence $\Ind(-):\Cat_{\cE^{\omega}}^{\perf}\xto{\sim}\Cat_{\cE}^{\cg}.$ In the special case when $\cE$ is generated by the (compact) unit object, we have $\cE\simeq \Mod\hy\mk$ for an $\bE_2$-ring $\mk,$ and we write
\begin{equation*}
\Prr^L_{\mk}=\Prr^L_{\Mod\hy\mk},\quad \Cat_{\mk}^{\dual}=\Cat_{\Mod\hy\mk}^{\dual},\quad \Cat_{\mk}^{\cg}=\Cat_{\Mod\hy\mk}^{\cg},\quad \Cat_{\mk}^{\perf}=\Cat_{\Perf(\mk)}^{\perf}. 
\end{equation*}

For $\cC\in\Cat_{\cE}^{\dual}$ we use the following notation for the relative evaluation and coevaluation functors:
\begin{equation*}
\ev_{\cC/\cE}:\cC\otimes\cC^{\vee}\to\cE,\quad\coev_{\cC/\cE}:\Sp\to\cC^{\vee}\tens{\cE}\cC.
\end{equation*}
Here $ev_{\cC/\cE}$ is $\cE\hy\cE$-linear. We recall the relation with the absolute evaluation and coevaluation. We put $\Gamma=\Hom(1_{\cE},-):\cE\to\Sp,$ and denote by $F:\cC^{\vee}\tens{\cE}\cC\to\cC^{\vee}\otimes\cC$ the (continuous) right adjoint to the natural functor. Then we have $\ev_{\cC}=\Gamma\circ\ev_{\cC/\cE}$ and $\coev_{\cC}=F\circ\coev_{\cC/\cE}.$

Recall that by \cite[Theorem 1.14]{E25} the category $\Cat_{\cE}^{\cg}$ is compactly assembled. Also by loc. cit. (essentially by \cite[Proposition 1.71]{E24}) the inclusion functor $\Cat_{\cE}^{\cg}\to\Cat_{\cE}^{\dual}$ has a right adjoint which commutes with filtered colimits. These facts will be extremely important for the proof of our result on the rigidity of the category of localizing motives (Theorem \ref{th:dualizability_and_rigidity}). We will also use the $\omega_1$-presentability of the category $\Cat_{\cE}^{\dual},$ see \cite[Theorem A]{Ram24a}, \cite[Theorem D.1]{E24} and \cite[Theorem 1.13]{E25}. A dualizable left $\cE$-module $\cC$ is $\omega_1$-compact in $\Cat_{\cE}^{\dual}$ if and only if both functors $\ev_{\cC/\cE}$ and $\coev_{\cC/\cE}$ preserve $\omega_1$-compact objects.

Recall from \cite[Section 1.9]{E25} that $\cC\in\Cat_{\cE}^{\dual}$ is called smooth over $\cE$ if the relative coevaluation functor $\coev_{\cC/\cE}$ is strongly continuous, i.e. if the object $\coev_{\cC/\cE}(\bS)\in\cC^{\vee}\tens{\cE}\cC$ is compact. Similarly, $\cC$ is called proper over $\cE$ if the relative evaluation functor $\ev_{\cC/\cE}$ is strongly continuous. These notions are generalizations of \cite[Definitions 8.2 and 8.8]{KonSob09}. If $\cE$ is compactly generated, then a small stable left $\cE^{\omega}$-module $\cA\in\Cat_{\cE^{\omega}}^{\perf}$ is called smooth resp. proper if so is $\Ind(\cA)\in\Cat_{\cE}^{\dual}.$

For $\cC\in\Cat_{\cE}^{\dual}$ and $x\in\cE,$ $c\in\cC$ we denote by $c^x\in\cC$ the version of internal $\Hom$ in the following sense: we have $\Map_{\cC}(c',c^x)\simeq \Map(x\otimes c',c)$ for $c'\in\cC.$ It follows from the above that we have
\begin{equation*}
(d^{\vee})^x\cong (d\otimes x)^{\vee},\quad d\in\cC^{\vee},\,x\in\cE.
\end{equation*}

If $\cC\in\Cat_{\cE}^{\cg}$ is generated over $\cE$ by a single compact object $x\in\cC^{\omega},$ then we have an $\cE$-linear equivalence
\begin{equation*}
\cC\simeq\Mod\hy A,\quad A=\un{\End}_{\cC/\cE}(x)\in\Alg_{\bE_1}(\cE).
\end{equation*}
More generally, if $\cC$ is generated over $\cE$ by a small collection of compact objects $\{x_i\in\cC^{\omega}\}_{i\in I},$ then we can consider the following category $\cA$ enriched over $\cE:$ its set of objects is given by $I,$ and the morphisms are given by $\cA(i,j)=\un{\Hom}_{\cC/\cE}(x_i,x_j).$ We denote by $\cA^{op}$ the corresponding category enriched over $\cE^{mop},$ and let $\Fun(\cA^{op},\cE)$ be the category of $\cE^{mop}$-enriched functors, considered as a left $\cE$-module. Then we have an equivalence of left $\cE$-modules
\begin{equation*}
	\Fun(\cA^{op},\cE)\simeq \cC.
\end{equation*} 
We will need this point of view on relatively compactly generated $\cE$-modules for a certain important construction in Proposition \ref{prop:resolution_by_nuclear}.

We will say that an $\bE_1$-algebra $A\in\Alg_{\bE_1}(\cE)$ is smooth resp. proper over $\cE$ if so is the category $\Mod\hy A\in\Cat_{\cE}^{\dual}.$ Note that properness simply means that the underlying object of $A$ is compact in $\cE.$

For $A\in\Alg_{\bE_1}(\cE)$ and $\cC\in\Cat_{\cE}^{\dual}$ we denote by $\Rep(A,\cC)$ the category of left $A$-modules in $\cC,$ and similarly for $\Rep(A,\cC^{\omega}).$ We have the standard equivalences 
\begin{equation*}
\Fun_{\cE}^L(\Mod\hy A,\cC)\xto{\sim}\Rep(A,\cC),\quad \Fun_{\cE}^{LL}(\Mod\hy A,\cC)\xto{\sim}\Rep(A,\cC^{\omega}),
\end{equation*}
sending an object $F$ to $F(A)$ with the natural $A$-module structure.

We recall that for $\cC\in\Cat_{\cE}^{\dual}$ and for an uncountable regular cardinal $\kappa$ the compactly generated categories $\Ind(\cC^{\kappa})$ and $\Calk_{\kappa}(\cC)$ are naturally left $\cE$-modules, and we have a short exact sequence in $\Cat_{\cE}^{\dual}:$
\begin{equation*}
	0\to \cC\xto{\hat{\cY}}\Ind(\cC^{\kappa})\to\Calk_{\kappa}(\cC)\to 0.
\end{equation*}

Finally, recall that for $\cC,\cD$ we denote by $\un{\Hom}_{\cE}^{\dual}(\cC,\cD)$ the relative internal $\Hom$ in $\Cat_{\cE}^{\dual}$ over $\Cat_{\st}^{\dual},$ studied in \cite[Section 3]{E25}. If $\cE$ is symmetric monoidal, then $\un{\Hom}_{\cE}^{\dual}(\cC,\cD)$ is naturally an $\cE$-module, and it is exactly the internal $\Hom$ in the symmetric monoidal category $\Cat_{\cE}^{\dual}.$

\subsection{Localizing invariants of dualizable modules}

Let $\cE$ be a rigid $\bE_1$-monoidal category, and let $\cT$ be an accessible stable category. The notion of an accessible localizing invariant $\Cat_{\cE}^{\dual}\to \cT$ is defined in the same way as in \cite{BGT}, see also \cite[Section 4]{E24}. If $\kappa$ is a regular cardinal, such that $\cT$ has $\kappa$-filtered colimits, then we denote by $\Fun_{\loc,\kappa}(\Cat_{\cE}^{\dual},\cT)$ the category of localizing invariants which commute with $\kappa$-filtered colimits.

We recall the relative version of the universal localizing invariants (depending on the choice of a regular cardinal).

\begin{defi}
	As above, let $\cE$ be a rigid $\bE_1$-monoidal category, and let $\kappa$ be a regular cardinal. We denote by
	\begin{equation*}
		\cU_{\loc,\kappa}:\Cat_{\cE}^{\dual}\to \Mot^{\loc}_{\cE,\kappa}
	\end{equation*}
	the universal $\kappa$-finitary localizing invariant. This means that the category $\Mot^{\loc}_{\cE,\kappa}$ is accessible stable with $\kappa$-filtered colimits, the functor $\cU_{\loc,\kappa}$ commutes with $\kappa$-filtered colimits, and for any other accessible stable category $\cT$ with $\kappa$-filtered colimits we have an equivalence
	\begin{equation}\label{eq:universal_property_Mot^loc_kappa}
		\Fun_{\loc,\kappa}(\Cat_{\cE}^{\dual},\cT)\simeq \Fun^{\kappa\hy\cont}(\Mot^{\loc}_{\cE,\kappa},\cT).
	\end{equation}
	For $\kappa=\omega$ we use the notation $\Mot^{\loc}_{\cE}=\Mot^{\loc}_{\cE,\omega}$ and say ``finitary'' instead of ``$\omega$-finitary''. We also write $\Mot^{\loc}=\Mot^{\loc}_{\Sp},$ $\Mot^{\loc}_{\kappa}=\Mot^{\loc}_{\Sp,\kappa}.$
\end{defi}

Here the target of \eqref{eq:universal_property_Mot^loc_kappa} is the category of exact functors which commute with $\kappa$-filtered colimits. We can also consider the localizing invariants defined on $\Cat_{\cE}^{\cg},$ in particular the universal $\kappa$-finitary localizing invariants. By \cite[Theorem 4.10]{E24} and its straightforward generalization \cite[Theorem 1.19]{E25} there is no difference between these two versions: the composition $\Cat_{\cE}^{\cg}\to\Cat_{\cE}^{\dual}\xto{\cU_{\loc,\kappa}}\Mot^{\loc}_{\cE,\kappa}$ is the universal $\kappa$-finitary localizing invariant of relatively compactly generated left $\cE$-modules.

Recall that for an accessible localizing invariant $F:\Cat_{\cE}^{\cg}\to\cT,$ the corresponding localizing invariant $F^{\cont}:\Cat_{\cE}^{\dual}\to \cT$ is given by
\begin{equation*}
F^{\cont}(\cC)=\Omega F(\Calk_{\omega_1}(\cC)).	
\end{equation*} 

Below we will denote by the same symbol $\cU_{\loc}$ either of the universal finitary localizing invariants $\Cat_{\cE}^{\dual}\to\Mot^{\loc}_{\cE},$ $\Cat_{\cE}^{\cg}\to\Mot^{\loc}.$ In the case when $\cE$ is compactly generated, the latter is identified with $\Cat_{\cE^{\omega}}^{\perf}\to\Mot^{\loc}_{\cE}.$

In the case $\cE=\Mod\hy\mk$ we put $\Mot^{\loc}_{\mk}=\Mot^{\loc}_{\Mod\hy\mk}.$




\subsection{Right and left trace-class $2$-morphisms}
\label{ssec:right_left_trace_class}

We will deal with the following abstract situation. Let $\cA$ be a category enriched over $\Pr^L$ (not necessarily over $\Pr^L_{\st}$). For the purposes of this paper we can restrict to the case when $\cA$ has finitely many objects. In the case of a single object we are simply dealing with a presentable $\bE_1$-monoidal category. However, it is convenient to formulate the definitions and statements when $\cA$ has a possibly large set of objects.

We denote the categories of morphisms by $\cA(X,Y)\in\Pr^L,$ we denote by 
\begin{equation*}
\mu_{X,Y,Z}:\cA(Y,Z)\otimes\cA(X,Y)\to\cA(X,Z)
\end{equation*} 
the (colimit-preserving) composition functors, and we denote by $1_X\in\cA(X,X)$ the unit objects. For brevity we write
\begin{equation*}
F\circ G=\mu_{X,Y,Z}(F\boxtimes G),\quad F\in\cA(Y,Z),\,G\in\cA(X,Y).
\end{equation*}
We sometimes refer to the objects of $\cA(X,Y)$ as $1$-morphisms in $\cA$, and to the $1$-morphisms in $\cA(X,Y)$ as $2$-morphisms in $\cA.$

For a $1$-morphism $F\in\cA(X,Y)$ we denote by $F^{r\vee},F^{l\vee}\in\cA(Y,X)$ the $1$-morphisms defined by the following universal properties:
\begin{equation*}
\Map_{\cA(Y,X)}(G,F^{r\vee})\simeq \Map_{\cA(Y,Y)}(F\circ G,\id_Y),\quad \Map_{\cA(Y,X)}(G,F^{l\vee})\simeq \Map_{\cA(Y,Y)}(G\circ F,\id_X).
\end{equation*}

\begin{defi}
Let $\cA$ be a category enriched over $\Pr^L.$ Let $X,Y\in\Ob(\cA),$ and let $F,G\in\cA(X,Y).$ A $2$-morphism $\varphi:F\to G$ is said to be right resp. left trace-class if there exists a morphism $\psi':1_X\to F^{r\vee}\circ G$ resp. $\psi'':1_Y\to G\circ F^{l\vee}$ such that the composition
\begin{equation*}
F\cong F\circ\id_X\xto{F\circ \psi'}F\circ F^{r\vee}\circ G\to \id_Y\circ G\cong G
\end{equation*}
resp.
\begin{equation*}
	F\cong \id_Y\circ F\xto{\psi''\circ F} G\circ F^{l\vee}\circ F\to G\circ \id_X\cong G
\end{equation*}
is homotopic to $\varphi.$
\end{defi}

\begin{example}
If $\cA$ has a single object $X,$ then $\cA(X,X)\in\Alg_{\bE_1}(\Pr^L)$ is a presentable $\bE_1$-monoidal category, and we get the same notions of right resp. left trace-class maps in $\cA(X,X)$ as in Subsection \ref{ssec:rigid_monoidal}.
\end{example}

\begin{example}
Let $\cE$ be a rigid $\bE_1$-monoidal category, and let $\cA=\mPr^L_{\cE}$ be the category of presentable stable left $\cE$-modules, enriched over $\Pr^L$ in the natural way, i.e. $\cA(\cC,\cD)=\Fun_{\cE}^L(\cC,\cD).$ Then for $F\in\cA(\cC,\cD)$ we have $F^{r\vee}=F^{R,\cont}$ -- the continuous approximation of the (accessible lax $\cE$-linear) right adjoint functor $F^R:\cD\to\cC.$ In this case right trace-class $2$-morphism were defined in \cite[Definition 1.45]{E25}.
\end{example}

\subsection{Infinite products of dualizable modules}
\label{ssec:infinite_products}

Let $\cE$ be a rigid $\bE_1$-monoidal category. As we mentioned in Subsection \ref{ssec:dualizable_modules}, the category $\Cat_{\cE}^{\dual}$ is equivalent to the category of left $\cE$-modules in $\Cat_{\st}^{\dual}.$ In particular, for a family $(\cC_s)_{s\in S}$ of dualizable left $\cE$-modules, their product in $\Cat_{\cE}^{\dual}$ is computed in $\Cat_{\st}^{\dual}.$ We denote it by $\prodd[s]^{\dual}\cC_s,$ following the notation from \cite[Section 1.17]{E24}, where the properties of these products are studied. We have a fully faithful strongly continuous $\cE$-linear functor 
\begin{equation*}\label{eq:from_prod^dual_to_Ind_of_prod}
\prodd[s]^{\dual}\cC_s\to\Ind(\prodd[s]\cC_s^{\omega_1}).
\end{equation*}  
Here the explicit formula for $\cE$-action on the target is the following. For $x\in\cE$ with $\hat{\cY}(x)=\inddlim[i]x_i\in\Ind(\cE^{\omega_1}),$ and an object $c=\inddlim[j](c_{j,s})_{s\in S}\in \Ind(\prodd[s]\cC_s^{\omega_1})$ we have
\begin{equation}\label{eq:E_action}
x\otimes c = \inddlim[i,j] (x_i\otimes c_{j,s})_{s\in S}.
\end{equation}
We identify the category $\prodd[s]^{\dual}\cC_s$ with its essential image in $\Ind(\prodd[s]\cC_s^{\omega_1}).$ This essential image consists of ind-objects $\inddlim[i\in I](c_{i,s})_{s\in S},$ where $I$ is directed, such that for any $i\in I$ there exists $j\geq i$ such that for each $s\in S$ the morphism $c_{i,s}\to c_{j,s}$ is compact in $\cC_s.$

There is a subtle difference between the special case when $\cE$ is compactly generated and the general case, which we address in the following proposition.

\begin{prop}\label{prop:prod_dual_vs_Cat_cg}
\begin{enumerate}[label=(\roman*),ref=(\roman*)]
  \item Let $\cE$ be a compactly generated rigid $\bE_1$-monoidal category. Then the inclusion functor $\Cat_{\cE}^{\cg}\to\Cat_{\cE}^{\dual}$ commutes with infinite products. \label{prod_dual_over_compactly_generated}
  \item Let $\mk$ be a field, and let $\cE=\Shv([0,1],D(\mk))$ be the (rigid) symmetric monoidal category of sheaves on the closed unit interval with values in $D(\mk).$ Then the $\cE$-module $\prodd[\N]^{\dual}\cE$ is not relatively compactly generated over $\cE.$ In particular, the inclusion functor $\Cat_{\cE}^{\cg}\to\Cat_{\cE}^{\dual}$ does not commute with countable products. \label{prod_dual_is_not_compactly_generated}
\end{enumerate}
\end{prop}

\begin{proof}
\ref{prod_dual_over_compactly_generated} If $\cE$ is compactly generated then any $\cC\in\Cat_{\cE}^{\cg}$ is (absolutely) compactly generated by the objects $x\otimes c,$ $x\in\cE^{\omega},$ $c\in \cC^{\omega}.$ Hence, the assertion follows from \cite[Proposition 1.103]{E24}.

\ref{prod_dual_is_not_compactly_generated} We consider the category $\Shv(\{0\},D(\mk))\simeq D(\mk)$ as a quotient of $\cE$ by the smashing ideal $\Shv((0,1],D(\mk)),$ in particular it is a dualizable module over $\cE.$ It suffices to prove that the category
\begin{equation*}
\cA=(\prodd[\N]^{\dual}\cE)\tens{\cE}\Shv(\{0\},D(\mk))
\end{equation*}
is not generated by the image of $(\prodd[\N]^{\dual}\cE)^{\omega}.$ We will directly show that this image has a non-zero right orthogonal in $\cA.$  

For $V\in D(\mk)$ denote by $\un{V}\in\cE$ the constant sheaf on $[0,1]$ with value $V.$ By \cite[Proposition 6.15]{E24} we have an equivalence $\Perf(\mk)\xto{\sim} \cE^{\omega},$ $V\mapsto\un{V}.$ Hence, we also have an equivalence
\begin{equation*}
\prodd[\N]\Perf(\mk)\xto{\sim}(\prodd[\N]^{\dual}\cE)^{\omega}\simeq\prodd[\N]\cE^{\omega},\quad (V_n)_{n\geq 0}\mapsto (\un{V_n})_{n\geq 0}.
\end{equation*}
By definition, the functor $\prodd[\N]^{\dual}\cE\to \cA$  is a strongly continuous $\cE$-linear quotient functor. We identify $\cA$ with the essential image of the right adjoint $\cA\to \prodd[\N]^{\dual}\cE.$ Hence, $\cA$ consists of objects $X\in\prodd[\N]^{\dual}\cE$ such that $\mk_{(0,1]}\otimes X=0.$

For $n,k\geq 0$ we put 
\begin{equation*}
a_{k,n}=\frac{1}{(k+1)(n+2)},\quad b_{k,n}=\frac{1}{k+1},\quad \cF_{k,n}=\mk_{(a_{k,n},b_{k,n}]}\in\cE.
\end{equation*} 
We have a natural map $\cF_{k,n}\to\cF_{k+1,n},$ which is moreover compact in $\cE.$ Hence, we have a well-defined object 
\begin{equation*}
X=\inddlim[k](\cF_{k,n})_{n\geq 0}\in\prodd[\N]^{\dual}\cE.
\end{equation*}
Moreover, it follows from the formula \eqref{eq:E_action} that 
$\mk_{(0,1]}\otimes X=0$ since $\hat{\cY}(\mk_{(0,1]})=\inddlim[0<\veps<1]\mk_{(\veps,1]}.$ Hence, we have $X\in\cA.$ Moreover, $X\ne 0$ since for all $k\geq 0$ the map $\cF_{0,k}\to\cF_{k,k}$ is non-zero.
	
It remains to see that $X$ is contained in the right orthogonal to $(\prodd[\N]^{\dual}\cE)^{\omega}\simeq\prodd[\N]\Perf(\mk)$ (since we now consider $\cA$ as a subcategory of $\prodd[\N]^{\dual}\cE$). This is clear since $\Gamma([0,1],\cF_{k,n})=0$ for $n,k\geq 0.$ This proves \ref{prod_dual_is_not_compactly_generated}. 
\end{proof}

\subsection{Deformed tensor algebra}
\label{ssec:deformed_tensor_algebra}

We recall the following notion of a deformed tensor algebra which is essentially standard but not widely known. It was used in \cite[Definition 5.17, Proposition 5.18]{E25}. We use the notation from loc. cit.

\begin{defi}\label{def:deformed_tensor_algebra}
Let $\cE\in\Alg_{\bE_1}(\Prr^L_{\st})$ be an $\bE_1$-monoidal presentable stable category.
\begin{enumerate}[label=(\roman*),ref=(\roman*)]
	\item The deformed tensor algebra functor
	\begin{equation*}
		T^{\deff}(-):\cE_{/1_{\cE}[1]}\to \Alg_{\bE_1}(\cE)
	\end{equation*}
	is defined to be the left adjoint to the composition
	\begin{equation*}
		\Alg_{\bE_1}(\cE)\to\Alg_{\bE_0}(\cE)\simeq\cE_{1_{\cE}/}\simeq\cE_{/1_{\cE}[1]}.
	\end{equation*}
	\item We consider the category $\Fun(\N,\cE)$ as an $\bE_1$-monoidal category of non-negatively filtered objects of $\cE.$ Then the filtered deformed tensor algebra functor
	\begin{equation*}
	\Fil_{\bullet} T^{\deff}(-):\cE_{/1_{\cE}[1]}\to\Alg_{\bE_1}(\Fun(\N,\cE))
	\end{equation*} 
	is defined to be the left adjoint to the functor
	\begin{equation*}
	\Fil_{\bullet} A\mapsto (\Cone(1_{\cE}\to\Fil_1 A)\to 1_{\cE}[1]).
	\end{equation*}
\end{enumerate}
\end{defi}

Note that we have a commutative square
\begin{equation*}
\begin{tikzcd}
\cE_{/1_{\cE}[1]}\ar{r}{\Fil_{\bullet}T^{\deff}(-)}\ar[equal]{d} & [3em] \Alg_{\bE_1}(\Fun(\N,\cE))\ar[d, "\colim"]\\
\cE_{/1_{\cE}[1]}\ar{r}{T^{\deff}(-)} & \Alg_{\bE_1}(\cE).
\end{tikzcd}
\end{equation*}
If we denote by $\cE_{\geq 0}^{\Gr}\simeq \prodd[\N]\cE$ the monoidal category of non-negatively graded objects of $\cE,$ then we have a monoidal functor $\gr_{\bullet}:\Fun(\N,\cE)\to \cE_{\geq 0}^{\Gr}$ (the associated graded). We have a commutative square
\begin{equation*}
	\begin{tikzcd}
		\cE_{/1_{\cE}[1]}\ar{r}{\Fil_{\bullet}T^{\deff}(-)}\ar[equal]{d} & [3em] \Alg_{\bE_1}(\Fun(\N,\cE))\ar[d, "\gr_{\bullet}"]\\
		\cE_{/1_{\cE}[1]}\ar{r}{T(-)} & \Alg_{\bE_1}(\prodd[\N]\cE),
	\end{tikzcd}
\end{equation*}
where we take the standard grading on the usual tensor algebra $T(-).$

We will sometimes write $T^{\deff}(x)$ for $x\in\cE,$ assuming the choice of a morphism $x\to 1_{\cE}[1].$ For $A\in\Alg_{\bE_1}(\cE)$ we denote by $T_A^{\deff}(-)$ and $\Fil_{\bullet}T_A^{\deff}(-)$ the compositions
\begin{equation*}
T_A^{\deff}(-):(A\hy\Mod\hy A)_{/A[1]}\to\Alg_{\bE_1}(A\hy\Mod\hy A)\simeq \Alg_{\bE_1}(\cE)_{A/},
\end{equation*}
\begin{equation*}
\Fil_{\bullet} T_A^{\deff}(-):(A\hy\Mod\hy A)_{/A[1]}\to\Alg_{\bE_1}(\Fun(\N,A\hy\Mod\hy A))\simeq \Alg_{\bE_1}(\Fun(\N,\cE))_{A/}.
\end{equation*}
Here to make sense of the latter category we consider $A$ as a non-negatively filtered algebra with constant filtration.

\section{Duality between categories of $1$-morphisms}
\label{sec:duality_between_cats_of_1_morphisms}

In this section we prove a non-standard sufficient condition for dualizability of a presentable stable category. It is conceptually similar to the criterion of rigidity for monoidal presentable stable categories (see Proposition \ref{prop:rigidity_criterion} and references therein). We will apply the following theorem in the next section to prove the duality between $\Mot_{\cE}^{\loc}$ and $\Mot^{\loc}_{\cE^{mop}}$ for a rigid $\bE_1$-monoidal category $\cE.$

\begin{theo}\label{th:conditions_for_E_0_rigidity}
	Let $\cA$ be a category enriched over $(\Pr^L_{\st},\otimes),$ and let $X,Y\in\cA$ be two objects. Suppose that the following condition hold. 
	\begin{enumerate}[label=(\roman*),ref=(\roman*)]
		\item The identity $1$-morphism $1_X$ is compact in $\cA(X,X).$ \label{1_X_compact}
		\item The category $\cA(X,Y)$ is generated via colimits by $1$-morphisms $F$ such that $F\cong\indlim[n\in\N]F_n,$ where each $2$-morphism $F_n\to F_{n+1}$ is right trace-class. \label{cond_right_trace_class}
		\item The category $\cA(Y,X)$ is generated via colimits by $1$-morphisms $G$ such that $G\cong\indlim[n\in\N]G_n,$ where each $2$-morphism $G_n\to G_{n+1}$ is left trace-class. \label{cond_left_trace_class}
	\end{enumerate}
	Then the categories $\cA(X,Y)$ and $\cA(Y,X)$ are dualizable and the composition functor $\mu_{Y,X,Y}:\cA(X,Y)\otimes\cA(Y,X)\to \cA(Y,Y)$ is strongly continuous. Moreover, we have an equivalence $\cA(Y,X)^{\vee}\simeq\cA(X,Y),$ and the evaluation and coevaluation functors are given by the compositions 
	\begin{equation}\label{eq:ev_for_A_Y_X}
		\ev:\cA(Y,X)\otimes\cA(X,Y)\xto{\mu_{X,Y,X}}\cA(X,X)\xto{\Hom(1_X,-)}\Sp,
	\end{equation}
	\begin{equation}\label{eq:coev_for_A_Y_X}
		\coev:\Sp\xto{1_Y}\cA(Y,Y)\xto{\mu_{Y,X,Y}^R}\cA(X,Y)\otimes\cA(Y,X).
	\end{equation}
\end{theo}

\begin{lemma}\label{lem:strcont_from_tensor_product}
Let $\cC,\cD,\cE$ be presentable stable categories, and suppose that $\cC$ is dualizable. Let $\Phi:\cC\otimes\cD\to \cE$ be a continuous functor. The following are equivalent.
\begin{enumerate}[label=(\roman*),ref=(\roman*)]
	\item $\Phi$ is strongly continuous. \label{strcont_from_tensor_product}
	\item For any compact morphism $x\to y$ in $\cC^{\omega_1},$ the morphism of functors $\Phi(x\boxtimes-)\to \Phi(y\boxtimes -)$ is right trace-class in $\Fun^L(\cD,\cE).$ \label{criterion_via_right_trace_class}
\end{enumerate}
\end{lemma}

\begin{proof}
\Implies{strcont_from_tensor_product}{criterion_via_right_trace_class}. Let us identify $x$ and $y$ with the corresponding continuous functors $\Sp\to\cC.$ Then compactness of a morphism $x\to y$ means that it is right trace-class in $\Fun^L(\Sp,\cC).$ Hence so is the morphism $\Phi(x\boxtimes-)\to \Phi(y\boxtimes -)$ in $\Fun^L(\cD,\cE)$ by \cite[Proposition 1.47]{E25}

\Implies{criterion_via_right_trace_class}{strcont_from_tensor_product}. The right adjoint $\Phi^R$ is given by the composition
\begin{equation*}
\cE\xto{\Psi} \Fun^{\omega_1\hy\lex}(\cC^{\omega_1,op},\cD)\simeq\cC\otimes\cD,
\end{equation*}
where $\Psi(y)(x)=\Phi(x\boxtimes-)^R(y).$ The inclusion
\begin{equation*}
\incl:\Fun^{\omega_1\hy\lex}(\cC^{\omega_1,op},\cD)\to \Fun(\cC^{\omega_1,op},\cD)
\end{equation*}
has a right adjoint $\incl^R.$ For an (exact) functor $F:\cC^{\omega_1,op}\to\cD$ and for an object $x\in\cC^{\omega_1}$ with $\hat{\cY}(x)=\inddlim[n]x_n$ we have
\begin{equation*}
\incl^R(F)(x)=\prolim[n]F(x_n).
\end{equation*}
Using \ref{criterion_via_right_trace_class} we obtain that for an ind-system $(y_i)_{i\in I}$ in $\cE$ we have
\begin{multline*}
(\indlim[i]\Psi(y_i))(x)\cong \prolim[n]\indlim[i]\Psi(y_i)(x_n)\cong \prolim[n]\indlim[i]\Phi(x_n\boxtimes -)^R(y_i)\\
\cong \prolim[n]\indlim[i]\Phi(x_n\boxtimes -)^{R,\cont}(y_i)\cong \prolim[n] \Phi(x_n\boxtimes-)^{R,\cont}(\indlim[i]y_i)\cong \prolim[n]\Phi(x_n\boxtimes-)^{R}(\indlim[i]y_i)\\
\cong \Phi(x\boxtimes-)^{R}(\indlim[i]y_i)\cong \Psi(\indlim[i]y_i)(x).
\end{multline*}
This shows that $\Phi^R$ is continuous.
\end{proof}

The following application is almost immediate.

\begin{cor}\label{cor:dualizability_of_A_X_Y} Let $\cA$ be a category enriched over $\Pr^L_{\st},$ and let $X\in\cA$ be an object such that $1_X\in\cA(X,X)^{\omega}.$
\begin{enumerate}[label=(\roman*),ref=(\roman*)]
	\item Let $Y\in\cA$ be an object such that the category $\cA(X,Y)$ is generated via colimits by $1$-morphisms $F$ such that $F\cong\indlim[n\in\N]F_n,$ where each $2$-morphism $F_n\to F_{n+1}$ is right trace-class. Then $\cA(X,Y)$ is a dualizable category and for any object $Z\in\cA$ the functor
	\begin{equation}\label{eq:composition_is_strcont}
	\mu_{Z,X,Y}:\cA(X,Y)\otimes \cA(Z,X)\to \cA(Z,Y)
	\end{equation}
	is strongly continuous. \label{dualizability_and_strcont_right_trace_class}
	\item Dually, let $Y\in\cA$ be an object such that the category $\cA(Y,X)$ is generated via colimits by $1$-morphisms $G$ such that $G\cong\indlim[n\in\N]G_n,$ where each $2$-morphism $G_n\to G_{n+1}$ is left trace-class. Then $\cA(Y,X)$ is a dualizable category and for any object $Z\in\cA$ the functor
	\begin{equation*}
		\mu_{Y,X,Z}:\cA(X,Z)\otimes \cA(Y,X)\to \cA(Y,Z)
	\end{equation*}
	is strongly continuous. \label{dualizability_and_strcont_left_trace_class}
\end{enumerate} \end{cor}

\begin{proof}
We prove \ref{dualizability_and_strcont_right_trace_class}, and \ref{dualizability_and_strcont_left_trace_class} is obtained by passing to the $1$-opposite category $\cA^{1\hy op}.$ 

Note that compactness of $1_X$ in $\cA(X,X)$ implies that each right trace-class $2$-morphism in $\cA(X,Y)$ is compact. Indeed, given such a $2$-morphism $f:F\to F'$ with a right trace-class witness $\wt{f}:1_X\to F^{r\vee}\circ F',$ we obtain a factorization
\begin{equation*}
	\Hom(F',-)\to \Hom(F^{r\vee}\circ F',F^{r\vee}\circ -)\to \Hom(1_X,F^{r\vee}\circ -)\to \Hom(F,-),
\end{equation*}
and the functor $\Hom(1_X,F^{r\vee}\circ -)$ commutes with colimits. Therefore, the category $\cA(X,Y)$ is dualizable by \cite[Theorem 1.45]{E24}. 

Next, we show that for any $F,F'\in\cA(X,Y)$ we have an isomorphism
\begin{equation}\label{eq:right_trace_class_to_compact}
\Hom(1_X,F'^{r\vee}\circ F)\xto{\sim} \Hom_{\Ind(\cA(X,Y))}(\cY(F'),\hat{\cY}(F)) .
\end{equation} 
More precisely, if $\hat{\cY}(F)=\inddlim[i] F_i,$ then the map \eqref{eq:right_trace_class_to_compact} is given by the composition
\begin{multline*}
\Hom(1_X,F'^{r\vee}\circ F)\cong \indlim[i] \Hom(1_X,F'^{r\vee}\circ F_i)\to \indlim[i]\Hom(F',F_i)\\
\cong \Hom_{\Ind(\cA(X,Y))}(\cY(F'),\hat{\cY}(F)).
\end{multline*}
Since the source and the target of \eqref{eq:right_trace_class_to_compact} commute with colimits in $F,$ we may assume that $F\cong\indlim[n\in\N] F_n,$ where each map $F_n\to F_{n+1}$ is right trace-class. Then by the above argument we have $\hat{\cY}(F)=\inddlim[n] F_n,$ and the isomorphism \eqref{eq:right_trace_class_to_compact} follows from the factorizations
\begin{equation*}
\Hom(F',F_n)\to \Hom(1_X,F'^{r\vee}\otimes F_{n+1})\to \Hom(F',F_{n+1}).
\end{equation*}

Now the strong continuity of \eqref{eq:composition_is_strcont} follows from Lemma \ref{lem:strcont_from_tensor_product}. Namely, take a compact morphism $\varphi:F'\to F$ in $\cA(X,Y).$ By the above, $\varphi$ is right trace-class and we can choose a right trace-class witness $\wt{\varphi}:1_X\to F'^{r\vee}\circ F.$ Then the morphism of functors $\mu_{Z,X,Y}(F'\boxtimes -)\to \mu_{Z,X,Y}(F\boxtimes -)$ has a right trace-class witness, given by the composition
\begin{equation*}
\id\xto{\mu_{Z,X,X}(\wt{\varphi}\boxtimes-)}\mu_{Z,X,X}(F'^{r\vee}\circ F\boxtimes -)\to \mu_{Z,X,Y}(F'\boxtimes -)^{R,\cont}\circ \mu_{Z,X,Y}(F\boxtimes -).
\end{equation*}
Applying Lemma \ref{lem:strcont_from_tensor_product}, we see that the functor \eqref{eq:composition_is_strcont} is strongly continuous.
\end{proof}

\begin{proof}[Proof of Theorem \ref{th:conditions_for_E_0_rigidity}]
Dualizability of $\cA(X,Y)$ and $\cA(Y,X)$ follows from Corollary \ref{cor:dualizability_of_A_X_Y}, as well as the strong continuity of  $\mu_{Y,X,Y}.$


To establish the duality between $\cA(X,Y)$ and $\cA(Y,X),$ we first show that the following (lax commutative) square commutes:
\begin{equation}\label{eq:comm_square1}
\begin{tikzcd}
\cA(Y,X)\otimes\cA(Y,Y)\ar{r}{\mu_{Y,Y,X}}\ar{d}{\id\boxtimes\mu_{Y,X,Y}^R} & [2em] \cA(Y,X)\ar{d}{\mu_{Y,X,X}^R}\\
\cA(Y,X)\otimes\cA(X,Y)\otimes\cA(Y,X)\ar{r}{\mu_{X,Y,X}\boxtimes\id} & \cA(X,X)\otimes\cA(Y,X),
\end{tikzcd}
\end{equation}
Here the functor $\mu_{Y,X,X}^R$ is continuous by Corollary \ref{cor:dualizability_of_A_X_Y}. It is convenient to use the following notation: for presentable stable categories $\cC,\cD,$ a continuous functor $\Phi:\cA(X,Y)\otimes\cC\to\cD$ and an object $G\in\cA(X,Y),$ we denote by $\un{\Hom}^r(G,-):\cD\to\cC$ the right adjoint to $\Phi(G\boxtimes -).$ We apply this notation to the functors
\begin{equation*}
\mu_{Y,X,Y}:\cA(X,Y)\otimes\cA(Y,X)\to\cA(Y,Y),
\end{equation*}
\begin{equation*}
\mu_{X,X,Y}\boxtimes\id:\cA(X,Y)\otimes\cA(X,X)\otimes\cA(Y,X)\to \cA(X,Y)\otimes\cA(Y,X).
\end{equation*}

 Within this notation, for $G\in \cA(X,Y),$ $H\in\cA(Y,Y)$ we have a natural isomorphism
\begin{equation}\label{eq:obvious_isom}
\mu_{Y,X,X}^R(\un{\Hom}^r(G,H))\cong \un{\Hom}^r(G,\mu_{Y,X,Y}^R(H))
\end{equation}
in $\cA(X,X)\otimes\cA(Y,X).$ Indeed, if we fix $G$ and consider the source and the target of \eqref{eq:obvious_isom} as functors of $H$ from $\cA(Y,Y)$ to $\cA(X,X)\otimes\cA(Y,X),$ then the left adjoints of these functors are naturally isomorphic.

Now let $F\in \cA(Y,X)$ be an object such that $F\cong\indlim[n\in\\N] F_n,$ where each map $F_n\to F_{n+1}$ is left trace-class. Since such objects generated $\cA(Y,X),$ it suffices to prove the lax commutativity morphism in \eqref{eq:comm_square1} is an isomorphism on $F\boxtimes-.$ We choose left trace-class witnesses $1_X\to F_{n+1}\circ F_n^{l\vee}.$ Then for each $n\in\N$ we have a composition morphism.
\begin{multline*}
\mu_{Y,X,X}^R(\mu_{Y,Y,X}(F_n,-))\to \mu_{Y,X,X}^R(\un{\Hom}^r(F_n^{l\vee},-))\cong \un{\Hom}^r(F_n^{l\vee},\mu_{Y,X,Y}^R(-))\\
\to (\mu_{X,Y,X}\boxtimes\id)(F_{n+1}\boxtimes\mu_{Y,X,Y}^R(-)).
\end{multline*} 
Passing to the colimit over $n,$ we obtain a morphism of functors
\begin{equation*}
\mu_{Y,X,X}^R(\mu_{Y,Y,X}(F,-))\to (\mu_{X,Y,X}\boxtimes\id)(F\boxtimes\mu_{Y,X,Y}^R(-)),
\end{equation*}
which is the inverse to the lax commutativity morphism of \eqref{eq:comm_square1}.

Now for the functors $\ev$ and $\coev$ from \eqref{eq:ev_for_A_Y_X} and \eqref{eq:coev_for_A_Y_X} we obtain the following commutative diagram:
\begin{equation*}
\begin{tikzcd}
\cA(Y,X)\ar[d, "\id\boxtimes 1_Y"]\ar[r, "\sim"] & [2em] \cA(Y,X)\otimes\Sp \ar[d, "\id\boxtimes\coev"]\\
\cA(Y,X)\otimes\cA(Y,Y)\ar[r, "\id\boxtimes\mu_{Y,X,Y}^R"]\ar[d, "\mu_{Y,Y,X}"] & \cA(Y,X)\otimes\cA(X,Y)\otimes\cA(Y,X)\ar[r, "\ev\boxtimes\id"] \ar[d, "\mu_{X,Y,X}\boxtimes\id"] & \Sp\otimes\cA(Y,X)\ar[d, "\sim"] \\
\cA(Y,X)\ar[r, "\mu_{Y,X,X}^R"] & \cA(X,X)\otimes \cA(Y,X)\ar{r}{\Hom(1_X,-)\boxtimes\id} & \cA(Y,X).  
\end{tikzcd}
\end{equation*}
The composition of functors in the left column is isomorphic to the identity, and so is the composition of functors in the bottom row. 

Passing to the $1$-opposite category $\cA^{1\hy op},$ we see that the following (lax commutative) square commutes:
\begin{equation}\label{eq:comm_square2}
	\begin{tikzcd}
		\cA(Y,Y)\otimes\cA(X,Y)\ar{r}{\mu_{X,Y,Y}}\ar{d}{\mu_{Y,X,Y}^R\boxtimes\id} & [2em] \cA(X,Y)\ar{d}{\mu_{X,X,Y}^R}\\
		\cA(X,Y)\otimes\cA(Y,X)\otimes\cA(X,Y)\ar{r}{\id\boxtimes\mu_{X,Y,X}} & \cA(X,Y)\otimes\cA(X,X).
	\end{tikzcd}
\end{equation}
This gives a commutative diagram
\begin{equation*}
	\begin{tikzcd}
		\cA(X,Y)\ar[d, "1_Y\boxtimes \id"]\ar[r, "\sim"] & [2em] \Sp\otimes \cA(X,Y) \ar[d, "\coev\boxtimes\id"]\\
		\cA(Y,Y)\otimes\cA(X,Y)\ar[r, "\mu_{Y,X,Y}^R\boxtimes\id"]\ar[d, "\mu_{X,Y,Y}"] & \cA(X,Y)\otimes\cA(Y,X)\otimes\cA(X,Y)\ar[r, "\id\boxtimes\ev"] \ar[d, "\id\boxtimes\mu_{X,Y,X}"] & \cA(X,Y)\otimes\Sp\ar[d, "\sim"] \\
		\cA(X,Y)\ar[r, "\mu_{X,X,Y}^R"] & \cA(X,Y)\otimes \cA(X,X)\ar{r}{\id\boxtimes \Hom(1_X,-)} & \cA(X,Y).  
	\end{tikzcd}
\end{equation*}
Again, both in the left column and in the bottom row the compositions are isomorphic to the identity. This proves the duality between $\cA(X,Y)$ and $\cA(Y,X).$
\end{proof}

\section{Dualizability and rigidity of categories of localizing motives}
\label{sec:dualizability_and_rigidity}

In this section we prove the following result.

\begin{theo}\label{th:dualizability_and_rigidity}
Let $\cE$ be a rigid $\bE_1$-monoidal category.
\begin{enumerate}[label=(\roman*),ref=(\roman*)]
\item The category $\Mot_{\cE}^{\loc}$ is dualizable. Moreover, we have an equivalence $(\Mot_{\cE}^{\loc})^{\vee}\simeq \Mot_{\cE^{mop}}^{\loc},$ and the evaluation functor is given by
\begin{equation*}
\ev:\Mot_{\cE}^{\loc}\otimes \Mot_{\cE^{mop}}^{\loc}\to\Sp,\quad \cU_{\loc}(\cC)\boxtimes \cU_{\loc}(\cD)\mapsto K^{\cont}(\cD\tens{\cE}\cC).
\end{equation*} \label{E_0_rigidity}
\item If $\cE$ is in addition $\bE_2$-monoidal (for example, symmetric monoidal), then the $\bE_1$-monoidal category $\Mot_{\cE}^{\loc}$ is rigid. \label{E_1_rigidity}
\end{enumerate}
\end{theo}

\subsection{Nuclear left $\cE$-modules}

In this subsection we fix a rigid $\bE_1$-monoidal category $\cE.$ We will define and study the notion of nuclearity for relatively compactly generated left $\cE$-modules. The main result is Proposition \ref{prop:resolution_by_nuclear} which loosely speaking says that any $\omega_1$-compact relatively compactly generated left $\cE$-module is a quotient of a nuclear $\cE$-module by a nuclear $\cE$-submodule. This will be crucial for the proof of Theorem \ref{th:dualizability_and_rigidity}. Using our study of nuclear $\cE$-modules we give a proof of Theorem \ref{th:dualizability_and_rigidity} in the end of this subsection.

The main reason for restricting to the relatively compactly generated case is that the category $\Cat_{\cE}^{\cg}$ is compactly assembled by \cite[Theorem 1.14]{E25}. On the other hand, the category $\Cat_{\cE}^{\dual}$ is not compactly assembled even for $\cE=\Sp$ by \cite[Corollary 1.86]{E24}.

As in \cite[Section 3]{E25}, for $\cC,\cD\in\Cat_{\cE}^{\dual},$ we denote by $\un{\Hom}_{\cE}^{\dual}(\cC,\cD)\in\Cat_{\st}^{\dual}$ the relative internal $\Hom$ in $\Cat_{\cE}^{\dual}$ over $\Cat_{\st}^{\dual}.$ Moreover, we consider the category $\un{\Hom}_{\cE}^{\dual}(\cC,\cE)$ as a dualizable right $\cE$-module.



\begin{defi}\label{def:nuclear_E_modules} Let $\cE$ be a rigid $\bE_1$-monoidal category. 
\begin{enumerate}[label=(\roman*),ref=(\roman*)]
	\item Let $\cC,\cD\in\Cat_{\cE}^{\dual}$ be dualizable left $\cE$-modules. A strongly continuous $\cE$-linear functor $F:\cC\to\cD$ is called right trace-class over $\cE$ if the object $F\in\Fun_{\cE}^{LL}(\cC,\cD)$ is contained in the essential image of the functor
	\begin{equation*}
	(\un{\Hom}_{\cE}^{\dual}(\cC,\cE)\tens{\cE} \cD)^{\omega}\to (\un{\Hom}_{\cE}^{\dual}(\cC,\cD))^{\omega}\simeq \Fun_{\cE}^{LL}(\cC,\cD).
	\end{equation*}
	We say that an object $X\in (\un{\Hom}_{\cE}^{\dual}(\cC,\cE)\tens{\cE} \cD)^{\omega}$ is a right trace-class witness for $F$ if the image of $X$ in $\Fun_{\cE}^{LL}(\cC,\cD)$ is isomorphic to $F.$
	\item A relatively compactly generated left $\cE$-module  $\cC\in\Cat_{\cE}^{\cg}$ is called nuclear over $\cE$ if for any $\omega_1$-compact $\cD\in (\Cat_{\cE}^{\cg})^{\omega_1}$ and for any strongly continuous $\cE$-linear functor $F:\cD\to\cC,$ if $F$ is a compact morphism in $\Cat_{\cE}^{\cg}$ then $F$ is right trace-class over $\cE.$ We denote by $\Cat_{\cE}^{\cg,\nuc}\subset\Cat_{\cE}^{\cg}$ the full subcategory of nuclear left $\cE$-modules.
	
	\item An $\cE$-module $\cC\in\Cat_{\cE}^{\cg}$ is called basic nuclear if there is a direct sequence $(\cC_n)_{n\geq 0}$ in $\Cat_{\cE}^{\cg}$ such that $\cC\simeq\indlim[n]\cC_n$ and each functor $\cC_n\to\cC_{n+1}$ is right trace-class over $\cE.$
\end{enumerate}	
\end{defi}

\begin{remark}
\begin{enumerate}
\item If $\cE$ is symmetric monoidal, then the above definition of right trace-class functor over $\cE$ simply means being a trace-class morphism in the symmetric monoidal category $\Cat_{\cE}^{\dual}.$
\item Let $\cA$ be the following category enriched over $\Pr^L:$ the objects are rigid $\bE_1$-monoidal categories, and $\cA(\cE,\cE')=\Cat_{\cE^{mop}\otimes\cE'}^{\dual}.$ Then $\Cat_{\cE}^{\dual}=\cA(\Sp,\cE),$ and under this identification the notion of a right trace-class functor over $\cE$ corresponds to the notion of a right trace-class $2$-morphism $\cA.$
\end{enumerate}
\end{remark}

We first explore some basic properties and equivalent characterizations of nuclear left $\cE$-modules. These are similar to the results of \cite[Section 13]{CS20}, but the proofs are slightly more difficult. We start with a trivial observation which will be tacitly used below in many situations. 

\begin{prop}\label{prop:diagonal_arrow_dualizable_trace_class}
Let $\cC,\cC',\cD\in\Cat_{\cE}^{\dual}$ be dualizable left $\cE$-module. Let $F:\cC\to\cC'$ be a strongly continuous $\cE$-linear functor which is right trace-class over $\cE.$ Then there exists a strongly continuous functor $\un{\Hom}_{\cE}^{\dual}(\cC',\cD)\to \un{\Hom}_{\cE}^{\dual}(\cC,\cE)\tens{\cE}\cD,$ making the following diagram commute:
\begin{equation*}
\begin{tikzcd}
\un{\Hom}_{\cE}^{\dual}(\cC',\cE)\tens{\cE}\cD \ar{r}{\un{\Hom}_{\cE}^{\dual}(F,\cE)\boxtimes\id} \ar[d] & [4em] \un{\Hom}_{\cE}^{\dual}(\cC,\cE)\tens{\cE}\cD\ar[d]\\
\un{\Hom}_{\cE}^{\dual}(\cC',\cD) \ar{r}{\un{\Hom}_{\cE}^{\dual}(F,\cD)}\ar[ru] & \un{\Hom}_{\cE}^{\dual}(\cC,\cD).
\end{tikzcd}
\end{equation*}
\end{prop}

\begin{proof}
Indeed, choosing a right trace-class witness $X\in (\un{\Hom}_{\cE}^{\dual}(\cC,\cE)\tens{\cE}\cC')^{\omega}$ for $F,$ we obtain the required functor by taking the composition
\begin{equation*}
\un{\Hom}_{\cE}^{\dual}(\cC',\cD)\xto{X\boxtimes \id} \un{\Hom}_{\cE}^{\dual}(\cC,\cE)\tens{\cE}\cC'\otimes \un{\Hom}_{\cE}^{\dual}(\cC,\cD) \to \un{\Hom}_{\cE}^{\dual}(\cC,\cE)\tens{\cE}\cD.\qedhere
\end{equation*}
\end{proof}

Next, we show that basic nuclear $\cE$-modules are in fact nuclear. More generally, the following holds.

\begin{prop}\label{prop:nuclear_equiv_cond}
	Let $(\cC_i)_{i\in I}$ be an ind-system in $\Cat_{\cE}^{\cg},$ where $I$ is directed, and let $\cC = \indlim[i\in I] \cC_i\in \Cat_{\cE}^{\cg}.$ The following are equivalent.
	\begin{enumerate}[label=(\roman*),ref=(\roman*)]
		\item $\cC$ is nuclear and we have an isomorphism $\hat{\cY}(\cC)\xto{\sim} \inddlim[i] \cC_i$ in $\Ind(\Cat_{\cE}^{\cg}).$ \label{C_nuclear_and_hat_Y}
		\item For any $i\in I$ there exists $j\geq i$ such that the functor $\cC_i\to \cC_j$ is right trace-class over $\cE.$ \label{eventually_right_trace_class}
	\end{enumerate}
In particular, if $\cD\in\Cat_{\cE}^{\cg}$ is basic nuclear then $\cD$ is nuclear. 
\end{prop}

\begin{proof} We denote by $F_{ij}:\cC_i\to \cC_j$ the transition functors, and by $F_i:\cC_i\to \cC$ the functors to the colimit.
	
	\Implies{C_nuclear_and_hat_Y}{eventually_right_trace_class}. Take some $i\in I,$ and choose $k\geq i$ such that the functor $F_{ik}:\cC_i\to \cC_k$ is a compact morphism in $\Cat_{\cE}^{\cg}.$ Since the functor $F_k:\cC_k\to \cC$ is a compact morphism in $\Cat_{\cE}^{\cg}$ and since $\cC$ is nuclear, we can find an object $X_k\in (\un{\Hom}_{\cE}^{\dual}(\cC_k,\cE)\tens{\cE}\cC)^{\omega}$ such that the image of $X_k$ in $\Fun_{\cE}^{LL}(\cC_k,\cC)$ is isomorphic to $F_k.$ It follows from \cite[Proposition 1.74]{E24} that we have an equivalence
	\begin{equation*}
	(\un{\Hom}_{\cE}^{\dual}(\cC_k,\cE)\tens{\cE}\cC)^{\omega}\simeq \indlim[l\geq k] (\un{\Hom}_{\cE}^{\dual}(\cC_k,\cE)\tens{\cE}\cC_l)^{\omega}, 
	\end{equation*}
	where the colimit is of course taken in $\Cat^{\perf}.$ Hence, for some $l\geq k$ the object $X_k$ can be lifted to an object $X_{kl}\in (\un{\Hom}_{\cE}^{\dual}(\cC_k,\cE)\tens{\cE}\cC_l)^{\omega}.$ 
	Denote by $F_{kl}'$ the image of $X_{kl}$ in $\Fun_{\cE}^{LL}(\cC_k,\cC_l).$ 
	Denote by $G,G'$ the following compositions in $\Ind(\Cat_{\cE}^{\cg}):$
	\begin{equation*}
	G:\cY(\cC_k)\xto{\cY(F_{kl})} \cY(\cC_l)\to \hat{\cY}(\cC),\quad G':\cY(\cC_k)\xto{\cY(F_{kl}')} \cY(\cC_l)\to \hat{\cY}(\cC).
	\end{equation*}
	The compositions $\hat{\cY}(\cC_k)\to \cY(\cC_k)\xto{G}\hat{\cY}(\cC)$ and  $\hat{\cY}(\cC_k)\to \cY(\cC_k)\xto{G'}\hat{\cY}(\cC)$ are both identified with $\hat{\cY}(F_k).$ The compactness of the morphism $F_{ik}$ in $\Cat_{\cE}^{\cg}$ implies that the morphisms $G\circ\cY(F_{ik})$ and $G'\circ \cY(F_{ik})$ are homotopic. This exactly means that for some $j\geq l$ the composition functor $F_{lj}\circ F_{kl}'\circ F_{ik}$ is isomorphic to $F_{ij}.$ Since $F_{kl}'$ is right trace-class by construction, so is $F_{ij}.$
	
	\Implies{eventually_right_trace_class}{C_nuclear_and_hat_Y}. We first observe the following: for $\cD,\cD'\in\Cat_{\cE}^{\cg}$ and for a strongly continuous $\cE$-linear functor $F:\cD\to\cD',$ if $F$ is right trace-class over $\cE$ then $F$ is a compact morphism in $\Cat_{\cE}^{\cg}.$ Indeed, this follows directly from \cite[Proposition 1.71]{E24}. Hence, we have an isomorphism $\hat{\cY}(\cC)\xto{\sim}\inddlim[i]\cC_i.$ Now, any compact morphism $F:\cC'\to \cC$ in $\Cat_{\cE}^{\cg}$ factors through some $\cC_i.$ Our assumption implies that the map $F_i:\cC_i\to \cC$ is right trace-class, hence so is $F.$ This proves that $\cC$ is nuclear.      
\end{proof}

We deduce the following properties of the category $\Cat_{\cE}^{\cg,\nuc}.$

\begin{prop}\label{prop:properties_of_nuc_E_modules} The full subcategory $\Cat_{\cE}^{\cg,\nuc}\subset \Cat_{\cE}^{\cg}$ is closed under filtered colimits. The category $\Cat_{\cE}^{\cg,\nuc}$ is $\omega_1$-accessible and the $\omega_1$-compact objects in it are exactly the basic nuclear left $\cE$-modules. Moreover, the inclusion  $\Cat_{\cE}^{\cg,\nuc}\to \Cat_{\cE}^{\cg}$ preserves and reflects $\omega_1$-compactness.
\end{prop}

\begin{proof}
	First, take some ind-system $(\cC_i)_{i\in I}$ in $\Cat_{\cE}^{\cg,\nuc},$ and let $\cC=\indlim[i] \cC_i,$ the colimit taken in $\Cat_{\cE}^{\cg}.$ Then for any compact morphism $\cC'\to \cC$ in $\Cat_{\cE}^{\cg}$ there exists $i\in I$ and a factorization $\cC'\xto{F}\cC_i\to\cC,$ such that $F$ is also a compact morphism. Hence, $F$ is right trace-class over $\cE$ by the nuclearity of $\cC_i.$ It follows that the functor $\cC'\to\cC$ is also right trace-class over $\cE.$ Therefore, $\cC$ is nuclear.
	
	By Proposition \ref{prop:nuclear_equiv_cond} any basic nuclear left $\cE$-module is $\omega_1$-compact in $\Cat_{\cE}^{\cg}.$ Hence, it is also $\omega_1$-compact in $\Cat_{\cE}^{\cg,\nuc}.$ Similarly, by loc.cit. if $\cC\in(\Cat_{\cE}^{\cg})^{\omega_1}$ is nuclear, then $\cC$ is basic nuclear.
	
	It remains to show that any $\cC\in\Cat_{\cE}^{\cg,\nuc}$ is an $\omega_1$-directed colimit of basic nuclear left $\cE$-modules. This is proved in the same way as \cite[Proposition 1.24]{E24}. Namely, let $\hat{\cY}(\cC) = \inddlim[i\in I] \cC_i,$ where $I$ is directed. Put $J''=\Fun(\N,I),$ and let $J'\subset J''$ be the full subposet consisting of maps $f:\N\to I$ such that each functor $\cC_{f(n)}\to \cC_{f(n+1)}$ is right trace-class over $\cE.$ It follows from Proposition \ref{prop:nuclear_equiv_cond} that $J'$ is cofinal in $J''.$ Next, define an equivalence relation on $J':$ $f\sim g$ if $f(n)=g(n)$ for $n\gg 0.$ We put $J=J'/\sim.$ The partial order on $J$ is the natural one: $\bar{f}\leq \bar{g}$ iff $f(n)\leq g(n)$ for $n\gg 0,$ where $f,g\in J'$ are representatives of $\bar{f},\bar{g}\in J.$ Then $J$ is $\omega_1$-directed and the functor $J'\to\Cat_{\cE}^{\cg},$ $f\mapsto\indlim[n]\cC_{f(n)},$ factors uniquely through $J.$ We obtain
	\begin{equation*}
		\cC\simeq\indlim[\bar{f}\in J]\indlim[n]\cC_{f(n)},
	\end{equation*}
	which gives a presentation of $\cC$ as an $\omega_1$-directed colimit of basic nuclear left $\cE$-modules.
\end{proof}

Next, we wish to show that the class of nuclear $\cE$-modules is sufficiently large in a certain sense. For this purpose it is technically convenient to introduce the following stronger notion of absolute nuclearity. Below we use the notation from \cite[Definition 1.56]{E24} in the case of small categories: for an exact functor $F:\cA\to\cB$ between idempotent-complete small stable categories we denote by $\im(F)\subset\cB$ the full idempotent-complete stable subcategory generated by the essential image of $F.$ Note that $\im(F)$ is the fiber of the cofiber of $F$ in $\Cat^{\perf}.$

\begin{defi}\label{def:absolute_nuclearity}
	Let $\cE$ be a rigid $\bE_1$-monoidal category.
	\begin{enumerate}[label=(\roman*),ref=(\roman*)]
		\item Let $\cC,\cD\in\Cat_{\cE}^{\dual}$ be dualizable left $\cE$-modules. We define the following category:
		\begin{equation*}
			(\un{\Hom}_{\cE}^{\dual}(\cC,\cE)\tens{\cE}\cD)^{\omega,\dec}=\im(\Fun_{\cE}^{LL}(\cC,\cE)\otimes \cD^{\omega}\to (\un{\Hom}_{\cE}^{\dual}(\cC,\cE)\tens{\cE}\cD)^{\omega})\in\Cat^{\perf}.
		\end{equation*}
		We say that a strongly continuous $\cE$-linear functor $F:\cC\to\cD$ is absolutely right trace-class over $\cE$ if $F$ is contained in the essential image of the functor
		\begin{equation*}
		(\un{\Hom}_{\cE}^{\dual}(\cC,\cE)\tens{\cE}\cD)^{\omega,\dec}\to \Fun_{\cE}^{LL}(\cC,\cD).
		\end{equation*}
		\item A relatively compactly generated left $\cE$-module $\cC\in\Cat_{\cE}^{\cg}$ is called absolutely nuclear over $\cE$ if for any $\omega_1$-compact $\cD\in(\Cat_{\cE}^{\cg})^{\omega_1}$ and for any strongly continuous $\cE$-linear functor $F:\cD\to\cC,$ if $F$ is a compact morphism in $\Cat_{\cE}^{\cg}$ then $F$ is absolutely right trace-class over $\cE.$  
	\end{enumerate}
\end{defi}

Clearly, absolute nuclearity implies nuclearity, and this stronger condition seems to be not very meaningful. However, certain important results below (Corollaries \ref{cor:abs_nuclear_cat_equiv_cond} and \ref{cor:properties_of_class_of_abs_nuclear}) would be much more difficult to prove if we replace ``absolutely nuclear'' by ``nuclear''.

The key statement is the following lemma. It is quite difficult but also crucial for the proof of Theorem \ref{th:dualizability_and_rigidity}.

\begin{lemma}\label{lem:map_from_compact_functor}
Let $\cB,\cC,\cD\in\Cat_{\cE}^{\cg}$ be relatively compactly generated left $\cE$-modules, and let $F:\cB\to\cC$ be a strongly continuous $\cE$-linear functor, which is compact as a $1$-morphism in $\Cat_{\cE}^{\cg}.$ We define the idempotent-complete stable categories
$\cV\subset \Fun_{\cE}^{LL}(\cB,\cD)$ and $\cW\subset \Fun_{\cE}^{LL}(\cC,\cD)$ as follows:
\begin{equation*}
\cV=\im((\un{\Hom}_{\cE}^{\dual}(\cB,\cE)\tens{\cE} \cD)^{\omega,\dec}\to \Fun_{\cE}^{LL}(\cB,\cD)),
\end{equation*}
\begin{equation*}
\cW=\im((\un{\Hom}_{\cE}^{\dual}(\cC,\cE)\tens{\cE} \cD)^{\omega,\dec}\to \Fun_{\cE}^{LL}(\cC,\cD)).
\end{equation*}
Then there exist an exact functor $\cW\to (\un{\Hom}_{\cE}^{\dual}(\cB,\cE)\tens{\cE} \cD)^{\omega,\dec},$ which is the diagonal arrow in a commutative diagram
\begin{equation}\label{eq:comm_diagram_for_E_0-rigidity}
\begin{tikzcd}
(\un{\Hom}_{\cE}^{\dual}(\cC,\cE)\tens{\cE} \cD)^{\omega,\dec}\ar[r]\ar[d] & (\un{\Hom}_{\cE}^{\dual}(\cB,\cE)\tens{\cE} \cD)^{\omega,\dec}\ar[d]\\
\cW\ar[r]\ar[ru] & \cV.
\end{tikzcd}
\end{equation}
Here the vertical arrows are the natural functors, and the horizontal arrows are induced by $F:\cB\to\cC.$
\end{lemma}

\begin{proof}
We first reduce to the case when $\cC$ is generated over $\cE$ by a single compact object. For a finite set $S$ of isomorphism classes of objects in $\cC^{\omega},$ denote by $\cC_S\subset \cC$ the localizing $\cE$-submodule generated by the objects in $S.$ Then $\cC\simeq\indlim[S]\cC_S,$ hence also $\hat{\cY}(\cC)\simeq\indlim[S]\hat{\cY}(\cC_S).$ It follows that $F$ can be decomposed as $\cB\xto{F'}\cC_S\to\cC,$ where $F'$ is compact in $\Cat_{\cE}^{\cg}.$ The $\cE$-module $\cC_S$ is generated by the compact object $\biggplus[x\in S] x,$ and we have $\cC_S\simeq\Mod\hy \un{\End}_{\cC/\cE}(\biggplus[x\in S]x).$ Hence, we may and will assume that $\cC\simeq \Mod\hy C$ for some $C\in\Alg_{\bE_1}(\cE).$

Next, by \cite[Theorem 1.14]{E25}, the category $\Alg_{\bE_1}(\cE)$ is compactly assembled, and the functor $\Alg_{\bE_1}(\cE)\to\Cat_{\cE}^{\cg},$ $A\mapsto \Mod\hy A,$ is strongly continuous. Hence, we have
\begin{equation*}
\hat{\cY}(\Mod\hy C)\simeq \inddlim[i\in I]\Mod\hy C_i,\quad\text{where}\quad \hat{\cY}(C)=\inddlim[i\in I]C_i.
\end{equation*}
It follows that $F:\cB\to\cC$ can be decomposed as $\cB\to \Mod\hy C_i\to Mod\hy C\simeq \cC$ for some $i.$ Hence, we may and will assume that $\cB\simeq Mod\hy B$ for some $B\in\Alg_{\bE_1}(\cE),$ and the functor $F:\Mod\hy B\to \Mod\hy C$ is given by extension of scalars via a morphism $f:B\to C,$ where $f$ is compact in $\Alg_{\bE_1}(\cE).$ Moreover, we may and will assume that both $B$ and $C$ are $\omega_1$-compact in $\Alg_{\bE_1}(\cE),$ or equivalently $\omega_1$-compact in $\cE.$

We have equivalences 
\begin{equation*}
\Fun_{\cE}^{LL}(\Mod\hy B,\cD)\simeq \Rep(B,\cD^{\omega}),\quad \Fun_{\cE}^{LL}(\Mod\hy C,\cD)\simeq\Rep(C,\cD^{\omega}).
\end{equation*} 
So we need to construct a dashed arrow, making the following diagram commute: 
\begin{equation}\label{eq:comm_diagram_for_E_0-rigidity_rewritten}
	\begin{tikzcd}
		(\un{\Hom}^{\dual}(\Mod\hy C,\cE)\tens{\cE} \cD)^{\omega,\dec}\ar[r]\ar[d] & (\un{\Hom}^{\dual}(\Mod\hy B,\cE)\tens{\cE} \cD)^{\omega,\dec}\ar[d]\\
		\cW\ar[hookrightarrow]{d}\ar[r]\ar[ru,dashed] & \cV\ar[hookrightarrow]{d}\\
		\Rep(C,\cE^{\omega})\ar[r] & \Rep(B,\cE^{\omega}).
	\end{tikzcd}
\end{equation}

The construction is quite non-trivial, and the first step is to obtain a suitable interpretation of categories in \eqref{eq:comm_diagram_for_E_0-rigidity_rewritten}. Consider the category $C\hy\Mod\hy C$ of $C\hy C$-bimodules. It is naturally $\bE_1$-monoidal, more precisely we choose the monoidal structure coming from the equivalence
\begin{equation*}
C\hy\Mod\hy C\xto{\sim}\Fun_{\cE^{mop}}^L(C\hy\Mod,C\hy\Mod),\quad M\mapsto M\tens{C}-.
\end{equation*}
By assumption, $C$ is an $\omega_1$-compact object of $\cE,$ hence the full subcategory $(C\hy\Mod\hy C)^{\omega_1}\subset C\hy\Mod\hy C$ is closed under the tensor product and it contains the unit object $C.$ Here and below we use the fact that a $C\hy C$-bimodule is $\omega_1$-compact in $C\hy\Mod\hy C$ if and only if its underlying object is $\omega_1$-compact in $\cE$ (this follows for example from \cite[Proposition C.3]{E24}). 

We obtain a well-defined compactly generated $\bE_1$-monoidal category 
\begin{equation*}
\cS_C=\Ind((C\hy\Mod\hy C)^{\omega_1})\in\Alg_{\bE_1}(\Prr^L_{\st,\omega}).
\end{equation*} 
Now the key idea is to define an auxiliary small category $\cT$ enriched over $\cS_C.$ The objects of $\cT$ are given by pairs $(X,Y),$ where $X\in \Rep(C,\cE^{\omega}),$ $Y\in \cD^{\omega}.$ The morphisms are given by
\begin{equation}\label{eq:morphisms_in_T_formal_description}
  \Hom_{\cT}((X_1,Y_1),(X_2,Y_2)) = \un{\Hom}_{\cE}^l(X_1,X_2\otimes \hat{\cY}_{\cE}(\un{\Hom}_{\cD/\cE}(Y_1,Y_2)))\in\cS_C.
\end{equation}
We elaborate on the meaning of \eqref{eq:morphisms_in_T_formal_description}. Put $Z=\un{\Hom}_{\cD/\cE}(Y_1,Y_2)\in \cE,$ and let $\hat{\cY}_{\cE}(Z)=\inddlim[i]Z_i\in\Ind(\cE^{\omega_1}).$ Then each object $X_2\otimes Z_i$ is naturally a left $C$-module, hence the object $\un{\Hom}_{\cE}^l(X_1,X_2\otimes Z_i)$ is naturally a $C\hy C$-bimodule, which is moreover $\omega_1$-compact. So the definition of morphisms \eqref{eq:morphisms_in_T_formal_description} can be rewritten as follows:
\begin{equation*}
  \Hom_{\cT}((X_1,Y_1),(X_2,Y_2)) = \inddlim[i] \Hom_{\cE}^l(X_1,X_2\otimes Z_i)\in\Ind(C\hy\Mod\hy C).
\end{equation*}
To explain the compositions in $\cT,$ consider a third object $(X_3,Y_3)\in\cT,$ and put $U=\un{\Hom}_{\cD/\cE}(Y_2,Y_3),$ and let $\hat{\cY}(U)=\inddlim[j] U_j.$
Then the composition is given by
\begin{multline*}
\Hom_{\cT}((X_2,Y_2),(X_3,Y_3))\otimes \Hom_{\cT}((X_1,Y_1),(X_2,Y_2)) \\
\cong \inddlim[i,j] [\un{\Hom}_{\cE}^l(X_2,X_3\otimes U_j)\tens{C}\un{\Hom}_{\cE}^l(X_1,X_2\otimes Z_i)]\to \inddlim[i,j]\un{\Hom}_{\cE}^l(X_1.X_3\otimes U_j\otimes Z_i)\\
\cong \un{\Hom}_{\cE}^l(X_1,X_3\otimes\hat{\cY}_{\cE}(\un{\Hom}_{\cD/\cE}(Y_2,Y_3)\otimes \un{\Hom}_{\cD/\cE}(Y_1,Y_2)))\\
\to \un{\Hom}_{\cE}^l(X_1,X_3\otimes\hat{\cY}_{\cE}(\un{\Hom}_{\cD/\cE}(Y_1,Y_3))) = \Hom_{\cT}((X_1,Y_1),(X_3,Y_3)).
\end{multline*}
Here we of course use the strong continuity of the multiplication functor $:\cE\otimes\cE\to\cE.$ 

We also use the compactness of the unit object $1_{\cE}\in\cE$ to see that the unit morphisms are well-defined. Namely, for an object $(X,Y)\in\cT$ the unit morphism $\cY(C)\to\Hom_{\cT}((X,Y),(X,Y))$ is given by the composition
\begin{equation*}
\cY(C)\to \cY(\un{\Hom}_{\cE}^l(X,X))\cong \un{\Hom}_{\cE}^l(X,X\otimes\hat{\cY}_{\cE}(1_{\cE}))\to \un{\Hom}_{\cE}^l(X,X\otimes\hat{\cY}_{\cE}(\un{\Hom}_{\cD/\cE}(Y,Y))).
\end{equation*}
in $\cS_C.$

Next, suppose that we have an $\bE_1$-coalgebra $E$ in $\cS_C$ (which is the same thing as a $\bE_1$-algebra in $\cS_C^{op}$). The coalgebra structure on $E$ gives an oplax monoidal structure on the functor $E\otimes-:\Sp\to \cS_C.$ Hence, its right adjoint functor $\Hom(E,-):\cS_C\to\Sp$ is naturally lax monoidal. In particular, we obtain the category $\cT^E$ enriched over $\Sp:$ its objects are the same as in $\cT,$ and the morphisms are given by 
\begin{equation*}
\Hom_{\cT^E}((X_1,Y_1),(X_2,Y_2))=\Hom_{\cS_C}(E,\Hom_{\cT}((X_1,Y_1),(X_2,Y_2)))\in \Sp.
\end{equation*}
We denote by $\Perf(\cT^E)\in\Cat^{\perf}$ the idempotent-complete stable envelope of $\cT^E.$ Namely, $\Perf(\cT^E)$ is the category of compact objects in the (compactly generated) category of $\Sp$-enriched functors $(\cT^E)^{op}\to\Sp.$ The assignment $E\mapsto \cT^E$ is clearly (contravariantly) functorial in $E\in\Coalg_{\bE_1}(\cS_C).$

We now apply this construction to the relevant coalgebras in $\cS_C.$ First, note that the functor
\begin{equation*}
\hat{\cY}:C\hy\Mod\hy C\to\cS_C.
\end{equation*}
is oplax $\bE_1$-monoidal, since its right adjoint is monoidal. This gives an $\bE_1$-coalgebra structure on $\hat{\cY}(C)\in\cS_C.$ We claim that we have a natural equivalence
\begin{equation}\label{eq:W_via_coalgebra_hat_Y_of_C}
\Perf(\cT^{\hat{\cY}(C)})\simeq \cW.
\end{equation}
Indeed, by definition the full subcategory $\cW\subset\Rep(C,\cD^{\omega})$ is generated by the tensor products $X\otimes Y,$ where $X\in\Rep(C,\cE^{\omega}),$ $Y\in\cD^{\omega}.$ Given $X_1,X_2\in\Rep(C,\cE^{\omega})$ and $Y_1,Y_2\in\cD^{\omega},$ we have natural isomorphisms
\begin{multline*}
\Hom_{\cT^{\hat{\cY}(C)}}((X_1,Y_1),(X_2,Y_2)) = \Hom_{\cS_C}(\hat{\cY}(C),\un{\Hom}_{\cE}^l(X_1,X_2\otimes \hat{\cY}_{\cE}(\Hom_{\cD/\cE}^l(Y_1,Y_2))))\\
 \cong \Hom_{C\hy\Mod\hy C}(C,\un{\Hom}_{\cE}^l(X_1,X_2\otimes \Hom_{\cD/\cE}^l(Y_1,Y_2)))\\
 \cong \Hom_{C\hy\Mod\hy C}(C, \un{\Hom}_{\cE}^l(X_1,\Hom_{\cD/\cE}^l(Y_1,X_2\otimes Y_2)))\\
 \cong \Hom_{C\hy\Mod\hy C}(C, \un{\Hom}_{\cD/\cE}(X_1\otimes Y_1,X_2\otimes Y_2))\cong \Hom_{\Rep(C,\cD^{\omega})}(X_1\otimes Y_1,X_2\otimes Y_2). 
\end{multline*}
Here we used the dualizability of $X_2$ in $\cE.$ This gives the equivalence \eqref{eq:W_via_coalgebra_hat_Y_of_C}.

Next, consider the unit object $\cY(C)\in\cS_C$ as an $\bE_1$-coalgebra. We claim that we have an equivalence
\begin{equation}\label{eq:equiv_for_first_tensor_product}
\Perf(\cT^{\cY(C)})\simeq (\un{\Hom}^{\dual}(\cC,\cE)\tens{\cE} \cD)^{\omega,\dec}.
\end{equation}
To see this, note that by definition the right hand side of \eqref{eq:equiv_for_first_tensor_product} is generated by the objects $X\boxtimes Y,$ where $X\in\Rep(C,\cE^{\omega}),$ $Y\in\cD^{\omega}.$ Take some objects $(X_1,Y_1),(X_2,Y_2)\in\cT.$ We put $Z=\un{\Hom}_{\cD/\cE}(Y_1,Y_2),$ and let $\hat{\cY}_{\cE}(Z)=\inddlim[i]Z_i\in\Ind(\cE^{\omega_1}).$ We consider the objects $X_1,X_2$ as compact objects of the right $\cE$-module $\un{\Hom}_{\cE}^{\dual}(\Mod\hy C,\cE)\subset \Ind(\Rep(C,\cE^{\omega_1})).$ Applying the right multiplication by $Z$ to $X_2,$ we obtain exactly the ind-object $X_2\otimes\hat{\cY}_{\cE}(Z)=\inddlim[i]X_2\otimes Z_i\in\Ind(\Rep(C,\cE^{\omega_1})).$ Using this observation we obtain
\begin{multline*}
\Hom_{\cT^{\cY(C)}}((X_1,Y_1),(X_2,Y_2))= 
\Hom_{\cS_C}(\cY(C),\un{\Hom}_{\cE}^l(X_1,X_2\otimes \hat{\cY}(Z)))\\
\cong \indlim[i]\Hom_{C\hy\Mod\hy C}(C,\Hom_{\cE}^l(X_1,X_2\otimes Z_i))\cong \indlim[i]\Hom_{\Rep(C,\cE)}(X_1,X_2\otimes Z_i)\\
\cong \Hom_{\un{\Hom}_{\cE}^{\dual}(\Mod\hy C,\cE)}(X_1,X_2\otimes\hat{\cY}_{\cE}(Z))\cong \\
\cong \Hom_{\cE}(1,\un{\Hom}^r_{\un{\Hom}_{\cE}^{\dual}(\Mod\hy C,\cE)/\cE}(X_1,X_2\otimes\hat{\cY}_{\cE}(Z)))\\
\cong \Hom_{\cE}(1,\un{\Hom}^r_{\un{\Hom}_{\cE}^{\dual}(\Mod\hy C,\cE)/\cE}(X_1,X_2)\otimes \un{\Hom}_{\cD/\cE}(Y_1,Y_2))\\
\cong \Hom_{\un{\Hom}_{\cE}^{\dual}(\Mod\hy C,\cE)\tens{\cE}\cD}(X_1\boxtimes Y_1,X_2\boxtimes Y_2).
\end{multline*}
This gives the equivalence \eqref{eq:equiv_for_first_tensor_product}.

Next, note that the functor
\begin{equation*}
	\Ind(C\tens{B} - \tens{B} C):\Ind((B\hy\Mod\hy B)^{\omega_1})\to \Ind((C\hy\Mod\hy C)^{\omega_1})=\cS_{C}
\end{equation*}
is oplax monoidal. To simplify the notation we denote this functor simply by $C\tens{B}-\tens{B}C.$ Then the object $C\tens{B}\hat{\cY}(B)\tens{B}C\in \cS_C$ is naturally an $\bE_1$-coalgebra. Using the adjunction, we deduce from the above that we have an equivalence
\begin{equation*}
\im(\cW\to\cV)\simeq \Perf(\cT^{C\tens{B}\hat{\cY}(B)\tens{B}C}). 
\end{equation*}
Similarly, we have an equivalence
\begin{equation*}
\im((\un{\Hom}_{\cE}^{\dual}(\Mod\hy C,\cE)\tens{\cE}\cD)^{\omega,\dec}\to (\un{\Hom}_{\cE}^{\dual}(\Mod\hy B,\cE)\tens{\cE}\cD)^{\omega,\dec})\simeq \Perf(\cT^{\cY(C\tens{B}C)}).
\end{equation*}
Here we assume the tautological isomorphism $C\tens{B}C\cong C\tens{B}B\tens{B}C$ of $C\hy C$-bimodules.

Therefore, in order to construct a required dashed arrow in \eqref{eq:comm_diagram_for_E_0-rigidity_rewritten}, it suffices to construct a morphism $\cY(C\tens{B}C)\to\hat{\cY}(C)$ of $\bE_1$-coalgebras in $\cS_C,$ making the following diagram commute in $\Coalg_{\bE_1}(\cS_C):$
\begin{equation}\label{eq:comm_diag_for_E_1_coalgebras}
\begin{tikzcd}
\cY(C) & \cY(C\tens{B}C)\ar[l]\ar[ld]\\
\hat{\cY}(C)\ar[u] & C\tens{B}\hat{\cY}(B)\tens{B}C\ar[l]\ar[u].
\end{tikzcd}
\end{equation}
Let $\hat{\cY}_{\Alg_{\bE_1}(\cE)}(C)=\inddlim[n\in\N]C_n\in\Ind(\Alg_{\bE_1}(\cE)^{\omega_1}),$ and assume that all maps $C_n\to C_{n+1}$ are compact in $\Alg_{\bE_1}(\cE).$ Note that $\Alg_{\bE_1}(\cE)^{\omega_1}\simeq \Alg_{\bE_1}(\cE^{\omega_1})$ by \cite[Proposition C.2]{E24}, hence each $C_n$ is $\omega_1$-compact as an object of $\cE.$ We also may and will assume that our morphism of $\bE_1$-algebras $f:B\to C$ factors through $C_0.$

Now, we have a direct sequence $(C\tens{C_n}C)_{n\geq 0}$ of $\bE_1$-coalgebras in $C\hy\Mod\hy C.$ Its colimit is isomorphic to $C.$ Hence, by adjunction between $\Coalg_{\bE_1}(C\hy\Mod\hy C)$ and $\Coalg_{\bE_1}(\Ind(C\hy\Mod\hy C)),$ we have a natural morphism
\begin{equation*}
\varphi:\hat{\cY}(C)\to\indlim[n]\cY(C\tens{C_n}C)\text{ in }\Coalg_{\bE_1}(\cS_{C}).
\end{equation*}
We claim that $\varphi$ is an isomorphism. To prove this we can ignore the coalgebra structures. By Proposition \ref{prop:Toen_Vaquie_compact_morphism} below each map $C_{n+1}\tens{C_n}C_{n+1}\to C_{n+1}$ is compact in $C_{n+1}\hy\Mod\hy C_{n+1}$ for $n\geq 0.$ Extending the scalars, we see that each map $C\tens{C_n}C\to C\tens{C_{n+1}}C$ is compact in $C\hy\Mod\hy C.$ This proves that $\varphi$ is an isomorphism. 

Now the required diagonal arrow in \eqref{eq:comm_diag_for_E_1_coalgebras} is obtained as a composition of maps of $\bE_1$-coalgebras
\begin{equation*}
\cY(C\tens{B}C)\to \cY(C\tens{C_0}C)\to \indlim[n]\cY(C\tens{C_n}C)\cong\hat{\cY}(C).
\end{equation*}
This proves the lemma.
\end{proof}

The following is a slightly more sophisticated version of a result of T\"oen and Vaqui\'e \cite{TV07}, with essentially the same proof.

\begin{prop}\label{prop:Toen_Vaquie_compact_morphism}
Let $f:A\to B$ be a compact morphism in $\Alg_{\bE_1}(\cE).$ Then the morphism of $B\hy B$-bimodules $B\tens{A}B\to B$ is compact in $B\hy\Mod\hy B.$ In particular, if $A$ is compact in $\Alg_{\bE_1}(\cE),$ then $A$ is smooth.
\end{prop}

\begin{proof}
We use a modification of the argument from \cite[Proposition 2.14]{TV07}. Recall that for an $\bE_1$-algebra $C$ in $\cE,$ the bimodule $\Omega_C = \Fiber(C\otimes C\to C)$ is the cotangent bimodule for the $\bE_1$-operad. That is, for a $C\hy C$-bimodule $M$ and the corresponding trivial square-zero extension $C\oplus M$ we have an equivalence
\begin{equation*}
\Map_{\Alg_{\bE_1}(\cE)/C}(C,C\oplus M)\simeq \Map_{C\hy\Mod\hy C}(\Omega_C,M),
\end{equation*}
induced by the universal section $C\to C\oplus\Omega_C,$ which informally is given by $x\mapsto (x,1\otimes x-x\otimes 1).$ Note also that the map of $C\hy C$-bimodules $C\to \Omega_C[1]$ is an isomorphism in $\Calk^{\cont}(C\hy\Mod\hy C)$ (because $C\otimes C$ is a compact $C\hy C$-bimodule). 

It follows that the claimed compactness of the map $B\tens{A}B\to B$ of $B\hy B$-bimodules is equivalent to the compactness of the map $B\tens{A}\Omega_A\tens{A}B\to\Omega_B.$ To prove the latter, consider an ind-system of $B\hy B$-bimodules $(M_i)_{i\in I},$ where $I$ is a directed poset. Let $M=\indlim[i]M_i,$ and take some map $g:\Omega_B\to M.$ It corresponds to a section $s:B\to B\oplus M.$ Compactness of $f:A\to B$ in $\Alg_{\bE_1}(\cE)$ implies that the composition $s\circ f:A\to B\oplus M,$ considered as a map in $\Alg_{\bE_1}(\cE)/B,$ factors through $B\oplus M_{i_0}$ for some $i_0\in I.$ This gives a map
\begin{equation*}
A\to A\times_B (B\oplus M_{i_0})\cong A\oplus M_{i_0}\text{ in }\Alg_{\bE_1}(\cE)/A.
\end{equation*}
The latter map corresponds via adjunction to a map of $B\hy B$-bimodules $h:B\tens{A}\Omega_A\tens{A}B\to M_{i_0}.$

We obtain a commutative square
\begin{equation*}
\begin{tikzcd}
B\tens{A}\Omega_A\tens{A}B\ar{r}{h}\ar[d] & M_{i_0}\ar[d]\\
\Omega_B\ar{r}{g} & M.
\end{tikzcd}
\end{equation*}
This proves the compactness of the map $B\tens{A}\Omega_A\tens{A}B\to \Omega_B$ in $B\hy\Mod\hy B.$ 
\end{proof}

We deduce non-trivial corollaries about the class of (absolutely) nuclear left $\cE$-modules.



\begin{cor}\label{cor:abs_nuclear_cat_equiv_cond}
Let $\cC\in\Cat_{\cE}^{\cg}$ be a relatively compactly generated left $\cE$-module. The following are equivalent.
\begin{enumerate}[label=(\roman*),ref=(\roman*)]
	\item $\cC$ is absolutely nuclear over $\cE.$ \label{nuclear_cat_cond1}
	\item For any $\cD\in(\Cat_{\cE}^{\cg})^{\omega_1}$ and for any strongly continuous $\cE$-linear functor $F:\cD\to\cC,$ if $F$ is a compact morphism in $\Cat_{\cE}^{\cg}$ then we have
	\begin{equation}\label{eq:inclusion_in_the_image}
		F\in\im((\un{\Hom}_{\cE}^{\dual}(\cD,\cE)\tens{\cE}\cC)^{\omega,\dec}\to \Fun_{\cE}^{LL}(\cD,\cC)).
	\end{equation} \label{nuclear_cat_cond2}
\end{enumerate}
\end{cor}

\begin{remark}\label{rem:clarification_of_cond_for_abs_nuclear_cat}
Note that if in \eqref{eq:inclusion_in_the_image} we replace the image in $\Cat^{\perf}$ with the usual essential image, we get exactly the condition from Definition \ref{def:absolute_nuclearity}. A priori \ref{nuclear_cat_cond2} is a much weaker condition, which for quite non-trivial reasons (Lemma \ref{lem:map_from_compact_functor}) turns out to be still equivalent to the absolute nuclearity of $\cC.$
\end{remark}

\begin{proof}[Proof of Corollary \ref{cor:abs_nuclear_cat_equiv_cond}]
As we already mentioned in Remark \ref{rem:clarification_of_cond_for_abs_nuclear_cat}, the implication \Implies{nuclear_cat_cond1}{nuclear_cat_cond2} is tautological.

\Implies{nuclear_cat_cond2}{nuclear_cat_cond1}. Let $\cD\in(\Cat_{\cE}^{\cg})^{\omega_1}$ and let $G:\cD\to \cC$ be a strongly continuous $\cE$-linear functor which is a compact morphism in $\Cat_{\cE}^{\cg}.$ Choose a factorization $\cD\xto{G'}\cD'\xto{G''}\cC$ of $G$ in $\Cat_{\cE}^{\cg},$ such that both $G'$ and $G''$ are compact morphisms in $\Cat_{\cE}^{\cg}.$ Applying Lemma \ref{lem:map_from_compact_functor}, we obtain a functor
\begin{equation}\label{eq:tricky_functor}
\im((\un{\Hom}_{\cE}^{\dual}(\cD',\cE)\tens{\cE}\cC)^{\omega,\dec}\to \Fun_{\cE}^{LL}(\cD',\cC))\to (\un{\Hom}_{\cE}^{\dual}(\cD,\cE)\tens{\cE}\cC)^{\omega,\dec},
\end{equation}
which is the diagonal arrow in a commutative diagram similar to \eqref{eq:comm_diagram_for_E_0-rigidity} (with $\cB,\cC,\cD$ replaced with resp. $\cD,\cD',\cC$). Now condition \ref{nuclear_cat_cond2} and our assumptions imply that $F''$ is an object of the source of the functor \eqref{eq:tricky_functor}. Applying this functor to $F'',$ we obtain an object of the category $(\un{\Hom}_{\cE}^{\dual}(\cD,\cE)\tens{\cE}\cC)^{\omega,\dec},$ whose image in $\Fun_{\cE}^{LL}(\cD,\cC)$ is isomorphic to $F.$ This proves the absolute nuclearity of $\cC$ over $\cE.$
\end{proof}

The latter characterization of absolutely nuclear left $\cE$-modules allows to deduce quite non-trivial inheritance properties for this class of $\cE$-modules. We recall the notion of a (possibly infinite) semi-orthogonal decomposition in $\Cat_{\cE}^{\cg},$ similar to \cite[Definition 1.80]{E24}. Namely, for $\cC\in\Cat_{\cE}^{\cg}$ and a poset $P,$ a $P$-indexed semi-orthogonal decomposition in $\Cat_{\cE}^{\cg}$ is a collection of $\cE$-submodules $\cC_i\subset\cC,$ $i\in P,$ such that $\cC_i$ are relatively compactly generated over $\cE,$ the inclusion functors are strongly continuous, the subcategories $\cC_i$ generate $\cC$ via colimits, and $\Hom(x,y)=0$ whenever $x\in\cC_i,$ $y\in\cC_j$ and $i\nleq j.$ In this case we write $\cC=\la \cC_i; i\in P\ra.$

\begin{cor}\label{cor:properties_of_class_of_abs_nuclear}
\begin{enumerate}[label=(\roman*),ref=(\roman*)]
	\item Let $P$ be a poset, and let $\cC=\la \cC_i; i\in P\ra$ be a $P$-indexed semi-orthogonal decomposition in $\Cat_{\cE}^{\cg}.$ If all $\cC_i$ are absolutely nuclear over $\cE,$ then so is $\cC.$ \label{rnuc_closed_under_SOD}
	\item Let $\iota:\cC'\to\cC$ be a monomorphism in $\Cat_{\cE}^{\cg},$ i.e. $\iota$ is a fully faithful strongly continuous $\cE$-linear functor between relatively compactly generated left $\cE$-modules. If $\cC$ is absolutely nuclear over $\cE,$ then $\cC'$ is nuclear over $\cE.$ \label{rnuc_closed_under_taking_subcats}
\end{enumerate}
\end{cor}

\begin{proof}
We prove \ref{rnuc_closed_under_SOD}. Arguing as in the proof of Proposition \ref{prop:properties_of_nuc_E_modules} we see that the class of absolutely nuclear left $\cE$-modules is closed under filtered colimits in $\Cat_{\cE}^{\cg}.$ Using this we reduce to the case of a finite semi-orthogonal decomposition, which then reduces to the case of $2$ components. So let $\cC=\la\cC_1,\cC_2\ra$ be an SOD in $\Cat_{\cE}^{\cg},$ and denote by $i_1:\cC_1\to\cC$ and $i_2:\cC_2\to\cC$ the inclusion functors. We also denote by $i_1^L$ resp. $i_2^R$ the left adjoint to $i_1$ resp. the right adjoint to $i_2$ (both $i_1^L$ and $i_2^R$ are strongly continuous and $\cE$-linear). Suppose that $\cC_1$ and $\cC_2$ are absolutely nuclear over $\cE.$ To prove the absolute nuclearity of $\cC,$ consider a functor $F:\cD\to\cC,$ which is a compact morphism in $\Cat_{\cE}^{\cg}.$ The absolute nuclearity of $\cC_1$ and $\cC_2$ implies the inclusions
\begin{equation}\label{eq:incl_for_C_1}
i_1^L\circ F\in \im((\un{\Hom}_{\cE}^{\dual}(\cD,\cE)\tens{\cE}\cC_1)^{\omega,\dec}\to \Fun_{\cE}^{LL}(\cD,\cC_1)),
\end{equation} 
\begin{equation}\label{eq:incl_for_C_2}
	i_2^R\circ F\in \im((\un{\Hom}_{\cE}^{\dual}(\cD,\cE)\tens{\cE}\cC_2)^{\omega,\dec}\to \Fun_{\cE}^{LL}(\cD,\cC_2)),
\end{equation}
We have the cofiber sequence
\begin{equation*}
i_2\circ i_2^R\circ F\to F\to i_1\circ i_1^L\circ F
\end{equation*}
in $\Fun_{\cE}^{LL}(\cD,\cC).$ Hence, \eqref{eq:incl_for_C_1} and \eqref{eq:incl_for_C_2} imply the inclusion
\begin{equation*}
F\in \im((\un{\Hom}_{\cE}^{\dual}(\cD,\cE)\tens{\cE}\cC)^{\omega}\to \Fun_{\cE}^{LL}(\cD,\cC)).
\end{equation*}
By Corollary \ref{cor:abs_nuclear_cat_equiv_cond} this implies the absolute nuclearity of $\cC.$

We prove \ref{rnuc_closed_under_taking_subcats}. Let $F:\cD\to\cC'$ be a functor which is a compact morphism in $\Cat_{\cE}^{\cg}.$ We choose a factorization $\cD\xto{F'}\cD'\xto{F''}\cC'$ of $F$ in $\Cat_{\cE}^{\cg},$ such that both $F'$ and $F''$ are compact morphisms in $\Cat_{\cE}^{\cg}.$ The absolute nuclearity of $\cC$ implies that there is an object $X\in(\un{\Hom}_{\cE}^{\dual}(\cD',\cE)\tens{\cE}\cC)^{\omega,\dec},$ such that its image in $\Fun_{\cE}^{LL}(\cD',\cC)$ is isomorphic to the composition $\iota\circ F''.$ Denote by $\bar{X}\in (\un{\Hom}_{\cE}^{\dual}(\cD',\cE)\tens{\cE}\cC/\cC')^{\omega,\dec}$ the image of $X.$ By Lemma \ref{lem:map_from_compact_functor}, the functor
\begin{equation*}
(\un{\Hom}_{\cE}^{\dual}(\cD',\cE)\tens{\cE}\cC/\cC')^{\omega,\dec}\to (\un{\Hom}_{\cE}^{\dual}(\cD,\cE)\tens{\cE}\cC/\cC')^{\omega,\dec}
\end{equation*}
factors through
\begin{equation*}
\cV = \im((\un{\Hom}_{\cE}^{\dual}(\cD',\cE)\tens{\cE}\cC/\cC')^{\omega,\dec}\to \Fun_{\cE}^{LL}(\cD',\cC/\cC')).
\end{equation*}
By construction, the image of $\bar{X}$ in $\cV$ vanishes, hence also the image of $\bar{X}$ in $(\un{\Hom}_{\cE}^{\dual}(\cD,\cE)\tens{\cE}\cC/\cC')^{\omega,\dec}$ vanishes. But we have a short exact sequence
\begin{equation*}
0\to \un{\Hom}_{\cE}^{\dual}(\cD,\cE)\tens{\cE}\cC'\to \un{\Hom}_{\cE}^{\dual}(\cD,\cE)\tens{\cE}\cC\to\un{\Hom}_{\cE}^{\dual}(\cD,\cE)\tens{\cE}\cC/\cC'\to 0
\end{equation*}
in $\Cat_{\st}^{\dual}.$ It follows that the image of $X$ in $(\un{\Hom}_{\cE}^{\dual}(\cD,\cE)\tens{\cE}\cC)^{\omega}$ is isomorphic to an image of some object $Y\in (\un{\Hom}_{\cE}^{\dual}(\cD,\cE)\tens{\cE}\cC')^{\omega}.$ By construction, the image of $Y$ in $\Fun_{\cE}^{LL}(\cD',\cC')$ is isomorphic to $F.$ This proves the nuclearity of $\cC'$ over $\cE.$
\end{proof}

Essentially the only remaining ingredient for the proof of Theorem \ref{th:dualizability_and_rigidity} is the following construction of a ``resolution of length $1$'' by nuclear left $\cE$-modules. We need some notation. Let $\cA$ be a small category enriched over $\cE,$ and denote the $\cE$-valued morphisms by $\cA(x,y)\in\cE$ for $x,y\in\cA.$ Given a presentable left $\cE$-module $\cD,$ we can consider $\cD$ as a category enriched over $\cE,$ and denote by $\Fun(\cA,\cD)$ the category of $\cE$-enriched functors. We also have the $\cE^{mop}$-enriched category $\cA^{op},$ and the category $\Fun(\cA^{op},\cE)$ is naturally a left $\cE$-module, which is relatively compactly generated. More precisely, for $x\in\cA$ we put $\h_x=\cA(-,x)\in\Fun(\cA^{op},\cE).$ Then $\h_x$ is compact, and the collection $\{\h_x\}_{x\in\cA}$ generates $\Fun(\cA^{op},\cE)$ as a left $\cE$-module. For a presentable left $\cE$-module $\cD$ we have an equivalence
\begin{equation*}
\Fun(\cA,\cD)\simeq\Fun_{\cE}^L(\Fun(\cA^{op},\cE),\cD).
\end{equation*}

\begin{prop}\label{prop:resolution_by_nuclear}
Let $\cA$ be a small category enriched over $\cE.$ We define the $\cE$-enriched category $\cB$ as follows. Its set of objects is given by 
\begin{equation*}
\Ob(\cB)=\Ob(\cA)\times\N.
\end{equation*} 
The morphisms in $\cB$ are given by
\begin{equation*}
\cB((x,n),(y,m)) = \begin{cases} \cA(x,y) & \text{for }n<m;\\
1_{\cE} & \text{for }n=m,\,x=y;\\
0 & \text{else}.\end{cases}
\end{equation*}
The composition in $\cB$ is induced by the composition in $\cA$ in the natural way. 

\begin{enumerate}[label=(\roman*),ref=(\roman*)]
	\item Consider the strongly continuous $\cE$-linear functor $\Phi:\Fun(\cB^{op},\cE)\to \Fun(\cA^{op},\cE),$ given on the representable functors by $\h_{(x,n)}\mapsto \h_x.$ Then we have a short exact sequence in $\Cat_{\cE}^{\cg}:$
	\begin{equation}\label{eq:resolution_by_abs_nuclear}
		0\to \cC\to\Fun(\cB^{op},\cE)\xto{\Phi}\Fun(\cA^{op},\cE)\to 0.
	\end{equation}
	Here $\cC\subset\Fun(\cB^{op},\cE)$ is the localizing $\cE$-submodule generated by the compact objects $Z_{(x,n)} = \Fiber(\h_{(x,n)}\xto{1_x} \h_{(x,n+1)}),$ $x\in\Ob(\cA),$ $n\geq 0.$ \label{description_of_resolution}
	\item The relatively compactly generated left $\cE$-modules $\Fun(\cB^{op},\cE)$ and $\cC$ are nuclear over $\cE.$ \label{nuclearity_of_resolution}
	\item  Suppose that in addition the set $\Ob(\cA)$ is at most countable and the objects $\cA(x,y)\in\cE$ are $\omega_1$-compact for $x,y\in\Ob(\cA).$ Then both $\Fun(\cB^{op},\cE)$ and $\cC$ are $\omega_1$-compact in $\Cat_{\cE}^{\cg},$ hence they are basic nuclear over $\cE.$ \label{basic_nuclearity_of_resolution}
\end{enumerate}
\end{prop}

\begin{proof}
We prove \ref{description_of_resolution}. The argument showing the short exactness of \eqref{eq:resolution_by_abs_nuclear} is quite standard, we give it for completeness. By definition of $\Phi,$ its right adjoint is given by $\Phi^R(M)((x,n)) = M(x).$ In particular, we have an isomorphism
\begin{equation}\label{eq:F^R_of_A_as_colimit}
\Phi^R(\h_x)\xto{\sim}\indlim[n]h_{(x,n)},\quad x\in\Ob(\cA).
\end{equation}
which is seen by evaluating at each $(y,k)\in\Ob(\cB).$ Hence, the adjunction counit is an isomorphism on $\h_x,$ i.e. $\Phi(\Phi^R(\h_x))\xto{\sim} \h_x$ for $x\in\Ob(\cA).$ Therefore, $\Phi\circ \Phi^R\cong\id,$ i.e. $\Phi$ is a quotient functor. Furthermore, it follows from \eqref{eq:F^R_of_A_as_colimit} that for each $(x,k)\in\Ob(\cB)$ we have
\begin{equation*}
\Fiber(\h_{(x,k)}\to \Phi^R(\Phi(\h_{(x,k)})))\cong \indlim[n\geq k] \Fiber(\h_{(x,k)}\xto{1_x} \h_{(x,n)}).
\end{equation*}
For each $n\geq k$ the object $\Fiber(\h_{(x,k)}\xto{1_x} \h_{(x,n)})$ is contained in $\cC,$ since has a finite filtration with subquotients $Z_{(x,i)},$ $k\leq i\leq n-1.$ This proves that the kernel of $\Phi$ is exactly $\cC.$

Next, we prove \ref{nuclearity_of_resolution}. By construction, for each $(x,n)\in\Ob(\cB)$ the functor $G_{x,n}:\cE\to \Fun(\cB^{op},\cE),$ $G_{x,n}(w)=w\otimes h_{(x,n)},$ is fully faithful. Its right adjoint is given by the evaluation at $(x,n).$ In particular, if $n>m$ or $n=m,$ $x\ne y,$ then $G_{x,n}^R\circ G_{y,m} = 0.$ This shows that we have semi-orthogonal decomposition, indexed by $\Ob(\cB):$
\begin{equation*}
\Fun(\cB^{op},\cE)=\la G_{x,n}(\cE); (x,n)\in\Ob(\cB)\ra
\end{equation*}
Here the partial order on $\Ob(\cB)$ is given by
\begin{equation*}
(x,n)\leq (y,m)\quad\text{if}\quad n< m\text{ or }(x,n)=(y,m).
\end{equation*}
Since $\cE$ is absolutely nuclear over $\cE,$ so is $\Fun(\cB^{op},\cE)$ by Corollary \ref{cor:properties_of_class_of_abs_nuclear}. Moreover, by loc. cit. the left $\cE$-module $\cC$ is nuclear. This proves \ref{nuclearity_of_resolution}.

Finally, to prove \ref{basic_nuclearity_of_resolution} note that by assumption both $\cE$-modules $\Fun(\cB^{op},\cE)$ and $\cC$ have countable generating collections of compact objects, and the internal $\Hom$ over $\cE$ between these objects takes values in $\cE^{\omega_1}.$ Hence, both $\Fun(\cB^{op},\cE)$ and $\cC$ are $\omega_1$-compact in $\Cat_{\cE}^{\cg}.$ The basic nuclearity follows from nuclearity and Proposition \ref{prop:properties_of_nuc_E_modules}. 
\end{proof}

\begin{proof}[Proof of Theorem \ref{th:dualizability_and_rigidity}]
We prove \ref{E_0_rigidity}. It is convenient to introduce the following naive version of $2$-motives $2\hy\Mot^{\naive}$ as a category enriched over $\Pr^L_{\st}.$ The objects are rigid $\bE_1$-monoidal categories, and the morphisms are given by $2\hy\Mot^{\naive}(\cE_1,\cE_2) = \Mot^{\loc}_{\cE_1^{mop}\otimes \cE_2}.$ To prove part \ref{E_0_rigidity} of the theorem we will apply Theorem \ref{th:conditions_for_E_0_rigidity} to the category $2\hy\Mot^{\naive}$ and to its objects $X=\Sp,$ $Y=\cE.$ It suffices to prove that the conditions of loc. cit. are satisfied.

We know the compactness of $1_{\Sp}=\cU_{\loc}(\Sp)\in \Mot^{\loc}=2\hy\Mot^{\naive}(\Sp,\Sp)$ by \cite{BGT}, i.e. the condition \ref{1_X_compact} of Theorem \ref{th:conditions_for_E_0_rigidity} is satisfied. Next, let $\cD,\cD'\in\Cat_{\cE}^{\dual}$ and let $F:\cD\to\cD'$ be a strongly continuous $\cE$-linear functor which is right trace-class over $\cE.$ Then the induced map $\cU_{\loc}(\cD)\to\cU_{\loc}(\cD')$ is right trace-class in $2\hy\Mot^{\naive}(\Sp,\cE).$ It follows that if $\cD\in\Cat_{\cE}^{\cg}$ is basic nuclear, then $\cU_{\loc}(\cD)\in 2\hy\Mot^{\naive}(\Sp,\cE)$ is a sequential colimit with right trace-class transition maps. Thus, to check the condition \ref{cond_right_trace_class} of Theorem \ref{th:conditions_for_E_0_rigidity}, it suffices to prove that the category $\Mot^{\loc}_{\cE}$ is generated as a localizing subcategory by the objects of the form $\cU_{\loc}(\cD),$ where $\cD$ is basic nuclear over $\cE.$ Moreover, once this is established, the condition \ref{cond_left_trace_class} of Theorem \ref{th:conditions_for_E_0_rigidity} would follow by replacing $\cE$ with $\cE^{mop}.$

Now, any $\cD\in\Cat_{\cE}^{\cg}$ is a filtered colimit of full sub-$\cE$-modules generated by a single compact object (equivalently, by a finite collection of compact objects). Hence, the category $\Mot^{\loc}_{\cE}$ is generated by objects of the form $\cU_{\loc}(\Mod\hy A)$ for an $\bE_1$-algebra $A.$ Moreover, any such $A$ is an filtered colimit of $\omega_1$-compact $\bE_1$-algebras, so $\Mot^{\loc}_{\cE}$ is generated by the collection $\{\cU_{\loc}(\Mod\hy A),\,A\in\Alg_{\bE_1}(\cE)^{\omega_1}\}.$ By Proposition \ref{prop:resolution_by_nuclear}, for any such $A$ we have a short exact sequence of the form
\begin{equation*}
0\to\cC\to\cD\to\Mod\hy A\to 0
\end{equation*}
in $\Cat_{\cE}^{\cg},$ where $\cC$ and $\cD$ are basic nuclear over $\cE.$ It gives a cofiber sequence
\begin{equation*}
\cI_{\loc}(\cC)\to\cU_{\loc}(\cD)\to\cU_{\loc}(\Mod\hy A),
\end{equation*}
which proves that the category $\Mot^{\loc}_{\cE}$ is generated by the motives of basic nuclear left $\cE$-modules. This proves part \ref{E_0_rigidity} of the theorem.

Next, we deduce part \ref{E_1_rigidity} of the theorem from the proof of part \ref{E_0_rigidity}. Suppose that $\cE$ is a rigid $\bE_2$-monoidal category, so the categories $\Cat_{\cE}^{\dual},$ $\Cat_{\cE}^{\cg}$ and $\Mot^{\loc}_{\cE}$ are $\bE_1$-monoidal. We know that the unit object $\cU_{\loc}(\cE)$ is compact in $\Mot^{\loc}_{\cE}.$ It remains to prove that $\Mot^{\loc}_{\cE}$ is generated by sequential colimits in which the transition maps are both left and right trace-class.

We make the following almost tautological observation: if a morphism $F:\cD\to\cD'$ in $\Cat_{\cE}^{\dual}$ comes from an object $\tilde{F}\in(\un{\Hom}_{\cE}^{\dual}(\cD,\cE)\tens{\cE}\cD')^{\omega},$ then $F$ is both left and right trace-class in the $\bE_1$-monoidal category $\Cat_{\cE}^{\dual}.$ Hence, in this case $\cU_{\loc}(F):\cU_{\loc}(\cD)\to\cU_{\loc}(\cD')$ is both left and right trace-class in $\Mot^{\loc}_{\cE}.$ Therefore, if $\cC\in\Cat_{\cE}^{\cg}$ is basic nuclear in the sense of Definition \ref{def:nuclear_E_modules}, then $\cU_{\loc}(\cC)$ is a sequential colimit in which the transition maps are both left and right trace-class in $\Mot^{\loc}_{\cE}.$ By the above, the category $\Mot^{\loc}_{\cE}$ is generated by the motives of basic nuclear $\cE$-modules, which proves the rigidity of $\Mot^{\loc}_{\cE}.$   
\end{proof}

\subsection{Proper $\cE$-modules are nuclear}

As in the previous subsection, we fix a rigid $\bE_1$-monoidal category $\cE.$ The following statement will be very useful for computations of morphisms in the categories of localizing motives.

\begin{prop}\label{prop:proper_are_nuclear}
Let $\cE$ be a rigid $\bE_1$-monoidal category, and let $\cC\in\Cat_{\cE}^{\cg}$ be a proper relatively compactly generated left $\cE$-module. Then $\cC$ is nuclear over $\cE.$
\end{prop}

To prove this we will need some general properties of proper $\cE$-modules. The following statement is non-trivial only when $\cE$ is not compactly generated.

\begin{prop}\label{prop:proper_are_colimits_of_finitely_presented}
Let $\cC\in\Cat_{\cE}^{\cg}$ be proper over $\cE.$ Then $\cC\simeq\indlim[i\in I]\cC_i,$ where $I$ is directed and $\cC_i\in(\Cat_{\cE}^{\cg})^{\omega}$ for $i\in I.$
\end{prop}

\begin{proof}
Since $\cC$ is a filtered colimit of $\cE$-submodules generated by a single compact object, we may and will assume that $\cC$ itself is generated by a single compact object. This means that $\cC\simeq \Mod\hy A,$ where $A\in\Alg_{\bE_1}(\cE).$ Properness of $\cC$ means that the underlying object of $A$ is compact in $\cE.$

It suffices to prove that $A\cong \indlim[n] A_n,$ where $A_n\in(\Alg_{\bE_1}(\cE))^{\omega}$ for $n\in\N.$ Such a sequence $(A_n)$ can be constructed inductively exactly as in \cite[Proof of Proposition 5.18]{E25}, using the deformed tensor algebra. We use the notation from Subsection \ref{ssec:deformed_tensor_algebra}. We put $A_0=1_{\cE},$ and $A_{n+1}=T_{A_n}^{\deff}(\Cone(A_n\to A))$ for $n\in\N.$ 
Here we consider $\Cone(A_n\to A)$ with its natural morphism to $A_n[1]$ as an object of the category $(A_n\hy\Mod\hy A_n)_{/A_n[1]}.$ By construction, on the underlying objects of $\cE$ we have a factorization $A_n\to A\to A_{n+1}.$ Hence, we have $\indlim[n]A_n\xto{\sim} A.$

It remains to check that $A_n$ is compact in $\Alg_{\bE_1}(\cE).$ We use induction on $n.$ Clearly, $A_0$ is compact. Suppose that $A_n$ is compact for some $n.$ By Proposition \ref{prop:Toen_Vaquie_compact_morphism} $A_n$ is smooth over $\cE.$ Since $A\in\cE^{\omega},$ we deduce that $A$ is compact as an $A_n\hy A_n$-bimodule. Hence, the object $\Cone(A_n\to A)$ is compact in $(A_n\hy\Mod\hy A_n)_{/A_n[1]}.$ Therefore $A_{n+1}$ is compact in $\Alg_{\bE_1}(\cE)_{A_n/},$ which implies that $A_{n+1}$ is compact in $\Alg_{\bE_1}(\cE).$
\end{proof}

\begin{prop}\label{prop:from_smooth_to_proper}
Let $\cC,\cD\in\Cat_{\cE}^{\dual}$ such that $\cC$ is proper and $\cD$ is smooth over $\cE.$ Then we have an equivalence
\begin{equation}\label{eq:from_smooth_to_proper}
\un{\Hom}_{\cE}^{\dual}(\cD,\cE)\tens{\cE}\cC\xto{\sim} \un{\Hom}_{\cE}^{\dual}(\cD,\cC).
\end{equation}
In particular, any strongly continuous $\cE$-linear functor $\cD\to\cC$ is right trace-class over $\cE.$
\end{prop}

\begin{proof}
Denote by $\Phi$ the functor \eqref{eq:from_smooth_to_proper}. Since $\cD$ is smooth, by \cite[Proposition 3.4]{E25} we have fully faithful strongly continuous functor $\Psi:\un{\Hom}_{\cE}^{\dual}(\cD,\cC)\hto \cD^{\vee}\tens{\cE}\cC.$ This in particular applies to the case $\cD=\cE,$ hence in general the composition $
\Psi\circ\Phi$ is also fully faithful, therefore so is $\Phi.$ 
It suffices to prove that the essential image of $\Psi$ is contained in the essential image of $\Psi\circ\Phi$ inside $\cD^{\vee}\tens{\cE}\cC.$

Recall that properness of $\cC$ over $\cE$ means that the relative evaluation $\cC\otimes\cC^{\vee}\to\cE$ is strongly continuous. We choose an uncountable regular cardinal $\kappa$ such that $\cC$ is $\kappa$-compact in $\Cat_{\cE}^{\dual}$ (hence $\cC^{\vee}$ is $\kappa$-compact in $\Cat_{\cE^{mop}}^{\dual}$). We consider the following composition of strongly continuous functors:
\begin{multline*}
\Psi':\Hom_{\cE}^{\dual}(\cD,\cC)\to \Hom_{\cE\otimes\cE^{mop}}^{\dual}(\cD\otimes\cC^{\vee},\cC\otimes\cC^{\vee})\to \Hom_{\cE\otimes\cE^{mop}}^{\dual}(\cD\otimes\cC^{\vee},\cE)\\
\simeq \Hom_{\cE^{mop}}^{\dual}(\cC^{\vee},\un{\Hom}_{\cE}^{\dual}(\cD,\cE))
\hto \Ind((\un{\Hom}_{\cE}^{\dual}(\cD,\cE)\tens{\cE}\cC)^{\kappa}).
\end{multline*}
Recall that by \cite[Proposition 1.89]{E24} the functor $\Ind((-)^{\kappa}):\Prr^L_{\st,\kappa}\to\Cat_{\st}^{\dual}$ is right adjoint to the (not fully faithful) inclusion. Hence, by construction the following square commutes:
\begin{equation}\label{eq:diagram_of_inclusions}
\begin{tikzcd}
\un{\Hom}_{\cE}^{\dual}(\cD,\cC)\ar[r, "\Psi'"]\ar[d, "\Psi"] & \Ind((\un{\Hom}_{\cE}^{\dual}(\cD,\cE)\tens{\cE}\cC)^{\kappa})\ar{d}{\Ind((\Psi\circ\Phi)^{\kappa})}\\
\cD^{\vee}\tens{\cE}\cC\ar[r, "\hat{\cY}"] & \Ind((\cD^{\vee}\tens{\cE}\cC)^{\kappa}).
\end{tikzcd}
\end{equation}
All the functors in \eqref{eq:diagram_of_inclusions} are fully faithful and strongly continuous. The intersection of the essential images of $\cD^{\vee}\tens{\cE}\cC$ and $\Ind((\un{\Hom}_{\cE}^{\dual}(\cD,\cE)\tens{\cE}\cC)^{\kappa})$ inside $\Ind((\cD^{\vee}\tens{\cE}\cC)^{\kappa})$ is exactly the essential image of the composition
\begin{equation*}
\un{\Hom}_{\cE}^{\dual}(\cD,\cE)\tens{\cE}\cC\xto{\Psi\circ\Phi} \cD^{\vee}\tens{\cE}\cC\xto{\hat{\cY}}\Ind((\cD^{\vee}\tens{\cE}\cC)^{\kappa}).
\end{equation*}
It follows that the essential image of $\Psi$ coincides with the essential image of $\Psi\circ\Phi,$ as stated.
\end{proof}

\begin{proof}[Proof of Proposition \ref{prop:proper_are_nuclear}]
By Propositions \ref{prop:proper_are_colimits_of_finitely_presented} we have $\cC\simeq\indlim[i]\cC_i,$ where $\cC_i\in(\Cat_{\cE}^{\cg})^{\omega}.$ By Proposition \ref{prop:Toen_Vaquie_compact_morphism} each $\cC_i$ is smooth over $\cE.$ By Proposition \ref{prop:from_smooth_to_proper} each functor $\cC_i\to\cC$ is right trace-class over $\cE.$ Any compact morphism $\cD\to\cC$ in $\Cat_{\cE}^{\cg}$ factors through some $\cC_i,$ hence $\cC$ is nuclear over $\cE.$
\end{proof}

We will also use the following closely related statement on proper $\cE$-modules.

\begin{prop}\label{prop:stronger_nuclearity_for_proper}
Let $\cE$ be a rigid $\bE_1$-monoidal category, and let $\cC\in\Cat_{\cE}^{\cg}$ be proper over $\cE.$ Let $\hat{\cY}(\cC)=\inddlim[i\in I]\cC_i$ in $\Ind(\Cat_{\cE}^{\cg}).$ For $i\in I$ denote by $\cT_i\subset \un{\Hom}_{\cE}^{\dual}(\cC_i,\cE)$ the localizing right $\cE$-submodule generated by the image of the composition $\cC^{\vee}\to\un{\Hom}_{\cE}^{\dual}(\cC,\cE)\to\un{\Hom}_{\cE}^{\dual}(\cC_i,\cE).$ Then $\cT_i$ is relatively compactly generated over $\cE,$ the inclusion functor from $\cT_i$ to $\un{\Hom}_{\cE}^{\dual}(\cC_i,\cE)$ is strongly continuous, and we have a pro-equivalence
\begin{equation}\label{eq:pro_equivalence}
\proolim[i] \cT_i\xto{\sim}\proolim[i] \un{\Hom}_{\cE}^{\dual}(\cC_i,\cE).
\end{equation}
in $\Pro(\Cat_{\cE^{mop}}^{\dual}).$
\end{prop}

\begin{proof}
The relative compact generation for $\cT_i$ and the strong continuity of the inclusions follow from \cite[Corollary 1.57]{E25}. We prove the pro-equivalence. By Proposition \ref{prop:proper_are_colimits_of_finitely_presented} we may and will assume that each $\cC_i$ is compact in $\Cat_{\cE}^{\cg}.$ We will also assume that $I$ is directed. We denote by $F_{ij}:\cC_i\to\cC_j$ the transition functors, and by $F_i:\cC_i\to\cC$ the functors to the colimit.

Take some $i\in I.$ We need to prove that for some $j\geq i$ the essential image of the functor
\begin{equation}\label{eq:functor_between_internal_Homs}
\un{\Hom}_{\cE}^{\dual}(\cC_j,\cE)\to \un{\Hom}_{\cE}^{\dual}(\cC_i,\cE)
\end{equation}
is contained in $\cT_i.$ Note that by Proposition \ref{prop:from_smooth_to_proper} we have an equivalence
\begin{equation}\label{eq:from_functors_to_tensor_product}
\Fun_{\cE}^{LL}(\cC_i,\cC)\xto{\sim}(\un{\Hom}_{\cE}^{\dual}(\cC_i,\cE)\tens{\cE}\cC)^{\omega}.
\end{equation}
The proof of loc. cit. in fact gives more: the equivalence \eqref{eq:from_functors_to_tensor_product} takes $F$ to an object $X_i\in (\cT_i\tens{\cE}\cC)^{\omega}.$ By \cite[Proposition 1.71]{E24} there exists $k\geq i$ such that $X_i$ can be lifted to an object $X_{ik}\in (\cT_i\tens{\cE}\cC_k)^{\omega}.$ Arguing as in the proof of Proposition \ref{prop:nuclear_equiv_cond} we see that for some $j\geq k$ the image of $X_{ik}$ in $\Fun_{\cE}^{LL}(\cC_i,\cC_j)$ is isomorphic to $F_{ij}.$ Denote by $X_{ij}\in(\cT_i\tens{\cE}\cC_j)^{\omega}$ the image of $X_{ik}.$ Then the functor \eqref{eq:functor_between_internal_Homs} is identified with the following composition:
\begin{equation*}
\un{\Hom}_{\cE}^{\dual}(\cC_j,\cE)\xto{X_{ij}\boxtimes\id} \cT_i\tens{\cE}\cC_j\otimes \un{\Hom}_{\cE}^{\dual}(\cC_j,\cE)\to\cT_i\to\un{\Hom}_{\cE}^{\dual}(\cC_i,\cE).
\end{equation*}
In particular, it factors through $\cT_i,$ as required.
\end{proof}

\section{Morphisms in the categories of localizing motives}
\label{sec:morphisms_in_Mot^loc}

In this section we prove general results describing the morphisms in the category $\Mot^{\loc}_{\cE}$ for a rigid $\bE_1$-monoidal category $\cE,$ and on the internal $\Hom$ in $\Mot^{\loc}_{\cE}$ in the case when $\cE$ is symmetric monoidal. We do not assume that $\cE$ is compactly generated. We formulate the main statements in Subsection \ref{ssec:general_results_on_morphisms}, but some of the proofs are postponed till Subsection \ref{ssec:internal_injectivity_of_Calkin}. As an application, we prove a generalization of Ramzi-Sosnilo-Winges stating the equivalence between the categories $\Mot^{\loc}_{\cE}$ and $\Mot^{\loc}_{\cE,\omega_1}.$  

\subsection{General results on morphisms in $\Mot_{\cE}^{\loc}$}
\label{ssec:general_results_on_morphisms}

We first prove a result describing morphisms in $\Mot^{\loc}_{\cE}$ from of a motive of a nuclear left $\cE$-module $\cC,$ and a slightly different version in the case when $\cC$ is proper. Here we use the notion of nuclearity from Definition \ref{def:nuclear_E_modules}.

\begin{theo}\label{th:morphisms_in_Mot^loc_via_limits}
Let $\cE$ be a rigid $\bE_1$-monoidal category, and let $\cC\in\Cat_{\cE}^{\cg}$ be a relatively compactly generated left $\cE$-module, and Let $\cD\in\Cat_{\cE}^{\dual}$ be a dualizable left $\cE$-module. Let $\hat{\cY}(\cC)=\inddlim[i\in I]\cC_i$ in $\Ind(\Cat_{\cE}^{\cg}).$  

\begin{enumerate}[label=(\roman*),ref=(\roman*)]
	\item Suppose that $\cC$ is nuclear over $\cE.$ Then we have isomorphisms
	\begin{multline}\label{eq:Hom_via_inverse_limits_dualizable}
		\prolim[i]K^{\cont}(\un{\Hom}_{\cE}^{\dual}(\cC_i,\cE)\tens{\cE}\cD)\xto{\sim} \prolim[i] K^{\cont}(\un{\Hom}_{\cE}^{\dual}(\cC_i,\cD))\\
		\xto{\sim}\Hom_{\Mot^{\loc}_{\cE}}(\cU_{\loc}(\cC),\cU_{\loc}(\cD)).
	\end{multline}
	If $\cE$ is symmetric monoidal, then we have an isomorphisms in $\Mot_{\cE}^{\loc}$
	\begin{multline}\label{eq:internal_Hom_via_inverse_limits_dualizable}
	 \prolim[i]\cU_{\loc}(\un{\Hom}_{\cE}^{\dual}(\cC_i,\cE)\tens{\cE}\cD)\xto{\sim} \prolim[i] \cU_{\loc}(\un{\Hom}_{\cE}^{\dual}(\cC_i,\cD))\\
	 \xto{\sim}\un{\Hom}_{\Mot^{\loc}_{\cE}}(\cU_{\loc}(\cC),\cU_{\loc}(\cD)).  
	\end{multline}
	\label{Hom_via_inverse_limit_dualizable}
	\item Suppose that $\cC$ is proper over $\cE.$ As in Proposition \ref{prop:stronger_nuclearity_for_proper}, for $i\in I$ denote by $\cT_i\subset\un{\Hom}_{\cE}^{\dual}(\cC_i,\cE)$ the localizing right $\cE$-submodule generated by the image of the composition $\cC^{\vee}\to\un{\Hom}_{\cE}^{\dual}(\cC,\cE)\to \un{\Hom}_{\cE}^{\dual}(\cC_i,\cE).$ Then we have an isomorphism
	\begin{equation*}
		\prolim[i] K^{\cont}(\cT_i\tens{\cE}\cD)\xto{\sim} \Hom_{\Mot^{\loc}_{\cE}}(\cU_{\loc}(\cC),\cU_{\loc}(\cD)).
	\end{equation*}
	If $\cE$ is symmetric monoidal, then we have an isomorphism in $\Mot^{\loc}_{\cE}:$
	\begin{equation*}
	\prolim[i] \cU_{\loc}(\cT_i\tens{\cE}\cD)\xto{\sim} \un{\Hom}_{\Mot^{\loc}_{\cE}}(\cU_{\loc}(\cC),\cU_{\loc}(\cD)).
	\end{equation*} \label{Hom_via_inverse_limit_for_proper}
\end{enumerate}
\end{theo}

\begin{proof}
We may and will assume that $I$ is directed, and denote by $F_{ij}:\cC_i\to\cC_j$ the transition functors.

We prove \ref{Hom_via_inverse_limit_dualizable}. We explain the isomorphisms \eqref{eq:Hom_via_inverse_limits_dualizable}, and analogous argument gives \eqref{eq:internal_Hom_via_inverse_limits_dualizable}. By Proposition \ref{prop:nuclear_equiv_cond} for each $i\in I$ there exists $j\geq i$ such that the functor $F_{ij}:\cC_i\to\cC_j$ is right trace-class over $\cE.$ In particular, we have a pro-equivalence
\begin{equation*}
\proolim[i]\un{\Hom}_{\cE}^{\dual}(\cC_i,\cE)\tens{\cE}\cD\xto{\sim}\proolim[i]\un{\Hom}_{\cE}^{\dual}(\cC_i,\cD)
\end{equation*}
in $\Pro(\Cat_{\st}^{\dual}).$ This gives the first isomorphism in \eqref{eq:Hom_via_inverse_limits_dualizable}. It remains to prove that the composition of morphisms in \eqref{eq:Hom_via_inverse_limits_dualizable} is an isomorphism. 

We first fix a pair of elements $i\leq j$ in $I$ such that $F{ij}$ is right trace-class over $\cE.$ We choose a right trace-class witness $\wt{F}_{ij}\in (\un{\Hom}_{\cE}^{\dual}(\cC_i,\cE)\tens{\cE}\cC_j)^{\omega}$ for $F_{ij}.$ We obtain the following morphism in $\Mot^{\loc},$ given by the composition
\begin{multline*}
\un{\Hom}_{\Mot^{\loc}_{\cE}/\Mot^{\loc}}(\cU_{\loc}(\cC_j),\cU_{\loc}(\cD))\\
\xto{[\wt{F}_{ij}]\boxtimes \id} \cU_{\loc}(\un{\Hom}_{\cE}^{\dual}(\cC_i,\cE)\tens{\cE}\cC_j)\otimes \un{\Hom}_{\Mot^{\loc}_{\cE}/\Mot^{\loc}}(\cU_{\loc}(\cC_j),\cU_{\loc}(\cD))\\
\to \cU_{\loc}(\un{\Hom}_{\cE}^{\dual}(\cC_i,\cE)\tens{\cE}\cD).
\end{multline*}
This shows that we have an isomorphism in $\Pro(\Mot^{\loc}):$
\begin{equation*}
\proolim[i]\cU_{\loc}(\un{\Hom}_{\cE}^{\dual}(\cC_i,\cE)\tens{\cE}\cD)\xto{\sim}\xto{\sim} \proolim[i]\un{\Hom}_{\Mot^{\loc}_{\cE}/\Mot^{\loc}}(\cU_{\loc}(\cC_i),\cU_{\loc}(\cD)).
\end{equation*}
Applying the functor $\prolim:\Pro(\Mot^{\loc})\to\Mot^{\loc},$ followed by $\Hom(\cU_{\loc}(\Sp),-):\Mot^{\loc}\to\Sp,$ we see that the composition in \eqref{eq:Hom_via_inverse_limits_dualizable} is an isomorphism.

Now \ref{Hom_via_inverse_limit_for_proper} follows from \ref{Hom_via_inverse_limit_dualizable} and Proposition \ref{prop:stronger_nuclearity_for_proper}: if $\cE$ is $\bE_1$-monoidal resp. symmetric monoidal then we have a pro-equivalence
\begin{equation*}
\proolim[i]\cT_i\tens{\cE}\cD\xto{\sim} \proolim[i]\un{\Hom}_{\cE}^{\dual}(\cC_i,\cE)\tens{\cE}\cD
\end{equation*}
in $\Pro(\Cat_{\st}^{\dual})$ resp. $\Pro(\Cat_{\cE}^{\dual}).$
\end{proof}

We obtain the following non-trivial application.

\begin{cor}\label{cor:U_loc_commutes_with_products}
Let $\cE$ be a rigid $\bE_1$-monoidal category. Then the functor $\cU_{\loc}:\Cat_{\cE}^{\dual}\to \Mot^{\loc}_{\cE}$ commutes with infinite products. If $\cE$ is compactly generated, then the functor $\cU_{\loc}:\Cat_{\cE}^{\cg}\to \Sp$ commutes with infinite products. 
\end{cor}

\begin{proof} If $\cE$ is compactly generated, then by Proposition \ref{prop:prod_dual_vs_Cat_cg} the functor $\Cat_{\cE}^{\cg}\to\Cat_{\cE}^{\dual}$ commutes with products. Hence, the second assertion follows from the first one.
	 
Now consider the general case when $\cE$ is not necessarily compactly generated. By the proof of Theorem \ref{th:dualizability_and_rigidity}, the category $\Mot^{\loc}_{\cE}$ is generated as a localizing subcategory by the objects of the form $\cU_{\loc}(\cC),$ where $\cC\in\Cat_{\cE}^{\cg}$ is basic nuclear. Hence, it suffices to prove that for such $\cC$ the functor
\begin{equation*}
\Hom(\cU_{\loc}(\cC),-):\Mot^{\loc}_{\cE}\to\Sp
\end{equation*}
commutes with products. Let $\hat{\cY}(\cC)=\inddlim[n\in\N]\cC_n$ in $\Ind(\Cat_{\cE}^{\cg}).$ Recall that by \cite[Theorem 4.28]{E24} the functor $K^{\cont}:\Cat_{\st}^{\dual}\to\Sp$ commutes with products. Applying Theorem \ref{th:morphisms_in_Mot^loc_via_limits}, for a family $(\cD_j)_{j\in J}$ in $\Cat_{\cE}^{\dual}$ we obtain
\begin{multline*}
\Hom(\cU_{\loc}(\cC),\cU_{\loc}(\prodd[j]^{\dual}\cD_j))\cong \prolim[n] K^{\cont}(\un{\Hom}_{\cE}^{\dual}(\cC_n,\prodd[j]^{\dual}\cD_j))\\
\cong \prolim[n] K^{\cont}(\prodd[j]^{\dual}\un{\Hom}_{\cE}^{\dual}(\cC_n,\cD_j))\cong \prolim[n] \prodd[j] K^{\cont}(\un{\Hom}_{\cE}^{\dual}(\cC_n,\cD_j))\\
\cong \prodd[j] \Hom(\cU_{\loc}(\cC),\cU_{\loc}(\cD_j)).
\end{multline*}
This proves the corollary.
\end{proof}

We now formulate a much more difficult result which describes morphisms in $\Mot^{\loc}_{\cE}$ in terms of dualizable internal $\Hom$ over $\cE.$ Recall from Subsection \ref{ssec:notation_and_terminology} that for a dualizable category $\cD$ and an uncountable regular cardinal $\kappa$ we use the notation 
\begin{equation*}
\Calk_{\kappa}(\cD)=\Ind(\Calk_{\kappa}^{\cont}(\cD))\simeq \ker(\colim:\Ind(\cC^{\kappa})\to \cC).
\end{equation*}
For any $n\geq 1$ we denote by $\Calk_{\kappa}^n(\cD)$ the $n$-th iteration of this construction, i.e. $\Calk_{\kappa}^1(\cD)=\Calk_{\kappa}(\cD)$ and $\Calk_{\kappa}^{n+1}(\cD)=\Calk_{\kappa}(\Calk_{\kappa}^n(\cD)).$ If $\cD$ is a dualizable left module over a rigid $\bE_1$-monoidal category $\cE,$ then so are all $\Calk_{\kappa}^n(\cD).$ 

\begin{theo}\label{th:morphisms_in_Mot^loc_via_internal_Hom}
Let $\cE$ be a rigid $\bE_1$-monoidal category. Let $\cC$ and $\cD$ be dualizable left $\cE$-modules, and suppose that $\cC$ is $\omega_1$-compact in $\Cat_{\cE}^{\dual}.$
\begin{enumerate}[label=(\roman*),ref=(\roman*)]
\item The category $\un{\Hom}_{\cE}^{\dual}(\cC,\Calk_{\omega_1}^2(\cD))$ is compactly generated. 
We have the following isomorphisms:
\begin{multline}\label{eq:morphisms_in_Mot^loc_via_Calk^2}
\Omega^2 K(\Fun_{\cE}^{LL}(\cC,\Calk_{\omega_1}^2(\cD)))\xto{\sim}\Omega K^{\cont}(\un{\Hom}_{\cE}^{\dual}(\cC,\Calk_{\omega_1}(\cD)))\\
\xto{\sim}\Hom_{\Mot_{\cE}^{\loc}}(\cU_{\loc}(\cC),\cU_{\loc}(\cD)).
\end{multline}
If $\cE$ is symmetric monoidal, then we have isomorphisms in $\Mot_{\cE}^{\loc}:$
\begin{multline}\label{eq:internal_Hom_in_Mot^loc_via_Calk^2}
\Omega^2\cU_{\loc}(\Ind(\Fun_{\cE}^{LL}(\cC,\Calk_{\omega_1}^2(\cD))))\xto{\sim}\Omega\cU_{\loc}(\un{\Hom}_{\cE}^{\dual}(\cC,\Calk_{\omega_1}(\cD)))\\
\xto{\sim}\un{\Hom}_{\Mot_{\cE}^{\loc}}(\cU_{\loc}(\cC),\cU_{\loc}(\cD)).
\end{multline} \label{internal_Hom_into_Calk}
\item Suppose that $\cC$ is proper. Then the category $\un{\Hom}_{\cE}^{\dual}(\cC,\Calk_{\omega_1}(\cD))$ is compactly generated. We have the following isomorphisms: 
\begin{multline*}
	\Omega K(\Fun_{\cE}^{LL}(\cC,\Calk_{\omega_1}(\cD)))\xto{\sim} K^{\cont}(\un{\Hom}_{\cE}^{\dual}(\cC,\cD))\\
	\xto{\sim}\Hom_{\Mot_{\cE}^{\loc}}(\cU_{\loc}(\cC),\cU_{\loc}(\cD)).
\end{multline*}
If $\cE$ is symmetric monoidal, then we have isomorphisms in $\Mot_{\cE}^{\loc}:$
\begin{multline*}
	\Omega\cU_{\loc}(\Ind(\Fun_{\cE}^{LL}(\cC,\Calk_{\omega_1}(\cD))))\xto{\sim}\cU_{\loc}(\un{\Hom}_{\cE}^{\dual}(\cC,\cD))\\
	\xto{\sim}\un{\Hom}_{\Mot_{\cE}^{\loc}}(\cU_{\loc}(\cC),\cU_{\loc}(\cD)).
\end{multline*} \label{internal_Hom_from_proper}
\end{enumerate}
\end{theo}

The proof of this theorem is given in Subsection \ref{ssec:proof_of_theorem_on_morphisms_in_Mot^loc}. We note that part \ref{internal_Hom_into_Calk} was used in \cite[Proof of Theorem 6.1]{RSW25} in the case $\cE=\Sp$ and $\cC,\cD\in\Cat_{\st}^{\cg}$ to prove the equivalence $\Mot^{\loc}_{\omega_1}\xto{\sim}\Mot^{\loc}.$ In fact, the general case of Theorem \ref{th:morphisms_in_Mot^loc_via_internal_Hom} allows to prove that for any rigid $\bE_1$-monoidal category $\cE$ we have an equivalence $\Mot^{\loc}_{\cE,\omega_1}\xto{\sim}\Mot^{\loc}_{\cE},$ see Theorem \ref{th:equivalence_Mot^loc_omega_1_and_Mot^loc} below.  

The first assertion in \ref{internal_Hom_from_proper} has already appeared in \cite[Proof of Corollary 3.14]{E25}. We deduce the isomorphisms in \ref{internal_Hom_from_proper} from part \ref{internal_Hom_into_Calk} using our internal projectivity result \cite[Theorem 3.6, Corollary 3.14]{E25}. To prove \ref{internal_Hom_into_Calk} we give a detailed study of the injectivity properties of Calkin categories. First, in Subsection \ref{ssec:formally_injective_categories} we introduce and study the notion of formal $\omega_1$-injectivity for dualizable categories (Definition \ref{def:formal_omega_1_injectivity}). In Subsection \ref{ssec:formal_injectivity_of_Calkin} we prove that Calkin categories are formally $\omega_1$-injective (Theorem \ref{th:hat_Y_for_Calkin}). Then in Subsection \ref{ssec:internal_injectivity_of_Calkin} we prove general results about the (relative) dualizable internal $\Hom$ into formally $\omega_1$-injective categories, including the (relative) internal $\omega_1$-injectivity statement (Theorem \ref{th:internal_omega_1_injectivity}). These results together with our theorem on localizing invariants of inverse limits \cite[Theorem 6.1]{E25} allow to prove Theorem \ref{th:morphisms_in_Mot^loc_via_internal_Hom}.

\subsection{Formally $\omega_1$-injective dualizable categories}
\label{ssec:formally_injective_categories}

We introduce the following class of dualizable categories.

\begin{defi}\label{def:formal_omega_1_injectivity}
We say that a dualizable category $\cD$ is formally $\omega_1$-injective if the functor $\hat{\cY}:\cD\to\Ind(\cD^{\omega_1})$ commutes with countable limits.
\end{defi}

Equivalently, we can replace ``countable limits'' with ``countable products''. To give a feeling of this notion we first make the following trivial observation.

\begin{prop}\label{prop:when_hat_Y_is_cocontinuous}
Let $\cD$ be a dualizable category. The following are equivalent.
\begin{enumerate}[label=(\roman*),ref=(\roman*)]
\item The functor $\hat{\cY}:\cD\to\Ind(\cD^{\omega_1})$ commutes with all limits. \label{hat_Y_cocontinuous}
\item $\cD$ is compactly generated and the category $\cD^{\omega}$ has countable colimits (equivalently, countable coproducts). \label{compact_objects_countably_cocomplete}
\end{enumerate}
\end{prop}

\begin{proof}
\Implies{hat_Y_cocontinuous}{compact_objects_countably_cocomplete}. Commutation of $\hat{\cY}$ with limits means that it has a left adjoint $F:\Ind(\cD^{\omega_1})\to\cD,$ Then $F$ is a strongly continuous quotient functor, hence $\cD$ is compactly generated and $F$ preserves compact objects. The induced functor $\cD^{\omega_1}\simeq\Ind(\cD^{\omega})^{\omega_1}\to\cD^{\omega}$ is left adjoint to the Yoneda embedding. Hence, the category $\cD^{\omega}$ has sequential colimits, therefore all countable colimits.

The implication \Implies{compact_objects_countably_cocomplete}{hat_Y_cocontinuous} is proved by reversing the previous argument.
\end{proof}

However, the formal $\omega_1$-injectivity is much less restrictive than the condition \ref{hat_Y_cocontinuous} of Proposition \ref{prop:when_hat_Y_is_cocontinuous}. For example, the Calkin categories certainly don't satisfy the condition \ref{compact_objects_countably_cocomplete}, but they turn out to be formally $\omega_1$-injective by Theorem \ref{th:hat_Y_for_Calkin} below.

It will be convenient to have equivalent reformulations of the formal $\omega_1$-injectivity. We first prove the following lemma, which is almost formal.

\begin{lemma}\label{lem:when_ev_with_x_commutes_with_countable_limits}
	Let $\cD$ be a dualizable category. For an object $x\in\cD,$ the following conditions are equivalent.
	\begin{enumerate}[label=(\roman*),ref=(\roman*)]
		\item The functor $\ev_{\cD}(x,-):\cD^{\vee}\to\Sp$ commutes with countable limits. \label{cond_on_ev_with_x}
		\item The object $\hat{\cY}(x)\in\Ind(\cD^{\omega_1})$ is contained in $\Ind_{\omega_1}(\cD^{\omega_1})\subset\Ind(\cD^{\omega_1}).$ \label{cond_on_hat_Y_of_x}
		\item For any $y\in\cD^{\omega_1},$ we have an isomorphism
		\begin{equation*}
			\ev_{\cD}(x,y^{\vee})\xto{\sim}\Hom_{\cD}(y,x).
		\end{equation*} \label{cond_on_comp_maps_to_x}
		\item The image of $x$ in $\Calk^{\cont}(\cD)$ is in the right orthogonal to the full subcategory $\Calk_{\omega_1}^{\cont}(\cD).$ \label{cond_on_orthogonality_in_Calk}
	\end{enumerate}
\end{lemma}

\begin{proof}
	The equivalence \Iff{cond_on_comp_maps_to_x}{cond_on_orthogonality_in_Calk} is tautological.
	
	\Implies{cond_on_hat_Y_of_x}{cond_on_ev_with_x}. Suppose that $\hat{\cY}(x)=\inddlim[i\in I]x_i,$ where $I$ is an $\omega_1$-directed poset and $x_i\in\cD^{\omega_1}.$ We have a pro-isomorphism $\proolim[i]x_i^{\vee}\cong \proolim[j\in J]y_j,$ where $J$ is $\omega_1$-codirected and $y_j\in(\cD^{\vee})^{\omega_1}.$ Hence, we obtain isomorphisms of functors
	\begin{equation*}
		\ev_{\cD}(x,-)\cong \indlim[i]\Hom_{\cD^{\vee}}(x_i^{\vee})\cong \indlim[j]\Hom_{\cD^{\vee}}(y_j,-).
	\end{equation*}
	The latter is an $\omega_1$-directed colimit of functors which commute with countable limits. Hence, it also commutes with countable limits.
	
	\Implies{cond_on_ev_with_x}{cond_on_comp_maps_to_x}. Let $y\in\cD^{\omega_1},$ with $\hat{\cY}(y)=\inddlim[n\in\N]y_n.$ Then we have
	\begin{equation*}
		\ev_{\cD}(x,y^{\vee})\cong \prolim[n]\ev_{\cD}(x,y_n^{\vee})\cong\prolim[n]\Hom_{\cD}(y_n,x)\cong \Hom_{\cD}(y,x).
	\end{equation*}
	
	\Implies{cond_on_comp_maps_to_x}{cond_on_hat_Y_of_x}. Under the equivalence $\Ind(\cD^{\omega_1})\simeq\Fun((\cD^{\omega_1})^{op},\Sp)$ the object $\hat{\cY}(x)$ corresponds to the functor $\ev_{\cD}(x,(-)^{\vee}).$ Hence, the condition \ref{cond_on_comp_maps_to_x} exactly means that the object $\hat{\cY}(x)\in\Ind(\cD^{\omega_1})$ is identified with the image of $x$ under the composition $\cD\simeq\Ind_{\omega_1}(\cD^{\omega_1})\hto \Ind(\cD^{\omega_1}).$ This proves the implication.
\end{proof}

\begin{prop}\label{prop:formal_omega_1_injectivity_reformulations}
	Let $\cD$ be a dualizable category. The following are equivalent.
	\begin{enumerate}[label=(\roman*),ref=(\roman*)]
		\item $\cD$ is formally $\omega_1$-injective. \label{hat_Y_for_D_countably_cocont}
		\item $\cD^{\vee}$ is formally $\omega_1$-injective. \label{hat_Y_for_D_dual_countably_cocont}
		\item For any $x\in\cD^{\omega_1}$ and for any $y\in (\cD^{\vee})^{\omega_1}$ we have an isomorphism
		\begin{equation}\label{eq:ev_of_duals}
			\ev_{\cD}(y^{\vee}, x^{\vee})\xto{\sim}\Hom_{\cD^{\vee}\otimes\cD}(y\boxtimes x,\coev_{\cD}(\bS)).
		\end{equation}
		More precisely, the map \eqref{eq:ev_of_duals} is given by the composition
		\begin{equation*}
			\ev_{\cD}(y^{\vee}, x^{\vee})\to \Hom_{\cD}(x,y^{\vee})\\
			\cong \Hom_{\cD^{\vee}\otimes\cD}(y\boxtimes x,\coev_{\cD}(\bS)),
		\end{equation*}
		or equivalently by a similar composition with the roles of $x$ and $y$ interchanged. \label{ev_of_duals}
	\end{enumerate}
\end{prop}

\begin{remark}
	A more symmetric description of the map \eqref{eq:ev_of_duals} is the following. We have a natural map
	\begin{equation*}
		\alpha:y\boxtimes y^{\vee}\boxtimes x^{\vee}\boxtimes x\to \coev(\bS)\boxtimes\coev(\bS)\quad\text{in }\cD^{\vee}\otimes\cD\otimes\cD^{\vee}\otimes\cD.
	\end{equation*}
	Applying the functor
	\begin{equation*}
		\id\boxtimes\ev_{\cD}\boxtimes\id:\cD^{\vee}\otimes\cD\otimes\cD^{\vee}\otimes\cD\to \cD^{\vee}\otimes\cD
	\end{equation*}
	to $\alpha,$ we obtain the map
	\begin{equation*}
		\beta:\ev_{\cD}(y^{\vee},x^{\vee})\otimes (y\boxtimes x)\to \coev(\bS)\quad \text{in }\cD^{\vee}\otimes\cD.
	\end{equation*}
	By adjunction the map $\beta$ corresponds to the map \eqref{eq:ev_of_duals}.
\end{remark}

\begin{proof}[Proof of Proposition \ref{prop:formal_omega_1_injectivity_reformulations}]
	
	
	We prove \Iff{hat_Y_for_D_countably_cocont}{ev_of_duals}. The condition \ref{hat_Y_for_D_countably_cocont} means that for $y\in(\cD^{\vee})^{\omega_1}$ the functor $\ev_{\cD}(-,y^{\vee}):\cD\to\Sp$ commutes with countable limits. By Lemma \ref{lem:when_ev_with_x_commutes_with_countable_limits}, this is equivalent to the condition \ref{ev_of_duals}.
	
	Since the condition \ref{ev_of_duals} is self-dual in $\cD,$ we also obtain the equivalence \Iff{hat_Y_for_D_dual_countably_cocont}{ev_of_duals}.
\end{proof}

We will need a deeper analysis of formal $\omega_1$-injectivity to study the (relative) dualizable internal $\Hom$ into such categories in Subsection \ref{ssec:internal_injectivity_of_Calkin}. It is convenient to introduce the following notation.

\begin{defi}\label{def:D_omega_1}
	Let $\cD\in\Cat_{\st}^{\dual}$ be a formally $\omega_1$-injective dualizable category. We denote by $\cD_{\omega_1}\subset \cD$ the full subcategory of objects $x\in\cD$ satisfying the equivalent conditions of Lemma \ref{lem:when_ev_with_x_commutes_with_countable_limits}. In particular, $x\in\cD_{\omega_1}$ if and only if
	the object $\hat{\cY}(x)$ is contained in the full subcategory $\cD\simeq\Ind_{\omega_1}(\cD^{\omega_1})\subset \Ind(\cD^{\omega_1}).$ 
\end{defi}

We list some structural properties of formally $\omega_1$-injective dualizable categories which follow almost directly from the above results. We recall that for $\cC\in\Cat_{\st}^{\dual},$ $\cD\in\Pr^L_{\st}$ and an accessible exact functor $F:\cC\to\cD,$ the continuous functor $F^{\cont}:\cC\to\cD$ is given by the composition $\cC\xto{\hat{\cY}}\Ind(\cC^{\omega_1})\xto{\Ind(F)}\Ind(\cD)\xto{\colim}\cD.$

\begin{prop}\label{prop:properties_of_formally_omega_1_injective}
	Let $\cD$ be a formally $\omega_1$-injective dualizable category.
	\begin{enumerate}[label=(\roman*),ref=(\roman*)]
		\item The full subcategory $\cD_{\omega_1}\subset \cD$ is closed under countable limits and $\omega_1$-filtered colimits. \label{D_omega_1_with_limits_colimits}
		\item Denote by $\cA\subset\cD$ the small idempotent-complete stable subcategory generated by $x^{\vee},$ $x\in(\cD^{\vee})^{\omega_1}.$ Then $\cA\subset \cD_{\omega_1}.$ \label{duals_are_omega_1_filtered}
		\item Any object of $\cD_{\omega_1}$ is an $\omega_1$-directed colimit of objects of $\cA.$ \label{omega_1_directed_colimit_of_duals}
		\item Let $\cC$ be a presentable stable category. Then we have the following equivalence of categories of functors:
		\begin{equation}\label{eq:restriction_to_D_omega_1}
			\Fun^L(\cD,\cC)\xto{\sim}\Fun^{\omega_1\hy\cont}(\cD_{\omega_1},\cC),\quad F\mapsto F_{\mid \cD_{\omega_1}}.
		\end{equation}
		Here the target is the category of exact functors commuting with $\omega_1$-filtered colimits. \label{restricting_to_D_omega_1}
		\item Let $\cC$ be a dualizable category, and let $F:\cD\to\cC$ be a continuous functor. Then $F$ commutes with countable limits if and only if $F_{\mid \cD_{\omega_1}}$ commutes with countable limits. \label{countable_cocontinuity_suffices_on_D_omega_1}
		\item Let $\cC$ be a dualizable category, and let $F:\cD\to\cC$ be an $\omega_1$-accessible functor which commutes with countable limits. Then $F^{\cont}:\cD\to\cC$ also commutes with countable limits. \label{when_F^cont_commutes_with_countable_limits}
	\end{enumerate}
\end{prop}

\begin{proof}
	\ref{D_omega_1_with_limits_colimits} The full subcategory $\cD\simeq\Ind_{\omega_1}(\cD^{\omega_1})\subset\Ind(\cD^{\omega_1})$ is closed under all limits since the inclusion functor has a left adjoint. This subcategory is also closed under $\omega_1$-filtered colimits, and the functor $\hat{\cY}:\cD\to\Ind(\cD^{\omega_1})$ commutes with countable limits (by assumption) and with all colimits. This directly implies \ref{D_omega_1_with_limits_colimits}.
	
	Next, \ref{duals_are_omega_1_filtered} follows directly from Lemma \ref{lem:when_ev_with_x_commutes_with_countable_limits} and Proposition \ref{prop:formal_omega_1_injectivity_reformulations}.
	
	To see \ref{omega_1_directed_colimit_of_duals}, note that by definition of $\cD_{\omega_1}$ and by Lemma \ref{lem:duals_strongly_generate} below for any object $x\in\cD_{\omega_1}$ the image of $\hat{\cY}(x)$ in $\Ind(\cD)$ is contained in the intersection $\Ind(\cA)\cap \Ind_{\omega_1}(\cD^{\omega_1})\subset\Ind(\cD).$ The latter intersection is contained in $\Ind_{\omega_1}(\cA).$ This gives a presentation of $x$ as an $\omega_1$-directed colimit of objects of $\cA.$
	
	
	To see \ref{restricting_to_D_omega_1}, we use the small stable subcategory $\cA\subset\cD$ from \ref{omega_1_directed_colimit_of_duals} and the retraction
	$\cD\to\Ind(\cA)\to\cD$ (the first functor is well-defined by Lemma \ref{lem:duals_strongly_generate} below). By \ref{omega_1_directed_colimit_of_duals}, this retraction induces another retraction $\cD_{\omega_1}\to\Ind_{\omega_1}(\cA)\to\cD_{\omega_1}.$ In the latter both functors are exact and commute with $\omega_1$-filtered colimits. It follows that \eqref{eq:restriction_to_D_omega_1} is a retract of the functor
	\begin{equation*}
		\Fun^L(\Ind(\cA),\cC)\to \Fun^{\omega_1\hy\cont}(\Ind_{\omega_1}(\cA),\cC),
	\end{equation*}
	which is clearly an equivalence since the source and the target are identified with the category of exact functors $\cA\to\cC.$
	
	Next, \ref{countable_cocontinuity_suffices_on_D_omega_1} follows directly from \ref{restricting_to_D_omega_1}. Indeed, the ``only if'' direction is clear. To prove the ``if'' direction, it suffices to show that $F$ commutes with countable products. By \ref{duals_are_omega_1_filtered} and \ref{omega_1_directed_colimit_of_duals}, any object of $\cD$ is a directed colimit of objects of $\cD_{\omega_1}.$ Using (AB6) for $\cC$ and $\cD,$ we see that the commutation of $F_{\mid \omega_1}$ with countable products implies that $F$ also commutes with countable products.
	
	Finally, \ref{when_F^cont_commutes_with_countable_limits} follows from \ref{countable_cocontinuity_suffices_on_D_omega_1}, since the assumptions imply the isomorphism $(F^{\cont})_{\mid \cD_{\omega_1}}\xto{\sim} F_{\mid \cD^{\omega_1}}.$
\end{proof}

We used the following general statement about dualizable categories.

\begin{lemma}\label{lem:duals_strongly_generate}
Let $\cD$ be a dualizable category. As in Proposition \ref{prop:properties_of_formally_omega_1_injective}, denote by $\cA\subset\cD$ the idempotent-complete stable subcategory generated by the objects $x^{\vee},$ $x\in(\cD^{\vee})^{\omega_1}.$ Then the essential image of the composition $F:\cD\xto{\hat{\cY}}\Ind(\cD^{\omega_1})\to\Ind(\cD)$ is contained in $\Ind(\cA)\subset\Ind(\cD).$ 
\end{lemma}

\begin{proof}
It suffices to prove that for $y\in\cD^{\omega_1}$ the object $F(y)$ is contained in $\Ind(\cA).$ We have $\hat{\cY}(y)\cong\inddlim[n\in\N]y_n,$ where $y_n\in\cD^{\omega_1}$ and each map $y_n\to y_{n+1}$ is compact in $\cD.$ For each $n\geq 0$ we choose a factorization $y_{n+1}^{\vee}\to z_n\to y_n^{\vee}$ in $\cD^{\vee}$ with $z_n\in(\cD^{\vee})^{\omega_1}.$ Then we have $F(y)\cong \inddlim[n]z_n^{\vee}\in\Ind(\cA),$ as required.
\end{proof}

The next statement is more subtle.

\begin{prop}\label{prop:Ind_of_D_omega_1_universal_property}
	Let $\cC$ and $\cD$ be dualizable categories, and suppose that $\cD$ is formally $\omega_1$-injective. Let $F:\cC\to\cD$ be a continuous functor, and let $G=\Ind(F)\circ\hat{\cY}_{\cC}:\cC\to\Ind(\cD)$ be the corresponding strongly continuous functor. The following are equivalent.
	\begin{enumerate}[label=(\roman*),ref=(\roman*)]
		\item The essential image of $G$ is contained in $\Ind(\cD_{\omega_1})\subset\Ind(\cD).$ \label{image_in_Ind_D_omega_1}
		\item The functor $F^{\vee}:\cD^{\vee}\to\cC^{\vee}$ commutes with countable limits. \label{dual_commutes_with_countable_limits}
	\end{enumerate}
\end{prop}

\begin{proof} \Implies{image_in_Ind_D_omega_1}{dual_commutes_with_countable_limits}. Let $x\in\cC^{\omega_1}$ be an $\omega_1$-compact object with $\hat{\cY}(x)=\inddlim[n\in\N]x_n,$ and let $y_0,y_1,\dots$ be a sequence of objects of $\cD.$ By assumption, the ind-sequence $(F(x_n))_{n\geq 0}$ is ind-isomorphic to a sequence in $\cD_{\omega_1}.$ We obtain the isomorphisms 
	\begin{multline*}
		\ev_{\cC}(x,F^{\vee}(\prodd[k]y_k))\cong \ev_{\cD}(F(x),\prodd[k]y_k)\cong \indlim[n]\ev_{\cD}(F(x_n),\prodd[k]y_k)\\
		\cong \indlim[n]\prodd[k]\ev_{\cD}(F(x_n),y_k)\cong \indlim[n]\prodd[k]\ev_{\cC}(x_n,F^{\vee}(y_k))\cong \ev_{\cC}(x,\prodd[k]F^{\vee}(y_k)).
	\end{multline*}
	This shows that $F^{\vee}$ commutes with countable products, hence with countable limits.
	
	\Implies{dual_commutes_with_countable_limits}{image_in_Ind_D_omega_1}. Take some object $x\in\cC^{\omega_1}.$ By \cite[Proposition 1.82]{E24}, there exists a functor $\Phi:\Q_{\leq}\to \cC,$ where $\Q_{\leq}$ is the poset of rational numbers with the usual order, with an isomorphism $\indlim[a]\Phi(a)\cong x,$ such that for $a<b$ the map $\Phi(a)\to\Phi(b)$ is compact in $\cC.$ Consider another functor
	\begin{equation*}
		\Psi:\Q_{\leq}\to \cD,\quad \Psi(a)=\prolim[b>a]F(\Phi(b)).
	\end{equation*}
	
	We claim that $\Psi(a)\in\cD_{\omega_1}$ for all $a\in\Q.$ To see this, note that for $b<c$ the morphism $\ev_{\cC}(\Phi(b),-)\to \ev_{\cC}(\Phi(c),-)$ (of functors $\cC^{\vee}\to\Sp$) factors through $\Hom_{\cC^{\vee}}(\Phi(b)^{\vee},-).$ Hence, for $a\in\Q$ the functor $\prolim[b>a]\ev_{\cC}(\Phi(b),-)$ commutes with all limits. By assumption, the functor $F^{\vee}$ commutes with countable limits, hence for any $a\in\Q$ the functor \begin{equation*}
		H_a:\prolim[b>a]\ev_{\cC}(\Phi(b),F^{\vee}(-))\cong \prolim[b>a]\ev_{\cD}(F(\Phi(a)),-):\cD^{\vee}\to\Sp\end{equation*} also commutes with countable limits. Clearly, $H_a$ commutes with $\omega_1$-filtered colimits. Applying Proposition \ref{prop:properties_of_formally_omega_1_injective} \ref{when_F^cont_commutes_with_countable_limits}, we see that the functor
	\begin{equation*}
		\ev_{\cD}(\Psi(a),-)\cong H_a^{\cont}:\cD^{\vee}\to\Sp
	\end{equation*}
	commutes with countable limits. This exactly means that $\Psi(a)\in\cD_{\omega_1},$ as stated.
	
	We obtain the required inclusion:
	\begin{equation*}
		G(x)\cong\inddlim[n]F(\Phi(n))\cong \inddlim[n]\Psi(n)\in\Ind(\cD_{\omega_1}).\qedhere
	\end{equation*}
\end{proof}

Next, we obtain sufficient conditions for formal $\omega_1$-injectivity of subcategories.

\begin{prop}\label{prop:formal_omega_1_injectivity_for_subcategory}
Let $F:\cC\to\cD$ be a strongly continuous fully faithful functor between dualizable categories and denote by $F^R$ its right adjoint. Suppose that $\cD$ is formally $\omega_1$-injective. 
\begin{enumerate}[label=(\roman*),ref=(\roman*)]
	\item If $F$ commutes with countable limits, then $\cC$ is formally $\omega_1$-injective. \label{inclusion_commutes_with_prod_N}
	\item If the composition $F\circ F^R:\cD\to\cD$ preserves the subcategory $\cD_{\omega_1},$ then $\cC$ is formally $\omega_1$-injective. \label{composition_preserves_D_omega_1}
\end{enumerate}
\end{prop}

\begin{proof} We prove \ref{inclusion_commutes_with_prod_N}. We claim that the functor $\Ind(F^{\omega_1}):\Ind(\cC^{\omega_1})\to\Ind(\cD^{\omega_1})$ commutes with countable limits. Note that the source and the target are formally $\omega_1$-injective by Proposition \ref{prop:when_hat_Y_is_cocontinuous}. We have $\Ind(\cC^{\omega_1})_{\omega_1}=\Ind_{\omega_1}(\cC^{\omega_1})\simeq\cC$ and similarly for $\cD.$ Hence, the restriction of $\Ind(F^{\omega_1})$ to $\Ind(\cC^{\omega_1})_{\omega_1}$ commutes with countable limits. Hence, by Proposition \ref{prop:properties_of_formally_omega_1_injective} \ref{countable_cocontinuity_suffices_on_D_omega_1} the functor  $\Ind(F^{\omega_1})$ itself commutes with countable limits.

We have an isomorphism $\Ind(F^{\omega_1})\circ\hat{\cY}_{\cC}\cong \hat{\cY}_{\cD}\circ F$ of functors from $\cC$ to $\Ind(\cD^{\omega_1}).$ The latter composition commutes with countable limits by our assumptions. Since the functor $\Ind(F^{\omega_1})$ is fully faithful and commutes with countable limits, it follows that $\hat{\cY}_{\cC}:\cC\to\Ind(\cC^{\omega_1})$ commutes with countable limits, i.e. $\cC$ is formally $\omega_1$-injective.

Next, we deduce \ref{composition_preserves_D_omega_1} from \ref{inclusion_commutes_with_prod_N}. Denote by $\Phi:\cD\to\Ind(\cD)$ the strongly continuous functor corresponding to the composition $F\circ F^R:\cD\to\cD.$ Then it follows from our assumption together with Proposition \ref{prop:properties_of_formally_omega_1_injective} \ref{duals_are_omega_1_filtered} and Lemma \ref{lem:duals_strongly_generate} that the essential image of $\Phi$ is contained in $\Ind(\cD_{\omega_1}).$ By Proposition \ref{prop:Ind_of_D_omega_1_universal_property} the functor $\Phi^{\vee}:\cD^{\vee}\to \cD^{\vee}$ commutes with countable limits. We put $G=(F^R)^{\vee}:\cC^{\vee}\to\cD^{\vee},$ then $G$ is strongly continuous and fully faithful. We have $\Phi^{\vee}\cong G\circ G^R$ and the functor $G^R$ is essentially surjective and commutes with all limits, hence $G$ commutes with countable limits. It follows from \ref{inclusion_commutes_with_prod_N} that $\cC^{\vee}$ is formally $\omega_1$-injective, hence  so is $\cC$ by Proposition \ref{prop:formal_omega_1_injectivity_reformulations}.
\end{proof}

We close this subsection with the following simple observation.

\begin{prop}\label{prop:C_omega_1_for_subcategory}
Let $F:\cC\to\cD$ be a fully faithful strongly continuous functor between formally $\omega_1$-injective dualizable categories. Then an object $x\in\cC$ is contained in $\cC_{\omega_1}$ if and only if $F(x)$ is contained in $\cD_{\omega_1}.$
\end{prop}

\begin{proof}
Note that the functor $F^{\vee}:\cD^{\vee}\to\cC^{\vee}$ is essentially surjective and commutes with all limits. Therefore, $x\in\cC_{\omega_1}$ iff the functor $\ev_{\cC}(x,-)$ commutes with countable products iff the functor $\ev_{\cC}(x,F^{\vee}(-))\cong \ev_{\cD}(F(x),-)$ commutes with countable products iff $F(x)\in\cD_{\omega_1}.$
\end{proof}

\subsection{Formal $\omega_1$-injectivity of Calkin categories}
\label{ssec:formal_injectivity_of_Calkin}

We prove the following quite difficult statement about Calkin categories of dualizable categories, which was announced in \cite[Theorem 3.31]{E25}.

\begin{theo}\label{th:hat_Y_for_Calkin}
Let $\cC$ be a dualizable category, and let $\kappa$ be an uncountable regular cardinal. Then the dualizable category $\cD=\Calk_{\kappa}(\cC)$ is formally $\omega_1$-injective, i.e. the functor $\hat{\cY}:\cD\to\Ind(\cD^{\omega_1})$
commutes with countable limits.
\end{theo}

We will deduce Theorem \ref{th:hat_Y_for_Calkin} from the following general statement from \cite{E25} on the combinations of products and directed colimits.

\begin{prop}\label{prop:pullback_square_for_products}\cite[Corollary A.4]{E25}
Let $\cC$ be a (not necessarily stable) presentable category which satisfies strong (AB5) (i.e. filtered colimits commute with finite limits) and (AB6) for countable products. Let $(I_n)_{n\geq 0}$ and $(J_k)_{k\geq 0}$ be sequences of directed posets. Suppose that for any $n,k\geq 0$ we have a functor $F_{n,k}:I_n\times J_k\to\cC.$ Then we have a pullback square
\begin{equation}\label{eq:pullback_square_for_products}
\begin{tikzcd}
\indlim[(i_n)_n\in\prod\limits_{n} I_n]\indlim[(j_k)_k\in\prod\limits_{k} J_k]\prodd[n]\prodd[k]F_{n,k}(i_n,j_k)\ar[r]\ar[d] & \prodd[n]\indlim[i_n\in I_n]\prodd[k]\indlim[j_k\in J_k] F_{n,k}(i_n,j_k)\ar[d]\\
\prodd[k]\indlim[j_k\in J_k]\prodd[n]\indlim[i_n\in I_n] F_{n,k}(i_n,j_k)\ar[r] & \prodd[n]\prodd[k]\indlim[i_n\in I_n]\indlim[j_k\in J_k]F_{n,k}(i_n,j_k). 
\end{tikzcd}
\end{equation} 
\end{prop}

\begin{proof}[Proof of Theorem \ref{th:hat_Y_for_Calkin}]
It suffices to check the condition \ref{ev_of_duals} from Proposition \ref{prop:formal_omega_1_injectivity_reformulations}. Take some objects $x\in\cD^{\omega_1},$ $y\in(\cD^{\vee})^{\omega_1}.$ Since $\cD$ is compactly generated, we may assume that $x\cong\biggplus[n\in\N]\bbar{x}_n,$ $y\cong\biggplus[k\in\N]\bbar{y}_k^{\vee},$
where $\bbar{x}_n,\bbar{y}_k\in\cD^{\omega}=\Calk_{\kappa}^{\cont}(\cC).$ Moreover, we may and will assume that each object $\bbar{x}_n$ resp. $\bbar{y}_k$ is isomorphic to the image of an object $x_n\in\cC^{\kappa}$ resp. $y_k\in\cC^{\kappa},$ and we fix these isomorphisms.

Since the functor $\pi:\cC^{\kappa}\to\Calk_{\kappa}^{\cont}(\cC)$ is a homological epimorphism, we need to prove that the following map is an isomorphism:
\begin{multline}\label{eq:source_and_target}
(\prodd[k]\Hom_{\Calk_{\kappa}^{\cont}(\cC)}(\pi(-),\pi(y_k)))\tens{\cC^{\kappa}} (\prodd[n]\Hom_{\Calk_{\kappa}^{\cont}(\cC)}(\pi(x_n),\pi(-)))\\
\to\prodd[n,k]\Hom_{\Calk_{\kappa}^{\cont}(\cC)}(\pi(x_n),\pi(y_k)).
\end{multline}
Recall the following isomorphisms of functors $(\cC^{\kappa})^{op}\to\Sp$ resp. $\cC^{\kappa}\to\Sp:$
\begin{equation*}
\Hom_{\Calk_{\kappa}^{\cont}(\cC)}(\pi(-),\pi(y_k))\cong\Cone(\ev_{\cC}(y_k,(-)^{\vee})\to\Hom_{\cC}(-,y_k)),
\end{equation*}
\begin{equation*}
	\Hom_{\Calk_{\kappa}^{\cont}(\cC)}(\pi(x_n),\pi(-))\cong\Cone(\ev_{\cC}(-,x_n^{\vee})\to\Hom_{\cC}(x_n,-)).
\end{equation*}
It follows that the source of the map \eqref{eq:source_and_target} is given by the total cofiber of the following square:
\begin{equation}\label{eq:square_for_total_cofiber}
\begin{tikzcd}
(\prodd[k]\ev_{\cC}(y_k,(-)^{\vee}))\tens{\cC^{\kappa}}(\prodd[n]\ev_{\cC}(-,x_n^{\vee}))\ar[r]\ar[d] & (\prodd[k]\Hom_{\cC}(-,y_k))\tens{\cC^{\kappa}}(\prodd[n]\ev_{\cC}(-,x_n^{\vee}))\ar[d]\\
(\prodd[k]\ev_{\cC}(y_k,(-)^{\vee}))\tens{\cC^{\kappa}}(\prodd[n]\Hom_{\cC}(x_n,-))\ar[r] & (\prodd[k]\Hom_{\cC}(-,y_k))\tens{\cC^{\kappa}}(\prodd[n]\Hom_{\cC}(x_n,-)).
\end{tikzcd}
\end{equation}

We start with the simplest computation of the tensor product in \eqref{eq:square_for_total_cofiber}:
\begin{multline*}
(\prodd[k]\Hom_{\cC}(-,y_k))\tens{\cC^{\kappa}}(\prodd[n]\Hom_{\cC}(x_n,-))\cong (\prodd[k]\Hom_{\cC}(-,y_k))\tens{\cC^{\kappa}} \Hom_{\cC}(\biggplus[n]x_n,-)\\
\cong \prodd[k]\Hom_{\cC}(\biggplus[n]x_n,y_k)\cong\prodd[n]\prodd[k]\Hom_{\cC}(x_n,y_k).
\end{multline*}
Similarly, we have
\begin{multline*}
(\prodd[k]\ev_{\cC}(y_k,(-)^{\vee}))\tens{\cC^{\kappa}}(\prodd[n]\Hom_{\cC}(x_n,-))\cong (\prodd[k]\ev_{\cC}(y_k,(-)^{\vee}))\tens{\cC^{\kappa}}\Hom_{\cC}(\biggplus[n]x_n,-)\\
\cong \prodd[k]\ev_{\cC}(y_k,(\biggplus[n]x_n)^{\vee})\cong \prodd[k]\ev_{\cC}(y_k,\prodd[k]x_n^{\vee}).
\end{multline*}
The next computation is only slightly subtler:
\begin{multline*}
(\prodd[k]\Hom_{\cC}(-,y_k))\tens{\cC^{\kappa}}(\prodd[n]\ev_{\cC}(-,x_n^{\vee}))\cong \Hom_{\cC}(-,\prodd[k]y_k)\tens{\cC^{\kappa}}(\prodd[n]\ev_{\cC}(-,x_n^{\vee}))\\
\cong \prodd[n]\ev_{\cC}(\prodd[k]y_k,x_n^{\vee}).
\end{multline*}
Here the second isomorphism follows from the fact that the functor $\Hom(-,\prodd[k]y_k):(\cC^{\kappa})^{op}\to\Sp$ is a $\kappa$-filtered colimit of representable functors, and the functor $(\prodd[n]\ev_{\cC}(-,x_n^{\vee})):\cC\to\Sp$ commutes with $\kappa$-filtered colimits.

The above computations show that the square \eqref{eq:square_for_total_cofiber} is identified with the following:

\begin{equation}\label{eq:square_for_total_cofiber_rewritten}
\begin{tikzcd}
(\prodd[k]\ev_{\cC}(y_k,(-)^{\vee}))\tens{\cC^{\kappa}}(\prodd[n]\ev_{\cC}(-,x_n^{\vee}))\ar[r]\ar[d] & \prodd[n]\ev_{\cC}(\prodd[k]y_k,x_n^{\vee})\ar[d]\\
\prodd[k]\ev_{\cC}(y_k,\prodd[n]x_n^{\vee})\ar[r] & \prodd[n]\prodd[k]\Hom_{\cC}(x_n,y_k).
\end{tikzcd}
\end{equation}

Next, we deal with the tensor product in the top left corner. Let $\hat{\cY}_{\cC^{\vee}}(x_n^{\vee})=\inddlim[i_n\in I_n]z_{n,i_n},$ $\hat{\cY}_{\cC}(y_k)=\inddlim[j_k\in J_k]y_{k,j_k},$ where $I_n$ and $J_k$ are directed posets, and $z_{n,i_n}\in(\cC^{\vee})^{\omega_1},$ $y_{k,j_k}\in \cC^{\omega_1}.$ We have
\begin{multline*}
(\prodd[k]\ev_{\cC}(y_k,(-)^{\vee}))\tens{\cC^{\kappa}}(\prodd[n]\ev_{\cC}(-,x_n^{\vee}))\cong (\prodd[k]\indlim[j_k]\Hom_{\cC}(-,y_{k,j_k}))\tens{\cC^{\kappa}}(\prodd[n]\ev_{\cC}(-,x_n^{\vee}))\\
\cong (\indlim[(j_k)_k\in\prod\limits_{k}J_k]\prodd[k]\Hom_{\cC}(-,y_{k,j_k}))\tens{\cC^{\kappa}}(\prodd[n]\ev_{\cC}(-,x_n^{\vee}))\\
\cong (\indlim[(j_k)_k]\Hom_{\cC}(-,\prodd[k]y_{k,j_k}))\tens{\cC^{\kappa}}(\prodd[n]\ev_{\cC}(-,x_n^{\vee}))\\
\cong \indlim[(j_k)_k]\prodd[n]\ev_{\cC}(\prodd[k]y_{k,j_k},x_n^{\vee})\cong \indlim[(j_k)_k]\prodd[n]\indlim[i_n]\Hom_{\cC}(z_{n,i_n}^{\vee},\prodd[k]y_{k,j_k})\\
\cong \indlim[(i_n)_n]\indlim[(j_k)_k]\prodd[n]\prodd[k]\Hom_{\cC}(z_{n,i_n}^{\vee},y_{k,j_k}).
\end{multline*}
We now apply Proposition \ref{prop:pullback_square_for_products} to the functors
\begin{equation*}
F_{n,k}:I_n\times J_k\to \Sp,\quad F_{n,k}(i_n,j_k) = \Hom_{\cC}(z_{n,i_n}^{\vee},y_{k,j_k}).
\end{equation*}
We already know the top left corner of the associated pullback square \eqref{eq:pullback_square_for_products}. We can directly compute the other three objects as follows. First, we have
\begin{multline*}
	\prodd[n]\indlim[i_n]\prodd[k]\indlim[j_k]F_{n,k}(i_n,j_k)\cong \prodd[n]\indlim[i_n]\prodd[k]\Hom_{\cC}(z_{n,i_n}^{\vee},y_k)\\
	\cong \prodd[n]\indlim[i_n]\Hom_{\cC}(z_{n,i_n}^{\vee},\prodd[k]y_k)\cong \prodd[n]\ev_{\cC}(\prodd[k]y_k,x_n^{\vee}).
\end{multline*}
Here the first isomorphism follows from the fact that for each $n\geq 0$ we have an ind-isomorphism of functors $\inddlim[i_n]\Hom_{\cC}(z_{n,i_n}^{\vee},-)\cong \inddlim[i_n]\ev_{\cC}(-,z_{n,i_n}).$ 

Next, we have
\begin{multline*}
\prodd[k]\indlim[j_k]\prodd[n]\indlim[i_n]F_{n,k}(i_n,j_k)\cong \prodd[k]\indlim[j_k]\prodd[n]\ev_{\cC}(y_{k,j_k},x_n^{\vee})\cong \prodd[k]\indlim[j_k]\ev_{\cC}(y_{k,j_k},\prodd[n]x_n^{\vee})\\
\cong \prodd[k]\ev_{\cC}(y_k,\prodd[n]x_n^{\vee}).
\end{multline*}
Here the second isomorphism uses the fact that for each $k\geq 0$ we have an ind-isomorphism of functors $\inddlim[j_k]\ev_{\cC}(y_{k,j_k},-)\cong\inddlim[j_k]\Hom_{\cC^{\vee}}(y_{k,j_k}^{\vee},-).$

Finally, we have
\begin{equation*}
\prodd[n]\prodd[k]\indlim[i_n]\indlim[j_k]F_{n,k}(i_n,j_k)\cong \prodd[n]\prodd[k]\ev_{\cC}(y_k,x_n^{\vee}).
\end{equation*}
Applying Proposition \ref{prop:pullback_square_for_products}, we obtain the following pullback square:
\begin{equation}\label{eq:application_of_abstract_pullback_square}
\begin{tikzcd}
(\prodd[k]\ev_{\cC}(y_k,(-)^{\vee}))\tens{\cC^{\kappa}}(\prodd[n]\ev_{\cC}(-,x_n^{\vee}))\ar[r]\ar[d] & \prodd[n]\ev_{\cC}(\prodd[k]y_k,x_n^{\vee})\ar[d]\\
\prodd[k]\ev_{\cC}(y_k,\prodd[n]x_n^{\vee})\ar[r] & \prodd[n]\prodd[k]\ev_{\cC}(y_k,x_n^{\vee}).
\end{tikzcd}
\end{equation}
Now, we have a natural map from the square \eqref{eq:application_of_abstract_pullback_square} to the square \eqref{eq:square_for_total_cofiber_rewritten}, which is an isomorphism on the objects in the top left corner, top right corner and the bottom left corner. We conclude that the total cofiber of \eqref{eq:square_for_total_cofiber_rewritten} is naturally isomorphic to
\begin{multline*}
\Cone(\prodd[n]\prodd[k]\ev_{\cC}(x_n,y_k^{\vee})\to \prodd[n]\prodd[k]\Hom_{\cC}(y_k,x_n))\\
\cong \prodd[n]\prodd[k] \Hom_{\Calk_{\kappa}^{\cont}(\cC)}(\pi(y_k),\pi(x_n)).
\end{multline*}
Therefore, the map \eqref{eq:source_and_target} is an isomorphism, which proves the theorem.
\end{proof}

\subsection{Internal $\omega_1$-injectivity of Calkin categories}
\label{ssec:internal_injectivity_of_Calkin}

In this subsection we prove the following result, which in particular justifies the term ``formally $\omega_1$-injective''.

\begin{theo}\label{th:internal_omega_1_injectivity}
Let $\cE$ be a rigid $\bE_1$-monoidal category. Let $\cD$ be a dualizable left $\cE$-module which is formally $\omega_1$-injective, i.e. the functor $\hat{\cY}:\cD\to\Ind(\cD^{\omega_1})$ commutes with countable limits. Then $\cD$ is relatively internally $\omega_1$-injective in $\Cat_{\cE}^{\dual}$ over $\Cat_{\st}^{\dual}.$ This means that the functor
\begin{equation*}
\un{\Hom}_{\cE}^{\dual}(-,\cD):((\Cat_{\cE}^{\dual})^{\omega_1})^{op}\to \Cat_{\st}^{\dual}
\end{equation*}
takes short exact sequences to short exact sequences.

In particular, for any $\cC\in\Cat_{\cE}^{\dual}$ and for any uncountable regular cardinal $\kappa,$ the dualizable $\cE$-module $\Calk_{\kappa}(\cC)$ is relatively internally $\omega_1$-injective in $\Cat_{\cE}^{\dual}$ over $\Cat_{\st}^{\dual}.$
\end{theo}

We also prove other non-trivial results on the dualizable internal $\Hom$ into formally $\omega_1$-injective dualizable $\cE$-modules, which eventually lead to the proof of Theorem \ref{th:morphisms_in_Mot^loc_via_internal_Hom} in the end of this subsection. 

\begin{prop}\label{prop:well_defined_submodules}
Let $\cD\in\Cat_{\cE}^{\dual}$ be formally $\omega_1$-injective. As in Proposition \ref{prop:properties_of_formally_omega_1_injective}, denote by $\cA\subset\cD$ the idempotent-complete stable subcategory generated by the objects $x^{\vee},$ $x\in(\cD^{\vee})^{\omega_1}.$ We choose an uncountable regular cardinal $\kappa$ such that $\cA\subset\cD^{\kappa}.$

Then both $\Ind(\cA)$ and $\Ind(\cD_{\omega_1}\cap \cD^{\kappa})$ are $\cE$-submodules of $\Ind(\cD^{\kappa}).$
\end{prop}

\begin{proof}
We first show that $\Ind(\cA)\subset\Ind(\cD^{\kappa})$ is an $\cE$-submodule. It suffices to prove that for $x\in\cE^{\omega_1}$ and for $y\in (\cD^{\vee})^{\omega_1}$ we have the inclusion $x\otimes\cY(y^{\vee})\in\Ind(\cA)\subset\Ind(\cD^{\kappa}).$ Let $\hat{\cY}(x)=\inddlim[n\in\N]x_n\in\Ind(\cE^{\omega_1}),$ and assume that each map $x_n\to x_{n+1}$ is compact in $\cE.$ Choose factorizations $x_{n+1}^{l\vee}\to z_n\to x_n^{l\vee}$ with $z_n\in\cE^{\omega_1}.$ Then we have
\begin{equation*}
x\otimes\cY(y^{\vee})\cong \inddlim[n](x_n\otimes y^{\vee})\cong \inddlim[n]((y^{\vee})^{z_n})\cong \inddlim[n]((y\otimes z_n)^{\vee})\in\Ind(\cA).
\end{equation*}

A similar argument shows that $\Ind(\cD_{\omega_1}\cap \cD^{\kappa})$ is an $\cE$-submodule: we only need to show that for $y'\in \cD_{\omega_1}\cap\cD^{\kappa}$ and for $z\in\cE^{\omega_1}$ we have $(y')^z\in\cD_{\omega_1}\cap \cD^{\kappa}.$ To see this, note that by Proposition \ref{prop:properties_of_formally_omega_1_injective} \ref{omega_1_directed_colimit_of_duals} we have $y'\cong\indlim[i\in I]y_j,$ where $J$ is $\kappa$-small and $\omega_1$-directed, and $y_j\in\cA.$ By the above, we have $y_j^z\in\cA.$ Since $(-)^z:\cD\to\cD$ commutes with $\omega_1$-directed colimits, we have $(y')^z\cong\indlim[i\in I]y_j^z\in\cD_{\omega_1}\cap\cD^{\kappa}.$ This proves the proposition.
\end{proof}

In the next proposition we use the following notation. For $\cD\in\Cat_{\cE}^{\dual}$ we consider $\Ind(\cD)$ as a directed union of the $\cE$-modules $\Ind(\cD^{\kappa})$ over all uncountable regular cardinals $\kappa.$ In particular, for $\cC\in\Cat_{\cE}^{\dual}$ we consider the category $\Fun_{\cE}^{LL}(\cC,\Ind(\cD))$ as a directed union of $\Fun_{\cE}^{LL}(\cC,\Ind(\cD^{\kappa})).$ The same notation applies to $\Ind(\cD_{\omega_1})$ when $\cD$ is formally $\omega_1$-injective. Here we use the fact that by Proposition \ref{prop:well_defined_submodules} for sufficiently large $\kappa$ the category $\Ind(\cD_{\omega_1}\cap\cD^{\kappa})$ is a well-defined $\cE$-module.

\begin{prop}\label{prop:category_of_functors_to_Ind_D_omega_1}
Let $\cC,\cD\in\Cat_{\cE}^{\dual}$ be dualizable left $\cE$-modules, and suppose that $\cD$ is formally $\omega_1$-injective. 
\begin{enumerate}[label=(\roman*),ref=(\roman*)]
\item The full subcategory 
\begin{equation*}
\Fun_{\cE}^{LL}(\cC,\Ind(\cD_{\omega_1}))\subset \Fun_{\cE}^{LL}(\cC,\Ind(\cD))\simeq \Fun_{\cE}^L(\cC,\cD)\end{equation*}
is closed under $\omega_1$-filtered colimits and under countable limits.\label{closed_under_omega_1_filtered_colimits}

\item Suppose that moreover $\cC$ is $\omega_1$-compact in $\Cat_{\cE}^{\dual}.$ Let $F\in\cC^{\vee}\tens{\cE}\cD\simeq\Fun^L_{\cE}(\cC,\cD)$ be a functor such that $F^{\vee}:\cD^{\vee}\to\cC^{\vee}$ commutes with countable limits. Then for any $G\in(\cC^{\vee}\tens{\cE}\cD)^{\omega_1}$ we have an isomorphism
\begin{equation}\label{eq:abs_nuclearity_of_F}
\Hom_{\cC^{\vee}\tens{\cE}\cC}(\Id_{\cC},G^{R,\cont}\circ F)\xto{\sim} \Hom_{\cC^{\vee}\tens{\cE}\cD}(G,F).
\end{equation} 
 \label{image_in_Ind_D_omega_1_implies_abs_nuclearity}
\end{enumerate}
\end{prop}

\begin{proof}
\ref{closed_under_omega_1_filtered_colimits} 
Consider the full subcategory $\cT\subset\Fun_{\cE^{mop}}^L(\cD^{\vee},\cC^{\vee}),$ which consists of continuous $\cE^{mop}$-linear functors $G:\cD^{\vee}\to\cC^{\vee}$ which commute with countable limits.
By Proposition \ref{prop:Ind_of_D_omega_1_universal_property} we have a commutative square in which the horizontal arrows are equivalences:
\begin{equation*}
\begin{tikzcd}
\Fun_{\cE}^{LL}(\cC,\Ind(\cD_{\omega_1}))\ar[hookrightarrow]{d}\ar[r, "\sim"] & \cT\ar[hookrightarrow]{d}\\
\Fun_{\cE}^L(\cC,\cD)\ar[r, "\sim"] & \Fun_{\cE^{mop}}^L(\cD^{\vee},\cC^{\vee}).
\end{tikzcd}
\end{equation*}
Now, $\cT$ is closed under $\omega_1$-filtered colimits in $\Fun_{\cE^{mop}}^L(\cD^{\vee},\cC^{\vee}),$ since in $\cC^{\vee}$ the $\omega_1$-filtered colimits commute with countable limits. It remains to prove that $\cT$ is closed under countable products in $\Fun_{\cE^{mop}}^L(\cD^{\vee},\cC^{\vee}).$ Take a sequence $F_0,F_1,\dots$ in $\cT,$ and denote by $\prodd[n]^{\naive}F_n$ the naive product taken in $\Fun^{\acc,\lax}_{\cE^{mop}}(\cD^{\vee},\cC^{\vee}).$ Clearly, this naive product commutes with countable limits. By Proposition \ref{prop:properties_of_formally_omega_1_injective} \ref{when_F^cont_commutes_with_countable_limits} the functor $(\prodd[n]^{\naive}F_n)^{\cont}$ also commutes with countable limits. The latter functor is simply the product of $F_n$ in $\Fun_{\cE^{mop}}^L(\cD^{\vee},\cC^{\vee}).$ This proves \ref{closed_under_omega_1_filtered_colimits}.

\ref{image_in_Ind_D_omega_1_implies_abs_nuclearity} First note that $\omega_1$-compactness of $\cC$ implies that $(\cC^{\vee}\tens{\cE}\cD)^{\omega_1}\simeq \Fun_{\cE}^L(\cC,\cD)^{\omega_1}$ is simply the category of continuous $\cE$-linear functors which take $\cC^{\omega_1}$ to $\cD^{\omega_1}.$ Hence, $G^R$ commutes with $\omega_1$-filtered colimits. By the proof of Proposition \ref{prop:properties_of_formally_omega_1_injective} \ref{when_F^cont_commutes_with_countable_limits} we deduce that $(G^{R,\cont})_{\mid\cD_{\omega_1}}\cong (G^R)_{\mid\cD_{\omega_1}}.$ By Proposition \ref{prop:Ind_of_D_omega_1_universal_property} the functor $\Ind(F)\circ\hat{\cY}:\cC\to\Ind(\cD)$ takes values in $\Ind(\cD_{\omega_1}).$ Hence, we obtain the isomorphism $G^{R,\cont}\circ F\cong (G^R\circ F)^{\cont}.$ This gives the required isomorphism \eqref{eq:abs_nuclearity_of_F}:
\begin{multline*}
\Hom_{\cC^{\vee}\tens{\cE}\cD}(G,F)\cong\Hom_{\Fun_{\cE}^{\acc,\lax}(\cC,\cC)}(\Id_{\cC},G^R\circ F)\\
\cong\Hom_{\cC^{\vee}\tens{\cE}\cC}(\Id_{\cC},(G^R\circ F)^{\cont})
\cong \Hom_{\cC^{\vee}\tens{\cE}\cC}(\Id_{\cC},G^{R,\cont}\circ F).\qedhere
\end{multline*}
\end{proof}



We obtain (under some assumptions) a more explicit description of the class of continuous $\cE$-linear functors whose dual commutes with countable limits.

\begin{prop}\label{prop:another_description_of_functors_to_Ind_D_omega_1}
Let $\cC,\cD$ be dualizable left $\cE$-modules such that $\cC$ is $\omega_1$-compact in $\Cat_{\cE}^{\dual}$ and $\cD$ is formally $\omega_1$-injective. For a functor $F\in\Fun_{\cE}^L(\cC,\cD),$ the following are equivalent.
\begin{enumerate}[label=(\roman*),ref=(\roman*)]
\item The functor $F^{\vee}:\cD^{\vee}\to\cC^{\vee}$ commutes with countable limits. \label{dual_commutes_with_countable_limits_again}
\item  We have $F\cong \indlim[i\in I]F_i,$ where $I$ is an $\omega_1$-directed poset, $F_i\in (\cC^{\vee}\tens{\cE}\cD)^{\omega_1},$ and for any $i\in I$ there exists $j\geq i$ such that the morphism $F_i\to F_j$ is right trace-class in $\Fun_{\cE}^L(\cC,\cD).$ \label{omega_1_directed_colimit_of_right_trace_class_maps}
\end{enumerate}
\end{prop}

\begin{proof}
\Implies{dual_commutes_with_countable_limits_again}{omega_1_directed_colimit_of_right_trace_class_maps}. Take any $\omega_1$-directed system $(F_i)_{i\in I}$ in $(\cC^{\vee}\tens{\cE}\cD)^{\omega_1}$ with an isomorphism $F\cong \indlim[i] F_i.$ Denote by $\alpha_{ij}:F_i\to F_j$ the transition maps, and by $\alpha_i:F_i\to F$ the maps to the colimit. Since $\cC$ is $\omega_1$-compact in $\Cat_{\cE}^{\dual},$ the object $\Id_{\cC}$ is $\omega_1$-compact in $\cC^{\vee}\tens{\cE}\cC.$ Hence, for each $i\in I$ the functor $\Hom_{\cC^{\vee}\tens{\cE}\cC}(\Id_{\cC},F_i^{R,\cont}\circ -)$ commutes with $\omega_1$-filtered colimits. Hence, by Proposition \ref{prop:category_of_functors_to_Ind_D_omega_1} \ref{image_in_Ind_D_omega_1_implies_abs_nuclearity} applied to $G=F_i$ we can find some $k\in I$ and a right trace-class map $\beta:F_i\to F_k$ such that the composition $F_i\xto{\beta} F_k\xto{\alpha_k} F$ is homotopic to $\alpha_i.$ Since $F_i$ is $\omega_1$-compact, it follows that for some $j\geq i,k$ we have $\alpha_{kj}\circ \beta\sim \alpha_{ij},$ hence $\alpha_{ij}$ is right trace-class, as required.

\Implies{omega_1_directed_colimit_of_right_trace_class_maps}{dual_commutes_with_countable_limits_again}. We need to prove that the functor $F^{\vee}\cong\indlim[i]F_i^{\vee}$ commutes with countable limits. Let $i\leq j$ be elements of $I$ such that $F_i\to F_j$ is a right trace-class morphism in $\Fun_{\cE}^L(\cC,\cD).$ We put $G_i:=F_i^{R,\cont}$ and choose a right trace-class witness $\Id_{\cC}\to G_i\circ F_j.$ We observe that we have a factorization in $\Fun_{\cE^{mop}}^{\acc,\lax}(\cD^{\vee},\cC^{\vee}):$
\begin{equation*}
F_i^{\vee}\xto{\beta} (G_i^{\vee})^R\xto{\gamma} F_j^{\vee}.
\end{equation*}
Here $\beta$ is the composition
\begin{equation*}
\beta:F_i^{\vee}\to ((G_i^{\vee})^R\circ G_i^{\vee})\circ F_i^{\vee}\cong (G_i^{\vee})^R\circ (F_i\circ G_i)^{\vee}\to (G_i^{\vee})^R,
\end{equation*}
and $\gamma$ is the composition
\begin{equation*}
\gamma:(G_i^{\vee})^R\to (G_i\circ F_j)^{\vee}\circ (G_i^{\vee})^R\cong F_j^{\vee}\circ (G_i^{\vee}\circ (G_i^{\vee})^R)\to F_j^{\vee}.
\end{equation*}
It follows that we have an isomorphism in $\Fun_{\cE^{mop}}^{\acc,\lax}(\cD^{\vee},\cC^{\vee}):$
\begin{equation*}
\indlim[i]F_i^{\vee}\cong \indlim[i]((F_i^{R,\cont})^{\vee})^R.
\end{equation*}
The latter is an $\omega_1$-directed colimit of functors which commute with (all) limits, hence this functor commutes with countable limits, as required.
\end{proof}

Part \ref{hom_epi_onto_image} of the following proposition can be considered as an analogue of \cite[Corollary 3.11]{E25}, although the proof is quite different.

\begin{prop}\label{prop:Hom^dual_into_formally_omega_1_injective}
Let $\cC,\cD$ be dualizable left $\cE$-modules, and suppose that $\cC$ is $\omega_1$-compact and $\cD$ is formally $\omega_1$-injective. 
\begin{enumerate}[label=(\roman*),ref=(\roman*)]
\item The functor
\begin{equation}\label{eq:key_hom_epi_for_formally_omega_1_injective}
	(\cC^{\vee}\tens{\cE}\cD)^{\omega_1}\simeq\Fun_{\cE}^{LL}(\cC,\Ind(\cD^{\omega_1}))\to \Fun_{\cE}^{LL}(\cC,\Calk_{\omega_1}(\cD))
\end{equation}
is a homological epimorphism onto its image. \label{hom_epi_onto_image}
\item The essential image of the inclusion $\un{\Hom}_{\cE}^{\dual}(\cC,\cD)\to\Ind((\cC^{\vee}\tens{\cE}\cD)^{\omega_1})$ consists of formal colimits $\inddlim[i\in I]F_i,$ where $I$ is directed and for any $i\in I$ there exists $j\geq i$ such that the map $F_i\to F_j$ is right trace-class in $\Fun_{\cE}^L(\cC,\cD).$ \label{descr_of_Hom^dual}
\item The category $\un{\Hom}_{\cE}^{\dual}(\cC,\cD)$ is formally $\omega_1$-injective, and we have an equivalence
\begin{equation*}
	\un{\Hom}_{\cE}^{\dual}(\cC,\cD)_{\omega_1}\simeq\Fun_{\cE}^{LL}(\cC,\Ind(\cD_{\omega_1})).
\end{equation*} \label{formal_omega_1_injectivity_of_Hom^dual}
\end{enumerate}
\end{prop}

\begin{proof}
\ref{hom_epi_onto_image} The statement is equivalent to the following: for $F,G\in(\cC^{\vee}\tens{\cE}\cD)^{\omega_1}$ the map 
\begin{multline*}
\Hom_{\cC^{\vee}\tens{\cE}\cC}(\Id_{\cC},(-)^{R,\cont}\circ G)\tens{(\cC^{\vee}\tens{\cE}\cD)^{\omega_1}}\Hom_{\cC^{\vee}\tens{\cE}\cC}(\Id_{\cC},F^{R,\cont}\circ -)\\
\to\Hom_{\cC^{\vee}\tens{\cE}\cC}(\Id_{\cC},F^{R,\cont}\circ G)
\end{multline*}
is an isomorphism. We first observe that by Proposition \ref{prop:properties_of_formally_omega_1_injective} \ref{when_F^cont_commutes_with_countable_limits} the functor $F^{R,\cont}$ commutes with countable limits. It is convenient to replace $F^{R,\cont}$ with a general functor $H\in\Fun_{\cE}^L(\cD,\cC)$ which commutes with countable limits. We will prove that for such $H$ and for $G$ as above the map 
\begin{multline}\label{eq:tricky_tensor_product}
\Hom_{\cC^{\vee}\tens{\cE}\cC}(\Id_{\cC},(-)^{R,\cont}\circ G)\tens{(\cC^{\vee}\tens{\cE}\cD)^{\omega_1}}\Hom_{\cC^{\vee}\tens{\cE}\cC}(\Id_{\cC},H\circ -)\\
\to\Hom_{\cC^{\vee}\tens{\cE}\cC}(\Id_{\cC},H\circ G).
\end{multline}
is an isomorphism.
Consider the $\cE^{mop}$-linear functor $H^{\vee}:\cC^{\vee}\to\cD^{\vee}.$ By Proposition \ref{prop:category_of_functors_to_Ind_D_omega_1} we have $H^{\vee}\cong\indlim[i\in I]H_i,$ where $I$ is $\omega_1$-directed, $H_i\in(\cD^{\vee}\tens{\cE}\cC)^{\omega_1},$ and for any $i\in I$ there exists $j\geq i$ such that the map $H_i\to H_j$ is right trace-class in $\Fun_{\cE^{mop}}(\cC^{\vee},\cD^{\vee}).$ For such $i\leq j$ the map $(H_j^{R,\cont})^{\vee}\to (H_i^{R,\cont})^{\vee}$ is right trace-class in $\Fun_{\cE}^L(\cC,\cD),$ hence it factors through an object of $(\cC^{\vee}\tens{\cE}\cD)^{\omega_1}.$ It follows that we have a pro-isomorphism $\proolim[i](H_i^{R,\cont})^{\vee}\cong\proolim[j\in J]H_j',$ where $J$ is $\omega_1$-codirected and $H_j'\in (\cC^{\vee}\tens{\cE}\cD)^{\omega_1}.$ We obtain the isomorphisms
\begin{multline*}
\Hom_{\cC^{\vee}\tens{\cE}\cC}(\Id_{\cC},H\circ -)\cong \Hom_{\Fun_{\cE^{mop}}^L(\cC^{\vee},\cC^{\vee})}(\Id_{\cC^{\vee}},(-)^{\vee}\circ H^{\vee})\\
\cong \Hom_{\Fun_{\cE^{mop}}^L(\cC^{\vee},\cC^{\vee})}(\Id_{\cC^{\vee}},(-)^{\vee}\circ \indlim[i]H_i)\cong \indlim[i]\Hom_{\Fun_{\cE^{mop}}^L(\cC^{\vee},\cC^{\vee})}(\Id_{\cC^{\vee}},(-)^{\vee}\circ H_i)\\
\cong \indlim[i]\Hom_{\Fun_{\cE^{mop}}^L(\cD^{\vee},\cC^{\vee})}(H_i^{R,\cont},(-)^{\vee})
\cong \indlim[i]\Hom_{\Fun_{\cE}^L(\cC,\cD)}((H_i^{R,\cont})^{\vee},-)\\
\cong \indlim[j]\Hom_{\Fun_{\cE}^L(\cC,\cD)}(H_j',-). 
\end{multline*}
Using this we can describe the source of \eqref{eq:tricky_tensor_product} as a colimit:
\begin{multline*}
\Hom_{\cC^{\vee}\tens{\cE}\cC}(\Id_{\cC},(-)^{R,\cont}\circ G)\tens{(\cC^{\vee}\tens{\cE}\cD)^{\omega_1}}\Hom_{\cC^{\vee}\tens{\cE}\cC}(\Id_{\cC},H\circ -)\\
\cong \Hom_{\cC^{\vee}\tens{\cE}\cC}(\Id_{\cC},(-)^{R,\cont}\circ G)\tens{(\cC^{\vee}\tens{\cE}\cD)^{\omega_1}}(\indlim[j]\Hom_{(\cC^{\vee}\tens{\cE}\cD)^{\omega_1}}(H_j',-))\\
\cong \indlim[j]\Hom_{\cC^{\vee}\tens{\cE}\cC}(\Id_{\cC},(H_j')^{R,\cont}\circ G).
\end{multline*}
By construction, for any $j\in J$ there exists $k\leq j$ such that the map $H_k'\to H_j'$ is right trace-class in $\Fun_{\cE}^L(\cC,\cD).$ It follows that we have an ind-isomorphism $\inddlim[j](H_j')^{R,\cont}\cong \inddlim[i]H_i^{\vee}.$ Hence, we obtain the isomorphisms
\begin{equation*}
\indlim[j]\Hom_{\cC^{\vee}\tens{\cE}\cC}(\Id_{\cC},(H_j')^{R,\cont}\circ G)\cong \indlim[i]\Hom_{\cC^{\vee}\tens{\cE}\cC}(\Id_{\cC},H_i^{\vee}\circ G)\cong \Hom_{\cC^{\vee}\tens{\cE}\cC}(\Id_{\cC},H\circ G).
\end{equation*}
This proves that the map \eqref{eq:tricky_tensor_product} is indeed an isomorphism.

Next, \ref{descr_of_Hom^dual} follows directly from \ref{hom_epi_onto_image}: we obtain a description of the dualizable internal $\Hom$ as a naive kernel:
\begin{multline*}
\un{\Hom}_{\cE}^{\dual}(\cC,\cD)\simeq\ker^{\dual}(\Ind((\cC^{\vee}\tens{\cE}\cD)^{\omega_1})\to \Ind(\Fun_{\cE}^{LL}(\cC,\Calk_{\omega_1}(\cD))))\\
\simeq \ker(\Ind((\cC^{\vee}\tens{\cE}\cD)^{\omega_1})\to \Ind(\Fun_{\cE}^{LL}(\cC,\Calk_{\omega_1}(\cD)))).
\end{multline*}
The latter category is described exactly as in \ref{descr_of_Hom^dual}.

It remains to prove \ref{formal_omega_1_injectivity_of_Hom^dual}. To prove the formal $\omega_1$-injectivity we consider the (strongly continuous) fully faithful functor
\begin{equation*}
\un{\Hom}_{\cE}^{\dual}(\cC,\cD)^{\vee}\xto{\iota} \Ind((\cC^{\vee}\tens{\cE}\cD)^{\omega_1,op})\simeq \Fun((\cC^{\vee}\tens{\cE}\cD)^{\omega_1},\Sp).
\end{equation*}
Note that the target of this functor is formally $\omega_1$-injective by Proposition \ref{prop:when_hat_Y_is_cocontinuous}. To avoid confusion, for $F\in(\cC^{\vee}\tens{\cE}\cD)^{\omega_1}$ we denote by $F^{op}$ the corresponding object of the opposite category. By \ref{hom_epi_onto_image}, the composition of $\iota$ with its right adjoint $\iota^R$ is given on the compact objects by 
\begin{equation*}
\iota(\iota^R(F^{op}))=\Hom_{\cC^{\vee}\tens{\cE}\cC}(\Id,F^{R,\cont}\circ -),\quad F\in(\cC^{\vee}\tens{\cE}\cD)^{\omega_1}.
\end{equation*}
By the proof of \ref{hom_epi_onto_image}, we have
\begin{equation*}
\iota(\iota^R(F^{op}))\in \Ind_{\omega_1}((\cC^{\vee}\tens{\cE}\cD)^{\omega_1,op})\simeq \Ind((\cC^{\vee}\tens{\cE}\cD)^{\omega_1,op})_{\omega_1},\quad F\in (\cC^{\vee}\tens{\cE}\cD)^{\omega_1}.
\end{equation*}
It follows from Proposition \ref{prop:formal_omega_1_injectivity_for_subcategory} \ref{composition_preserves_D_omega_1} that the category $\un{\Hom}_{\cE}^{\dual}(\cC,\cD)^{\vee}$ is formally $\omega_1$-injective, hence so is $\un{\Hom}_{\cE}^{\dual}(\cC,\cD)$ by Proposition \ref{prop:formal_omega_1_injectivity_reformulations}.

By Proposition \ref{prop:C_omega_1_for_subcategory} the category $(\un{\Hom}_{\cE}^{\dual}(\cC,\cD))_{\omega_1}$ is simply the intersection of $\un{\Hom}_{\cE}^{\dual}(\cC,\cD)$ and $\cC^{\vee}\tens{\cE}\cD\simeq\Ind_{\omega_1}((\cC^{\vee}\tens{\cE}\cD)^{\omega_1})$ inside $\Ind((\cC^{\vee}\tens{\cE}\cD)^{\omega_1}).$ It follows from \ref{descr_of_Hom^dual} and Propositions \ref{prop:another_description_of_functors_to_Ind_D_omega_1} and \ref{prop:Ind_of_D_omega_1_universal_property} that this intersection is exactly the category 
\begin{equation*}
\Fun_{\cE}^{LL}(\cC,\Ind(\cD_{\omega_1}))\subset \Fun_{\cE}^{LL}(\cC,\Ind(\cD))\simeq \cC^{\vee}\tens{\cE}\cD.
\end{equation*}
\end{proof}

We obtain the following description of generators of the dual category of the relative dualizable internal $\Hom,$ analogous to \cite[Proposition 3.12]{E25}.

\begin{prop}\label{prop:generators_of_dual_of_Hom^dual}
Let $\cC,\cD\in\Cat_{\cE}^{\dual}$ be as in Proposition \ref{prop:Hom^dual_into_formally_omega_1_injective}. Then the essential image of the fully faithful functor
\begin{equation*}
\Phi:\un{\Hom}_{\cE}^{\dual}(\cC,\cD)^{\vee}\to \Fun((\cC^{\vee}\tens{\cE}\cD)^{\omega_1},\Sp)
\end{equation*}
contains all objects of the form $\Hom_{\cC^{\vee}\tens{\cE}\cC}(\Id_{\cC},G\circ -),$ where $G\in\Fun_{\cE}^L(\cD,\cC)$ is a functor which commutes with countable limits. Moreover, the image of $\Phi$ is generated by such objects with $G=F^{R,\cont},$ $F\in(\cC^{\vee}\tens{\cE}\cD)^{\omega_1}.$ 
\end{prop}

\begin{proof}
This follows directly from the proof of Proposition \ref{prop:Hom^dual_into_formally_omega_1_injective}.
\end{proof}

The statements of the following corollary are analogous to \cite[Proposition 3.13, Corollary 3.14]{E25}, but the assumptions are quite different.

\begin{cor}\label{cor:internal_Hom_to_functors_commuting_with_countable_limits}
Let $\cC,\cD,\cD'$ be dualizable left $\cE$-modules, such that $\cC$ is $\omega_1$-compact in $\Cat_{\cE}^{\dual}$ and both $\cD$ and $\cD'$ are formally $\omega_1$-injective.
\begin{enumerate}[label=(\roman*),ref=(\roman*)]
\item Let $F:\cD\to\cD'$ be a strongly continuous $\cE$-linear functor which commutes with countable limits. Then the following square commutes:
\begin{equation}\label{eq:Beck_Chevalley}
\begin{tikzcd}
\un{\Hom}_{\cE}^{\dual}(\cC,\cD')^{\vee}\ar{r}{\un{\Hom}_{\cE}^{\dual}(\cC,F)^{\vee}}\ar[d, "\Phi'"]  & \un{\Hom}_{\cE}^{\dual}(\cC,\cD)^{\vee}\ar[d, "\Phi"]\\
\Fun((\cC^{\vee}\tens{\cE}\cD')^{\omega_1},\Sp)\ar[r] & \Fun((\cC^{\vee}\tens{\cE}\cD)^{\omega_1},\Sp).
\end{tikzcd}
\end{equation}
Here the vertical functors are as in Proposition \ref{prop:generators_of_dual_of_Hom^dual}. If moreover $F$ is a quotient functor, then
\begin{equation}\label{eq:Hom^dual_to_F}
\un{\Hom}_{\cE}^{\dual}(\cC,F):\un{\Hom}_{\cE}^{\dual}(\cC,\cD)\to \un{\Hom}_{\cE}^{\dual}(\cC,\cD')
\end{equation}
is also a quotient functor. \label{Beck_Chevalley_duals_of_Hom^dual}
\item The functor
\begin{equation*}
	(\cC^{\vee}\tens{\cE}\cD)^{\omega_1}\simeq\Fun_{\cE}^{LL}(\cC,\Ind(\cD^{\omega_1}))\to \Fun_{\cE}^{LL}(\cC,Calk_{\omega_1}(\cD))
\end{equation*}
is a homological epimorphism. Moreover, we have an equivalence
\begin{equation*}
\un{\Hom}_{\cE}^{\dual}(\cC,\Calk_{\omega_1}(\cD))\simeq\Ind(\Fun_{\cE}^{LL}(\cC,\Calk_{\omega_1}(\cD))).
\end{equation*} \label{key_hom_epi}
\end{enumerate}
\end{cor}

\begin{proof}
\ref{Beck_Chevalley_duals_of_Hom^dual} To prove the commutativity of \eqref{eq:Beck_Chevalley} we need to show that the lower horizontal functor in takes the essential image of $\Phi'$ to the essential image of $\Phi.$ If $G\in\Fun_{\cE}^L(\cD',\cC)$ commutes with countable limits, then so does $G\circ F$ by assumption, and the lower horizontal functor in \eqref{eq:Beck_Chevalley} takes $\Hom_{\cC^{\vee}\tens{\cE}\cC}(\Id_{\cC},G\circ -)$ to $\Hom_{\cC^{\vee}\tens{\cE}\cC}(\Id_{\cC},G\circ F\circ -).$ By Proposition \ref{prop:generators_of_dual_of_Hom^dual} we deduce the commutativity of \eqref{eq:Beck_Chevalley}.

If $F$ is a quotient functor, then the lower horizontal functor in \eqref{eq:Beck_Chevalley} is fully faithful, hence so is the upper horizontal functor. This exactly means that the functor \eqref{eq:Hom^dual_to_F} is a quotient functor, as stated.

\ref{key_hom_epi} Since $\cD$ is formally $\omega_1$-injective, the functor $\Ind(\cD^{\omega_1})\to \Calk_{\omega_1}(\cD)$ commutes with countable limits. Its source is formally $\omega_1$-injective by Proposition \ref{prop:when_hat_Y_is_cocontinuous}, and so is the target by Theorem \ref{th:hat_Y_for_Calkin}. Applying \ref{Beck_Chevalley_duals_of_Hom^dual}, we see that the functor
\begin{equation*}
\un{\Hom}_{\cE}^{\dual}(\cC,\Ind(\cD^{\omega_1})\to\un{\Hom}_{\cE}^{\dual}(\cC,\Calk_{\omega_1}(\cD))\end{equation*}
is a quotient functor. Since we have
\begin{equation*}
\un{\Hom}_{\cE}^{\dual}(\cC,\Ind(\cD^{\omega_1})\simeq \Ind((\cC^{\vee}\tens{\cE}\cD)^{\omega_1}),
\end{equation*}
we deduce both assertions of \ref{key_hom_epi}.
\end{proof}

\begin{proof}[Proof of Theorem \ref{th:internal_omega_1_injectivity}]
It suffices to prove the first statement since the assertion about Calkin categories is a special case by Theorem \ref{th:hat_Y_for_Calkin}. 

Let $\cA,\cB\in(\Cat_{\cE}^{\dual})^{\omega_1}$ be $\omega_1$-compact dualizable left $\cE$-modules, and let $F:\cA\to\cB$ be a strongly continuous fully faithful $\cE$-linear functor. We need to prove that
\begin{equation*}
\un{\Hom}_{\cE}^{\dual}(F,\cD):\un{\Hom}_{\cE}^{\dual}(\cB,\cD)\to\un{\Hom}_{\cE}^{\dual}(\cA,\cD)
\end{equation*}
is a quotient functor. It follows from Proposition \ref{prop:Hom^dual_into_formally_omega_1_injective} \ref{formal_omega_1_injectivity_of_Hom^dual} that we have a commutative square of accessible categories and exact accessible functors
\begin{equation}\label{eq:precomposition_with_F}
\begin{tikzcd}
\Fun_{\cE}^{LL}(\cB,\Ind(\cD_{\omega_1}))\ar{r}{\Phi}\ar[d] & [1em] \Fun_{\cE}^{LL}(\cA,\Ind(\cD_{\omega_1}))\ar[d]\\
\un{\Hom}_{\cE}^{\dual}(\cB,\cD)\ar{r}{\un{\Hom}_{\cE}^{\dual}(F,\cD)} & \un{\Hom}_{\cE}^{\dual}(\cA,\cD),
\end{tikzcd}
\end{equation}
where the upper horizontal functor $\Phi$ is given by precomposition with $F.$ By loc. cit. the vertical functors in \eqref{eq:precomposition_with_F} are fully faithful and their images generate the targets (via colimits). Hence, by Lemma \ref{lem:sufficient_for_localization} below it suffices to prove that $\Phi$ has a fully faithful right adjoint. 

Denote by $\cT(\cA,\cD)\subset \Fun_{\cE^{mop}}^L(\cD^{\vee},\cA^{\vee})$ the full subcategory which consists of functors commuting with countable limits, and similarly for $\cT(\cB,\cD)\subset \Fun_{\cE^{mop}}^L(\cD^{\vee},\cB^{\vee}).$ By Proposition \ref{prop:Ind_of_D_omega_1_universal_property}, the functor $\Phi$ from \eqref{eq:precomposition_with_F} is identified with the functor
\begin{equation*}
F^{\vee}\circ -:\cT(\cB,\cD)\to\cT(\cA,\cD)
\end{equation*}
(note that $F^{\vee}$ commutes with all limits). Now, the fully faithfulness of $F$ implies that the functor
\begin{equation*}
\Psi=F^{\vee}\circ -:\Fun_{\cE^{mop}}^L(\cD^{\vee},\cB^{\vee})\to \Fun_{\cE^{mop}}^L(\cD^{\vee},\cA^{\vee})
\end{equation*}
is a quotient functor, so its right adjoint $\Psi^R$ is fully faithful. Hence, we only need to show that $\Psi^R$ takes $\cT(\cA,\cD)$ to $\cT(\cB,\cD).$ So let $G:\cD^{\vee}\to\cA^{\vee}$ be a continuous $\cE^{mop}$-linear functor which commutes with countable limits. Then $\Psi^R(G)=((F^{\vee})^R\circ G)^{\cont}.$ The $\omega_1$-compactness of $\cA$ and $\cB$ in $\Cat_{\cE}^{\dual}$ implies that $F^{\vee}$ takes $(\cB^{\vee})^{\omega_1}$ to $(\cA^{\vee})^{\omega_1}.$ Hence, its right adjoint $(F^{\vee})^R$ commutes with $\omega_1$-filtered colimits (and all limits). By assumption on $G,$ we see that the functor $(F^{\vee})^R\circ G$ commutes with $\omega_1$-filtered colimits and countable limits. Applying Proposition \ref{prop:properties_of_formally_omega_1_injective} \ref{when_F^cont_commutes_with_countable_limits}, we conclude that the functor $((F^{\vee})^R\circ G)^{\cont}:\cD^{\vee}\to\cB^{\vee}$ commutes with countable limits, i.e. $\Psi^R(G)\in\cT(\cB,\cD).$ This proves the theorem. 
\end{proof}

We used a very general sufficient condition for a strongly continuous functor to be a localization, which is quite straightforward and probably known.

\begin{lemma}\label{lem:sufficient_for_localization}
Let $F:\cC\to\cD$ be a strongly continuous functor between presentable stable categories. Let $\cC'\subset\cC$ and $\cD'\subset\cD$ be full stable subcategories closed under $\kappa$-filtered colimits for some regular cardinal $\kappa.$ Suppose that the following conditions hold:
\begin{itemize}
\item  $\cC'$ resp. $\cD'$ generates $\cC$ resp. $\cD$ via colimits;
\item  $F$ takes $\cC'$ to $\cD';$
\item the induced functor $G=F_{\mid \cC'}:\cC'\to\cD'$ has a fully faithful right adjoint $G^R.$ 
\end{itemize}
Then $F$ is a quotient functor.
\end{lemma}

\begin{proof}
Denote by $F^R$ the right adjoint to $F.$ For $x\in\cD'$ we have a natural morphism $\varphi_x:G^R(x)\to F^R(x)$ in $\cC.$ Then the object $\Cone(\varphi_x)\in\cC$ is contained in the right orthogonal to $\cC',$ hence it vanishes by our assumption that $\cC'$ generates $\cC.$ Hence, the functor $F^R$ takes $\cD'$ to $\cC'.$ It follows that the for $x\in\cD'$ the adjunction counit $F(F^R(x))\to x$ is an isomorphism (by the fully faithfulness of $G^R$). Since $\cD'$ generates $\cD$ and $F^R$ continuous, we conclude that the counit $F\circ F^R\to\id_{\cD}$ is an isomorphism, i.e. $F$ is a quotient functor. 
\end{proof}

\subsection{Proof of Theorem \ref{th:morphisms_in_Mot^loc_via_internal_Hom}}
\label{ssec:proof_of_theorem_on_morphisms_in_Mot^loc}

We will need the following lemma. We refer to Definition \ref{def:nuclear_E_modules} for the notion of a right trace-class strongly continuous $\cE$-linear functor between relatively compactly generated left $\cE$-modules. 

\begin{lemma}\label{lem:diagonal_arrow_for_Inds_and_Calkins}
	Let $\cC,\cC'\in\Cat_{\cE}^{\cg}$ be relatively compactly generated left $\cE$-modules, and let $\cD\in\Cat_{\cE^{mop}}^{\dual}$ be a dualizable right $\cE$-module. Let $F:\cC\to\cC'$ be a strongly continuous $\cE$-linear functor which is right trace-class over $\cE.$ 
	\begin{enumerate}[label=(\roman*),ref=(\roman*)]
		\item There is a strongly continuous functor $H:\Ind((\cD\tens{\cE}\cC)^{\omega_1})\to\Ind(\cD^{\omega_1})\tens{\cE}\cC',$ which is the diagonal arrow in the following commutative diagram:
		\begin{equation}\label{eq:diagonal_arrow_for_Inds}
			\begin{tikzcd}
				\Ind(\cD^{\omega_1})\tens{\cE}\cC\ar[r]\ar[d] & \Ind(\cD^{\omega_1})\tens{\cE}\cC'\ar[d]\\
				\Ind((\cD\tens{\cE}\cC)^{\omega_1})\ar[r]\ar[ru, "H"] & \Ind((\cD\tens{\cE}\cC')^{\omega_1})
			\end{tikzcd}
		\end{equation} \label{diagonal_arrow_for_Inds}
		\item There is a strongly continuous functor $\bbar{H}:\Calk_{\omega_1}(\cD\tens{\cE}\cC)\to\Calk_{\omega_1}(\cD)\tens{\cE}\cC',$ which is the diagonal arrow in the following commutative diagram:
		\begin{equation}\label{eq:diagonal_arrow_for_Calkins}
			\begin{tikzcd}
				\Calk_{\omega_1}(\cD)\tens{\cE}\cC\ar[r]\ar[d] & \Calk_{\omega_1}(\cD)\tens{\cE}\cC'\ar[d]\\
				\Calk_{\omega_1}(\cD\tens{\cE}\cC)\ar[r]\ar[ru, "\bbar{H}"] & \Calk_{\omega_1}(\cD\tens{\cE}\cC')
			\end{tikzcd}
		\end{equation} \label{diagonal_arrow_for_Calkins}
	\end{enumerate}
\end{lemma}

\begin{proof}
	We prove \ref{diagonal_arrow_for_Inds}. The left vertical functor in \eqref{eq:diagonal_arrow_for_Inds} is induced by the strongly continuous functor
	$\Ind(\cD^{\omega_1})\tens{\cE}\cC\to\Ind((\cD\tens{\cE}\cC)^{\omega_1}),$ which corresponds to the $\omega_1$-strongly continuous functor
	\begin{equation*}
		\Ind(\cD^{\omega_1})\tens{\cE}\cC\xto{\colim\boxtimes\id}\cD\tens{\cE}\cC.
	\end{equation*}
	by adjunction from \cite[Proposition 1.89]{E24}. Similarly for the right vertical functor in \eqref{eq:diagonal_arrow_for_Inds}. Now, choose a right trace-class witness for $F,$ given by an object $\tilde{F}\in(\un{\Hom}_{\cE}^{\dual}(\cC,\cE)\tens{\cE}\cC')^{\omega}.$ We define  
	\begin{equation*}
	G:\Ind((\cD\tens{\cE}\cC)^{\omega_1})\otimes\un{\Hom}_{\cE}^{\dual}(\cC,\cE)\to \Ind(\cD^{\omega_1})
	\end{equation*} 
	to be the strongly continuous $\cE^{mop}$-linear functor which corresponds by adjunction from \cite[Proposition 1.16]{E25} to the $\omega_1$-strongly continuous composition
	\begin{equation*}
		\Ind((\cD\tens{\cE}\cC)^{\omega_1})\otimes\un{\Hom}_{\cE}^{\dual}(\cC,\cE)\xto{\colim\boxtimes\id}\cD\tens{\cE}\cC\otimes\un{\Hom}_{\cE}^{\dual}(\cC,\cE)\to \cD.
	\end{equation*}
	We obtain the required strongly continuous functor
	\begin{equation*}
		H:\Ind((\cD\tens{\cE}\cC)^{\omega_1})\xto{\id\boxtimes\tilde{F}} \Ind((\cD\tens{\cE}\cC)^{\omega_1})\otimes\un{\Hom}_{\cE}^{\dual}(\cC,\cE)\tens{\cE}\cC'\xto{G\boxtimes\id} \Ind(\cD^{\omega_1})\tens{\cE}\cC'.
	\end{equation*}
	By construction, the diagram \eqref{eq:diagonal_arrow_for_Inds} naturally commutes.
	
	To prove \ref{diagonal_arrow_for_Calkins}, note that the following square commutes:
	\begin{equation*}
		\begin{tikzcd}
			\cD\tens{\cE}\cC\ar[r, "\id\boxtimes F"]\ar[d, "\hat{\cY}"] & \cD\tens{\cE}\cC'\ar[d, "\hat{\cY}\boxtimes\id"]\\
			\Ind((\cD\tens{\cE}\cC)^{\omega_1})\ar[r, "H"] & \Ind(\cD^{\omega_1})\tens{\cE}\cC'.
		\end{tikzcd}
	\end{equation*}
	Hence, the functor $H$ induces the required functor $\bar{H}:\Calk_{\omega_1}(\cD\tens{\cE}\cC)\to\Calk_{\omega_1}(\cD)\tens{\cE}\cC'.$ By construction, the diagram \eqref{eq:diagonal_arrow_for_Calkins} naturally commutes.
\end{proof}


\begin{proof}[Proof of Theorem \ref{th:morphisms_in_Mot^loc_via_internal_Hom}] We prove \ref{internal_Hom_into_Calk}. First we assume that $\cE$ is only $\bE_1$-monoidal. By Theorem \ref{th:hat_Y_for_Calkin} the category $\Calk_{\omega_1}(\cD)$ is formally $\omega_1$-injective. Note that $\cU_{\loc}(\cD)\cong \Omega \cU_{\loc}(\Calk_{\omega_1}(\cD))$ in $\Mot^{\loc}_{\cE}.$ We replace the category $\Calk_{\omega_1}(\cD)$ with an arbitrary formally $\omega_1$-injective $\cD'\in\Cat_{\cE}^{\dual}.$ Then the category $\un{\Hom}_{\cE}^{\dual}(\cC,\Calk_{\omega_1}(\cD'))$ is compactly generated by Corollary \ref{cor:internal_Hom_to_functors_commuting_with_countable_limits}, which in particular gives the compact generation of $\un{\Hom}_{\cE}^{\dual}(\cC,\Calk_{\omega_1}^2(\cD))$ for any $\cD\in\Cat_{\cE}^{\dual}.$ We will prove the following generalization of the isomorphisms \eqref{eq:morphisms_in_Mot^loc_via_Calk^2}:
\begin{multline}\label{eq:morphisms_in_Mot^loc_via_Calk_of_formally_omega_1_injective}
\Omega K(\Fun_{\cE}^{LL}(\cC,\Calk_{\omega_1}(\cD')))\xto{\sim} K^{\cont}(\un{\Hom}_{\cE}^{\dual}(\cC,\cD'))\\
\cong \Hom_{\Mot^{\loc}_{\cE}}(\cU_{\loc}(\cC),\cU_{\loc}(\cD')),
\end{multline}	
where $\cC\in\Cat_{\cE}^{\dual}$ is assumed to be $\omega_1$-compact (and again $\cD'\in\Cat_{\cE}^{\dual}$ is assumed to be formally $\omega_1$-injective).

By Corollary \ref{cor:internal_Hom_to_functors_commuting_with_countable_limits} we have a short exact sequence
\begin{equation*}
0\to \un{\Hom}_{\cE}^{\dual}(\cC,\cD')\to \Ind((\cC^{\vee}\tens{\cE}\cD')^{\omega_1})\to \Ind(\Fun_{\cE}^{LL}(\cC,\Calk_{\omega_1}(\cD')))\to 0. 
\end{equation*} 
This gives the first isomorphism in \eqref{eq:morphisms_in_Mot^loc_via_Calk_of_formally_omega_1_injective}. It remains to prove that the map 
\begin{equation}\label{eq:K_theory_of_functors_to_Calk_of_formally_omega_1_injective}
K(\Fun_{\cE}^{LL}(\cC,\Calk_{\omega_1}(\cD')))\to \Hom_{\Mot^{\loc}_{\cE}}(\cU_{\loc}(\cC),\cU_{\loc}(\Calk_{\omega_1}(\cD')))
\end{equation}
is an isomorphism. For this we will need some reductions.
	
First, the same argument as in \cite[Proof of Theorem C.6]{E24} shows that there is a short exact sequence in $\Cat_{\cE}^{\dual}$ of the form
\begin{equation*}
	0\to\cC\to\cC'_1\to\cC'_2\to 0,
\end{equation*}
where $\cC'_1$ and $\cC'_2$ are in $(\Cat_{\cE}^{\cg})^{\omega_1}.$  Hence, we may and will assume that $\cC$ is relatively compactly generated over $\cE.$ 
	
Then, by Proposition \ref{prop:resolution_by_nuclear} there is a short exact sequence in $\Cat_{\cE}^{\cg}$ of the form
\begin{equation*}
		0\to\cC'_3\to\cC'_4\to\cC\to 0,
\end{equation*}
where $\cC'_3$ and $\cC'_4$ are basic nuclear over $\cE.$ Hence, we may and will assume that $\cC$ is basic nuclear over $\cE.$
	
Let $\hat{\cY}(\cC)=\inddlim[n\in\N]\cC_n,$ where $\cC_n\in(\Cat_{\cE}^{\cg})^{\omega_1}$ and each functor $\cC_n\to\cC_{n+1}$ is right trace-class over $\cE.$ We put $\cT=\Calk_{\omega_1}(\cD).$ We know that for each $n$ the category $\un{\Hom}_{\cE}^{\dual}(\cC_n,\cT)$ is compactly generated. By Theorem \ref{th:morphisms_in_Mot^loc_via_limits} we obtain an isomorphism
\begin{equation*}
	\Hom_{\Mot_{\cE}^{\loc}}(\cU_{\loc}(\cC),\cU_{\loc}(\cT))\cong \prolim[n]K(\Fun_{\cE}^{LL}(\cC_n,\cT)).
\end{equation*}
Hence, we need to prove that the map
\begin{equation}\label{eq:unusual_continuity_of_K_theory}
	K(\Fun_{\cE}^{LL}(\cC,\cT))\to \prolim[n]K(\Fun_{\cE}^{LL}(\cC_n,\cT))
\end{equation}
is an isomorphism. We put $\cA_n=\Fun_{\cE}^{LL}(\cC_n,\cT).$ To be able to apply \cite[Theorem 6.1]{E25}, we need to prove that the functor
	\begin{equation}\label{eq:lim_of_Ind_to_lim_of_Calk}
		\prolim[n]\Ind(\cA_n)^{\omega_1}\to\prolim[n]\Calk_{\omega_1}(\cA_n)
	\end{equation}
	(between small stable categories) is a homological epimorphism. To see this, first note that for any $\cD''\in\Cat_{\cE}^{\dual}$ we have an equivalence in $\Pro(\Cat_{\st}^{\dual}):$
	\begin{equation*}
		\proolim[n]\un{\Hom}_{\cE}^{\dual}(\cC_n,\cE)\tens{\cE}\cD''\xto{\sim} \proolim[n]\un{\Hom}_{\cE}^{\dual}(\cC_n,\cD'').
	\end{equation*} 
	It follows from the assumption that the functors $\cC_n\to\cC_{n+1}$ are right trace-class over $\cE.$ Using Lemma \ref{lem:diagonal_arrow_for_Inds_and_Calkins} \ref{diagonal_arrow_for_Inds}, we see that we have an equivalence in $\Pro(\Cat^{\perf}):$
	\begin{equation*}
		\proolim[n](\un{\Hom}_{\cE}^{\dual}(\cC_n,\cE)\tens{\cE}\Ind(\cT^{\omega_1}))^{\omega}\xto{\sim} \proolim[n](\un{\Hom}_{\cE}^{\dual}(\cC_n,\cE)\tens{\cE}\cT)^{\omega_1}
	\end{equation*} 
	
	We obtain equivalences
	\begin{multline*}
		\prolim[n]\Ind(\cA_n)^{\omega_1}\simeq\prolim[n]\un{\Hom}_{\cE}^{\dual}(\cC_n,\cT)^{\omega_1}\simeq \prolim[n](\un{\Hom}_{\cE}^{\dual}(\cC_n,\cE)\tens{\cE}\cT)^{\omega_1}\\
		\simeq \prolim[n](\un{\Hom}_{\cE}^{\dual}(\cC_n,\cE)\tens{\cE}\Ind(\cT^{\omega_1}))^{\omega}
		\simeq\prolim[n]\un{\Hom}_{\cE}^{\dual}(\cC_n,\Ind(\cT^{\omega_1}))^{\omega}\\
		\simeq\prolim[n]\Fun_{\cE}^{LL}(\cC_n,\Ind(\cT^{\omega_1}))\simeq\Fun_{\cE}^{LL}(\cC,\Ind(\cT^{\omega_1}))\simeq(\cC^{\vee}\tens{\cE}\cT)^{\omega_1}. 
	\end{multline*}
	Similarly, using Lemma \ref{lem:diagonal_arrow_for_Inds_and_Calkins} \ref{diagonal_arrow_for_Calkins} we obtain equivalences
	\begin{multline*}
		\prolim[n]\Calk_{\omega_1}(\cA_n)\simeq \prolim[n]\Calk_{\omega_1}^{\cont}(\un{\Hom}_{\cE}^{\dual}(\cC_n,\cT))\simeq
		\prolim[n]\Calk_{\omega_1}^{\cont}(\un{\Hom}_{\cE}^{\dual}(\cC_n,\cE)\tens{\cE}\cT)\\
		\simeq \prolim[n](\un{\Hom}_{\cE}^{\dual}(\cC_n,\cE)\tens{\cE}\Calk_{\omega_1}(\cT))^{\omega}\simeq \prolim[n]\un{\Hom}_{\cE}^{\dual}(\cC_n,\Calk_{\omega_1}(\cT))^{\omega}\\
		\simeq \prolim[n]\Fun_{\cE}^{LL}(\cC_n,\Calk_{\omega_1}(\cT))\simeq\Fun_{\cE}^{LL}(\cC,\Calk_{\omega_1}(\cT)).
	\end{multline*}
	The functor \eqref{eq:lim_of_Ind_to_lim_of_Calk} is now identified with the functor
	\begin{equation*}
		(\cC^{\vee}\tens{\cE}\cT)^{\omega_1}\to \Fun_{\cE}^{LL}(\cC,\Calk_{\omega_1}(\cT)),
	\end{equation*}
	which is a homological epimorphism by Corollary \ref{cor:internal_Hom_to_functors_commuting_with_countable_limits}. This also gives an equivalence
	\begin{equation*}
		\prolim[n]^{\dual}\Ind(\cA_n)\simeq\un{\Hom}_{\cE}^{\dual}(\cC,\cT)\simeq\Ind(\Fun_{\cE}^{LL}(\cC,\cT)).
	\end{equation*}
	Applying \cite[Theorem 6.1]{E25} and by the commutation of $K$-theory with products in $\Cat^{\perf}$ \cite[Theorem 1.3]{KW19}, we conclude that the map \eqref{eq:unusual_continuity_of_K_theory} is an isomorphism. This proves the isomorphisms \eqref{eq:morphisms_in_Mot^loc_via_Calk_of_formally_omega_1_injective}.
	
	Now suppose that that the base category $\cE$ is symmetric monoidal, and let $\cC,\cD\in\Cat_{\cE}^{\dual},$ where $\cC$ is $\omega_1$-compact. Using again the formal $\omega_1$-injectivity of the category $\Calk_{\omega_1}(\cD)$ and arguing as above, we obtain an isomorphism 
	\begin{equation*}
	\Omega \cU_{\loc}(\Ind(\Fun_{\cE}^{LL}(\cC,\Calk_{\omega_1}^2(\cD)))\xto{\sim} \cU_{\loc}(\un{\Hom}_{\cE}^{\dual}(\cC,\Calk_{\omega_1}(\cD)))
	\end{equation*}
	in $\Mot^{\loc}_{\cE},$ which is the first isomorphism in \eqref{eq:internal_Hom_in_Mot^loc_via_Calk^2}. It remains to prove that the map
	\begin{equation}\label{eq:internal_Hom_in_Mot^loc_via_Calk}
	 \cU_{\loc}(\un{\Hom}_{\cE}^{\dual}(\cC,\Calk_{\omega_1}(\cD)))\to \un{\Hom}_{\Mot^{\loc}_{\cE}}(\cU_{\loc}(\cC),\cU_{\loc}(\Calk_{\omega_1}(\cD)))
	\end{equation}
	is an isomorphism in $\Mot^{\loc}_{\cE}.$
	
	By Proposition \ref{prop:Hom^dual_into_formally_omega_1_injective} (and by Theorem \ref{th:hat_Y_for_Calkin}) the category $\un{\Hom}_{\cE}^{\dual}(\cC,\Calk_{\omega_1}(\cD))$ is formally $\omega_1$-injective. Using the second isomorphism from \eqref{eq:morphisms_in_Mot^loc_via_Calk_of_formally_omega_1_injective}, for any $\cC'\in(\Cat_{\cE}^{\dual})^{\omega_1}$ we obtain
	\begin{multline*}
	\Hom_{\Mot^{\loc}_{\cE}}(\cU_{\loc}(\cC'),\cU_{\loc}(\un{\Hom}_{\cE}^{\dual}(\cC,\Calk_{\omega_1}(\cD))))\\
	\cong K^{\cont}(\un{\Hom}_{\cE}^{\dual}(\cC',\un{\Hom}_{\cE}^{\dual}(\cC,\Calk_{\omega_1}(\cD))))\cong K^{\cont}(\un{\Hom}_{\cE}^{\dual}(\cC'\tens{\cE}\cC,\Calk_{\omega_1}(\cD)))\\
	\cong \Hom_{\Mot^{\loc}_{\cE}}(\cU_{\loc}(\cC'\tens{\cE}\cC),\cU_{\loc}(\Calk_{\omega_1}(\cD)))\\
	\cong \Hom_{\Mot^{\loc}_{\cE}}(\cU_{\loc}(\cC')\otimes\cU_{\loc}(\cC),\cU_{\loc}(\Calk_{\omega_1}(\cD)))\\
	\cong \Hom_{\Mot^{\loc}_{\cE}}(\cU_{\loc}(\cC'),\un{\Hom}_{\Mot^{\loc}_{\cE}}(\cU_{\loc}(\cC),\cU_{\loc}(\Calk_{\omega_1}(\cD)))).
	\end{multline*}
	It follows that the map \eqref{eq:internal_Hom_in_Mot^loc_via_Calk} is an isomorphism since the category $\Mot^{\loc}_{\cE}$ is generated by the objects $\cU_{\loc}(\cC'),$ $\cC'\in(\Cat_{\cE}^{\dual})^{\omega_1}.$ 
	
	Next, we deduce \ref{internal_Hom_from_proper} from \ref{internal_Hom_into_Calk}. Again, we first assume that $\cE$ is only $\bE_1$-monoidal. By assumption, $\cC\in\Cat_{\cE}^{\dual}$ is proper and $\omega_1$-compact. By \cite[Corollary 3.14 and its proof]{E25}, the category $\un{\Hom}_{\cE}^{\dual}(\cC,\Calk_{\omega_1}(\cD))$ is compactly generated and we have a short exact sequence
	\begin{equation*}
	0\to \un{\Hom}_{\cE}^{\dual}(\cC,\cD)\to \Ind((\cC^{\vee}\tens{\cE}\cD)^{\omega_1})\to \Ind(\Fun_{\cE}^{LL}(\cC,\Calk_{\omega_1}(\cD)))\to 0.
	\end{equation*}
	This gives an isomorphism
	\begin{equation*}
	\Omega K(\Fun_{\cE}^{LL}(\cC,\Calk_{\omega_1}(\cD)))\xto{\sim} K^{\cont}(\un{\Hom}_{\cE}^{\dual}(\cC,\cD)).
	\end{equation*}
	As a special case of \ref{internal_Hom_from_proper} we obtain
	\begin{multline*}
	\Omega K(\Fun_{\cE}^{LL}(\cC,\Calk_{\omega_1}(\cD)))\cong \Omega K^{\cont}(\un{\Hom}_{\cE}^{\dual}(\cC,\Calk_{\omega_1}(\cD)))\\
	\cong \Hom_{\Mot^{\loc}_{\cE}}(\cU_{\loc}(\cC),\cU_{\loc}(\cD)),
	\end{multline*}
	as required.
	
	In the case when $\cE$ is symmetric monoidal we similarly deduce from \ref{internal_Hom_into_Calk} the assertions about the internal $\Hom$ in $\Mot_{\cE}^{\loc}.$ This proves the theorem.
\end{proof}


\subsection{Equivalence between $\Mot_{\cE}^{\loc}$ and $\Mot_{\cE,\omega_1}^{\loc}$}
\label{ssec:equivalence_Mot^loc_omega_1_and_Mot^loc}



The following result is a generalization of a theorem of Ramzi-Sosnilo-Winges \cite[Theorem 6.1]{RSW25} to the case of an arbitrary rigid base category. 

\begin{theo}\label{th:equivalence_Mot^loc_omega_1_and_Mot^loc}
Let $\cE$ be a rigid $\bE_1$-monoidal category. The natural functor $\Mot_{\cE,\omega_1}^{\loc}\to\Mot_{\cE}^{\loc}$ is an equivalence.
\end{theo}

The argument is basically the same as in \cite[Proof of Theorem 6.1]{RSW25}, and it relies on Theorem \ref{th:morphisms_in_Mot^loc_via_internal_Hom}. We only indicate the necessary changes when dealing with an arbitrary rigid base, which is not necessarily compactly generated.

We need some notation and terminology. For an accessible additive $\infty$-category $\cT$ and an accessible functor $F:\Cat_{\cE}^{\dual}\to\cT,$ we say that $F$ is a {\it splitting} invariant if $F$ is additive in finite semi-orthogonal decompositions. Here we use the same terminology as in \cite{RSW25} for convenience; such functors are called additive invariants in \cite{BGT}. 

As usual, if $\kappa$ is a regular cardinal such that $\cT$ has $\kappa$-filtered colimits, then $F$ is called $\kappa$-finitary if $F$ commutes with $\kappa$-filtered colimits. We denote by 
\begin{equation*}
\cU_{\splt,\omega_1}^{\dual}:\Cat_{\cE}^{\dual}\to \Mot^{\splt,\dual}_{\cE,\omega_1}
\end{equation*} 
the universal splitting $\omega_1$-finitary invariant with values in a presentable stable category. We have
\begin{equation*}
\Mot^{\splt,\dual}_{\omega_1}\simeq \Fun^{\splt}((\Cat_{\cE}^{\dual})^{\omega_1,op},\Sp),
\end{equation*}
where the right hand-side is the category of partially defined contravariant splitting invariants. Here we of course use the $\omega_1$-presentability of $\Cat_{\cE}^{\dual}.$ We also put $\Mot^{\splt,\dual}_{\omega_1}=\Mot^{\splt,\dual}_{\Sp,\omega_1},$ and denote by
\begin{equation*}
\cU_{\splt,\omega_1}:\Cat^{\perf}\to \Mot^{\splt}_{\omega_1}
\end{equation*}
the universal splitting $\omega_1$-finitary invariant of small idempotent-complete categories with values in a presentable stable category.

We will need the following statement, which is standard but probably did not appear in the literature in this generality.

\begin{prop}\label{prop:morphisms_in_Mot^split}
Let $\cE$ be a rigid $\bE_1$-monoidal category. Then for $\cC\in(\Cat_{\cE}^{\dual})^{\omega_1}$ and for $\cD\in\Cat_{\cE}^{\dual}$ we have a natural isomorphism
\begin{equation*}
K_{\geq 0}(\Fun_{\cE}^{LL}(\cC,\cD))\xto{\sim} \Hom_{\Mot_{\cE,\omega_1}^{\splt,\dual}}(\cU_{\splt,\omega_1}^{\dual}(\cC),\cU_{\splt,\omega_1}^{\dual}(\cD)).
\end{equation*}
\end{prop}

\begin{proof}
We first treat the case $\cE=\cC=\Sp.$ Recall that by \cite[Proposition 1.77]{E24} a semi-orthogonal decomposition $\cC=\la\cC_1,\cC_2\ra$ in $\Cat_{\st}^{\dual}$ induces a semi-orthogonal decomposition $\cC^{\omega}=\la \cC_1^{\omega},\cC_2^{\omega}\ra$ in $\Cat^{\perf}.$ Also, by \cite[Proposition 1.71]{E24} the functor $(-)^{\omega}:\Cat_{\st}^{\dual}\to\Cat^{\perf}$ commutes with filtered colimits, in particular, with $\omega_1$-filtered colimits. Hence, we have a well-defined functor
\begin{equation*}
\Mot^{\splt,\dual}_{\omega_1}\to\Mot^{\splt}_{\omega_1},\quad \cU_{\splt,\omega_1}^{\dual}(\cD)\mapsto \cU_{\splt,\omega_1}(\cD^{\omega}).
\end{equation*}
Its left adjoint is given by
\begin{equation*}
\Mot^{\splt}_{\omega_1}\to\Mot^{\splt,\dual}_{\omega_1},\quad \cU_{\splt,\omega_1}(\cA)\mapsto\cU_{\splt,\omega_1}^{\dual}(\Ind(\cA)).
\end{equation*}
Hence, by \cite[Theorem 4.2]{BGMN21} for $\cD\in\Cat_{\st}^{\dual}$ we obtain
\begin{multline*}
\Hom_{\Mot_{\omega_1}^{\splt,\dual}}(\cU_{\splt,\omega_1}^{\dual}(\Sp),\cU_{\splt,\omega_1}^{\dual}(\cD))\cong \Hom_{\Mot^{\splt}_{\omega_1}}(\cU_{\splt,\omega_1}(\Sp^{\omega}),\cU_{\splt,\omega_1}(\cD^{\omega}))\\
\cong K_{\geq 0}(\cD^{\omega}).
\end{multline*}

Now, if $\cE$ is a rigid $\bE_1$-monoidal category and $\cC\in(\Cat_{\cE}^{\dual})^{\omega_1},$ then the functor $\un{\Hom}_{\cE}^{\dual}(\cC,-):\Cat_{\cE}^{\dual}\to\Cat_{\st}^{\dual}$ preserves finite demi-orthogonal decompositions and commutes with $\omega_1$-filtered colimits. Hence, we have a well-defined functor
\begin{equation*}
\Mot^{\splt,\dual}_{\cE,\omega_1}\to\Mot^{\splt,\dual}_{\omega_1},\quad \cU_{\splt,\omega_1}^{\dual}(\cD)\mapsto \cU_{\splt,\omega_1}^{\dual}(\un{\Hom}_{\cE}^{\dual}(\cC,\cD)).
\end{equation*}
Its left adjoint is given by
\begin{equation*}
\Mot^{\splt,\dual}_{\omega_1}\to\Mot^{\splt,\dual}_{\cE,\omega_1},\quad \cU_{\splt,\omega_1}^{\dual}(\cB)\mapsto \cU_{\splt,\omega_1}^{\dual}(\cC\otimes\cB)
\end{equation*}
Hence, for $\cD\in\Cat_{\cE}^{\dual}$ we obtain
\begin{multline*}
\Hom_{\Mot_{\cE,\omega_1}^{\splt,\dual}}(\cU_{\splt,\omega_1}^{\dual}(\cC),\cU_{\splt,\omega_1}^{\dual}(\cD))\\
\cong \Hom_{\Mot_{\omega_1}^{\splt,\dual}}(\cU_{\splt,\omega_1}^{\dual}(\Sp),\cU_{\splt,\omega_1}^{\dual}(\un{\Hom}_{\cE}^{\dual}(\cC,\cD)))
\cong K_{\geq 0}(\un{\Hom}_{\cE}^{\dual}(\cC,\cD)^{\omega})\\
\cong K_{\geq 0}(\Fun_{\cE}^{LL}(\cC,\cD)).
\end{multline*}	
This proves the proposition.
\end{proof}

\begin{proof}[Proof of Theorem \ref{th:equivalence_Mot^loc_omega_1_and_Mot^loc}]
Let $\cC$ and $\cD$ be dualizable left $\cE$-modules, and suppose that $\cC$ is $\omega_1$-compact in $\Cat_{\cE}^{\dual}.$ By Theorem \ref{th:morphisms_in_Mot^loc_via_internal_Hom} \ref{internal_Hom_into_Calk} and Proposition \ref{prop:morphisms_in_Mot^split} we have equivalences of spaces
\begin{multline*}
\Omega^{\infty}K(\Fun_{\cE}^{LL}(\cC,\Calk_{\omega_1}^2(\cD)))\xto{\sim} \Map_{\Mot^{\splt,\dual}_{\cE,\omega_1}}(\cU_{\splt,\omega_1}(\cC),\cU_{\splt,\omega_1}(\Calk_{\omega_1}^2\cD))\\
\xto{\sim}\Map_{\Mot^{\loc}_{\cE}}(\cU_{\loc}(\cC),\cU_{\loc}(\Calk_{\omega_1}^2(\cD))).
\end{multline*} 
The latter equivalence is the only ingredient needed to generalize the argument of Ramzi-Sosnilo-Winges. We refer to \cite[Section 6 and Appendix B]{RSW25} for details.   
\end{proof}

\section{Properties of the functor $\cE\mapsto\Mot_{\cE}^{\loc}$}
\label{sec:Mot^loc_E_as_functor_of_E}

In this section we study the assignment $\cE\mapsto\Mot_{\cE}^{\loc}$ as a functor from rigid $\bE_1$-monoidal (resp. symmetric monoidal) categories to dualizable (resp. rigid symmetric monoidal) categories. We also study the relation between localizing motives over a locally rigid symmetric monoidal category $\cE$ and localizing motives over its rigidification $\cE^{\rig}.$

We start with the following trivial consequence of Theorem \ref{th:dualizability_and_rigidity}. Here for a continuous monoidal functor $\cE\to\cE'$ and a dualizable left $\cE'$-module $\cD$ we denote by $\cD_{\mid \cE}\in\Cat_{\cE}^{\dual}$ the same category $\cD$ with the left $\cE$-module structure induced from the $\cE'$-module structure.

\begin{prop}
Let $\cE\to\cE'$ be a continuous monoidal functor between rigid $\bE_1$-monoidal categories. Then the induced functor $\Mot_{\cE}^{\loc}\to \Mot_{\cE'}^{\loc},$ $\cU_{\loc}(\cC)\mapsto\cU_{\loc}(\cE'\tens{\cE}\cC),$ is strongly continuous. Its right adjoint is given by $\cU_{\loc}(\cD)\mapsto\cU_{\loc}(\cD_{\mid \cE})$ for $\cD\in\Cat_{\cE'}^{\dual}.$

In particular, we obtain the functors
\begin{equation}\label{eq:Mot_as_functor_E_1}
\Alg_{\bE_1}^{\rig}(\Prr^L_{\st})\to \Cat_{\st}^{\dual},\quad \cE\mapsto\Mot_{\cE}^{\loc},
\end{equation}
and
\begin{equation}\label{eq:Mot_as_functor_E_infty}
\CAlg^{\rig}(\Prr^L_{\\st})\to \CAlg^{\rig}(\Prr^L_{\st}),\quad \cE\mapsto\Mot_{\cE}^{\loc}.
\end{equation}
\end{prop}

\begin{proof}
The dualizability resp. rigidity of $\Mot_{\cE}^{\loc}$ is proved in Theorem \ref{th:dualizability_and_rigidity}. The rest is straightforward.
\end{proof}

The results in the following subsections are quite elementary except for Theorem \ref{th:localizing_motives_over_rigidification} which uses Theorem \ref{th:morphisms_in_Mot^loc_via_internal_Hom}.

\subsection{Localizing motives over a quotient category}

We first show that the functor \eqref{eq:Mot_as_functor_E_1} preserves the class of quotient functors, hence so does \eqref{eq:Mot_as_functor_E_infty}.

\begin{prop}\label{prop:localizing_motives_over_quotients}
Let $F:\cE\to\cE'$ be a continuous monoidal functor between rigid $\bE_1$-monoical categories. If $F$ is a quotient functor, then so is the induced functor $\Mot^{\loc}_{\cE}\to\Mot^{\loc}_{\cE'}.$
\end{prop}

\begin{proof}
Indeed, we have an equivalence of $\cE'\hy\cE'$-bimodules $\cE'\tens{\cE}\cE'\xto{\sim}\cE',$ hence for $\cD\in\Cat_{\cE'}^{\dual}$ we have $\cE'\tens{\cE}(\cD_{\mid \cE})\xto{\sim}\cD.$ It follows that the composition $\Mot^{\loc}_{\cE'}\to\Mot^{\loc}_{\cE}\to\Mot^{\loc}_{\cE'}$ is isomorphic to identity. This proves the proposition. 
\end{proof}

\subsection{Localizing motives over filtered colimits}

The following result is slightly more subtle.

\begin{prop}
Both functors \eqref{eq:Mot_as_functor_E_1} and \eqref{eq:Mot_as_functor_E_infty} commute with filtered colimits. 
\end{prop}

\begin{proof}
It suffices to prove this for the functor \eqref{eq:Mot_as_functor_E_1}. Let $(\cE_i)_{i\in I}$ be a directed system of rigid $\bE_1$-monoidal categories, and put $\cE=\indlim[i]\cE_i.$ We need to prove that the functor
\begin{equation}
\Phi:\indlim[i]\Mot^{\loc}_{\cE_i}\to\Mot^{\loc}_{\cE}
\end{equation}
is an equivalence. We denote by $F_{ij}:\Mot^{\loc}_{\cE_i}\to\Mot^{\loc}_{\cE_j}$ the functor induced by $\cE_i\to\cE_j,$ $i\leq j.$ Also, denote by $F_i:\Mot^{\loc}_{\cE_i}\to \indlim[i]\Mot^{\loc}_{\cE_i}$ the functor to the colimit. We denote by $F_{ij}^R,$ $F_i^R$ and $\Phi^R$ the corresponding right adjoint functors. To prove the fully faithfulness of $\Phi,$ it suffices to prove that we have isomorphisms $F_i^R F_j\xto{sim}F_i^R\Phi^R\Phi F_j$ for $i,j\in I.$ Using \cite[roposition 1.66]{E24}, for $\cC\in\Cat_{\cE_j}^{\dual}$ we obtain
\begin{multline*}
F_i^R F_j(\cU_{\loc}(\cC))\cong \indlim[k\geq i,j] F_{ik}^R F_{jk}(\cU_{\loc}(\cC))\cong \indlim[k\geq i,j]\cU_{\loc}((\cE_k\tens{\cE_j}\cC)_{\mid \cE_i})\cong \cU_{\loc}((\cE\tens{\cE_j}\cC)_{\mid\cE_i})\\
\cong (\Phi F_i)^R\Phi F_j(\cU_{\loc}(\cC))\cong F_i^R\Phi^R\Phi F_j(\cU_{\loc}(\cC)).
\end{multline*}
Hence, $\Phi$ is fully faithful. To prove the essential surjectivity, it suffices to observe that for any $\cC\in\Cat_{\cE}^{\dual}$ we have
\begin{equation*}
\cU_{\loc}(\cC)\cong \indlim[i]\cU_{\loc}(\cE\tens{\cE_i}\cC)\cong\indlim[i]\Phi F_i F_i^R\Phi^R(\cU_{\loc}(\cC)).\qedhere
\end{equation*}
\end{proof}

\subsection{Localizing motives over tensor products}

In this subsection we consider localizing motives over rigid symmetric monoidal categories. The natural question is whether the K\"unneth formula holds. We give a partial affirmative answer.

\begin{prop}\label{prop:partial_Kunneth}
Let $\cE_0,\cE_1,\cE_2$ be rigid symmetric monoidal presentable stable categories, and let $F_1:\cE_0\to\cE_1,$ $F_2:\cE_0\to\cE_2$ be continuous symmetric monoidal functors. Then the natural functor
\begin{equation}
\Mot^{\loc}_{\cE_1}\tens{\Mot^{\loc}_{\cE_0}}\Mot^{\loc}_{\cE_1}\to \Mot^{\loc}_{\cE_1\tens{\cE_0}\cE_2}
\end{equation}
is fully faithful.
\end{prop}

We need the following standard statement on the tensor products of rigid categories.

\begin{lemma}\label{lem:Beck_Chevalley}
Consider a commutative square in the category $\CAlg^{\rig}(\Prr^L_{\st})$ of rigid symmetric monoidal presentable stable categories:
\begin{equation}\label{eq:square_of_rigid_cats}
\begin{tikzcd}
\cE_0\ar[r, "F_1"]\ar[d, "F_2"] & \cE_1\ar[d, "G_1"]\\
\cE_2\ar[r, "G_2"] & \cE_3.
\end{tikzcd}
\end{equation}
Then the following are equivalent.
\begin{enumerate}[label=(\roman*),ref=(\roman*)]
\item The natural functor $\Phi:\cE_1\tens{\cE_0}\cE_2\to\cE_3$ is fully faithful. \label{fully_faithful}
\item The square \eqref{eq:square_of_rigid_cats} satisfies the (dual) Beck-Chevalley condition, namely the following (lax commutative) square commutes:
\begin{equation}\label{eq:dual_Beck_Chevalley}
\begin{tikzcd}
	\cE_0\ar[d, "F_2"] & \cE_1\ar[d, "G_1"]\arrow{l}[swap]{F_1^R}\\
	\cE_2 & \cE_3\arrow{l}[swap]{G_2^R}.
\end{tikzcd}
\end{equation}\label{Beck_Chevalley}
\end{enumerate}
\end{lemma}

\begin{proof}
The fully faithfulness of the functor $\Phi$ from \ref{fully_faithful} is equivalent to the map $1_{\cE_1\tens{\cE_0}\cE_2}\to\Phi^R(1_{\cE_3})$ being an isomorphism. The latter condition translates into commutativity of \eqref{eq:dual_Beck_Chevalley} via the equivalence $\cE_1\tens{\cE_0}\cE_2\simeq\Fun^L_{\cE_0}(\cE_1,\cE_2).$
\end{proof}

\begin{proof}[Proof of Proposition \ref{prop:partial_Kunneth}] By Lemma \ref{lem:Beck_Chevalley} we only need to check the commutativity of the square
\begin{equation*}
\begin{tikzcd}
\Mot^{\loc}_{\cE_0}\ar[d] & \Mot^{\loc}_{\cE_1}\ar[l]\ar[d]\\
\Mot^{\loc}_{\cE_2} & \Mot^{\loc}_{\cE_1\tens{\cE_0}\cE_2}\ar[l].
\end{tikzcd}
\end{equation*}
This is clear: for $\cC\in\Cat_{\cE_1}^{\dual}$ we have
\begin{equation*}
\cE_2\tens{\cE_0}\cC\simeq (\cE_1\tens{\cE_0}\cE_2)\tens{\cE_1}\cC\quad\text{in }\Cat_{\cE_2}^{\dual}.\qedhere
\end{equation*}
\end{proof}

\begin{ques}
Is it true that the functor in Proposition \ref{prop:partial_Kunneth} is an equivalence?
\end{ques}

\subsection{Localizing motives over a locally rigid base and over its rigidification}

In this subsection we again deal only with localizing motives over symmetric monoidal categories. The main result (Theorem \ref{th:localizing_motives_over_rigidification}) allows to reduce the study of localizing motives over a complicated ``analytic'' base category like $\Nuc(\Z_p)$ to the study of localizing motives over a simpler base, which in this case would be the usual $p$-complete derived category $D_{p\hy\compl}(\Z).$

Note that one can define the category $\Mot_{\cE}^{\loc}$ for an arbitrary $\cE\in\CAlg(\Pr^L_{\st}),$ not necessarily rigid. Namely, we have the universal finitary localizing invariant $\Cat_{\cE}^{\dual}\to \Mot_{\cE}^{\loc}.$ Here $\Cat_{\cE}^{\dual}$ is the category of dualizable $\cE$-modules, where the $1$-morphisms are $\cE$-linear strongly continuous functors whose right adjoint is $\cE$-linear.

We first observe that in the case of a locally rigid base we can reduce the study of localizing motives to the case of a rigid base.

\begin{prop}\label{prop:Cat^dual_over_locally_rigid}
Let $\cE$ be a locally rigid symmetric monoidal category. We choose a fully faithful inclusion $j_!:\cE\to\cE',$ such that $\cE'$ is rigid and $\cE$ is a smashing ideal in $\cE',$ so the right adjoint $j^*:\cE'\to\cE$ is continuous and naturally symmetric monoidal. For example, we can take $\cE'$ to be the one-point rigidification $\cE_+$ \cite[Definition 1.39]{E25}.

Then the induced functor $\cE\tens{\cE'}-:\Cat_{\cE'}^{\dual}\to \Cat_{\cE}^{\dual}$ has a fully faithful left adjoint. The latter functor sends $\cC\in\Cat_{\cE}^{\dual}$ to $\cC$ considered as an $\cE'$-module via $j^*.$
\end{prop}

\begin{proof}
Note that we have fully faithful inclusion $\Mod_{\cE}(\Pr^L_{\st})\hto \Mod_{\cE'}(\Pr^L_{\st}),$ given by ``restriction of scalars'' via $j^*.$ If an $\cE$-module $\cC$ is dualizable over $\cE,$ then $\cC$ is also dualizable over $\cE'$ (hence over $\Sp$). Namely, if $\cC^{\vee}$ is the dual of $\cC$ over $\cE,$ then $\cC^{\vee}$ is also dual to $\cC$ over $\cE'.$ The corresponding evaluation and coevaluation functors are given by
\begin{equation*}
\ev_{\cC/\cE'}:\cC\tens{\cE'}\cC^{\vee}\xto{\sim}\cC\tens{\cE}\cC^{\vee}\xto{\ev_{\cC/\cE}}\cE\xto{j_!}\cE',
\end{equation*}
\begin{equation*}
\coev_{\cC/\cE'}:\cE'\xto{j^*}\cE\xto{\coev_{\cC/\cE}}\cC^{\vee}\tens{\cE}\cC\simeq \cC^{\vee}\tens{\cE'}\cC.
\end{equation*}
Furthermore, a lax $\cE$-linear continuous functor between $\cE$-modules is automatically $\cE$-linear. The stated adjunction is now straightforward: the unit and counit are given by
\begin{equation*}
\cC\xto{\sim}\cE\tens{\cE'}\cC,\quad \cC\in\Cat_{\cE}^{\dual},\quad \cE\tens{\cE'}\cD\xto{j_!\boxtimes\id}\cE'\tens{\cE'}\cD\simeq\cD,\quad \cD\in\Cat_{\cE'}^{\dual}.\qedhere
\end{equation*}
\end{proof}

\begin{prop}\label{prop:Mot^loc_over_locally_rigid}
In the situation of Proposition \ref{prop:Mot^loc_over_locally_rigid}, the induced functor $\Mot_{\cE'}^{\loc}\to \Mot_{\cE}^{\loc},$ $\cU_{\loc}(\cC)\mapsto\cU_{\loc}(\cE\tens{\cE'}\cC),$ has a fully faithful left adjoint. In particular, the essential image of $\Mot^{\loc}_{\cE}$ is a smashing ideal in $\Mot^{\loc}_{\cE'}.$ Hence, the symmetric monoidal category $\Mot^{\loc}_{\cE}$ is locally rigid.
\end{prop}

\begin{proof}
This follows directly from Proposition \ref{prop:Cat^dual_over_locally_rigid}.
\end{proof}

Now we prove the following result on the category of localizing motives over a rigidification of a locally rigid category.

\begin{theo}\label{th:localizing_motives_over_rigidification}
Let $\cE$ be a locally rigid symmetric monoidal category, such that the unit object $1_{\cE}$ is $\omega_1$-compact. Then the natural functor $\Mot_{\cE^{\rig}}\to(\Mot_{\cE}^{\loc})^{\rig}$ is fully faithful.
\end{theo}

We need the following almost tautological observation.

\begin{lemma}\label{lem:when_functor_to_rigidification_is_fully_faithful}
In the situation of Proposition \ref{prop:Cat^dual_over_locally_rigid}, denote by $j_*:\cE\to\cE'$ the right adjoint to $j^*.$ The following are equivalent.
\begin{enumerate}[label=(\roman*),ref=(\roman*)]
	\item The natural functor $\Phi:\cE'\to\cE^{\rig}$ is fully faithful.
	\item We have an isomorphism $1_{\cE'}\xto{\sim}j_* 1_{\cE}\cong \un{\Hom}_{\cE'}(j_! 1_{\cE},1_{\cE'}).$
\end{enumerate}
\end{lemma}

\begin{proof}
Denote by $\Phi^R$ the right adjoint to $\Phi.$ Consider the factorization $\cE'\xto{\Phi}\cE^{\rig}\to\cE$ of the functor $j^*.$ It suffices to prove that the natural map $\Phi^R(1_{\cE^{\rig}})\to j_* 1_{\cE}$ is an isomorphism. This in turn reduces to the special case $\cE'=\cE^{\rig},$ so we will assume this from now on. 

Consider the composition $j^*:\cE^{\rig}\hto\Ind(\cE^{\kappa})\xto{\colim}\cE,$ where $\kappa$ is a regular cardinal such that $1_{\cE}$ is $\kappa$-compact. The right adjoint to the functor $\colim$ sends $1_{\cE}$ to the unit object $\cY(1_{\cE}),$ which is contained in $\cE^{\rig}.$ Hence, we have $1_{\cE^{\rig}}\cong j_*1_{\cE},$ as stated. 
\end{proof}

\begin{proof}[Proof of Theorem \ref{th:localizing_motives_over_rigidification}]
We consider the category $\cE$ as a dualizable module over $\cE^{\rig}$ and also as a smashing ideal in $\cE^{\rig}$ (recall that the functor $\cE^{\rig}\to\cE$ has a fully faithful left adjoint). Then $\cE$ is proper over $\cE^{\rig}$ and $\omega_1$-compact in $\Cat_{\cE^{\rig}}^{\dual}$ by assumption. By Theorem \ref{th:morphisms_in_Mot^loc_via_internal_Hom} \ref{internal_Hom_from_proper}, we have
\begin{equation*}
\un{\Hom}_{\Mot^{\loc}_{\cE^{\rig}}}(\cU_{\loc}(\cE),\cU_{\loc}(\cE^{\rig}))\cong \cU_{\loc}(\un{\Hom}_{\cE^{\rig}}^{\dual}(\cE,\cE^{\rig})).
\end{equation*}
By \cite[Proposition 4.1]{E25} we have
\begin{equation*}
\un{\Hom}_{\cE^{\rig}}^{\dual}(\cE,\cE^{\rig})\simeq \cE^{\rig}.
\end{equation*}
The theorem now follows from Lemma \ref{lem:when_functor_to_rigidification_is_fully_faithful}.
\end{proof}

\begin{ques}
Is it true that the functor in Theorem \ref{th:localizing_motives_over_rigidification} is an equivalence?
\end{ques}

\section{$G$-equivariant localizing motives}
\label{sec:G_equivariant_motives}

Let $G$ be a (discrete) group. Denote by ($\Cat_{\st}^{\dual})^{BG}=\Fun(BG,\Cat_{\st}^{\dual})$ the category of dualizable categories with a $G$-action. In this section we consider the universal finitary localizing invariant $\cU_{\loc,BG}:(\Cat_{\st}^{\dual})^{BG}\to\Mot^{\loc}_{BG}.$ All the results in this section are essentially corollaries of our general results on localizing motives and on dualizable modules over a rigid $\bE_1$-monoidal category.

We can also replace $BG$ with a general space. However, to illustrate the ideas we restrict to this special case. . 

\subsection{Dualizability and self-duality}

We first deduce from Theorem \ref{th:dualizability_and_rigidity} the basic properties of the category $\Mot^{\loc}_{BG}.$

\begin{theo}\label{th:G_equivariant_motives}
	Let $G$ be a group. Then the category $\Mot^{\loc}_{BG}$ is dualizable and self-dual. More precisely, we have a perfect pairing
	\begin{equation*}
		\Mot^{\loc}_{BG}\otimes\Mot^{\loc}_{BG}\to \Sp,\quad \cU_{\loc,G}(\cC)\boxtimes\cU_{\loc,G}(\cD)\mapsto K^{\cont}((\cC\otimes\cD)^{\h G}).
	\end{equation*}
\end{theo}

\begin{proof}
	Consider the category $\cE=\prodd[G]\Sp$ -- the product of copies of $\Sp$ indexed by elements of $G.$ Then $\cE$ is naturally an $\bE_1$-monoidal category via the convolution product
	\begin{equation*}
		-\star-:\cE\otimes\cE\to \cE,\quad (x\star y)_g = \biggplus[h\in G] x_h\otimes y_{h^{-1}g}.
	\end{equation*}
	Formally, if for a space $X$ we denote by $\Loc(X)=\Sp^X$ the category of $\Sp$-valued local systems on $X,$ then the (covariant) functor $\cS\to\Cat_{\st}^{\dual},$ $X\mapsto \Loc(X),$ is naturally symmetric monoidal. Hence it sends $\bE_1$-algebras in $\cS$ to $\bE_1$-algebras in $\Cat_{\st}^{\dual}.$ This gives the above convolution product on $\Loc(G)=\prodd[G]\Sp=\cE.$
	
	We have the ``delta-function'' objects $\bS_g\in E,$ given by
	\begin{equation*}
		(\bS_g)_h=\begin{cases}\bS & \text{for }h=g;\\
			0 & \text{for }h\ne g.\end{cases},\quad g,h\in G.
	\end{equation*}
	Then $\bS_1$ is the unit object of $\cE,$ and $\bS_g\star \bS_h\cong \bS_{gh}.$ In particular, the objects $\bS_g$ are invertible. They are also compact in $\cE,$ which proves that $\cE$ is a rigid $\bE_1$-monoidal category.
	
	We have natural equivalences
	\begin{equation*}
		(\Cat_{\st}^{\dual})^{BG}\simeq \Cat_{\cE}^{\dual},\quad \Mot^{\loc}_{BG}\simeq \Mot_{\cE}^{\loc}.
	\end{equation*}
	Moreover, we have a natural monoidal equivalence $\cE\simeq\cE^{mop},$ induced by $G\xto{\sim}G^{op},$ $g\mapsto g^{-1}.$ The assertions of the theorem are now obtained as a special case of Theorem \ref{th:dualizability_and_rigidity} \ref{E_0_rigidity}.
\end{proof}

For a connected space $X$ one obtains a similar result for the universal finitary localizing invariant $\cU_{\loc,X}:(\Cat_{\st}^{\dual})^X\to \Mot^{\loc}_X,$ using the rigid $\bE_1$-monoidal category of local systems on the based loop space of $X.$ If $X$ is not connected, then $\Mot^{\loc}_X$ is simply the product of $\Mot^{\loc}_{X_i}$ over the connected components $X_i$ of $X,$ $i\in\pi_0(X).$  

We consider the category $\Mot^{\loc}_{BG}$ as a symmetric monoidal category, where the monoidal product is induced by the Lurie tensor product (over $\Sp$) in $(\Cat_{\st}^{\dual})^{BG}.$ This corresponds to the (cocommutative) Hopf algebra structure on $\prodd[G]\Sp$ in $\Cat_{\st}^{\dual}.$ One can show that the unit object $\cU_{\loc}(\Sp)\in\Mot^{\loc}_{BG}$ is not compact unless $G$ trivial.

\subsection{The case of finite groups}

In the case of a finite group $G$ we have the following result on morphisms from the unit object in $\Mot^{\loc}_{BG}$ (part \ref{invariants_exact} was announced in \cite[Section 3.6]{E25}).

\begin{theo}
	Let $G$ be a finite group. For $\cC\in(\Cat_{\st}^{\dual})^{BG},$ we denote by $\cC^{\h G,\dual}\in\Cat_{\st}^{\dual}$ the dualizable limit over $BG.$
	
	\begin{enumerate}[label=(\roman*),ref=(\roman*)]
		\item The functor $(-)^{\h G,\dual}:(\Cat_{\st}^{\dual})^{BG}\to \Cat_{\st}^{\dual}$ takes short exact sequences to short exact sequences. \label{invariants_exact}
		\item For $\cC\in(\Cat_{\st}^{\dual})^{BG}$ we have an isomorphism
		\begin{equation*}
			\Hom_{\Mot^{\loc}_{BG}}(\cU_{\loc}(\Sp),\cU_{\loc}(\cC))\cong K^{\cont}(\cC^{\h G,\dual}).
		\end{equation*}
		Here we consider $\Sp$ as a category with trivial $G$-action. \label{self_duality_Mot_loc_G}
	\end{enumerate}
\end{theo} 

\begin{proof}
	Consider the $\bE_1$-monoidal category $\cE=\prodd[G]\Sp$ from the proof of Theorem \ref{th:G_equivariant_motives}, and use the notation from therein. We consider the category $\Sp$ with trivial $G$-action as a left $\cE$-module. Then $\Sp$ is proper over $\cE.$ Namely, the evaluation functor functor $\Sp\simeq\Sp\otimes\Sp\to\cE$ sends $\bS$ to $\biggplus[g\in G]\bS_g,$ and the latter object is compact since $G$ is finite.
	
	Now \ref{invariants_exact} is a special case of \cite[Theorem 3.6]{E25}, and \ref{self_duality_Mot_loc_G} is a special case of Theorem \ref{th:morphisms_in_Mot^loc_via_internal_Hom} \ref{internal_Hom_from_proper}.
\end{proof}

The same argument applies when we replace $BG$ with a general loop-finite space $X,$ which means that each based loop space of $X$ is finitely dominated (compact in $\cS$). 

\subsection{Observations on the $K$-theoretic Farrell-Jones conjecture}

Self-duality of the category $\Mot^{\loc}_{BG}$ described in Theorem \ref{th:G_equivariant_motives} allows to obtain several equivalent formulations of the $K$-theoretic Farrell-Jones conjecture.  We first recall its statement. 

Let $G$ be a (discrete) group. We denote by $\Vcyc$ the family of virtually cyclic subgroups of $G.$ Recall that a group is called virtually cyclic if it contains a cyclic subgroup of finite index. We denote by $G\hy\Orb$ the category of $G$-orbits (transitive $G$-sets), and by $G_{\Vcyc}\hy\Orb\subset G\hy\Orb$ the full subcategory of $G$-orbits with stabilizers in $\Vcyc.$ The Farrell-Jones conjecture for the group $G$ and an idempotent-complete stable category $\cC$ with a $G$-action states that we have an isomorphism:
\begin{equation*}
	\colim\limits_{X\in G_{\Vcyc}\hy\Orb} K^{\cont}((\Ind(\cC)\otimes\prodd[X]\Sp)^{\h G})\xto{\sim} K^{\cont}(\Ind(\cC)^{\h G}).
\end{equation*}
Here we of course take continuous $K$-theory of compactly generated categories: for $X=G/H\in G\hy\Orb$ we have
\begin{equation*}
	K^{\cont}((\Ind(\cC)\otimes\prodd[X]\Sp)^{\h G})\simeq K(\cC_{\h H}).
\end{equation*}
We refer to \cite{Luc25} for a detailed account, the original version of the conjecture was formulated in \cite{FJ93}. The conjecture is known for many classes of groups \cite{AFR00, BLR08, BL12, Weg12, BFL14, GMR15, Weg15, KLR16, Rup16, Wu16, BB19, BKW21}

We will need some notation. For a subgroup $H\subset G$ the restriction functor $\Res^G_H:(\Cat_{\st}^{\dual})^{BG}\to (\Cat_{\st}^{\dual})^{BH}$ has a left adjoint which we call induction and denote by $\Ind_H^G:(\Cat_{\st}^{\dual})^{BH}\to (\Cat_{\st}^{\dual})^{BG}.$ It is given by
\begin{equation*}
	\Ind_H^G(\cC)=(\prodd[G]\Sp)\tens{\prodd[H]\Sp}\cC,\quad \cC\in (\Cat_{\st}^{\dual})^{BH}.
\end{equation*} 
These functors induce the adjoint pair on the corresponding categories of localizing motives:
\begin{equation*}
	\Ind_H^G:\Mot^{\loc}_{BH} \rightleftharpoons \Mot^{\loc}_{BG}:\Res^G_H.
\end{equation*}

\begin{prop}\label{prop:Farrell-Jones_reformulations}
	Let $G$ be a (discrete) group. The following are equivalent.
	\begin{enumerate}[label=(\roman*),ref=(\roman*)]
		\item The $K$-theoretic Farrell-Jones conjecture holds for all $\cC\in(\Cat^{\perf})^{BG}.$ \label{Farrell_Jones}
		\item The category $\Mot^{\loc}_{BG}$ is generated as a localizing subcategory by the images of induction functors $\Ind_H^G:\Mot^{\loc}_{BH}\to \Mot^{\loc}_{BG},$ where $H$ runs through virtually cyclic subgroups of $G.$ \label{inductions_from_virtually_cyclic}
		\item The unit object $\cU_{\loc,G}(\Sp)\in \Mot^{\loc}_{BG}$ is contained in the localizing subcategory generated by the images of induction functors as in \ref{inductions_from_virtually_cyclic}. \label{unit_object_generated}
		\item We have an isomorphism in $\Mot^{\loc}_{BG}:$
		\begin{equation}\label{eq:colimit_in_Mot_loc_G}
			\indlim[X\in G_{\Vcyc}\hy \Orb]\cU_{\loc,G}(\prodd[X]\Sp)\xto{\sim}\cU_{\loc,G}(\Sp).
		\end{equation} \label{colimit_in_Mot_loc_G}
	\end{enumerate}
\end{prop}

\begin{proof}
	The implications \Implies{colimit_in_Mot_loc_G}{unit_object_generated} is clear. The equivalence \Iff{Farrell_Jones}{colimit_in_Mot_loc_G} follows from the self-duality of $\Mot^{\loc}_{BG}$ from Theorem \ref{th:G_equivariant_motives}.
	
	The implication \Implies{unit_object_generated}{inductions_from_virtually_cyclic} follows from the standard projection formula: for any $H\in\Vcyc$ the essential image of induction $\Ind_H^G:\Mot^{\loc}_{BH}\to\Mot^{\loc}_{BG}$ is an ideal.
	
	Finally, to prove the implication \Implies{inductions_from_virtually_cyclic}{colimit_in_Mot_loc_G} it suffices to observe that \eqref{eq:colimit_in_Mot_loc_G} becomes an isomorphism after applying the restriction functor $\Res^G_H:\Mot^{\loc}_{BG}\to\Mot^{\loc}_{BH}$ for any $H\in\Vcyc.$
\end{proof}

\section{$\Mot^{\loc}$ is not generated by motives of connective $\bE_1$-rings}
\label{sec:motives_not_generated}

In this section we consider the case of the absolute base. We know by \cite{E24} that the target of the universal finitary localizing invariant $\cU_{\loc}:\Cat^{\perf}\to\Mot^{\loc}$ is the same as for $\Cat_{\st}^{\dual}.$ For convenience we put $\cU_{\loc}(R)=\cU_{\loc}(\Perf(R))$ for an $\bE_1$-ring $A.$

One can hope to have a generating collection of objects in $\Mot^{\loc}$ of the form $\cU_{\loc}(\cC),$ where $\cC$ has a certain nice property/structure. To make this more precise, we first observe that for any $\cC\in\Cat^{\perf}$ we can consider $\cC$ as an additive $\infty$-category, and take its stabilization. This gives a short exact sequence 
\begin{equation}\label{eq:ses_via_stabilization}
0\to\Stab(\cC)^{\acycl}\to \Stab(\cC)\xto{F} \cC\to 0.
\end{equation} 
Here $F$ is the adjunction counit, and $\Stab(\cC)^{\acycl}$ is the kernel of $F.$ Then $\Stab(\cC)$ has a bounded weight structure \cite{Bon}, and $\Stab(\cC)^{\acycl}$ has a bounded $t$-structure, see \cite[Proposition G.7]{E24} and \cite[Proposition 1.3]{SW25}. Furthermore, by \cite[Theorem 1.5]{RSW25} every object of $\Mot^{\loc}$ is isomorphic to $\cU_{\loc}(\cC)$ for some $\cC\in\Cat^{\perf}.$ Applying $\cU_{\loc}$ to \eqref{eq:ses_via_stabilization}, we see that the category $\Mot^{\loc}$ is generated as a stable subcategory by motives of categories which have either a bounded weight structure or a bounded $t$-structure.

Now, if $\cC\in\Cat^{\perf}$ has a bounded weight structure, then we have $\cC\simeq\indlim[i\in I]\Perf(R_i),$ where $I$ is directed and each $R_i$ is a connective $\bE_1$-ring. Namely, $\cC$ is a directed union of its full subcategories generated by finitely many objects of weight $0,$ and any such subcategory is equivalent to $\Perf(R)$ for a connective $\bE_1$-ring $R.$

This motivates a natural question: is it true that the category $\Mot^{\loc}$ is generated as a localizing subcategory by the motives of connective $\bE_1$-rings? By the above observation, it would be sufficient to prove that for any $\cC$ with a bounded $t$-structure the motive $\cU_{\loc}(\cC)$ can be generated by motives of connective $\bE_1$-rings. Unfortunately, this is not the case already for the simplest non-trivial examples.

\begin{theo}
Consider the dg $\Q$-algebra $\Q^{S^2}\cong \Q[y]/y^2,$ where $y$ has cohomological degree $2.$ Denote by $\Q[\veps]$ the usual algebra of dual numbers, i.e. $\veps^2=0$ and $\deg(\veps)=0.$ Then both objects $\cU_{\loc}(\Q^{S^2}),\cU_{\loc}(D^b_{\coh}(\Q[\veps]))\in \Mot^{\loc}$ are not contained in the localizing subcategory generated by motives of connective $\bE_1$-rings.
\end{theo}

\begin{proof}
We first observe that we can replace the category $\Mot^{\loc}$ with the category $\Mot^{\loc}_{\Q}$ by Proposition \ref{prop:localizing_motives_over_quotients}. Indeed, both the extension of scalars $-\otimes\Q$ and the restriction of scalars preserve connective $\bE_1$-rings.

Next, we show that the assertions about $\cU_{\loc}(\Q^{S^2})$ and $\cU_{\loc}(D^b_{\coh}(\Q[\veps]))$ are equivalent. Consider the category $\cA$ of triples $(V,W,\phi),$ where $V,W\in\Perf(\Q)$ and $\phi:V\oplus V[-1]\to W.$ Equivalently, $\cA$ is the category of $\Perf(\Q)$-valued representations of a graded quiver with two vertices and two parallel arrows of cohomological degrees $0$ and $1.$ Consider the object $X=(\Q,\Q[-1],(0,\id))\in\cA.$ Its endomorphism dg algebra is given by $\End(X)=\Q^{S^2}.$ The right orthogonal to $X$ in $\Ind(\cA)$ is identified with the category $\Mod\hy B,$ where $B=\Q[x]$ is the free associative $\Q$-algebra on one generator $x$ of cohomological degree $1.$ Hence, we have a short exact sequence
\begin{equation*}
0\to\Perf(\Q^{S^2})\to \cA\to\Perf(B)\to 0.
\end{equation*}
Applying $\cU_{\loc},$ we obtain a cofiber sequence in $\Mot^{\loc}_{\Q}:$ 
\begin{equation*}
\cU_{\loc}(\Q^{S^2})\to \cU_{\loc}(\Q)\oplus\cU_{\loc}(\Q)\to\cU_{\loc}(B).
\end{equation*}
Note that we have an equivalence $D^b_{\coh}(\Q[\veps])\simeq \Perf(B),$ since the (derived) endomorphism dg algebra $\bR\End_{\Q[\veps]}(\Q)$ is quasi-isomorphic to $B.$ Hence, it suffices to prove that $\cU_{\loc}(B)\cong\cU_{\loc}(D^b_{\coh}(\Q[\veps]))\in\Mot^{\loc}_{\Q}$ is not contained in the localizing subcategory generated by motives connective dg $\Q$-algebras. 

Next, we claim that the functor
\begin{equation*}
\Cat^{\perf}_{\Q}\to D(\Q),\quad \cC\mapsto \Fiber(\HC^-(\cC\otimes\Perf(\Q[\veps]))\to \HC^-(\cC)),
\end{equation*}
commutes with filtered colimits. To see this, note that the reduced Hochschild homology $\wt{HH}(\Q[\veps]),$ considered as a $\Q S^1$-module, is given by $\biggplus[n\geq 0]\Q S^1[2n].$ It follows that for any $\cC\in\Cat_{\Q}^{\perf}$ the $S^1$-representation $\HH(\cC)\otimes\wt{\HH}(\Q[\veps])$ is induced from the trivial subgroup of $S^1,$ hence its Tate cohomology vanishes. This gives an isomorphism
\begin{equation*}
\Fiber(\HC^-(\cC\otimes\Perf(\Q[\veps]))\to \HC^-(\cC))\cong \Cone(\HC(\cC\otimes\Perf(\Q[\veps]))\to \HC(\cC)).
\end{equation*}
The latter expression commutes with filtered colimits in $\cC,$ as stated.

Now, by Goodwillie's theorem \cite[Main Theorem]{Goo86} for any connective dg $\Q$-algebra $R$ we have
\begin{equation*}
\Fiber(K(R\otimes\Q[\veps])_{\Q}\to K(R)_{\Q}) \xto{\sim} \Fiber(\HC^-(R\otimes\Q[\veps])\to \HC^-(R)).
\end{equation*}
Therefore, it suffices to prove that the map
\begin{equation}\label{eq:not_an_iso_from_K_to_HC^-}
	\Fiber(K(B\otimes\Q[\veps])_{\Q}\to K(B)_{\Q}) \to \Fiber(\HC^-(B\otimes\Q[\veps])\to \HC^-(B))
\end{equation}
is not an isomorphism.

It is easy to compute the target of \eqref{eq:not_an_iso_from_K_to_HC^-}. Namely, its homology in each degree is a countable-dimensional $\Q$-vector space. Thus, it suffices to prove that
\begin{equation*}
K_{-1}(B\otimes\Q[\veps])=0.
\end{equation*}
Consider the fully faithful functor $\Perf(\Q[\veps])\to D^b_{\coh}(\Q[\veps])\simeq\Perf(B).$ It sends the object $\Q[\veps]$ to $\Cone(B[-1]\xto{x}B).$ Hence, we have an equivalence
\begin{equation*}
D^b_{\coh}(\Q[\veps])/\Perf(\Q[\veps])\simeq \Perf(\Q[x,x^{-1}]),
\end{equation*}
where as above $x$ has cohomological degree $1.$ This quotient category is well known under the name ``triangulated category of singularities'' \cite{Or04}. We obtain a short exact sequence
\begin{equation}\label{eq:short_exact_via_1_periodic_derived_category}
0\to \Perf(B\otimes\Q[\veps])\to\Perf(B\otimes B)\to \Perf(B\otimes\Q[x,x^{-1}])\to 0.
\end{equation}
Now, we have $\Perf(B\otimes B)\simeq D^b_{\coh}(\Q[\veps]\otimes\Q[\veps]),$ hence $K_{-1}(B\otimes B)=0$ by \cite[Theorem 6]{Sch06} and $K_0(B\otimes B)=\Z.$ Next, the category $\Perf(B\otimes\Q[x,x^{-1}])$ is identified with the $1$-periodic absolute derived category of finite-dimensional $\Q[\veps]$-modules in the sense of Positselski \cite{Pos11, EP15}. Hence, we have $K_0(B\otimes \Q[x,x^{-1}])=\Z/2.$ The long exact sequence of $K$-groups for \eqref{eq:short_exact_via_1_periodic_derived_category} gives
\begin{equation*}
K_{-1}(B\otimes \Q[\veps])\cong \coker(\Z\onto\Z/2)=0.
\end{equation*}
This proves the theorem.
\end{proof}

\section{$\A^1$-invariant localizing motives}
\label{sec:A^1_invariant_motives}

In this section we describe the category of $\A^1$-invariant localizing motives over a rigid $\bE_1$-monoidal category $\cE$ as a full subcategory of $\Mot^{\loc}_{\cE}.$ (Theorem \ref{th:A^1_invariant_localizing_motives}). Since we do not assume $\cE$ to be compactly generated, we consider localizing invariants of dualizable left $\cE$-modules. Similarly to the previous section, for an $\bE_1$-ring $A\in\Alg_{\bE_1}(\Sp)$ we put $\cU_{\loc}(A)=\cU_{\loc}(\Mod\hy A)\in\Mot^{\loc}.$ 

Consider the $\bE_{\infty}$-ring $\bS[x]=\Sigma^{\infty}\N_+$ -- the monoid ring (over $\bS$) of the monoid of natural numbers. Here $\N_+=\N\sqcup\{\ast\}$ is considered as a commutative algebra object in the symmetric monoidal $\infty$-category $\cS_{\ast}$ of pointed spaces, where the symmetric monoidal structure is given by the smash product. The affine scheme $\Spec\bS[x]$ over $\bS$ is known as the {\it flat} affine line $\A^{1,\flat}_{\bS}$ \cite{Lur18}. The latter is different from the smooth affine line $\Spec\bS\{x\},$ where $\bS\{x\}$ is the free $\bE_{\infty}$-ring on $x.$ Note however that as an $\bE_1$-ring $\bS[x]$ is freely generated by $x.$ 

A localizing invariant $F:\Cat_{\cE}^{\dual}\to \cT$ is called $\A^1$-invariant if for any $\cC\in\Cat_{\cE}^{\dual}$ we have an isomorphism $F(\cC)\to F(\Mod_{\bS[x]}(\cC)),$ induced by the functor
\begin{equation*}
\cC\to\Mod_{\bS[x]}(\cC),\quad M\mapsto\bS[x]\otimes M.
\end{equation*}
For notational convenience we use the term ``$\A^1$-invariant'' instead of ``$\A^{1,\flat}$''-invariant, although the latter would be more accurate. 

Consider the universal finitary $\A^1$-invariant localizing invariant $\cU_{\loc}^{\A^1}:\Cat_{\cE}^{\dual}\to \Mot^{\loc,\A^1},$ and put $\Mot^{\loc,\A^1}=\Mot^{\loc,\A^1}_{\Sp}.$ Clearly, we have an equivalence
\begin{equation*}
\Mot^{\loc,\A^1}_{\cE}\simeq \Mot^{\loc}_{\cE}\tens{\Mot^{\loc}}\Mot^{\loc,\A^1},
\end{equation*} and the category $\Mot^{\loc,\A^1}$ is the quotient of $\Mot^{\loc}$ by the localizing ideal generated by the reduced motive 
\begin{equation*}
\wt{\cU}_{\loc}(\bS[x])=\Cone(\cU_{\loc}(\bS)\to\cU_{\loc}(\bS[x])).
\end{equation*}

It turns out that the latter ideal is in fact smashing. To formulate a precise statement, for each $n\geq 0$ define the algebraic simplex $\Delta^n$ to be the flat affine $n$-dimensional space over $\bS,$ i.e.
\begin{equation*}
\Delta^n = \Spec\bS[\Delta^n], \quad  \bS[\Delta^n]=\bS[x_1,\dots,x_n]=\Sigma^{\infty} (\N^n)_+.
\end{equation*}
Then the simplices $\Delta^n$ form a cosimplicial object $\Delta^{\bullet}$ in the category of affine schemes over $\bS.$ More precisely, consider the symmetric monoidal category $\Set_{\ast}$ of pointed (discrete) sets with the smash product. Then we have a functor
\begin{equation}\label{eq:simplices_over_F_1}
\Delta^{op}\to \CAlg(\Set_*),\quad [n]\mapsto (\N^n)_+=\N^n\sqcup \{\ast\},
\end{equation}
If we denote by $e_i\in\N^n$ the basis elements, $1\leq i\leq n,$ then the functor \eqref{eq:simplices_over_F_1} sends
an order-preserving map $f:[n]\to [m]$ to the map
\begin{equation*}
f^*:(\N^m)_+\to(\N^n)_+,\quad f^*(e_i)=
\begin{cases}
e_j & \text{if }f(j-1)\leq i-1,\,f(j)\geq i;\\
\ast & \text{if }i\leq f(0);\\
0 & \text{if }i\geq f(n)+1.
\end{cases} 
\end{equation*}
Applying $\Sigma^{\infty}(-),$ we obtain a simplicial object in the category of connective $\bE_{\infty}$-rings, or equivalently a cosimplicial object in the category of affine schemes over $\bS.$ In particular, we obtain a simplicial object in the category of localizing motives $\cU_{\loc}(\Delta^{\bullet})\in (\Mot^{\loc})^{\Delta^{op}},$ where $\cU_{\loc}(\Delta^n)=\cU_{\loc}(\bS[\Delta^n]).$

\begin{theo}\label{th:A^1_invariant_localizing_motives}
\begin{enumerate}[label=(\roman*),ref=(\roman*)]
	\item The object $\wt{\cU}_{\loc}(\bS[x])\in\Mot^{\loc}$ generates a smashing ideal. The corresponding idempotent $\bE_{\infty}$-algebra in $\Mot^{\loc}$ is given by the geometric realization $|\cU_{\loc}(\Delta^{\bullet})|=\colim\limits_{\Delta^{op}} \cU_{\loc}(\Delta^{\bullet}).$ In particular, we have an equivalence 
	\begin{equation*}
		\Mot^{\loc,\A^1}\simeq\Mod\hy |\cU_{\loc}(\Delta^{\bullet})|\subset \Mot^{\loc}.
	\end{equation*} \label{Mot^loc_A^1_over_Sp}
	\item Let $\cE$ be a rigid $\bE_1$-monoidal category, and consider the category $\Mot^{\loc}_{\cE}$ as a module over $\Mot^{\loc}.$ Then we have an equivalence
	\begin{equation*}
		\Mot^{\loc,\A^1}_{\cE}\simeq\Mod_{|\cU_{\loc}(\Delta^{\bullet})|}(\Mot^{\loc}_{\cE}).
	\end{equation*} \label{Mot^loc_A^1_over_E}
\end{enumerate}
\end{theo}

This follows from the standard argument from \cite{MV99}. However, it will be convenient for us to deduce Theorem \ref{th:A^1_invariant_localizing_motives} from the following general statement on localizing ideals in symmetric monoidal presentable stable categories. We use the notion of a deformed tensor algebra from Definition \ref{def:deformed_tensor_algebra}.

\begin{prop}\label{prop:A_2_coalgebra_generates_smashing_ideal} Let $\cE\in\Alg_{\bE_{\infty}}(\Pr^L_{\st})$ be a symmetric monoidal presentable stable category. Let $C\in\cE$ be an object equipped with a morphism $\veps:C\to 1_{\cE}$ (i.e. $C$ is an $\bE_0$-coalgebra in $\cE$). Suppose that there exists a morphism $\Delta:C\to C\otimes C$ such that the composition
\begin{equation*}
C\xto{\Delta}C\otimes C\xto{\id\otimes \veps}C\otimes 1_{\cE} \cong C
\end{equation*}
is homotopic to identity (we do not impose any coassociativity or cocommutativity constraints on $\Delta$). We denote by $\la C\ra^{\otimes}\subset\cE$ the localizing ideal generated by $C.$

Then $\la C\ra^{\otimes}$ is a smashing ideal of $\cE.$ The corresponding idempotent $\bE_{\infty}$-algebra in $\cE$ is given by the deformed tensor algebra $T^{\deff}(C[1],\veps[1]).$ In particular, we have an equivalence
\begin{equation*}
\cE/\la C\ra^{\otimes}\simeq \Mod\hy T^{\deff}(C[1],\veps[1])\subset\cE.
\end{equation*}  
\end{prop}

\begin{proof}
Consider the natural increasing filtration $\Fil_{\bullet}T^{\deff}(C[1],\veps[1]).$ Its associated graded is given by the (non-deformed) tensor algebra of $C[1].$ Hence, the object $\Fiber(1_{\cE}\to T^{\deff}(C[1],\veps[1]))$ is contained in the ideal $\la C\ra^{\otimes}.$ Thus, it suffices to prove that the object $C\otimes T^{\deff}(C[1],\veps[1])$ is zero.

Consider the complex $(\cK_{\bullet},d)$ in the homotopy category $\h\cE,$ whose terms are given by $\cK_n=C\otimes \gr_n^{\Fil}T^{\deff}(C[1],\veps[1])[-n],$ and the differential $d$ is given by
\begin{multline*}
d_n:C\otimes \gr_n^{\Fil}T^{\deff}(C[1],\veps[1])[-n]\to C\otimes \Fil_{n-1}T^{\deff}(C[1],\veps[1])[1-n]\\
\to C\otimes\gr_{n-1}^{\Fil}T^{\deff}(C[1],\veps[1])[1-n],\quad n\geq 1.
\end{multline*}
It suffices to prove that the complex $\cK_{\bullet}$ has a contracting homotopy.

The complex $\cK_{\bullet}$ is simply the tensor product of $C$ and the ($E_0$-coalgebra version of) Amitsur complex of $C:$ we have $\cK_n\cong C^{\otimes n+1}$ for $n\geq 0,$ and the differential is given by
\begin{equation*}
d_n:C^{\otimes n+1}\to C^{\otimes n},\quad d_n=\sum\limits_{i=1}^{n} (-1)^{i-1} \id^{\otimes i}\otimes \veps\otimes \id^{\otimes n-i},\quad n\geq 1.
\end{equation*}
We define the contracting homotopy $h$ by the standard formula:
\begin{equation*}
h_n:C^{\otimes n+1}\to C^{\otimes n+2},\quad h_n=\Delta\otimes \id^{\otimes n},\quad n\geq 0.
\end{equation*}
Using our assumption on $\Delta:C\to C\otimes C$ it is straightforward to check that $dh+hd = \id,$ as required.
\end{proof}

\begin{proof}[Proof of Theorem \ref{th:A^1_invariant_localizing_motives}.] We prove \ref{Mot^loc_A^1_over_Sp}. To be able to apply Proposition \ref{prop:A_2_coalgebra_generates_smashing_ideal}, we only need to observe that the structure of a multiplicative monoid on the affine line induces an $\bE_{\infty}$-coalgebra structure on the object $\wt{\cU}_{\loc}(\bS[x])\in\Mot^{\loc}.$ To formalize this we argue as follows. First, the object $\N\in\CAlg(\Set)$ is naturally a cocommutative coalgebra: the comultiplication $\N\to \N^2$ is given by $n\mapsto(n,n).$ Applying $\Sigma^{\infty}((-)_+),$ we obtain an $\bE_{\infty}$-coalgebra structure on $\bS[x]$ as an object of $\Alg_{\bE_{\infty}}(\Sp).$ This gives an $\bE_{\infty}$-coalgebra structure on $\Mod\hy \bS[x]\in\Cat_{\st}^{\dual}.$ Applying $\cU_{\loc},$ we get an $\bE_{\infty}$-coalgebra structure on $\cU_{\loc}(\bS[x])\in\Mot^{\loc}.$ In particular, we obtain an $\bE_{\infty}$-ring structure on the spectrum $\Hom_{\Mot^{\loc}}(\cU_{\loc}(\bS[x]),\cU_{\loc}(\bS)).$ Similarly, we have a (discrete) commutative monoid structure on the set $\Map_{\CAlg(\Set_*)}(\N_+,\{0\}_+).$ The latter monoid is freely generated by the idempotent element, given by the map $\N_+\to\{\ast\}\to\{0\}_+.$ Therefore, we obtain a natural morphism of $\bE_{\infty}$-rings
\begin{equation*}
\bS\times\bS\to \Hom_{\Mot^{\loc}}(\cU_{\loc}(\bS[x]),\cU_{\loc}(\bS)).
\end{equation*}
This morphism induces a decomposition of the $\bE_{\infty}$-coalgebra $\cU_{\loc}(\bS[x])$ into a coproduct (direct sum) of $\bE_{\infty}$-coalgebras, which on the level of objects is given by $\cU_{\loc}(\bS[x])=\cU_{\loc}(\bS)\oplus \wt{\cU}_{\loc}(\bS[x]).$ In particular, we obtain an $\bE_{\infty}$-coalgebra structure on $\wt{\cU}_{\loc}(\bS[x]),$ as stated.

Denoting by $\veps:\wt{\cU}_{\loc}(\bS[x])\to\cU_{\loc}(\bS)$ the counit morphism, we see that the deformed tensor algebra $T^{\deff}(\wt{\cU}_{\loc}(\bS[x])[1],\veps[1])$ is isomorphic to the geometric realization $|\cU_{\loc}(\Delta^{\bullet})|.$ Hence, \ref{Mot^loc_A^1_over_Sp} follows from Proposition \ref{prop:A_2_coalgebra_generates_smashing_ideal}.

Finally, \ref{Mot^loc_A^1_over_E} follows directly from \ref{Mot^loc_A^1_over_Sp}.
\end{proof}

Next, we recall that Weibel's homotopy $K$-theory $KH$ was originally defined in the $\Z$-linear context \cite{Wei89, Wei13, Tab15}. It naturally generalizes to general stable categories as follows. For $\cC\in\Cat^{\perf}$ we put
\begin{equation*}
KH(\cC)=|\cU_{\loc}(\cC\otimes\Perf(\bS[\Delta^{\bullet}]))|=\colim\limits_{\Delta^{op}} \cU_{\loc}(\cC\otimes\Perf(\bS[\Delta^{\bullet}])).
\end{equation*}
By definition, this is a finitary localizing invariant of $\cC.$ We denote by $KH^{\cont}:\Cat_{\st}^{\dual}\to\Sp$ the corresponding localizing invariant of dualizable categories.

\begin{cor}\label{cor:KH}
For a dualizable category $\cC$ we have
\begin{equation*}
\Hom_{\Mot^{\loc,\A^1}}(\cU_{\loc}^{\A^1}(\bS),\cU_{\loc}^{\A^1}(\cC))\cong KH^{\cont}(\cC).
\end{equation*}
\end{cor}

\begin{proof}
This follows directly from Theorem \ref{th:A^1_invariant_localizing_motives}.
\end{proof}

\section{Corepresentability of $\TR$ and $\TC$}
\label{sec:corepresentability_TR_TC}

In this section we give first non-trivial examples of computations of morphisms in the category $\Mot^{\loc}=\Mot^{\loc}_{\Sp}.$ We prove the corepresentability of the classical invariants $\TR$ (topological restriction) and $\TC$ (topological cyclic homology) when restricted to connective $\bE_1$-rings. It will be convenient to consider localizing invariants of small categories, so we consider the universal finitary localizing invariant $\cU_{\loc}:\Cat^{\perf}\to\Mot^{\loc},$ and put $\cU_{\loc}(\A)=\cU_{\loc}(\Perf(A))\in\Mot^{\loc}$ for an $\bE_1$-ring.

\subsection{Corepresentability of $\TR$}
\label{ssec:corepresentability_TR}

As in Section \ref{sec:A^1_invariant_motives}, we put $\bS[x]=\Sigma^{\infty}\N_+,$ so that $\Spec \bS[x]$ is the flat affine line over $\bS.$ We denote by $\tilde{\cU}_{\loc}(\bS[x])$ the reduced motive, so we have $\cU_{\loc}(\bS[x])=\tilde{\cU}_{\loc}(\bS[x])\oplus \cU_{\loc}(\bS).$ We will prove the following result.

\begin{theo}\label{th:corepresentability_of_TR}
Let $R$ be an $\bE_1$-ring spectrum. Then we have an isomorphism
\begin{equation*}
\Hom_{\Mot^{\loc}}(\wt{\cU}_{\loc}(\bS[x]),\cU_{\loc}(R))\cong \Omega\prolim[n]K(R[x^{-1}]/x^{-n},(x^{-1})).
\end{equation*}
In particular, if $R$ is connective, then we have an isomorphism
\begin{equation}\label{eq:corepresentability_of_TR}
\Hom_{\Mot^{\loc}}(\wt{\cU}_{\loc}(\bS[x]),\cU_{\loc}(R))\cong \TR(R).
\end{equation}
\end{theo}

The proof is given below. Here we use the notation
\begin{equation*}
K(R[x^{-1}]/x^{-n},(x^{-1}))=\Fiber(K(R[x^{-1}]/x^{-n})\to K(R)).
\end{equation*}
We refer to \cite{BHM93, HM97, BM16} for the classical construction of the invariant $\TR.$ We use the description from \cite[Theorem A]{McC23} stating that for a connective $\bE_1$-ring $R$ we have
\begin{equation*}
\TR(R)\cong \Omega \prolim[n]K(R[x^{-1}]/x^{-n},(x^{-1})).
\end{equation*}
It will become clear from the proof why we consider the truncated polynomials in $x^{-1}$ instead of an abstract variable.

We will use a version of the bar-cobar adjunction (in fact, equivalence) between $\bE_1$-algebras and $\bE_1$-coalgebras which is explained in \cite[Appendix A]{Bur22}. We follow the notation from loc. cit. Let $\Sp^{\Gr}$ be the symmetric monoidal category of graded objects in $\Sp,$ and we denote its unit object simply by $\bS$ (placed in degree $0$). We denote by $\Sp_+^{\Gr}\subset \Sp^{\Gr}_{\bS/-/\bS}$ the full subcategory formed by objects which vanish in negative degrees and are isomorphic to $\bS$ in degree $0.$ Note that the assignment $X\mapsto\bbar{X}=\Cone(\bS\to X)$ defines a (non-monoidal) equivalence from $\Sp_+^{\Gr}$ to the full subcategory of $\Sp^{\Gr}$ formed by objects concentrated in strictly positive degrees. By \cite[Theorem A.6]{Bur22} we have a pair of mutually inverse equivalences
\begin{equation}\label{eq:bar_cobar_equivalence}
\Bbar:\Alg_{\bE_1}(\Sp_+^{\Gr})\xto{\sim}\Coalg_{\bE_1}(\Sp_+^{\Gr}),\quad \Cobar:\Coalg_{\bE_1}(\Sp_+^{\Gr})\xto{\sim}\Alg_{\bE_1}(\Sp_+^{\Gr}),
\end{equation}
given by the standard bar resp. cobar construction as defined in \cite[Section 5.2.3]{Lur17}

\begin{proof}[Proof of Theorem \ref{th:corepresentability_of_TR}]
Consider the flat projective line $\PP^{1,\flat}_{\bS}$ as defined in \cite[Construction 5.4.1.3]{Lur18}, with a distinguished $\bS$-point $\infty,$ whose complement is the flat affine line $\A^{1,\flat}_{\bS}=\Spec \bS[x].$ We have a short exact sequence in $\Cat^{\perf}:$
\begin{equation*}
0\to \Perf_{\{\infty\}}(\PP^{1,\flat}_{\bS})\to\Perf(\PP^{1,\flat}_{\bS})\to\Perf(\bS[x])\to 0.
\end{equation*}
The functor of global sections $\Gamma:\Perf_{\{\infty\}}(\PP^{1,\flat}_{\bS})\to \Sp^{\omega}$ has an (exact) right inverse sending $\bS$ to the skyscraper at $\infty.$ Consider the reduced motive $\wt{\cU}_{\loc}(\Perf_{\{\infty\}}(\PP^{1,\flat}_{\bS}))=\Fiber(\cU_{\loc}(\Perf_{\{\infty\}}(\PP^{1,\flat}_{\bS})\to\cU_{\loc}(\bS)),$ so we have 
\begin{equation*}
\cU_{\loc}(\Perf_{\{\infty\}}(\PP^{1,\flat}_{\bS})=\wt{\cU}_{\loc}(\Perf_{\{\infty\}}(\PP^{1,\flat}_{\bS}))\oplus \cU_{\loc}(\bS).
\end{equation*}

We claim that we have an isomorphism
\begin{equation}\label{eq:isomorphism_of_reduced_motives}
\wt{\cU}_{\loc}(\bS[x])\cong \Sigma\wt{\cU}_{\loc}(\Perf_{\{\infty\}}(\PP^{1,\flat}_{\bS})).
\end{equation}

To see this, recall that the Beilinson's full exceptional collection $\la \cO(-1),\cO\ra$ \cite{Bei78} gives a semi-orthogonal decomposition $\Perf(\PP^{1,\flat}_{\bS})=\la \Sp^{\omega}\otimes \cO(-1), \Sp^{\omega}\otimes \cO\ra.$ More precisely, Beilinson's theorem holds over $\bS$ by \cite[Theorem 5.4.2.6]{Lur17}. We obtain a direct sum decomposition $\cU_{\loc}(\PP^{1,\flat}_{\bS})=\cU_{\loc}(\bS)\oplus\cU_{\loc}(\bS).$ Now the components of the morphism $\cU_{\loc}(\Perf_{\{\infty\}}(\PP^{1,\flat}_{\bS})\to \cU_{\loc}(\bS)\oplus\cU_{\loc}(\bS)$ are given by $(\cU_{\loc}(\Gamma),\cU_{\loc}(\Gamma)).$ The components of the morphism $\cU_{\loc}(\bS)\oplus\cU_{\loc}(\bS)\to \cU_{\loc}(\bS[x])$ are given by $([\bS[x]],[\bS[x]]).$ This gives the isomorphism \eqref{eq:isomorphism_of_reduced_motives}.

We have $\Perf_{\{\infty\}}(\PP^{1,\flat}_{\bS})\simeq \Perf_{x^{-1}\hy\tors}(\bS[x^{-1}])$ -- the category of perfect $\bS[x^{-1}]$-modules which are annihilated by a power of $x^{-1}.$ This category is proper over $\bS.$
It is generated by a single object $\bS\cong\Cone(\bS[x^{-1}]\xto{x^{-1}}\bS[x^{-1}]).$ Consider its endomorphism $\bE_1$-ring
\begin{equation*}
A=\End_{\bS[x^{-1}]}(\bS)\cong \bS\oplus\bS[-1].
\end{equation*}
Note that $A$ naturally refines to an $\bE_1$-algebra object in $\Sp_+^{\Gr},$ formally the grading comes from the $\bG_{m,\bS}^{\flat}$-action on $\PP^{1,\flat}_{\bS}.$ Applying the bar construction from \eqref{eq:bar_cobar_equivalence}, we get 
\begin{equation*}
\Bbar(A)\cong \colim[n](\bS[x^{-1}]/x^{-n})^*\in\Coalg_{\bE_1}(\Sp_+^{\Gr}).
\end{equation*}
Here $x^{-1}$ has degree $-1$ and $(-)^*=\un{\Hom}_{\Sp^{\Gr}}(-,\bS).$ We put $A_n=\Cobar((\bS[x^{-1}]/x^{-n})^*).$ Forgetting the grading, we consider $A_n$ simply as $\bE_1$-algebras in $\Sp.$ By construction, we have a direct sequence $(A_n)_{n\geq 0},$ with $\indlim[n]A_n\cong A.$ Moreover, each $\bE_1$-ring $A_n$ is contained in $\Alg_{\bE_1}(\Sp)^{\omega}.$  By Theorem \ref{th:morphisms_in_Mot^loc_via_limits} \ref{Hom_via_inverse_limit_for_proper} we obtain
\begin{equation*}
\Hom_{\Mot^{\loc}}(\cU_{\loc}(\Perf_{\{\infty\}}(\PP^1)),\cU_{\loc}(R))\cong \prolim[n]K(\End_{A_n^{op}}(\bS)\otimes R).
\end{equation*}
By the standard Koszul duality, we have $\End_{A_n^{op}}(\bS)\cong \bS[x^{-1}]/x^{-n}.$ Hence, we obtain
\begin{equation*}
\Hom_{\Mot^{\loc}}(\cU_{\loc}(\Perf_{\{\infty\}}(\PP^1)),\cU_{\loc}(R))\cong \prolim[n]K(R[x^{-1}]/x^{-n}).
\end{equation*}
Passing to reduced motives and using \eqref{eq:isomorphism_of_reduced_motives}, we get
\begin{multline*}
\Hom_{\Mot^{\loc}}(\wt{\cU}_{\loc}(\bS[x]),\cU_{\loc}(\bS))\cong \Omega\Hom_{\Mot^{\loc}}(\wt{\cU}_{\loc}(\Perf_{\{\infty\}}(\PP^1)),\cU_{\loc}(R))\\
\cong \Omega\prolim[n] K(R[x^{-1}]/x^{-n},(x^{-1})),
\end{multline*}
as stated.
\end{proof}

We will need the following interpretation of Theorem \ref{th:corepresentability_of_TR} for connective $\bE_1$-rings. We denote by $\CycSp$ the category of cyclotomic spectra as defined by Nikolaus and Scholze \cite{NS18}, and consider the topological Hochschild homology as a finitary localizing invariant $\THH:\Cat^{\perf}\to\CycSp,$ and denote by the same symbol the induced functor from $\Mot^{\loc}$ to $\CycSp.$ By \cite[Theorem 6.11]{BM16}, \cite[Theorem 3.3.12]{McC23} for a connective $\bE_1$-ring $R$ we have
\begin{equation*}
\Hom_{\CycSp}(\wt{\THH}(\bS[x]),\THH(R))\cong\TR(R).
\end{equation*}
Hence, for connective $\bE_1$-rings Theorem \ref{th:corepresentability_of_TR} states that the functor $\THH:\Mot^{\loc}\to\CycSp$ induces an isomorphism
\begin{equation}\label{eq:isomorphism_on_morphisms}
\Hom_{\Mot^{\loc}}(\wt{\cU}_{\loc}(\bS[x]),\cU_{\loc}(\bS))\xto{\sim} \Hom_{\CycSp}(\wt{\THH}(\bS[x]),\THH(R)).
\end{equation}
Now if $R$ is a connective $\bE_{\infty}$-ring, then \eqref{eq:isomorphism_on_morphisms} is naturally an isomorphism of $\bE_{\infty}$-rings. Here the $\bE_{\infty}$-structure on the source comes from the $\bE_{\infty}$-coalgebra structure on $\wt{\cU}_{\loc}(\bS[x])$ and from the $\bE_{\infty}$-algebra structure on $\cU_{\loc}(R)$ (the former is explained in the proof of Theorem \ref{th:A^1_invariant_localizing_motives}).

As an application of Theorem \ref{th:corepresentability_of_TR}, we obtain the following new result on the so-called nil $K$-theory $NK(R)=\Cone(K(R)\to K(R[x]))$ \cite{BHS64, Bas68}. 
For a (discrete) commutative ring $A$ we denote by $W_{\biggg}(A)$ the ring of big Witt vectors.

\begin{cor}\label{cor:nil_K_theory_module_over_TR}
Let $R$ be a connective $\bE_{\infty}$-ring. Then the spectrum $NK(R)$ has a natural structure of a module over $\TR(R).$ In particular, the groups $NK_n(R)$ have a natural structure of a module over $W_{\biggg}(\pi_0(R)).$
\end{cor}

\begin{proof}
By definition of nil $K$-theory we have
\begin{equation*}
NK(R)\cong \Hom_{\Mot^{\loc}}(\cU_{\loc}(\bS),\wt{\cU}_{\loc}(\bS[x])\otimes \cU_{\loc}(R))
\end{equation*} The result now follows from the isomorphism
of $\bE_{\infty}$-rings \eqref{eq:corepresentability_of_TR}, where the $\bE_{\infty}$-structure on the source comes from the $\bE_{\infty}$-coalgebra structure on $\wt{\cU}_{\loc}(\bS[x])$ and the $\bE_{\infty}$-algebra structure on $\cU_{\loc}(R).$

The second assertion follows from the well known isomorphism $\pi_0(\TR(R))\cong W_{\biggg}(\pi_0(R)).$
\end{proof}

\subsection{Corepresentability of $\TC$}

We recall the idempotent $\bE_{\infty}$-algebra $|\cU_{\loc}(\Delta^{\bullet})|$ from Theorem \ref{th:A^1_invariant_localizing_motives}. Recall from \cite{NS18} that for a connective $\bE_1$-ring $R$ we have
\begin{equation*}
\TC(R)=\Hom_{\CycSp}(\bS^{\triv},\THH(R)),
\end{equation*}
where $\bS^{\triv}=\THH(\bS)$ is the unit object. We prove the following result.

\begin{theo}\label{th:corepresentability_of_TC}
Let $R$ be a connective $\bE_1$-ring spectrum. Then we have a natural isomorphism
\begin{equation*}
\Hom_{\Mot^{\loc}}(\Fiber(\cU_{\loc}(\bS)\to |\cU_{\loc}(\Delta^{\bullet})|),\cU_{\loc}(R))\cong \TC(R).
\end{equation*}
Moreover, the following square naturally commutes:
\begin{equation}\label{eq:compatibility_with_cyclotomic_trace}
\begin{tikzcd}
\Hom_{\Mot^{\loc}}(\cU_{\loc}(\bS),\cU_{\loc}(R))\ar[r]\ar[d] & \Hom_{\Mot^{\loc}}(\Fiber(\cU_{\loc}(\bS)\to |\cU_{\loc}(\Delta^{\bullet})|),\cU_{\loc}(R))\ar[d]\\
K(R)\ar[r] & \TC(R).
\end{tikzcd}
\end{equation}
Here the lower horizontal map is the cyclotomic trace. In particular, we have
\begin{equation*}
\Hom_{\Mot^{\loc}}(|\cU_{\loc}(\Delta^{\bullet})|,\cU_{\loc}(R))\cong K^{\inv}(R)=\Fiber(K(R)\to\TC(R)).
\end{equation*}
\end{theo}

We will need the following generalization of Theorem \ref{th:corepresentability_of_TR} (for connective $\bE_1$-rings). We denote by $\wt{\THH}(\bS[x])\in\CycSp$ the reduced topological Hochschild homology of $\bS[x]$ considered as a cyclotomic spectrum.

\begin{prop}\label{prop:morphisms_in_Mot^loc_and_in_CycSp}
Let $R$ be a connective $\bE_1$-ring. Then for $k\geq 1$ we have an isomorphism 
\begin{equation*}
\Hom_{\Mot^{\loc}}(\wt{\cU}_{\loc}(\bS[x])^{\otimes k},\cU_{\loc}(\bS))\xto{\sim}\Hom_{\CycSp}(\wt{\THH}(\bS[x])^{\otimes k},\THH(\bS)).
\end{equation*} 
\end{prop}

\begin{proof}
We use the isomorphism $\wt{\cU}_{\loc}(\bS[x])\cong \Sigma \wt{\cU}_{\loc}(\Perf_{\{\infty\}}(\PP^{1,\flat}_{\bS}))$ from the proof of Theorem \ref{th:corepresentability_of_TR}. We need to prove that the map
\begin{equation}\label{eq:from_morphisms_in_Mot^loc_to_morphisms_in_CycSp}
\Hom_{\Mot^{\loc}}(\wt{\cU}_{\loc}(\Perf_{\{\infty\}}(\PP^{1,\flat}_{\bS}))^{\otimes k},\cU_{\loc}(\bS))\to \Hom_{\CycSp}(\wt{\THH}(\Perf_{\{\infty\}}(\PP^{1,\flat}_{\bS}))^{\otimes k},\THH(\bS))
\end{equation}
is an isomorphism.
As in the proof of Theorem \ref{th:corepresentability_of_TR}, consider the ind-system of $\bE_1$-rings $(A_n)_{n\geq 0},$ $A_n=\Cobar((\bS[x^{-1}]/x^{-n})^*).$
We will use the Dundas-Goodwillie-McCarthy theorem and the following already established facts:
\begin{itemize}
	\item we have $\indlim[n]\Perf(A_n)\simeq \Perf_{\infty}(\PP^{1,\flat}_{\bS});$
	\item for any $n\geq 0$ there exists $m\geq n$ such that the morphism $\cU_{\loc}(A_n)\to \cU_{\loc}(A_m)$ is trace-class in $\Mot^{\loc};$
	\item we have an isomorphism of pro-objects
	\begin{equation}\label{eq:pro_isomorphism_truncated_polynomials}
	\proolim[n]\cU_{\loc}(A_n)^{\vee}\cong \proolim[n]\cU_{\loc}(\bS[x^{-1}]/x^{-n})\quad\text{in }\Pro(\Mot^{\loc}).
	\end{equation}
\end{itemize}
Here the second fact follows from the nuclearity of $\Ind(\Perf_{\infty}(\PP^{1,\flat}_{\bS}))\in\Cat_{\st}^{\cg}$ over $\Sp$ which holds by Proposition \ref{prop:proper_are_nuclear}. The isomorphism \eqref{eq:pro_isomorphism_truncated_polynomials} follows from Theorem \ref{th:morphisms_in_Mot^loc_via_limits} \ref{Hom_via_inverse_limit_for_proper} since $\End_{A_n^{op}}(\bS)\cong\bS[x^{-1}]/x^{-n}.$ It follows that we have similar assertions for $k$-th tensor powers of the objects involved, where $k\geq 1.$ Passing to reduced motives, we obtain the isomorphisms
\begin{multline*}
\Hom_{\Mot^{\loc}}(\wt{\cU}_{\loc}(\Perf_{\{\infty\}}(\PP^{1,\flat}_{\bS}))^{\otimes k},\cU_{\loc}(\bS))\\
\cong \prolim[n]\Hom_{\Mot^{\loc}}(\cU_{\loc}(\bS),((\wt{\cU}_{\loc}(A_n))^{\otimes k})^{\vee}\otimes\cU_{\loc}(R))\\
\cong \prolim[n]\Hom_{\Mot^{\loc}}(\cU_{\loc}(\bS),\wt{\cU}_{\loc}(\bS[x^{-1}]/x^{-n})^{\otimes k}\otimes\cU_{\loc}(R))
\end{multline*}
A similar chain of isomorphisms for morphisms in $\CycSp$ gives
\begin{multline*}
\Hom_{\CycSp}(\wt{\THH}(\Perf_{\{\infty\}}(\PP^{1,\flat}_{\bS}))^{\otimes k},\THH(\bS))\\
\cong \prolim[n]\Hom_{\CycSp}(\THH(\bS),\wt{\THH}(\bS[x^{-1}]/x^{-n})^{\otimes k}\otimes\THH(R)).
\end{multline*}
Now for each $n\geq 1$ the map
\begin{multline*}
\Hom_{\Mot^{\loc}}(\cU_{\loc}(\bS),\wt{\cU}_{\loc}(\bS[x^{-1}]/x^{-n})^{\otimes k}\otimes\cU_{\loc}(R))\\
\to \Hom_{\CycSp}(\THH(\bS),\wt{\THH}(\bS[x^{-1}]/x^{-n})^{\otimes k}\otimes\THH(R))
\end{multline*}
is a retract of the map
\begin{equation*}
\Fiber(K((\bS[x^{-1}]/x^{-n})^{\otimes k}\otimes R)\to K(R))\to \Fiber(\TC((\bS[x^{-1}]/x^{-n})^{\otimes k}\otimes R)\to \TC(R)).
\end{equation*}
The latter map is an isomorphism by \cite[Theorem 7.0.0.2]{DGM13}. Therefore, we have an isomorphism
\begin{multline*}
\prolim[n]\Hom_{\Mot^{\loc}}(\cU_{\loc}(\bS),\wt{\cU}_{\loc}(\bS[x^{-1}]/x^{-n})^{\otimes k}\otimes\cU_{\loc}(R))\\
\xto{\sim}\prolim[n]\Hom_{\CycSp}(\THH(\bS),\wt{\THH}(\bS[x^{-1}]/x^{-n})^{\otimes k}\otimes\THH(R)).
\end{multline*}
This proves that the map \eqref{eq:from_morphisms_in_Mot^loc_to_morphisms_in_CycSp} is an isomorphism, as stated.
\end{proof}

\begin{proof}[Proof of Theorem \ref{th:corepresentability_of_TC}]
The geometric realization $|\cU_{\loc}(\Delta^{\bullet})|$ has a standard increasing exhaustive filtration, whose associated graded pieces are the reduced components $\cU_{\loc}(\Delta^n)^{\red}[n]$ (the latter object is a direct summand of $\cU_{\loc}(\Delta^n)[n]$). We have $\cU_{\loc}(\Delta^n)^{\red}\cong \wt{\cU}_{\loc}(\bS[x])^{\otimes n}.$

Applying the functor $\THH:\Mot^{\loc}\to\CycSp$ gives the standard increasing exhaustive filtration on $|\THH(\Delta^{\bullet})|,$ with associated graded pieces $\THH(\Delta^n)^{\red}[n]\cong \wt{\THH}(\bS[x])^{\otimes n}.$ Moreover, we have $|\THH(\Delta^{\bullet})|=0:$ this geometric realization is naturally an $\bE_{\infty}$-ring with $1=0$ in $\pi_0.$ Applying Proposition \ref{prop:morphisms_in_Mot^loc_and_in_CycSp}, for a connective $\bE_1$-ring $R$ we obtain the isomorphisms
\begin{multline*}
\Hom_{\Mot^{\loc}}((\Fiber(\cU_{\loc}(\bS)\to |\cU_{\loc}(\Delta^{\bullet})|),\cU_{\loc}(R)))\\
\cong \Hom_{\CycSp}(\Fiber(\THH(\bS)\to|\THH(\Delta^{\bullet})|),\THH(R))\\
\cong \Hom_{\CycSp}(\THH(\bS),\THH(R))\cong \TC(R).
\end{multline*}
By construction, the square \eqref{eq:compatibility_with_cyclotomic_trace} commutes.
\end{proof}

We obtain the following immediate application.

\begin{cor}\label{cor:fiber_of_K_to_KH_module_over_TC}
Let $R$ be a connective $\bE_{\infty}$-ring. Then the spectrum $\Fiber(K(R)\to KH(R))$ is naturally a module over $\TC(R).$
\end{cor}

\begin{proof}
This is proved in the same way as Corollary \ref{cor:nil_K_theory_module_over_TR}.
\end{proof}

\begin{remark}\begin{enumerate}
		\item Denote by $\Mot^{\loc,\cyc}\subset \Mot^{\loc}$ the localizing ideal generated by the reduced motive $\wt{\cU}_{\loc}(\bS[x]).$ We know by Theorem \ref{th:A^1_invariant_localizing_motives} that this is a smashing ideal. As usual, this gives a recollement in the sense of \cite{BBD82}:
\begin{equation}\label{eqtn_recoll}
	\xymatrix{\Mot^{\loc,\A^1} \ar[rr]|{i_*}&& \Mot^{\loc} \ar@<-2ex>[ll]|{i^*} \ar@<2ex>[ll]|{i^!} \ar[rr]|{j^*} && \Mot^{\loc,\cyc} \ar@<-2ex>[ll]|{j_!} \ar@<2ex>[ll]|{j_*}} 
\end{equation}
Here $j_!$ is the (strongly continuous) inclusion of $\Mot^{\loc,\cyc}$ as a smashing ideal, and $j^*$ is its right adjoint, which itself has a right adjoint $j_*.$ Similarly, $i^*$ is the (strongly continuous) quotient functor, with right adjoints $i_*$ and $i^!.$ Denote by $\Gamma:\Mot^{\loc}\to\Sp$ the functor $\Hom(\cU_{\loc}(\bS),-),$ so that $\Gamma(\cU_{\loc}(\cC)) = K(\cC).$ Then by Corollary \ref{cor:KH}  we have
\begin{equation*}
\Gamma(i_* i^*\cU_{\loc}(\cC))\cong KH(\cC),\quad \cC\in\Cat^{\perf},
\end{equation*}
and by Theorem \ref{th:corepresentability_of_TC} for a connective $\bE_1$-ring $R$ we have
\begin{equation*}
\Gamma(j_* j^*\cU_{\loc}(R))\cong \TC(R),\quad \Gamma(i_*i^!\cU_{\loc}(R))\cong K^{\inv}(R).
\end{equation*}
\item The above isomorphisms motivate a natural question: is it true that $\cU_{\loc}^{\A^1}:\Cat^{\perf}\to\Mot^{\loc,\A^1}$ is a truncating invariant in the sense of \cite[Definition 3.1]{LT19}? Namely, is it true that for a connective $\bE_1$-ring $R$ we have an isomorphism $\cU_{\loc}^{\A^1}(R)\xto{\sim}\cU_{\loc}^{\A^1}(\pi_0(R))?$ In Subsection \ref{ssec:refined_HP_of_Q_over_Qx} we will give a negative answer to a similar question for $\Q[x]$-algebras. Namely, we will prove that the functor $\cU_{\loc}^{\A^1}:\Cat^{\perf}_{\Q[x]}\to\Mot^{\loc,\A^1}_{Q[x]}$ is not a truncating invariant using our construction of refined periodic cyclic homology.  
\end{enumerate}
\end{remark}

\section{$K$-homology of schemes}
\label{sec:K_homology_of_schemes}

In this section we will work with localizing motives over the (rigid) category $D(\mk),$ where $\mk$ is a discrete commutative ring. We put $\Mot^{\loc}_{\mk}=\Mot^{\loc}_{D(\mk)}.$ As in the previous section, it will be convenient to consider the universal finitary localizing invariant $\cU_{\loc}:\Cat_{\mk}^{\perf}\to\Mot^{\loc}_{\mk}.$ For simplicity we will assume $\mk$ to be noetherian. 

Given a quasi-compact separated scheme $X$ over $\mk,$ we use the notation $\QCoh(X)\simeq \Ind(\Perf(X))$ for the derived category of quasi-coherent sheaves. We put $\cU_{\loc}(X)=\cU_{\loc}(\Perf(X)).$ We will be interested in the following general question. 

\begin{ques}
Let $X$ be a separated scheme of finite type over a noetherian commutative ring $\mk.$ How to describe the spectrum $\Hom_{\Mot^{\loc}_{\mk}}(\cU_{\loc}(X),\cU_{\loc}(
\mk))$?
\end{ques} 

We obtain the answer in the situation when either $X$ is proper and $\mk$ is regular (Theorem \ref{th:K_homology_proper}), or $X$ is smooth and admits a smooth compactification (Theorem \ref{th:K_homology_smooth_with_compactification}). The latter result in fact generalizes to general smooth schemes (Theorem \ref{th:K_homology_of_smooth_general}), but in this case we only give a sketch of the proof, the details will appear in \cite{E}.

To deal with the proper case, we first develop the machinery of (large) relative derived categories in Subsection \ref{ssec:relative_derived}, which will allow to reduce to the special case $X\subset\Spec \mk.$ The latter case is deduced from a more general non-commutative result, which we prove in Subsection \ref{ssec:K_homology_proper}.

\subsection{Relative derived categories}
\label{ssec:relative_derived}

We first explain the general idea and motivation. For simplicity we work over an affine noetherian base scheme $\Spec \mk.$ Let $X/\mk$ be a (classical) separated scheme of finite type over $\mk.$ Recall from \cite{Lb06} (see also \cite{AJS23}) that an object $\cF\in D^b_{\coh}(X)$ is called {\it relatively perfect} over $\mk$ if for any affine open subset $U\subset X$ the complex $\Gamma(U,\cF)\in D(\mk)$ has finite $\Tor$-amplitude over $\mk.$ It is sufficient to check this condition for members of some affine Zariski cover of $X.$ We denote by $\Perf(X,\mk)\subset D^b_{coh}(X)$ the (idempotent-complete stable $\mk$-linear) full subcategory of relatively perfect complexes. 

Recall that $X$ is said to be of finite $\Tor$-dimension over $\mk$ if $\cO_X\in\Perf(X,\mk).$ This is equivalent to the inclusion $\Perf(X)\subset\Perf(X,\mk).$ 

If $\mk$ is regular, then we have $\Perf(X,\mk)=D^b_{\coh}(X).$ On the other hand, if $X$ is either smooth over $\mk$ or regular of finite $\Tor$-dimension over $\mk,$ then we have $\Perf(X,\mk)=\Perf(X).$

Recall that for an open subset $U\subset X$ the restriction functors $\Perf(X)\to\Perf(U)$ and $D^b_{\coh}(X)\to D^b_{\coh}(U)$ are quotient functors (up to direct summands in the former case). This motivates a natural question: assuming that $X$ is of finite $\Tor$-dimension over $\mk,$ is it true that for an open subset $U\subset X$ the restriction functor $\Perf(X,\mk)\to \Perf(U,\mk)$ is a quotient functor (up to direct summands)? It turns out that the answer is negative, and a counterexample is given in \cite[Section 3.3]{EP15}. We recall it for completeness.

Let $F$ be a field, and put $\mk=F[x,y,z,w]/(xy-zw).$ Let $X=\Spec\mk/w,$ and let $U=\{z\ne 0\}\subset X.$ It follows from the argument in loc. cit. that the functor $\Perf(X,\mk)\to \Perf(U,\mk)$ is not essentially surjective, even up to direct summands. In particular, it is not a quotient functor.

The goal of this section will be to introduce and study a certain nicely behaved dualizable category $\cC(X/\mk)$ when $X$ is separated of finite type and of finite $\Tor$-dimension over $\mk.$ It was announced in \cite[Remark 3.24]{E25}. We will in particular prove that the compact objects are given by $\cC(X/\mk)^{\omega}\simeq \Perf(X,\mk)$ and for an open subset $U\subset X$ we have a natural strongly continuous quotient functor $\cC(X/\mk)\to \cC(U/\mk).$ In general, the category $\cC(X/\mk)$ is strictly larger than the naive ind-completion $\Ind(\Perf(X,\mk)),$ even when $\mk$ is a field.  
 
We recall some notation. If $f:X\to Y$ is a morphism of finite type between noetherian schemes, then we denote by $f^!:\QCoh(Y)\to \QCoh(X)$ the Deligne's extraordinary inverse image. If $f$ is proper, then $f^!$ is the right adjoint to the pushforward functor $f_*.$ If $f$ is an open embedding, then we have $f^! = f^*.$ We denote by $\omega_{X/Y}=f^!\cO_Y$ the (canonical) relative dualizing complex. We will also use similar notation for derived schemes \cite{GaRo17}. All the fiber products of schemes are assumed to be derived.

Given a separated scheme $X$ of finite type over a noetherian ring $\mk$ we denote by $\Delta_X:X\to X\times_{\mk}X$ the diagonal morphism. We use the following notation for the shriek tensor product, relative to $\mk:$
\begin{equation*}
\cF\stens{!,\mk}\cG = \Delta_X^!(\cF\boxtimes\cG),\quad \cF,\cG\in\QCoh(X).
\end{equation*}
Note that this operation does not commute with filtered colimits unless $X$ is smooth over $\mk.$ However, it does commute with $\omega_1$-filtered colimits.

\begin{defi}\label{def:relative_derived_category} Let $\mk$ be a noetherian commutative ring. Let $X$ be a separated scheme of finite type over $\mk,$ which is of finite $\Tor$-dimension over $\mk.$ Consider the full subcategory
\begin{equation*}
\cC(X/\mk)\subset \Ind(\QCoh(X)^{\omega_1,op})\simeq \Fun_{\mk}(\QCoh(X)^{\omega_1},D(\mk)),
\end{equation*}
which is generated as a localizing subcategory by the objects of the form $\Gamma(X,-\stens{!,\mk}\cF),$ where $\cF\in\QCoh(X)^{\omega_1}.$ We call $\cC(X/\mk)$ the relative derived category of $X$ over $\mk.$
\end{defi}


We also define a closely related notion of a relatively compact morphism of (complexes of) quasi-coherent sheaves. For $\cF,\cG\in\QCoh(X)$ we denote by $\un{\Hom}_{\cO_X}(\cF,\cG)\in\QCoh(X)$ the internal $\Hom.$

\begin{defi}
Let $X/\mk$ be as in Definition \ref{def:relative_derived_category}. We say that a morphism $\varphi:\cF\to\cG$ in $\QCoh(X)$ is relatively compact over $\mk$ if the associated morphism $\cO_X\to\un{\Hom}_{\cO_X}(\cF,\cG)$ factors through $\cG\stens{!,\mk}\un{\Hom}_{\cO_X}(\cF,\omega_{X/\mk}).$ Here the canonical map from $\cG\stens{!,\mk}\un{\Hom}_{\cO_X}(\cF,\omega_{X/\mk})$ to $\un{\Hom}_{\cO_X}(\cF,\cG)$ is obtained by applying $\Delta_X^!$ to the composition
\begin{multline*}
\cG\boxtimes \un{\Hom}_{\cO_X}(\cF,\omega_{X/\mk})\to p_1^*(\cG)\otimes \un{\Hom}_{\cO_{X\times_{\mk} X}}(p_2^*(\cF),p_2^*(\omega_{X/\mk}))\\
\to \un{\Hom}_{\cO_{X\times_{\mk} X}}(p_2^*(\cF),p_1^*(\cG)\otimes p_2^*(\omega_{X/\mk}))\cong \un{\Hom}_{\cO_{X\times_{\mk} X}}(p_2^*(\cF),p_1^!(\cG)).
\end{multline*}
\end{defi}

We prove the following result, which roughly speaking means that the category $\cC(X/\mk)$ serves as a reasonable relative replacement of $\IndCoh(X).$ For an object $\cF\in\QCoh(X)^{\omega_1},$ we denote by $\cF^{op}$ the corresponding object of the opposite category $\QCoh(X)^{\omega_1,op}.$ As usual, we denote by $\cY:\QCoh(X)^{\omega_1,op}\to\Ind(\QCoh(X)^{\omega_1,op})$ the Yoneda embedding.

\begin{theo}\label{th:main_properties_of_relative_derived_categories}
Let $X/\mk$ be as in Definition \ref{def:relative_derived_category}.
\begin{enumerate}[label=(\roman*),ref=(\roman*)]
\item The category $\cC(X/\mk)$ is dualizable. \label{dualizability}
\item If $X$ is proper over $\mk,$ then we have an equivalence 
\begin{equation*}
\cC(X/\mk)\simeq \un{\Hom}_{\mk}^{\dual}(\QCoh(X),D(\mk))^{\vee}.	
\end{equation*} \label{equivalence_with_Hom^dual}
\item The inclusion functor $\cC(X/\mk)\to \Ind(\QCoh(X)^{\omega_1,op})$ is strongly continuous. Its right adjoint is given on compact objects by
\begin{equation*}
\cY(\cF^{op})\mapsto \Gamma(X,-\stens{!,\mk} \un{\Hom}_{\cO_X}(\cF,\omega_{X/\mk})).
\end{equation*} \label{description_of_right_adjoint}
\item Let $U\subset X$ be an open subset, and denote by $j:U\to X$ the inclusion. The functor 
\begin{equation*}
\Ind((j^*)^{op}):\Ind(\QCoh(X)^{\omega_1,op})\to \Ind(\QCoh(U)^{\omega_1,op})
\end{equation*} takes $\cC(X/\mk)$ to $\cC(U/\mk).$ The induced functor $j^*:\cC(X/\mk)\to\cC(U/\mk)$ is a strongly continuous quotient functor. \label{restricting_to_open}
\item The assignment $U\mapsto \cC(U/\mk)$ defines a sheaf on $X$ with values in $\Cat_{\mk}^{\dual}.$ Moreover, the same assignment defines a sheaf with values in $\Pr^L_{\mk}.$ \label{Zariski_descent}
\item The category of compact objects in $\cC(X/\mk)$ is equivalent to the category of relatively perfect complexes on $X$ over $\mk.$ The equivalence is given by
\begin{equation}\label{eq:relatively_perfect_are_compact_objects}
\Perf(X,\mk)\xto{\sim} \cC(X/\mk)^{\omega},\quad \cF\mapsto \Gamma(X,-\stens{!,\mk}\cF).
\end{equation} \label{compact_objects_of_relative_derived_categories}
\item The category $\cC(X/\mk)$ is self-dual. More precisely, we have the following perfect pairing over $\mk:$
\begin{equation}\label{eq:self_duality}
\cC(X/\mk)\tens{\mk}\cC(X/\mk)\to D(\mk),\quad (\inddlim[i]\cF_i^{op})\boxtimes (\inddlim[j]\cG_j^{op})\mapsto \indlim[(i,j)]\Hom(\cF\otimes\cG,\omega_{X/\mk}).
\end{equation} \label{self_duality}
\item The category $\cC(X/\mk)\subset \Ind(\QCoh(X)^{\omega_1,op})$ is generated by the objects of the form $\inddlim[n\in\N]\cF_n^{op},$ where each morphism $\cF_{n+1}\to\cF_n$ is relatively compact over $\mk.$ Moreover, these are exactly the $\omega_1$-compact objects in $\cC(X/\mk).$ \label{description_via_relatively_compact_maps}
\item Suppose that $X$ is either smooth over $\mk,$ or regular of finite Krull dimension. Then we have $\cC(X/\mk)\simeq \QCoh(X).$ \label{smooth_or_regular}
\end{enumerate}
\end{theo}

We will need the following observation on the partial independence from the base ring.

\begin{lemma}\label{lem:changing_base_ring}
	Let $X/\mk$ be as in Definition \ref{def:relative_derived_category}. Let $\mk'$ be a finitely generated smooth $\mk$-algebra, and suppose that we have a morphism $f:X\to\Spec \mk'$ over $\mk.$ Then we have equality $\cC(X/\mk')=\cC(X/\mk)$ of strictly full subcategories of $\Ind(\QCoh(X)^{\omega_1,op}).$
\end{lemma}

\begin{proof}
Indeed, for $\cF,\cG\in\QCoh(X)$ we have 
\begin{equation*}
	\cF\stens{!,\mk'}\cG\cong \cF\stens{!,\mk}(\cG\otimes f^*\omega_{\mk'/\mk}).
\end{equation*} Since $\omega_{\mk'/\mk}$ is invertible in $D(\mk'),$ we conclude that the localizing subcategories $\cC(X/\mk)$ and $\cC(X/\mk')$ of $\Ind(\QCoh(X)^{\omega_1,op})$ are generated by the same collection of objects.\end{proof}

The following statement on relatively compact maps is proved in the same way as the corresponding assertion on trace-class maps.

\begin{lemma}\label{lem:diagonal_arrow_for_double_dual}
Let $X/\mk$ be as in Definition \ref{def:relative_derived_category}. For brevity put $\bbD_X(\cF)=\un{\Hom}_{\cO_X}(\cF,\omega_{X/\mk})$ for $\cF\in\QCoh(X).$ Then for any morphism $\varphi:\cF\to\cG$ in $\QCoh(X),$ relatively compact over $\mk,$ there exists a morphism $\bbD_X(\bbD_X(\cF))\to\cG,$ making the following diagram commutative:
\begin{equation}\label{eq:diagonal_arrow_for_double_dual}
\begin{tikzcd}
\cF\ar[r, "\varphi"]\ar[d] & [2em] \cG\ar[d]\\
\bbD_X(\bbD_X(\cF))\ar[ru] \ar[r, "\bbD_X(\bbD_X(\varphi))"] & \bbD_X(\bbD_X(\cG))
\end{tikzcd}
\end{equation}
\end{lemma}

\begin{proof}
Choose a morphism $\tilde{\varphi}:\cO_X\to \cG\stens{!,\mk}\bbD_X(\cF),$ witnessing the relative compactness of $\varphi$ over $\mk.$ Then we define the diagonal arrow in \eqref{eq:diagonal_arrow_for_double_dual} to be the morphism which by adjunction corresponds to the following composition:
\begin{equation*}
\cO_X\xto{\tilde{\varphi}}\cG\stens{!,\mk} \bbD_X(\cF)\to \cG\stens{!,\mk} \bbD_X(\bbD_X(\bbD_X(\cF)))\to\un{\Hom}_{\cO_X}(\bbD_X(\bbD_X(\cF)),\cG).
\end{equation*}
It is straightforward to see that we get the required diagram.
\end{proof}



\begin{proof}[Proof of Theorem \ref{th:main_properties_of_relative_derived_categories}]
We prove \ref{equivalence_with_Hom^dual}. The dualizable category $\QCoh(X)$ is proper over $\mk$ and has a single compact generator, in particular, it is $\omega_1$-compact in $\Cat_{\mk}^{\dual}$. Hence, we can apply the results of \cite[Section 3.2]{E25}. We use the identifications 
\begin{equation*}
	\QCoh(X)\xto{\sim}\Fun^L_{\mk}(\QCoh(X),D(\mk)),\quad \cF\mapsto\Gamma(X,-\otimes\cF),
\end{equation*} 
and $\QCoh(X\times_{\mk}X)\simeq\QCoh(X)^{\vee}\otimes_{\mk}\QCoh(X).$ The coevaluation (relative to $\mk$) sends $\mk$ to $\Delta_{X,*}\cO_X.$ 
By \cite[Proposition 3.12]{E25} we have a strongly continuous fully faithful functor 
\begin{equation}\label{eq:embedding_of_Hom^dual_vee}
	\un{\Hom}_{\mk}^{\dual}(\QCoh(X),D(\mk))^{\vee}\hto \Ind(\QCoh(X)^{\omega_1,op})\simeq \Fun_{\mk}(\QCoh(X)^{\omega_1},D(\mk)),
\end{equation} 
whose essential image is generated by the objects of the form 
\begin{equation*}
\Hom(\Delta_{X,*}\cO_X,-\boxtimes\cF)\cong \Gamma(X,-\stens{!,\mk}\cF),\quad \cF\in\QCoh(X)^{\omega_1}.
\end{equation*}
This proves that $\cC(X/\mk)$ coincides with the essential image of \eqref{eq:embedding_of_Hom^dual_vee}, as required. 

Next, we prove \ref{restricting_to_open}. We identify $\Ind(\QCoh(X)^{\omega_1,op})$ with $\Fun_{\mk}(\QCoh(X)^{\omega_1},D(\mk)),$ and similarly for $U.$ Then the functor $\Ind((j^*)^{op})$ is simply the precomposition with $j_*:\QCoh(U)^{\omega_1}\to\QCoh(X)^{\omega_1}.$ The right adjoint $\Ind((j^*)^{op})^R$ is given by the precomposition with $j^*:\QCoh(X)^{\omega_1}\to\QCoh(U)^{\omega_1}.$ Now, for $\cF\in\QCoh(X)^{\omega_1}$ and $\cG\in\QCoh(U)^{\omega_1}$ we have
\begin{equation*}
\Gamma(X,j_*(-)\stens{!,\mk}\cF)\cong \Gamma(U,-\stens{!,\mk}j^*\cF),\quad \Gamma(U,j^*(-)\stens{!,\mk}\cG)\cong \Gamma(X,-\stens{!,\mk}j_*\cG).
\end{equation*} 
Hence, the functor $\Ind((j^*)^{op})$ takes $\cC(X/\mk)$ to $\cC(U/\mk),$ and the functor $\Ind((j^*)^{op})^R$ takes $\cC(X/\mk)$ to $\cC(U/\mk).$ Since the functor $\Ind((j^*)^{op})^R$ is continuous and fully faithful, we conclude that the induced functor $j^*:\cC(X/\mk)\to\cC(U/\mk)$ has a continuous right adjoint, as required.

Next, we prove part of \ref{Zariski_descent}, namely Zariski descent in $\Pr^L_{\mk}.$ Clearly, $\cC(\emptyset/\mk)=0.$ Let $U,V\subset X$ be open subset. Consider the composition
\begin{multline*}
\cC(U\cup V/\mk)\xto{\Phi} \cC(U/\mk)\times_{\cC(U\cap V/\mk)}\cC(V/\mk)\\
\xto{\Psi} \Ind(\QCoh(U)^{\omega_1,op})\times_{\Ind(\QCoh(U\cap V)^{\omega_1,op})} \Ind(\QCoh(V)^{\omega_1,op})\simeq \Ind(\QCoh(U\cup V)^{\omega_1,op}). 
\end{multline*}
Here the functors $\Psi$ and $\Psi\circ\Phi$ are fully faithful, hence so is $\Phi.$ It remains to prove that $\Phi$ is essentially surjective. We identify the target of $\Phi$ with its essential image in $\Ind(\QCoh(U\cup V)^{\omega_1,op}).$ Now, both functors $\cC(U/\mk)\to \cC(U\cap V/\mk)$ and $\cC(V/\mk)\to \cC(U\cap V/\mk)$ are quotient functors by \ref{restricting_to_open}. Hence, the fiber product is generated by the images of the fully faithful functors
\begin{equation*}
\pi_1^R:\cC(U/\mk)\to \cC(U/\mk)\times_{\cC(U\cap V/\mk)}\cC(V/\mk),\quad \pi_2^R:\cC(V/\mk)\to \cC(U/\mk)\times_{\cC(U\cap V/\mk)}\cC(V/\mk),
\end{equation*}
which are right adjoints to the projection functors. Denote by $j_1:U\to U\cup V,$ $j_2:V\to U\cup V$ the inclusion morphisms. By the proof of \ref{restricting_to_open}, for $\cF\in\QCoh(U)^{\omega_1},$ $\cG\in\QCoh(V)^{\omega_1}$ we have
\begin{equation*}
\Phi(\Gamma(U\cup V,-\stens{!,\mk}j_{1*}\cF))\cong \pi_1^R(\Gamma(U,-\stens{!,\mk}\cF)),\quad \Phi(\Gamma(U\cup V,-\stens{!,\mk}j_{2*}\cG))\cong \pi_2^R(\Gamma(V,-\stens{!,\mk}\cF)).
\end{equation*}
This proves the essential surjectivity of $\Phi.$

Now we prove \ref{dualizability}. First suppose that $X$ is affine. Choose a surjection of $\mk$-algebras $\mk'=\mk[x_1,\dots,x_n]\to\cO(X).$ Then $X$ is proper of finite $\Tor$-dimension over $\mk',$ hence by \ref{equivalence_with_Hom^dual} the category $\cC(X/\mk')$ is dualizable. By Lemma \ref{lem:changing_base_ring}, we have $\cC(X/\mk)\simeq \cC(X/\mk'),$ hence the category $\cC(X/\mk)$ is dualizable.

Now, suppose that $U,V\subset X$ are open subsets such that both $\cC(U/\mk)$ and $\cC(V/\mk)$ are dualizable. Since the functor $\cC(U/\mk)\to\cC(U\cap V/\mk)$ is a strongly continuous quotient functor by \ref{restricting_to_open}, it follows that the category $\cC(U\cap V/\mk)$ is dualizable. Applying \cite[Proposition 1.87]{E24}, we deduce that the fiber product $\cC(U/\mk)\times_{\cC(U\cap V/\mk)}\cC(V/\mk)$ is dualizable. By the above Zariski descent in $\Pr^L_{\mk},$ the above fiber product is equivalent to $\cC(X/\mk),$ which is therefore dualizable.

To prove the dualizability of $\cC(X/\mk)$ in general it remains to apply induction by the smallest number of open subsets in an affine cover of $X.$

Now the remaining assertion of \ref{Zariski_descent} about descent in $\Cat_{\mk}^{\dual}$ follows from descent in $\Pr^L_{\mk},$ together with \ref{restricting_to_open} and \cite[Proposition 1.87]{E24}.

We prove \ref{description_of_right_adjoint}. Denote by $\incl_X$ the inclusion functor of $\cC(X/\mk)$  into $\Ind(\QCoh(X)^{\omega_1,op}).$ Arguing as in the proof of \ref{dualizability} and using \ref{Zariski_descent}, we see that $\incl_X$ is a strongly continuous functor. Indeed, the question reduces to the case when $X\to\Spec \mk$ is a closed embedding, and in this case the assertion follows from \ref{equivalence_with_Hom^dual} and its proof.

It remains to compute the right adjoint $\incl^R.$ For $\cF\in\QCoh(X)^{\omega_1},$ we have a natural morphism in $\Ind(\QCoh(X)^{\omega_1,op})\simeq\Fun_{\mk}(\QCoh(X)^{\omega_1},D(\mk)):$
\begin{equation*}
\Gamma(X,-\stens{!,\mk}\un{\Hom}_{\cO_X}(\cF,\omega_{X/\mk}))\to \Gamma(X,\un{\Hom}_{\cO_X}(\cF,-))\cong \cY(\cF^{op}),
\end{equation*}
which gives a morphism in $\cC(X/\mk):$
\begin{equation}\label{eq:morphism_to_incl^R}
\Gamma(X,-\stens{!,\mk}\un{\Hom}_{\cO_X}(\cF,\omega_{X/\mk}))\to \incl^R(\cY(\cF^{op})).
\end{equation}
We need to prove that \eqref{eq:morphism_to_incl^R} is an isomorphism.

First suppose that $X$ is affine. As above, choose a surjection $\mk'=\mk[x_1,\dots,x_n]\to\cO(X).$ The functor
\begin{equation*}
\Gamma(X,-\otimes\cF)^{R,\cont}:D(\mk')\to \QCoh(X)
\end{equation*}
sends $\mk'$ to $\un{\Hom}_{\cO_X}(\cF,\omega_{X/\mk'}).$ Hence, by \cite[Corollary 3.11 and its proof]{E25},
the right adjoint to the functor
\begin{equation*}
\cC(X/\mk)\simeq\cC(X/\mk')\simeq\un{\Hom}_{\mk'}^{\dual}(\QCoh(X),D(\mk'))^{\vee}\to \Ind(\QCoh(X)^{\omega_1,op})
\end{equation*}
sends $\cY(\cF^{op})$ to 
\begin{equation*}
\Gamma(X,-\stens{!,\mk'}\un{\Hom}_{\cO_X}(\cF,\omega_{X/\mk'}))\simeq \Gamma(X,-\stens{!,\mk}\un{\Hom}_{\cO_X}(\cF,\omega_{X/\mk})).
\end{equation*}
This proves that \eqref{eq:morphism_to_incl^R} is an isomorphism when $X$ is affine.

Now suppose that $U,V\subset X$ are open subsets such that $X=U\cup V$ and the functors $\incl_U^R,$ $\incl_V^R$ and $\incl_{U\cap V}^R$ have the required descriptions. We will apply Zariski descent both for $\cC(-/\mk)$ and $\Ind(\QCoh(-)^{\omega_1,op}).$ Let $P\in\cC(X/\mk)$ be an object, and denote by $P_{\mid U},$ $P_{\mid V},$ $P_{\mid U\cap V}$ its images in $\cC(U/\mk),$ $\cC(V/\mk)$ and $\cC(U\cap V/\mk).$ Denote the inclusion morphisms by $j_U:U\to X,$ $j_V:V\to X,$ $j_{UV}:U\cap V\to X,$ $j_{UV,U}:U\cap V\to U,$ $j_{UV,V}:U\cap V\to V.$.

Put $T=X\setminus V,$ $S=X\setminus U.$ First suppose that $\cF\in\QCoh_T(X)^{\omega_1}.$ Then we have 
\begin{equation}\label{eq:dual_of_a_sheaf_supported_on_T}
\un{\Hom}_{\cO_X}(\cF,\omega_{X/\mk})\cong j_{U,*}\un{\Hom}_{\cO_U}(j_U^*\cF,\omega_{U/\mk}).
\end{equation} We obtain
\begin{multline}\label{eq:first_chain}
\Hom_{\Ind(\QCoh(X)^{\omega_1,op})}(P,\cY(\cF^{op}))\cong \Hom_{\Ind(\QCoh(U)^{\omega_1,op})}(P_{\mid U},\cY((j_U^*\cF)^{op}))\\
\cong \Hom_{\cC(U/\mk)}(P_{\mid U},\Gamma(U,-\stens{!,\mk}\un{\Hom}_{\cO_U}(j_U^*\cF,\omega_{U/\mk})))\\
\cong \Hom_{\cC(U/\mk)}(P_{\mid U},\Gamma(U,-\stens{!,\mk}j_U^*\un{\Hom}_{\cO_X}(\cF,\omega_{X/\mk}))),
\end{multline}
where the last isomorphism uses \eqref{eq:dual_of_a_sheaf_supported_on_T}. By adjunction and using \eqref{eq:dual_of_a_sheaf_supported_on_T} again, we obtain an isomorphism
\begin{multline}\label{eq:second_chain}
\Hom_{\cC(V/\mk)}(P_{\mid V},\Gamma(V,-\stens{!,\mk}j_V^*\un{\Hom}_{\cO_X}(\cF,\omega_{X/\mk})))\\
\cong \Hom_{\cC(V/\mk)}(P_{\mid U\cap V},\Gamma(U\cap V,-\stens{!,\mk}j_{UV}^*\un{\Hom}_{\cO_X}(\cF,\omega_{X/\mk}))).
\end{multline}
Combining \eqref{eq:first_chain} and \eqref{eq:second_chain}, we obtain an isomorphism
\begin{equation*}
\Hom_{\Ind(\QCoh(X)^{\omega_1,op})}(P,\cY(\cF^{op}))\cong \Hom_{\cC(X/\mk)}(P,\Gamma(X,-\stens{!,\mk}\un{\Hom}_{\cO_X}(\cF,\omega_{X/\mk}))),
\end{equation*}
which proves that \eqref{eq:morphism_to_incl^R} is an isomorphism for $\cF\in\QCoh_T(X)^{\omega_1}.$ Similarly, \eqref{eq:morphism_to_incl^R} is an isomorphism for $\cF\in\QCoh_S(X)^{\omega_1}.$ 

It remains to consider the case $\cF=j_{UV,*}\cF'$ for some $\cF'\in\QCoh(U\cap V)^{\omega_1}.$ Then $\un{\Hom}_{\cO_X}(\cF,\omega_{X/\mk})\in\QCoh(X)$ is contained in the essential image of the pushforward from $\QCoh(U\cap V),$ and similarly for $\un{\Hom}_{\cO_U}(j_U^*\cF,\omega_{U/\mk})\in\QCoh(U)$ and $\un{\Hom}_{\cO_V}(j_V^*\cF,\omega_{V/\mk})\in\QCoh(V).$ Arguing as above,
we obtain an isomorphism
\begin{equation*}
\Hom_{\Ind(\QCoh(X)^{\omega_1,op})}(P,\cY(\cF^{op}))\cong \Hom_{\cC(U\cap V/\mk)}(P_{\mid U\cap V},\Gamma(U\cap V,-\stens{~,\mk}\cG)), 
\end{equation*}
where $\cG$ is defined by the pullback square
\begin{equation*}
\begin{tikzcd}
\cG \ar[r]\ar[d] & j_{UV,U}^*\un{Hom}_{\cO_U}(j_U^*\cF,\omega_{U/\mk})\ar[d]\\
j_{UV,V}^*\un{Hom}_{\cO_V}(j_V^*\cF,\omega_{V/\mk})\ar[r] &  \un{\Hom}_{\cO_{U\cap V}}(j_{UV}^*\cF,\omega_{U\cap V/\mk}).
\end{tikzcd}
\end{equation*}
Similarly, we have
\begin{multline*}
\Hom_{\cC(X/\mk)}(P,\Gamma(X,-\stens{!,\mk}\un{\Hom}_{\cO_X}(\cF,\omega_{X/\mk})))\\
\cong \Hom_{\cC(U\cap V/\mk)}(P_{\mid U\cap V},\Gamma(U\cap V,-\stens{!,\mk}j_{UV}^*\un{\Hom}_{\cO_X}(\cF,\omega_{X/\mk}))).
\end{multline*}
It remains to observe that we have a pullback square
\begin{equation*}
	\begin{tikzcd}
		j_{UV}^*\un{\Hom}_{\cO_X}(\cF,\omega_{X/\mk}) \ar[r]\ar[d] & j_{UV,U}^*\un{\Hom}_{\cO_U}(j_U^*\cF,\omega_{U/\mk})\ar[d]\\
		j_{UV,V}^*\un{\Hom}_{\cO_V}(j_V^*\cF,\omega_{V/\mk})\ar[r] &  \un{\Hom}_{\cO_{U\cap V}}(j_{UV}^*\cF,\omega_{U\cap V/\mk}),
	\end{tikzcd}
\end{equation*}
which is obtained by applying $j_{UV}^*\un{\Hom}_{\cO_X}(\cF,-)$ to the pullback square
\begin{equation*}
\begin{tikzcd}
\omega_{X/\mk} \ar[r]\ar[d] & j_{U,*}\omega_{U/\mk}\ar[d]\\
j_{V,*}\omega_{V/\mk} \ar[r] & j_{UV,*}\omega_{U\cap V/\mk}.
\end{tikzcd}
\end{equation*}
Therefore, \eqref{eq:morphism_to_incl^R} is an isomorphism for $\cF=j_{UV,*}\cF'.$ This proves \ref{description_of_right_adjoint}.

Next, \ref{description_via_relatively_compact_maps} follows directly from \ref{description_of_right_adjoint}. Namely, the description of the right adjoint $\incl_X^R$ implies that a map $\cY(\cF^{op})\to\cY(\cG^{op})$ in $\Ind(\QCoh(X)^{\omega_1,op})$ factors through an object of $\cC(X/\mk)$ if and only if the corresponding map $\cG\to\cF$ is relatively compact over $\mk.$

We prove \ref{compact_objects_of_relative_derived_categories}. We first consider the assignment $\cF\mapsto \Gamma(X,-\stens{!,\mk}\cF)$ as a functor from $\Perf(X,\mk)$ to $\cC(X/\mk).$ We need to prove that it takes values in compact objects, and it gives an equivalence $\Perf(X,\mk)\xto{\sim}\cC(X/\mk)^{\omega}.$
Using Zariski descent, we reduce to the case when $X$ is affine. Arguing as above, we reduce to the case when $X\to\Spec\mk$ is a closed embedding. Then by \ref{equivalence_with_Hom^dual} we simply need to understand the compact objects of $\un{\Hom}_{\mk}^{\dual}(\QCoh(X),D(\mk)).$ These form the category of strongly continuous functors $\QCoh(X)\to D(\mk).$ The category $\QCoh(X)$ is self-dual, and these strongly continuous functors correspond to objects of $\QCoh(X)$ whose pushforward to $D(\mk)$ is contained in $\Perf(\mk).$ In other words, these are relatively perfect complexes on $X$ over $\mk.$ We obtain the equivalences
\begin{multline*}
\cC(X/\mk)^{\omega}\xto{\sim} (\un{\Hom}_{\mk}^{dual}(\QCoh(X),D(\mk))^{\vee})^{\omega}\simeq \un{\Hom}_{\mk}^{\dual}(\QCoh(X),D(\mk))^{\omega,op}\\
\simeq \Perf(X,\mk)^{op}\xto{\sim}\Perf(X,\mk),
\end{multline*}
where the latter equivalence is given by $\un{\Hom}_{\cO_X}(-,\omega_{X/\mk}).$ This composition is inverse to the stated equivalence \eqref{eq:relatively_perfect_are_compact_objects}.

Next, we prove \ref{self_duality}. First we observe that \ref{description_of_right_adjoint} implies that we have a strongly continuous fully faithful functor $(\incl_X^R)^{\vee}:\cC(X/\mk)^{\vee}\to \Ind(\QCoh(X))^{\omega_1}.$ Its essential image consists of objects of the form $\inddlim[i\in I]\cF_i,$ where $I$ is directed and for any $i\in I$ there exists $j\geq i$ such that the map $\cF_i\to\cF_j$ is relatively compact over $\mk.$ Choose some auxiliary full small idempotent-complete stable subcategory $\QCoh(X)^{\sim}\subset\QCoh(X)$
which contains $\QCoh(X)^{\omega_1}$ and such that for any $\cF\in \QCoh(X)^{\sim}$ we have $\un{\Hom}_{\cO_X}(\cF,\omega_{X/\mk})\in\QCoh(X)^{\sim}.$ Then we can consider $\cC(X/\mk)$ resp. $\cC(X/\mk)^{\vee}$ as a full subcategory of $\Ind(\QCoh(X)^{\sim,op})$ resp. $\Ind(\QCoh(X)^{\sim}),$ with the same description via relatively compact maps. It follows from Lemma \ref{lem:diagonal_arrow_for_double_dual} that we have mutually inverse equivalences
\begin{equation*}
\cC(X/\mk)\xto{\sim}\cC(X/\mk)^{\vee},\quad \inddlim[i]\cF_i^{op}\mapsto \inddlim[i]\un{\Hom}_{\cO_X}(\cF_i,\omega_{X/\mk}),
\end{equation*}
\begin{equation*}
	\cC(X/\mk)^{\vee}\xto{\sim}\cC(X/\mk),\quad \inddlim[i]\cF_i^{op}\mapsto \inddlim[i]\un{\Hom}_{\cO_X}(\cF_i,\omega_{X/\mk}),
\end{equation*}
The associated perfect pairing for $\cC(X/\mk)$ is given by \eqref{eq:self_duality}, as stated.

It remains to prove \ref{smooth_or_regular}. First suppose that $X$ is smooth over $\mk.$ Then for $\cF\in\QCoh(X)^{\omega_1}$ we have $\Gamma(X,-\stens{!,\mk}\cF)\cong \Gamma(X,-\otimes\cF\otimes\omega_{X/\mk}^{-1}).$ Hence, the essential image of $\cC(X/\mk)$ in $\Fun_{\mk}(\QCoh(X)^{\omega_1},D(\mk))$ is given by $\Fun_{\mk}^{\omega_1\hy\rex}(\QCoh(X)^{\omega_1},D(\mk))\simeq \QCoh(X),$ as required.

Now suppose that $X$ is regular of finite Krull dimension. Using Zariski descent, we reduce to the case when $X$ is affine. By the above argument we may also assume that $X\to\Spec\mk$ is a closed embedding. Then the assertion follows from \ref{equivalence_with_Hom^dual} and \cite[Corollary 3.21]{E25}.
\end{proof}

\begin{remark}
	The proof of Theorem \ref{th:main_properties_of_relative_derived_categories} simplifies if we assume that there is a closed embedding $X\hto Y$ such that $Y$ is smooth over $\mk$ (for example, if $X$ is quasi-projective over $\mk$). Then arguing as above we obtain for any open $U\subset Y$ a natural equivalence
	\begin{equation*}
		\cC(U\cap X/\mk)\simeq \un{\Hom}_{\QCoh(Y)}^{\dual}(\QCoh(X),\QCoh(U))^{\vee}.
	\end{equation*}
	Since the category $\QCoh(X)$ is proper and $\omega_1$-compact over $\QCoh(Y),$ the non-trivial assertions of Theorem \ref{th:main_properties_of_relative_derived_categories} follow almost directly from the results of \cite[Section 3.2]{E25}.
\end{remark}

We do not develop a $3$-functor formalism for the categories $\cC(X/\mk)$ here, but we mention the situations when these functors are strongly continuous.

\begin{prop}
	Let $\mk$ be a noetherian commutative ring, and let $X$ and $Y$ be separated schemes of finite type over $\mk,$ which are of finite $\Tor$-dimension over $\mk.$ Let $f:X\to Y$ be a morphism.
	\begin{enumerate}[label=(\roman*),ref=(\roman*)]
		\item If $f$ is proper, then the functor $\Ind((f_*)^{op}):\Ind(\QCoh(X)^{\omega_1,op})\to \Ind(\QCoh(Y)^{\omega_1,op})$ takes $\cC(X/\mk)$ to $\cC(Y/\mk).$ In particular, we obtain a well-defined strongly continuous functor $f_*:\cC(X/\mk)\to\cC(Y/\mk).$ \label{proper_pushfowards_defined}
		
		\item If $f$ has finite $\Tor$-dimension, then the functor $\Ind((f^*)^{op}):\Ind(\QCoh(Y)^{\omega_1,op})\to \Ind(\QCoh(X)^{\omega_1,op})$ takes $\cC(Y/\mk)$ to $\cC(X/\mk).$ We obtain the strongly continuous extraordinary inverse image functor $f^!:\cC(Y/\mk)\to\cC(X/\mk).$ We also obtain the strongly continuous inverse image functor, given by the composition
		\begin{equation*}
		f^*:\cC(Y/\mk)\xto{\sim}\cC(Y/\mk)^{\vee}\xto{((f^!)^R)^{\vee}}\cC(X/\mk)^{\vee}\xto{\sim}\cC(X/\mk)
		\end{equation*} \label{finite_Tor_diemnsion_pullbacks_defined}
	\end{enumerate}
\end{prop}

\begin{proof}
	\ref{proper_pushfowards_defined} We only need to show that for a morphism $\varphi:\cF\to\cG$ in $\QCoh(X),$ relatively compact over $\mk,$ the morphism $f_*\varphi: f_*\cF\to f_*\cG$ is also relatively compact over $\mk.$ Choose a morphism $\tilde{\varphi}:\Delta_{X,*}\cO_X\to \cG\boxtimes \un{\Hom}_{\cO_X}(\cF,\omega_{X/\mk}),$ witnessing the relative compactness of $\varphi.$ Denote by $(f,f):X\times_{\mk} X\to Y\times_{\mk} Y$ the product of two copies of $f,$ and consider the composition
	\begin{multline*}
		\psi:\Delta_{Y,*}\cO_Y\to \Delta_{Y,*}(f_*\cO_X)\cong (f,f)_*(\Delta_{X,*}\cO_X)\xto{(f,f)_*\tilde{\varphi}} (f,f)_*(\cG\boxtimes \un{\Hom}_{\cO_X}(\cF,\omega_{X/\mk}))\\
		\cong f_* \cG\boxtimes \un{\Hom}_{\cO_Y}(f_*\cF,\omega_{Y/\mk}).
	\end{multline*}
	Then $\psi$ witnesses the relative compactness of $f_*\tilde{\varphi},$ as required. Note that we used properness of $f$ to obtain an isomorphism $f_*\un{\Hom}_{\cO_X}(\cF,\omega_{X/\mk})\cong\un{\Hom}_{\cO_Y}(f_*\cF,\omega_{Y/\mk}).$
	
	\ref{finite_Tor_diemnsion_pullbacks_defined} Similarly, we only need to show that for a morphism $\varphi:\cF\to\cG$ in $\QCoh(Y),$ relatively compact over $\mk,$ the morphism $f^*\varphi: f^*\cF\to f^*\cG$ is also relatively compact over $\mk.$ Again, we choose a morphism $\tilde{\varphi}:\Delta_{Y,*}\cO_Y\to \cG\boxtimes \un{\Hom}_{\cO_Y}(\cF,\omega_{Y/\mk}),$ witnessing the relative compactness of $\varphi.$ Since $f$ has finite $\Tor$-dimension, the functor $f^!:\QCoh(Y)\to \QCoh(X)$ is continuous and $\QCoh(Y)$-linear, so we have $f^!(-)\cong f^*(-)\otimes \omega_{X/Y}.$. Consider the functor
	\begin{equation*}
		f^*\boxtimes f^!:\QCoh(Y\times_{\mk} Y)\to\QCoh(X\times_{\mk} X).
	\end{equation*}
	We have a natural map $\Delta_{X,*}\cO_X\to (f^*\boxtimes f^!)(\Delta_{Y,*}\cO_Y),$ which by adjunction corresponds to
	\begin{equation*}
		\cO_X\to (\id,f)^!((\id,f)_*\cO_X)\cong \Delta_X^!((f^*\boxtimes f^!)(\Delta_{Y,*}\cO_Y)),\quad (\id,f):X\to X\times_{\mk} Y.
	\end{equation*}
	Consider the composition
	\begin{multline*}
		\psi:\Delta_{X,*}\cO_X\to (f^*\boxtimes f^!)(\Delta_{Y,*}\cO_Y)\xto{(f^*\boxtimes f^!)(\tilde{\varphi})}(f^*\boxtimes f^!)(\cG\boxtimes \un{\Hom}_{\cO_Y}(\cF,\omega_{Y/\mk}))\\
		\to f^*\cG\boxtimes \un{\Hom}_{\cO_X}(f^*\cF,\omega_{X/\mk}).
	\end{multline*}
	Here we used the natural morphism $f^!\un{\Hom}_{\cO_Y}(\cF,\omega_{Y/\mk})\to \un{\Hom}_{\cO_X}(f^*\cF,\omega_{X/\mk}),$ which corresponds by adjunction to the composition
	\begin{equation*}
		f^!\un{\Hom}_{\cO_Y}(\cF,\omega_{Y/\mk})\otimes f^*\cF\cong f^!(\un{\Hom}_{\cO_Y}(\cF,\omega_{Y/\mk})\otimes\cF)\to f^!\omega_{Y/\mk}\cong \omega_{X/\mk}.
	\end{equation*}
	Then the above map $\psi$ witnesses the relative compactness of $f^*\varphi,$ as required.
\end{proof}

\subsection{$K$-homology of proper schemes and proper connective dg algebras}
\label{ssec:K_homology_proper}

In this subsection we prove two closely related results on the $K$-homology of proper schemes and proper connective (associative) dg algebras. We use the notation $G(X)=K(D^b_{\coh}(X))$ for a noetherian scheme $X.$ If $A$ is a right noetherian (discrete) ring, then we put $G(A)=K(D^b_{\coh}(A)).$ We recall that the relative derived categories $\cC(X/\mk)$ are introduced in Definition \ref{def:relative_derived_category} and their main properties are formulated in Theorem \ref{th:main_properties_of_relative_derived_categories}. Recall from loc. cit. that if $X$ is a separated scheme of finite type over a regular noetherian ring $\mk,$ then we have an equivalence $\cC(X/\mk)^{\omega}\simeq D^b_{\coh}(X).$ 

\begin{theo}\label{th:K_homology_proper}
Let $\mk$ be a regular noetherian commutative ring, and let $X$ be a separated scheme of finite type over $\mk.$
\begin{enumerate}[label=(\roman*),ref=(\roman*)]
\item The strongly continuous fully faithful functor $\IndCoh(X)\hto \cC(X/\mk)$ induces an isomorphism on $K^{\cont},$ i.e. we have
$G(X)\xto{\sim} K^{\cont}(\cC(X/\mk)).$ \label{K_theory_of_relative_derived_category}
\item If $X$ is proper over $\mk,$ then we have 
\begin{equation}\label{eq:K_homology_proper_scheme}
	\Hom_{\Mot^{\loc}_{\mk}}(\cU_{\loc}(X),\cU_{\loc}(\mk))\cong G(X).
\end{equation} \label{K_homology_for_proper}
\end{enumerate}
\end{theo}

We recall from \cite[Proposition 3.25]{E25} that even if $\mk$ is a field and $X=\Spec\mk[\veps]$ (where $\mk[\veps]$ is the algebra of dual numbers), then the category $\cC(X/\mk)$ is not compactly generated, i.e. the functor $\IndCoh(X)\to\cC(X/\mk)$ is not an equivalence.

\begin{theo}\label{th:K_homology_proper_connective_algebras}
Let $\mk$ be a regular noetherian ring commutative ring, and let $A$ be a proper connective dg algebra ($\bE_1$-algebra) over $\mk.$ Then we have the isomorphisms
\begin{equation}\label{eq:K_homology_proper_connective}
G(H_0(A))\xto{\sim}  K^{\cont}(\un{\Hom}_{\mk}^{\dual}(\Mod\hy A,D(\mk)))\xto{\sim} \Hom_{\Mot^{\loc}_{\mk}}(\cU_{\loc}(A),\cU_{\loc}(\mk)).
\end{equation}
\end{theo}

We will deduce Theorem \ref{th:K_homology_proper} from Theorem \ref{th:K_homology_proper_connective_algebras} using the descent properties of the categories $\cC(X/\mk)$ established in Theorem \ref{th:main_properties_of_relative_derived_categories}.

\begin{proof}[Proof of Theorem \ref{th:K_homology_proper_connective_algebras}] We first observe that the $\mk$-module $H_0(A)$ is finitely generated, hence the ring $H_0(A)$ is right (and left) noetherian. Thus, its $G$-theory is well defined.
		
The second isomorphism in \eqref{eq:K_homology_proper_connective} follows from Theorem \ref{th:morphisms_in_Mot^loc_via_internal_Hom}. It remains to prove that the composition in \eqref{eq:K_homology_proper_connective} is an isomorphism.
	
Now choose a sequence of $\bE_1$-$\mk$-algebras $A_0\to A_1\to\dots$ such that $\indlim[n]A_n\cong A$ and the following conditions hold:
\begin{itemize}
 \item each $A_n$ is finitely presented, i.e. $A_n\in (\Alg_{\bE_1}(\Mod_{\mk}))^{\omega}$ for $n\geq 0;$
 \item each $A_n$ is connective;
 \item We have $H_0(A_n)\xto{\sim}H_0(A)$ for $n\geq 0.$
\end{itemize}
For example, if $(B_n)_{n\geq 0}$ is the canonical approximation of $A$ constructed in the proof of Proposition \ref{prop:proper_are_colimits_of_finitely_presented} (originally in \cite[Proof of Proposition 5.18]{E25}), then the sequence $(A_n=B_{n+2})_{n\geq 0}$ satisfies the above conditions.
	
It follows from Proposition \ref{prop:stronger_nuclearity_for_proper} that we have a pro-equivalence in $\Pro(\Cat^{\perf}):$
\begin{equation*}
\proolim[n]\Perf(\End_{A_n^{op}}(A))\xto{\sim}\proolim[n]\Rep_{\mk}(A_n,\Perf(\mk)).
\end{equation*}
Hence, by Theorem \ref{th:morphisms_in_Mot^loc_via_limits} \ref{Hom_via_inverse_limit_for_proper} we have an isomorphism
\begin{equation}\label{eq:Hom_from_U_loc_of_connective_as_limit}
\prolim[n] K(\Rep_{\mk}(A_n,\Perf(\mk)))\xto{\sim} \Hom_{\Mot^{\loc}_{\mk}}(\cU_{\loc}(A),\cU_{\loc}(\mk)).
\end{equation}

The above assumptions imply that each category $\Rep_{\mk}(A_n,\Perf(\mk))$ has a bounded $t$-structure whose heart is identified with the category of finitely generated right $H_0(A)$-modules, which is noetherian. Moreover, for $n\geq 0$ the transition functor $\Rep_{\mk}(A_{n+1},\Perf(\mk))\to \Rep_{\mk}(A_n,\Perf(\mk))$ is $t$-exact and induces an equivalence on the hearts, compatible with the above identifications. Using Barwick's theorem of the heart \cite[Theorem 6.1]{Bar15} and Antieau-Gepner-Heller vanishing of negative $K$-groups \cite[Theorem 3.6]{AGH19} (which applies since the hearts are noetherian) we obtain the isomorphisms $G(H_0(A))\xto{\sim} K(\Rep(A_n,\Perf(\mk)))$ for $n\geq 0.$ Therefore, we have
	\begin{equation*}
		G(X)\xto{\sim}\prolim[n]K(\Rep_{\mk}(A_n,\Perf(\mk))).
	\end{equation*}
	Together with \eqref{eq:Hom_from_U_loc_of_connective_as_limit} this proves that the composition in \eqref{eq:K_homology_proper_connective} is an isomorphism.
\end{proof}

Before proving Theorem \ref{th:K_homology_proper} we first observe an immediate consequence of Theorem \ref{th:main_properties_of_relative_derived_categories}.

\begin{prop}\label{prop:Zariski_descent_K_theory_relative_derived}
Let $X/\mk$ be as in Theorem \ref{th:main_properties_of_relative_derived_categories}, i.e $\mk$ is a commutative noetherian ring and $X$ is a separated scheme of finite type and of finite $\Tor$-dimension over $\mk.$  Then the assignment $U\mapsto K^{\cont}(\cC(U/\mk))$ is a sheaf on $X$ with values in $\Sp.$
\end{prop}

\begin{proof}
Clearly, we have $K^{\cont}(\cC(\emptyset/\mk))=0.$ Let $U,V\subset X$ be open subsets. By Theorem \ref{th:main_properties_of_relative_derived_categories} \ref{restricting_to_open} the restriction $\cC(U/\mk)\to\cC(U\cap V/\mk)$ is a quotient functor. It follows from Theorem \ref{th:main_properties_of_relative_derived_categories} \ref{Zariski_descent} and \cite[Proposition 4.11]{E24} that we have isomorphisms
\begin{multline*}
K^{\cont}(\cC(U\cup V/\mk))\cong K^{\cont}(\cC(U/\mk)\times_{\cC(U\cap V/\mk)}\cC(V/\mk))\\
\cong  K^{\cont}(\cC(U/\mk))\times_{K^{\cont}(\cC(U\cap V/\mk))} K^{\cont}(\cC(V/\mk)),
\end{multline*}
as stated.
\end{proof}

\begin{proof}[Proof of Theorem \ref{th:K_homology_proper}]
\ref{K_theory_of_relative_derived_category} Using Proposition \ref{prop:Zariski_descent_K_theory_relative_derived} and Zariski descent for $G$-theory, we reduce to the case when $X$ is affine. By Lemma \ref{lem:changing_base_ring} we may and will assume that $X\to\Spec\mk$ is a closed embedding. By Theorem \ref{th:main_properties_of_relative_derived_categories} \ref{equivalence_with_Hom^dual} the category $\cC(X/\mk)$ is equivalent to $\un{\Hom}_{\mk}^{\dual}(\QCoh(X),D(\mk))^{\vee}.$ As a special case of Theorem \ref{th:K_homology_proper_connective_algebras} we obtain
\begin{equation*}
K^{\cont}(\un{\Hom}_{\mk}^{\dual}(\QCoh(X),D(\mk))^{\vee})\cong K^{\cont}(\un{\Hom}_{\mk}^{\dual}(\QCoh(X),D(\mk)))\cong G(X).
\end{equation*}
This proves \ref{K_theory_of_relative_derived_category}.

Next, we deduce \ref{K_homology_for_proper} from \ref{K_theory_of_relative_derived_category}. By assumption $\QCoh(X)\in\Cat_{\mk}^{\dual}$ is proper over $\mk.$ Hence, by Theorem \ref{th:morphisms_in_Mot^loc_via_internal_Hom} \ref{internal_Hom_from_proper} we have
\begin{equation*}
\Hom_{\Mot^{\loc}_{\mk}}(\cU_{\loc}(X),\cU_{\loc}(\mk))\cong K^{\cont}(\un{\Hom}_{\mk}^{\dual}(\QCoh(X),D(\mk))).
\end{equation*}
By Theorem \ref{th:main_properties_of_relative_derived_categories} \ref{equivalence_with_Hom^dual} we have
\begin{equation*}
 K^{\cont}(\un{\Hom}_{\mk}^{\dual}(\QCoh(X),D(\mk)))\cong  K^{\cont}(\un{\Hom}_{\mk}^{\dual}(\QCoh(X),D(\mk))^{\vee})\cong K^{\cont}(\cC(X/\mk)).
\end{equation*}
Finally, by \ref{K_theory_of_relative_derived_category} we have
\begin{equation*}
K^{\cont}(\cC(X/\mk))\cong G(X).
\end{equation*}
This proves \ref{K_homology_for_proper}.
\end{proof}

\subsection{$K$-homology of smooth schemes with a smooth compactification}
\label{ssec:K_homology_smooth_with_compactification}

We prove the following result. We assume the base ring to be noetherian for convenience; this assumption can be removed using a finite presentation argument.

\begin{theo}\label{th:K_homology_smooth_with_compactification}
Let $X$ be a smooth separated scheme over a commutative noetherian ring $\mk.$ Suppose that there exists a smooth compactification $\bbar{X}\supset X$ over $\mk.$ Then we have an isomorphism
\begin{equation}\label{eq:K_homology_smooth_with_compactification}
\Hom_{\Mot^{\loc}_{\mk}}(\cU_{\loc}(X),\cU_{\loc}(\mk))\cong \Fiber(K(X)\to K^{\cont}(\hat{X}_{\infty})).
\end{equation}
\end{theo}

We explain the meaning of the right hand side of \eqref{eq:K_homology_smooth_with_compactification}. The object $\hat{X}_{\infty}$ is the so-called formal punctured neighborhood of infinity, which is an adic space depending only on $X.$ If $\bbar{X}\supset X$ is a not necessarily smooth compactification over $\mk$ and $Z=\bbar{X}\setminus X,$ then $\hat{X}_{\infty}$ is the generic fiber of the formal scheme $\hhat{\bbar{X}}_Z$ -- the formal completion of $\bbar{X}$ along $Z.$ The original version of the category of categories nuclear modules (due to Clausen and Scholze) is defined for general adic spaces \cite{CS20, And21, And23}. We denote it by $\Nuc^{CS}(-)$ keeping the notation from \cite{E25}, since below we will also consider our version of these categories as defined in loc. cit. for formal schemes. The categories $\Nuc^{CS}(\bbar{X}_{\infty})$ and $\Nuc^{CS}(\hhat{\bbar{X}}_Z)$ are dualizable (in fact, rigid symmetric monoidal), and we have a short exact sequence in $\Cat_{\mk}^{\dual}:$
\begin{equation}\label{eq:ses_for_Nuc^CS_of_generic_fiber}
0\to \QCoh_{Z}(\bbar{X})\to \Nuc^{CS}(\hhat{\bbar{X}}_Z)\to \Nuc^{CS}(\hat{X}_{\infty})\to 0.
\end{equation}
We put 
\begin{equation*}
K^{\cont}(\hat{X}_{\infty})=K^{\cont}(\Nuc^{CS}(\hat{X}_{\infty})),\quad K^{\cont}(\hhat{\bbar{X}}_{Z})=K^{\cont}(\Nuc^{CS}(\hhat{\bbar{X}}_{Z})).
\end{equation*}
The map $K(X)\to K^{\cont}(\hat{X}_{\infty})$ is induced by the strongly continuous functor $\QCoh(X\to\Nuc^{CS}(\hat{X}_{\infty}).$ It follows from \eqref{eq:ses_for_Nuc^CS_of_generic_fiber} and the standard short exact sequence
\begin{equation}\label{eq:Thomason_ses}
	0\to \Perf_Z(\bbar{X})\to\Perf(\bbar{X})\to \Perf(X)\to 0
\end{equation}
from \cite{TT90} that we have an isomorphism
\begin{equation}\label{eq:two_fibers}
	\Fiber(K(X)\to K^{\cont}(\hat{X}_{\infty}))\cong \Fiber(K(\bbar{X})\to K^{\cont}(\hhat{\bbar{X}}_Z)).
\end{equation}
By \cite[Theorem 0.1]{E25} and descent results from \cite{And23} we have an isomorphism
\begin{equation*}
K^{\cont}(\hhat{\bbar{X}}_{\infty})\xto{\sim} \prolim[n] K(Z_n),
\end{equation*}
where $Z_n$ is the $n$-th infinitesimal neighborhood of $Z$ in $\bbar{X}.$ In other words, the above definition of continuous $K$-theory for $\hhat{\bbar{X}}_Z$ is equivalent to the classical one.

\begin{proof}[Proof of Theorem \ref{th:K_homology_smooth_with_compactification}]
We fix the smooth compactification $\bbar{X}$ which exists by assumption. As in the above discussion, we put $Z=\bbar{X}\setminus X$ and denote by $\hhat{\bbar{X}}_Z$ the formal completion.

It follows from the short exact sequence \eqref{eq:Thomason_ses} and from the smoothness and properness of $\Perf(\bbar{X})$ over $\mk$ that we have an isomorphism
\begin{equation}\label{eq:K_homology_smooth_intermediate}
\Hom_{\Mot^{\loc}_{\mk}}(\cU_{\loc}(X),\cU_{\loc}(\mk))\cong \Fiber(K(\bbar{X})\to \Hom_{\Mot^{\loc}_{\mk}}(\cU_{\loc}(\Perf_Z(\bbar{X})),\cU_{\loc}(\mk))).
\end{equation}
Now one way to argue would be to prove the equivalence $\un{\Hom}^{\dual}_{\mk}(\QCoh_Z(\bbar{X}),D(\mk))\cong \Nuc(\hhat{\bbar{X}}_Z)$ (here we mean our version of the category $\Nuc$ from \cite{E25}), and then apply Theorem \ref{th:morphisms_in_Mot^loc_via_internal_Hom} \ref{internal_Hom_from_proper} and \cite[Theorem 0.2]{E25}. Instead we use a more elementary argument, conceptually similar to the proof of Theorem \ref{th:corepresentability_of_TR}.

As above, for $n\geq 1$ we denote by $Z_n$ the $n$-th infinitesimal neighborhood of $Z$ in $\bbar{X}.$ We claim that for any $n\geq 1$ there exists $k\geq n$ such that the pullback functor $\Perf(Z_k)\to\Perf(Z_n)$ is trace-class in the symmetric monoidal category $\Cat_{\mk}^{\perf}.$ To see this, denote by $Y_n\subset Z_n\times_{\mk}\bbar{X}$ the graph of the inclusion $Z_n\hto\bbar{X}.$ Since $\bbar{X}$ is smooth over $\mk,$ we have $\cO_{Y_n}\in\Perf(Z_n\times_{\mk} \bbar{X})\simeq \Perf(Z_n)\tens{\mk}\Perf(\bbar{X}).$ Since the restriction of $\cO_{Y_n}$ to $Z_n\times_{\mk} X$ vanishes, we have $\cO_{Y_n}\in\Perf(Z_n)\tens{\mk}\Perf_Z(\bbar{X}).$ Now, again by the smoothness of $\bbar{X}$ we have an equivalence $\Perf_Z(X)\simeq\indlim[k]\Perf(Z_k,\mk).$ Here we follow the notation from Subsection \ref{ssec:relative_derived}, denoting by $\Perf(Z_k,\mk)$ the category of relatively perfect complexes. It follows that there exists some $k\geq n$ and an object $\cF\in\Perf(Z_n)\tens{\mk}\Perf(Z_k,\mk),$  whose image in $\Perf(Z_n)\otimes \Perf_Z(\bbar{X})$ is isomorphic to $\cO_{Y_n}.$ By construction, the image of $\cF$ in $\QCoh(Z_n\times_{\mk}Z_k)$ is isomorphic to $\cO_{Y_n}.$ By properness of $Z_k$ over $\mk$ we have $\Perf(Z_k,\mk)\simeq\Fun_{\mk}(\Perf(Z_k),\Perf(\mk)).$ Hence, $\cF$ is a trace-class witness for the pullback functor $\Perf(Z_k)\to\Perf(Z_n),$ as required.

It follows that for $n\leq k$ as above the pushforward functor $\Perf(Z_n,\mk)\to \Perf(Z_k,\mk)$ is also trace-class in $\Cat_{\mk}^{\perf}.$ We also obtain an equivalence in $\Pro(\Mot^{\loc}_{\mk}):$
\begin{equation*}
\proolim[n]\cU_{\loc}(\Perf(Z_n,\mk))^{\vee}\cong \proolim[n]\cU_{\loc}(Z_n).
\end{equation*}
We conclude that we have an isomorphism
\begin{multline*}
\Hom_{\Mot^{\loc}_{\mk}}(\cU_{\loc}(\Perf_Z(\bbar{X})),\cU_{\loc}(\mk))\cong\prolim[n]\Hom_{\Mot^{\loc}(\mk)}(\cU_{\loc}(\mk),\cU_{\loc}(Z_n))\\
\cong \prolim[n] K(Z_n)\cong K^{\cont}(\hhat{\bbar{X}}_Z).
\end{multline*}
Combining this with \eqref{eq:K_homology_smooth_intermediate} and \eqref{eq:two_fibers}, we obtain the stated isomorphism \eqref{eq:K_homology_smooth_with_compactification}.
\end{proof}

\subsection{$K$-homology of general smooth schemes}
\label{ssec:K_homology_smooth_general}

In \cite{E} we will prove the following general version of Theorem \ref{th:K_homology_smooth_with_compactification} removing the assumption on the existence of a smooth compactification.

\begin{theo}\label{th:K_homology_of_smooth_general}
Let $X$ be a smooth separated scheme over a noetherian ring $\mk.$ Then we have
\begin{equation}\label{eq:K_homology_of_smooth_general}
	\Hom_{\Mot^{\loc}_{\mk}}(\cU_{\loc}(X),\cU_{\loc}(\mk))\cong \Fiber(K(X)\to K^{\cont}(\hat{X}_{\infty})).
\end{equation}
\end{theo}

Here we will briefly sketch the idea, and the details are quite complicated. We will prove in \cite{E} that our version of the category of nuclear modules $\Nuc(\hat{X}_{\infty})$ is well defined. More precisely, this category is rigid symmetric monoidal, it depends only on $X,$ and for any (not necessarily smooth) compactification $\bbar{X}\supset X$ with $Z=\bbar{X}\setminus X$ we have a short exact sequence
\begin{equation*}
0\to \QCoh_Z(\bbar{X})\to\Nuc(\hhat{\bbar{X}}_Z)\to\Nuc(\hat{X}_{\infty})\to 0.
\end{equation*}
Here $\Nuc(\hhat{\bbar{X}}_Z)$ is defined as in \cite{E25}, i.e. it is the rigidification of the locally rigid category $\QCoh_Z(\bbar{X}).$

The intrinsic description of $\Nuc(\hat{X}_{\infty})$ (without choosing a compactification) is the following. Consider the natural (symmetric monoidal) functor 
\begin{equation*}
\Phi:\QCoh(X)\tens{\mk}\Ind(D(\mk)^{\omega_1})\to\Ind(\QCoh(X)^{\omega_1}).
\end{equation*}
Then $\Phi$ has a fully faithful strongly continuous right adjoint (it is in fact strongly continuous), which satisfies the projection formula. Hence, $\ker(\Phi)$ is naturally a compactly generated symmetric monoidal category, more precisely an object of $\CAlg(\Pr^L_{\st,\omega}).$ We have $\Nuc(\hat{X}_{\infty})\simeq \ker(\Phi)^{\rig}.$

To compute the $K$-homology of $X$ we use a different, ``noncommutative'' description of the category $\Nuc(\hat{X}_{\infty}).$ Namely, we will prove in \cite{E} that there is a natural short exact sequence in $\Cat_{\mk}^{\dual}:$
\begin{equation*}
0\to \Nuc(\hat{X}_{\infty})\to\un{\Hom}_{\mk}^{\dual}(\QCoh(X),\Ind(\Calk_{\omega_1}(\mk)))\to \Ind(\Calk_{\omega_1}(X))\to 0,
\end{equation*}
where $\Calk_{\omega_1}(\mk)=\Calk_{\omega_1}(\Perf(\mk)),$ $\Calk_{\omega_1}(X)=\Calk_{\omega_1}(\Perf(X)).$
Taking $K^{\cont}(-)$ and using Theorem \ref{th:morphisms_in_Mot^loc_via_internal_Hom} \ref{internal_Hom_into_Calk} and \cite[Theorem 0.2]{E25} we obtain
a cofiber sequence
\begin{equation*}
K^{\cont}(\hat{X}_{\infty})\to \Sigma\Hom_{\Mot^{\loc}_{\mk}}(\cU_{\loc}(X),\cU_{\loc}(\mk))\to \Sigma K(X).
\end{equation*} 
This gives the stated isomorphism \eqref{eq:K_homology_of_smooth_general}. 

\section{Continuous $K$-theory of categories of completed (co)sheaves}
\label{sec:completed_co_sheaves}

Recall that by \cite[Theorem 6.11]{E24} for a locally compact Hausdorff space $X$ and a presheaf $\un{\cC}$ on $X$ with values in $\Cat_{\st}^{\dual}$
we have an equivalence
\begin{equation*}
	K^{\cont}(\Shv(X;\un{\cC}))\cong \Gamma_c(X,K^{\cont}(\un{\cC})^{\sharp}).
\end{equation*}
Here $(-)^{\sharp}$ stands for sheafification. Informally speaking, this result says that the categories of sheaves on locally compact Hausdorff spaces categorify the cohomology with compact support. In this section we will apply Theorem \ref{th:morphisms_in_Mot^loc_via_internal_Hom} \ref{internal_Hom_from_proper} to obtain similar results for cohomology and for Borel-Moore homology, under some assumptions.

\subsection{Completed cosheaves} 
\label{ssec:completed_cosheaves}

For a locally compact Hausdorff space $X$ and a dualizable category $\cC$ we consider the category of $\cC$-valued completed cosheaves
\begin{equation*}
	\hhat{\Cosh}(X;\cC)=\un{\Hom}^{\dual}_{\Sp}(\Shv(X;\Sp),\cC),
\end{equation*}
studied in \cite{KNP}.

In this generality (without any further assumptions on $X$) the category $\Shv(X;\Sp)$ does not have to be internally projective in $\Cat_{\st}^{\dual},$ and the localizing invariants of the categories $\hhat{\Cosh}(X;C)$ probably don't have any reasonable description (unless $\cC$ is formally $\omega_1$-injective in the sense of Definition \ref{def:formal_omega_1_injectivity} and $X$ is second-countable). It is natural to impose conditions on $X$ which ensure that the category $\Shv(X;\Sp)$ is proper (over $\Sp$) and $\omega_1$-compact, so that the internal projectivity holds by \cite[Theorem 3.6]{E25}. Such conditions are formulated in \cite[Proposition 3.27]{E25}, and we use them to obtain the following result.

\begin{theo}
	Let $X$ be a locally compact Hausdorff space which is locally of constant shape and second-countable (e.g. a topological manifold which is countable at infinity). Let $p:X\to\pt$ be the projection, and consider the functors
	\begin{equation*}
		p_*;\Shv(X;\Sp)\to \Sp,\quad p^!:\Sp\to\Shv(X;\Sp).
	\end{equation*}
	Then we have a natural isomorphism
	\begin{equation*}
		K^{\cont}(\hhat{\Cosh}(X;\cC))\cong p_* p^! K^{\cont}(\cC).
	\end{equation*}
\end{theo}

Recall from \cite[Definition A.1.5, Proposition A.1.8]{Lur17} that $X$ is said to be locally of constant shape if the functor $p^*:\cS\to\Shv(X;\cS)$ has a left adjoint.

\begin{proof}
	As explained in \cite[Proof of Proposition 3.27]{E25}, our assumptions on $X$ imply that the dualizable category $\Shv(X;\Sp)$ is proper and $\omega_1$-compact. By Theorem \ref{th:morphisms_in_Mot^loc_via_internal_Hom} \ref{internal_Hom_from_proper} we have
	\begin{equation*}
		K^{\cont}(\hhat{\Cosh}(X;\cC))\xto{\sim} \Hom_{\Mot^{\loc}}(\cU_{\loc}(\Shv(X;\Sp)),\cU_{\loc}(\cC)).
	\end{equation*}
	We denote by $p^*,p_*,p_!,p^!$ the corresponding functors between the categories of sheaves with values in $\Mot^{\loc}.$ Then by \cite[Theorem 6.11]{E24}
	we have
	\begin{equation*}
		\un{\Hom}_{\Mot^{\loc}}(\cU_{\loc}(\Shv(X;\Sp)),\cU_{\loc}(\cC))\cong \un{\Hom}_{\Mot^{\loc}}(p_!p^*\cU_{\loc}(\Sp),\cU_{\loc}(\cC))\cong p_*p^!\cU_{\loc}(\cC).
	\end{equation*}
	Note that we have a commutative diagram
	\begin{equation*}
		\begin{tikzcd}
			\Mot^{\loc} \ar{r}{p^!}\ar[d] & \Shv(X;\Mot^{\loc})\ar{r}{p_*}\ar[d] & \Mot^{\loc}\ar[d]\\ 
			\Sp \ar{r}{p^!} & \Shv(X;\Sp)\ar{r}{p_*} & \Sp.
		\end{tikzcd}
	\end{equation*}
	Here the vertical functors are induced by $\Hom(\cU_{\loc}(\Sp),-).$ The commutativity is obtained by passing to left adjoints and using the equivalence $\Shv(X;\Mot^{\loc})\simeq \Shv(X;\Sp)\otimes \Mot^{\loc}.$ Therefore, we obtain
	\begin{multline*}
	\Hom_{\Mot^{\loc}}(\cU_{\loc}(\Shv(X;\Sp)),\cU_{\loc}(\cC))\cong\Hom_{\Mot^{\loc}}(\cU_{\loc}(\Sp),p_*p^!\cU_{\loc}(\cC))\\
	\cong p_*p^!\Hom_{\Mot^{\loc}}(\cU_{\loc}(\Sp),\cU_{\loc}(\cC))\cong p_*p^! K^{\cont}(\cC).
	\end{multline*}
This proves the theorem.
\end{proof}

\subsection{Completed sheaves}
\label{ssec:completed_sheaves}

Let $X$ be a locally compact Hausdorff space and let $\un{\cC}$ be a presheaf on $X$ with values in $\Cat_{\st}^{\dual}.$ One can take the sheafification of $\un{\cC}$ to obtain a sheaf $\un{\cC}^{\sharp}\in\Shv(X;\Cat_{\st}^{\dual}).$ Its global sections are described as follows:
\begin{equation}\label{eq:global_sections_category}
\Gamma(X,\un{\cC}^{\sharp})=\prolim[Y\in\msK(X)^{op}]^{\dual} \Shv(Y;\un{\cC}_{\mid Y}).
\end{equation}
Here $\msK(X)^{op}$ is the poset of compact subsets of $X$ with the reverse inclusion order. We will not prove the assertion about the sheafification of $\un{\cC}$ here, instead we simply use the notation $\Gamma(X,\un{\cC}^{\sharp})$ for the right hand side of \eqref{eq:global_sections_category}.

\begin{theo}\label{th:K_theory_of_global_sections}
Let $X$ be a locally compact Hausdorff space and let $\un{\cC}$ be a presheaf on $X$ with values in $\Cat_{\st}^{\dual}.$ Suppose that $X$ is countable at infinity. Then we have an isomorphism
\begin{equation}\label{eq:K_theory_of_global_sections}
K^{\cont}(\Gamma(X,\un{\cC}^{\sharp}))\cong \Gamma(X,K^{\cont}(\un{\cC})^{\sharp}).
\end{equation}
Equivalently, we have
\begin{equation}\label{eq:K_theory_of_inverse_limit_of_cats_of_sheaves}
K^{\cont}(\prolim[Y\in\msK(X)^{op}]^{\dual} \Shv(Y;\un{\cC}_{\mid Y}))\xto{\sim} \prolim[Y\in\msK(X)^{op}]K^{\cont}(\Shv(Y;\un{\cC}_{\mid Y}))
\end{equation}
\end{theo}

We will need the following interpretation of the category $\Gamma(X;\un{\cC}^{\sharp})$ via internal $\Hom.$ Below we take the one-point compactification of $X$ only for the sake of psychological convenience: it is more comfortable to work over a rigid base than over a locally rigid base.

\begin{prop}\label{prop:global_section_cats_via_internal_Hom}
Let $X$ and $\un{\cC}$ be as in Theorem \ref{th:K_theory_of_global_sections}, but we do not require $X$ to be countable at infinity. Denote by $\bbar{X}=X\cup\{\infty\}$ the one-point compactification of $X.$ Then we have have a natural equivalence
\begin{equation}\label{eq:global_sections_as_internal_Hom}
\Gamma(X,\un{\cC}^{\sharp})\simeq \un{\Hom}_{\Shv(\bbar{X};\Sp)}(\Shv(X;\Sp),\Shv(X;\un{\cC})).
\end{equation}
In particular, if $\un{\cC}=\Sp_X$ is the constant presheaf with value $\Sp,$ then the category $\Gamma(X,\Sp_X^{\sharp})$ is naturally symmetric monoidal and we have
\begin{equation}\label{eq:global_sections_as_rigidification}
\Gamma(X,\Sp_X^{\sharp})\simeq \Shv(X;\Sp)^{\rig}.
\end{equation}
\end{prop}

\begin{proof}
We first note that \eqref{eq:global_sections_as_rigidification} is indeed a special case of \eqref{eq:global_sections_as_internal_Hom} by \cite[Proposition 4.1]{E25}.

We put $\cE=\Shv(\bbar{X};\Sp).$ Then for any open subsets $U,V\subset X$ such that $U\Subset V$ the ($\cE$-linear) extension by zero (i.e. proper pushforward) functor $\Shv(U;\Sp)\to\Shv(V;\Sp)$ is trace-class over $\cE.$ More precisely, consider the $\cE$-linear functor $\Phi:\Shv(\bbar{U};\Sp)\to\un{\Hom}_{\cE}^{\dual}(\Shv(U;\Sp),\cE),$ which by adjunction corresponds to the composition
\begin{equation*}
\Shv(\bbar{U};\Sp)\tens{\cE}\Shv(U;\Sp)\xto{\sim}\Shv(U;\Sp)\to \cE,
\end{equation*}
where the latter functor is extension by zero. We have the object
\begin{equation*}
\bS_{\bbar{U}}\in \Shv(\bbar{U};\Sp)\simeq \Shv(\bbar{U};\Sp)\tens{\cE}\Shv(V;\Sp).
\end{equation*}
Then the composition
\begin{multline*}
\Shv(U;\Sp)\xto{\id\boxtimes \bS_{\bbar{U}}} \Shv(U;\Sp)\tens{\cE}\Shv(\bbar{U};\Sp)\tens{\cE}\Shv(V;\Sp)\\
\xto{\id\boxtimes\Phi\boxtimes\id}\Shv(U;\Sp)\tens{\cE}\un{\Hom}_{\cE}^{\dual}(\Shv(U;\Sp),\cE)\tens{\cE}\Shv(V;\Sp)\to\Shv(V;\Sp)
\end{multline*}
is isomorphic to the extension by zero functor. Next, we note that we have an equivalence 
\begin{equation*}
\Shv(\bbar{U};\Sp)\tens{\cE}\Shv(X;\un{\cC})\simeq \Shv(\bbar{U};\un{\cC}_{\mid U})
\end{equation*} 
and similarly for $V.$ Arguing as in Proposition \ref{prop:diagonal_arrow_dualizable_trace_class}, we obtain a commutative diagram
\begin{equation*}
\begin{tikzcd}
\Shv(\bbar{V};\un{\cC}_{\mid \bbar{V}})\ar[r]\ar[d] & \Shv(\bbar{U};\un{\cC}_{\mid \bbar{U}})\ar[d]\\
\un{\Hom}_{\cE}^{\dual}(\Shv(V;\Sp),\Shv(X;\un{\cC}))\ar[r]\ar[ru] & \un{\Hom}_{\cE}^{\dual}(\Shv(U;\Sp),\Shv(X;\un{\cC})).
\end{tikzcd}
\end{equation*}
This gives the equivalences
\begin{multline*}
\Gamma(X;\un{\cC}^{\sharp})\xto{\sim} \prolim[U\Subset X]^{\dual} \Shv(\bbar{U};\un{\cC}_{\mid \bbar{U}})\xto{\sim} \prolim[U\Subset X]^{\dual} \un{\Hom}_{\cE}^{\dual}(\Shv(U;\Sp),\Shv(X;\un{\cC}))\\
\simeq \un{\Hom}_{\cE}^{\dual}(\Shv(X;\Sp),\Shv(X;\un{\cC})),
\end{multline*}
as stated.
\end{proof}

\begin{proof}[Proof of Theorem \ref{th:K_theory_of_global_sections}]
We note that \eqref{eq:K_theory_of_inverse_limit_of_cats_of_sheaves} is indeed equivalent to \eqref{eq:K_theory_of_global_sections} by \cite[Theorem 6.11]{E24}.

As in the proof of Proposition \ref{prop:global_section_cats_via_internal_Hom}, we put $\cE=\Shv(\bbar{X};\Sp),$ where $\bbar{X}$ is the one-point compactification of $X.$ Then $\Shv(X;\Sp)$ is a smashing ideal in $\cE,$ hence it is proper over $\cE.$ Moreover, since $\cE$ is countable at infinity, we see that $\Shv(X;\Sp)$ is $\omega_1$-compact in $\Cat_{\cE}^{\dual}.$ By Proposition \ref{prop:global_section_cats_via_internal_Hom} and Theorem \ref{th:morphisms_in_Mot^loc_via_internal_Hom} we have
\begin{equation*}
K^{\cont}(\Gamma(X;\un{\cC}^{\sharp}))\xto{\sim}\Hom_{\Mot^{\loc}_{\cE}}(\cU_{\loc}(\Shv(X;\Sp)),\cU_{\loc}(\Shv(X;\un{\cC}))).
\end{equation*}
By the proof of Proposition \ref{prop:global_section_cats_via_internal_Hom} we obtain the isomorphisms in $\Mot^{\loc}_{\cE}:$
\begin{multline*}
\un{\Hom}_{\Mot^{\loc}_{\cE}}(\cU_{\loc}(\Shv(X;\Sp)),\cU_{\loc}(\Shv(X;\un{\cC})))\\
\cong \prolim[U\Subset X]\un{\Hom}_{\Mot^{\loc}_{\cE}}(\cU_{\loc}(\Shv(U;\Sp)),\cU_{\loc}(\Shv(X;\un{\cC})))\cong \prolim[U\Subset X] \cU_{\loc}(\Shv(\bbar{U};\un{\cC}_{\mid \bbar{U}})).
\end{multline*}
Applying $\Hom_{\Mot^{\loc}_{\cE}}(\cU_{\loc}(\cE),-),$ we obtain
\begin{equation*}
\Hom_{\Mot^{\loc}_{\cE}}(\cU_{\loc}(\Shv(X;\Sp)),\cU_{\loc}(\Shv(X;\un{\cC})))\cong \prolim[U\Subset X]K^{\cont}(\Shv(\bbar{U};\un{\cC}_{\mid \bbar{U}})).
\end{equation*}
This proves \eqref{eq:K_theory_of_inverse_limit_of_cats_of_sheaves} and the theorem.
\end{proof}

\section{Refined negative cyclic homology}
\label{sec:refined_HC^-}

In this section we explain another application of rigidity of the category of localizing motives -- the construction refined versions of Hochschild homology. This was announced in \cite{E24b}. Very interesting and highly non-trivial examples were computed by Meyer and Wagner in \cite{MW24}. Here we will give proofs of some of the results announced in \cite{E24b}, namely Propositions 2, 3, 6, 7 of loc. cit. The proofs of \cite[Theorems 4 and 5]{E24b} (in particular the comparison with rigid cohomology of smooth schemes over a perfect field of characteristic $p>0$) will appear elsewhere.

For simplicity, we work over a usual (discrete) commutative ring $\mk,$ and we will moreover assume that $\mk$ is a $\Q$-algebra. We consider the negative cyclic homology (relative to $\mk$) as a symmetric monoidal localizing invariant
\begin{equation}\label{eq:HC^-_from_Cat^perf}
\HC^-(-/\mk):\Cat_{\mk}^{\perf}\to \Mod_u^{\wedge}\hy\mk[[u]],
\end{equation}
where $u$ is a formal variable of cohomological degree $2.$ Here the target is the category of $u$-complete modules over the $\bE_{\infty}$-algebra $\mk[[u]].$ Here we identify $\mk[[u]]$ with the $\bE_{\infty}$-algebra $C^*(\C\PP^{\infty},\mk),$ using the assumption that $\mk$ is a $\Q$-algebra.

We have a commutative square
\begin{equation*}
\begin{tikzcd}
\Cat_{\mk}^{\perf}\ar[r, "\HH(-/\mk)"]\ar[equal]{d} & [2em] \Mod\hy\mk S^1\ar[d, "(-)^{\h S^1}", "\sim" {anchor=south, rotate=90}']\\
\Cat_{\mk}^{\perf}\ar[r, "\HC^-(-/\mk)"] & \Mod_u^{\wedge}\hy\mk[[u]].
\end{tikzcd}
\end{equation*}
Here the right vertical arrow is an equivalence which takes the compact generator $\mk S^1$ to the compact generator $\mk[1].$

\begin{remark}\begin{enumerate}[label=(\roman*),ref=(\roman*)]
		\item We of course have an isomorphism $\mk[[u]]\cong \mk[u].$ However, if we allow $\mk$ to be an $\bE_{\infty}$-$\Q$-algebra which is not bounded above (homologically), then the map $\mk[u]\to \mk[u[]]$ is not an isomorphism.
		\item Suppose that $\mk$ is not a $\Q$-algebra, but we still require $\mk$ to be $\Z$-linear, for example, discrete. Then we only have an isomorphism of $\bE_1$-$\mk$-algebras $\mk[[u]]\cong C^*(\C\PP^{\infty},\mk).$ This means that the target of \eqref{eq:HC^-_from_Cat^perf} is still the ``correct'' category, but the symmetric monoidal structure is different. The same applies more generally to the case when $\mk$ is an $\bE_{\infty}$-algebra over $\MU.$
	\end{enumerate}
\end{remark}

Since the functor \eqref{eq:HC^-_from_Cat^perf} commutes with filtered colimits, it induces a symmetric monoidal functor from the category of localizing motives, which we denote by the same symbol:
\begin{equation}\label{eq:HC^-_on_Mot^loc}
\HC^-:\Mot^{\loc}_{\mk}\to \Mod_u^{\wedge}\hy\mk[[u]].
\end{equation}
Now, the source of \eqref{eq:HC^-_on_Mot^loc} is a rigid category by Theorem \ref{th:dualizability_and_rigidity} \ref{E_1_rigidity}. On the other hand, the target is only locally rigid: the compact generator $\mk$ is dualizable, but the unit object $\mk[[u]]$ is not compact. Following \cite{E25}, we define the category of nuclear modules $\Nuc(\mk[[u]])$ to be the rigidification of $\Mod_u^{\wedge}\hy\mk[[u]].$ By the universal property of the rigidification, the functor $\eqref{eq:HC^-_on_Mot^loc}$ factors uniquely through $\Nuc(\mk[[u]]).$ Hence we obtain a symmetric monoidal continuous functor
\begin{equation*}
\HC^{-,\tref}(-/\mk):\Mot_{\mk}^{\loc}\to\Nuc(\mk[[u]]),
\end{equation*}
which we call the refined negative cyclic homology. For $\cA\in\Cat_{\mk}^{\perf}$ we put $\HC^{-,\tref}(\cA/\mk)=\HC^{-,\tref}(\cU_{\loc}(\cA)/\mk),$ and for an $\bE_1$-$\mk$-algebra $A$ we put $\HC^{-,\tref}(A/\mk)=\HC^{-,\tref}(\Perf(A)/\mk).$

In the next two subsections we give examples of computations of $\HC^{-,\tref}$ for $\Q[x]$ over $\Q$ (Proposition \ref{prop:refined_HC^-_affine_line}) and for $\Q[x^{\pm 1}]$ over $\Q[x]$ (Proposition \ref{prop:refined_HC^-_Laurent_polynomials}). In Subsection \ref{ssec:refined_HP} we introduce another closely related invariant, namely refined periodic cyclic homology $\HP^{\tref}(-/\mk).$ In Subsection \ref{ssec:refined_HP_of_Q_over_Qx} we compute $\HP^{\tref}(\Q/\Q[x]).$ As an application, we deduce that the category $\Mot^{\loc,\A^1}_{\Q[x]}$ is not compactly generated, which implies that the category $\Mot^{\loc}_{\Q[x]}$ is also not compactly generated. We also deduce that the universal finitary $\A^1$-invariant localizing invariant $\cU_{\loc}^{\A^1}:\Cat^{\perf}_{\Q[x]}\to\Mot^{\loc,\A^1}_{\Q[x]}$ is not truncating in the sense of \cite{LT19}.

We will need the following general statements about the category $\Nuc(\mk[[u]])$ and the functor $\HC^{-,\tref}(-/\mk).$ 

\begin{prop}\label{prop:Nuc_generated_by_sequential_colimits} We have a symmetric monoidal fully faithful strongly continuous functor $\Nuc(\mk[[u]])\to \Ind((\Mod_u^{\wedge}\hy\mk[[u]])^{\omega_1}).$ Its essential image is generated via colimits by the objects $\inddlim[n\in\N]M_n,$ where the transition maps $M_n\to M_{n+1}$ are trace-class.\end{prop}

\begin{proof}
This is a special case of \cite[Theorem 4.2]{E25}, since the unit object $\mk[[u]]$ is $\omega_1$-compact in $\Mod_u^{\wedge}\hy\mk[[u]].$ 
\end{proof}

We note that the second statement in the above proposition is quite difficult: it uses \cite[Lemma 3.10]{E25}, which only slightly simplifies in the case when the involved categories are compactly generated. Below we consider $\Nuc(\mk[[u]])$ as a full subcategory of $\Ind((\Mod_u^{\wedge}\hy\mk[[u]])^{\omega_1}).$

\begin{prop}\label{prop:refined_HC^-_of_nuclear}
Let $\cA\in\Cat_{\mk}^{\perf}$ be a category such that $\Ind(\cA)\in\Cat_{\mk}^{\cg}$ is nuclear in the sense of Definition \ref{def:nuclear_E_modules} (this holds if $\cA$ is proper over $\mk$ by Proposition \ref{prop:proper_are_nuclear}). Let $\cA\simeq\indlim[i\in I]\cA_i,$ where $I$ is directed and we have $\cA_i\in(\Cat_{\mk}^{\perf})^{\omega}$ for $i\in I.$ Then we have 
\begin{equation}\label{eq:refined_HC^-_of_nuclear}
\HC^{-,\tref}(\cA/\mk)\cong\inddlim[i]\HC^-(\cA_i/\mk).
\end{equation} 
\end{prop}

\begin{proof}

By the nuclearity of $\Ind(\cA),$ for any $i\in I$ there exists $j\geq i$ such that the functor $\Ind(\cA_i)\to\Ind(\cA_j)$ is trace-class in $\Cat_{\mk}^{\dual}.$ Hence, the morphism $\cU_{\loc}(\cA_i)\to\cU_{\loc}(\cA_j)$ is trace-class in $\Mot_{\mk}^{\loc}.$ We obtain $\hat{\cY}(\cU_{\loc}(\cA))\cong\inddlim[i]\cU_{\loc}(A_i).$ Now the isomorphism \eqref{eq:refined_HC^-_of_nuclear} follows from the general commutative diagram \eqref{eq:how_rigidification_works}.
\end{proof}

\subsection{Refined $\HC^-$ of $\Q[x]$ over $\Q$}

We work over $\Q,$ and by base change the result applies to any base ring which is a $\Q$-algebra. In the following proposition the notation $\bigooplus[n\geq 1]M_n$ stands for $\inddlim[n\geq 1](\biggplus[1\leq k\leq n]M_k).$

\begin{prop}\label{prop:refined_HC^-_affine_line}
We have an isomorphism
\begin{equation*}
\HC^{-,\tref}(\Q[x]/\Q)\cong \Q[[u]]\oplus \bigooplus[n\geq 1]\Q[1].
\end{equation*}
\end{prop}

This essentially reduces to the following lemma. For brevity we write $\HC^-(-)$ instead of $\HC^-(-/\Q),$ and we denote by $\wt{HC}^-(\Q[x]/x^n)$ the reduced negative cyclic homology. We will use the homological grading. 

\begin{lemma}\label{lem:compact_map_on_HC^-}
For any $n\geq 1$ the map $\wt{\HC}^-(\Q[x]/x^{n+1})\to\wt{\HC}^-(\Q[x]/x^n)$ is compact in $\Mod_u^{\wedge}\hy \Q[[u]].$ In other words, this map factors through an object of $\Perf_{u\hy\tors}(\Q[[u]]).$
\end{lemma} 

\begin{proof}
This is simply a direct computation. We put $A_n=\Q[x]/x^n,$ $n>0,$ and consider the reduced Hochschild homology $\wt{\HH}(A_n)$ as a $\Q S^1$-module. By Hochschild-Kostant-Rosenberg theorem, as a complex of $A_n$-modules, $\HH(A_n)$ is quasi-isomorphic to $\biggplus[k\geq 0]\Sym^k_{A_n}(L_{A_n/\Q}[1]).$ Now, $A_n$ is quasi-isomorphic to the semi-free (super-)commutative dg algebra $\Q[x,\xi_n],$ where $\deg(\xi_n)=1$ (the homological degree) and $d\xi_n = x^n.$ The cotangent complex $L_{A_n/\Q}$ is identified with the semi-free complex of $A_n$-modules, with basis elements $d_{dR}x$ and $d_{dR}\xi_n$ (of degree $0$ resp. $1$), and the differential is given by $d(d_{dR}\xi_n)=nx^{n-1}d_{dR}x.$ Hence, we have 
\begin{equation*}
\Sym^k(L_{A_n/\Q}[1])\cong \Fiber(A_n\xto{knx^{n-1}}A_n)[2k],\quad k>0.
\end{equation*} 
We obtain the following description of $\wt{\HH}(A_n)$ on the level of homology:
\begin{equation*}
\wt{\HH}_k(A_n)=\begin{cases}(x\Q[x]/x^n)\cdot (d_{dR}\xi_n)^l & \text{for }k=2l\geq 0;\\
(\Q[x]/x^{n-1})\cdot d_{dR}x\wedge (d_{dR}\xi_n)^l & \text{for }k=2l+1>0.\end{cases}
\end{equation*}
Now the Connes-Tsygan differential $B:\wt{\HH}_k(A_n)\to \wt{\HH}_{k+1}(A_n)$ is given by de Rham differential, hence we have
\begin{equation*}
B(x^i(d_{dR}\xi_n)^l) = i x^{i-1} d_{dR}x\wedge (d_{dR}\xi_n)^l,\quad l\geq 0,\,1\leq i\leq n-1,
\end{equation*}
and $B:\wt{\HH}_{2l+1}(A_n)\to \wt{\HH}_{2l+2}(A_n)$ is zero. It follows that $\wt{HH}(A_n)$ is a free $\Q S^1$-module, more precisely we have an isomorphism
\begin{equation}\label{eq:HH_of_truncated_polynomials}
\wt{HH}(A_n)\cong \biggplus[l\geq 0] (\Q S^1[2l])^{n-1}\quad\text{in }\Mod\hy\Q S^1.
\end{equation}
We claim that the map of $\Q S^1$-modules $\wt{HH}(A_{n+1})\to \wt{HH}(A_n)$ factors through $\biggplus[0\leq l\leq n-2](\Q S^1[2l])^n.$ Indeed, by \eqref{eq:HH_of_truncated_polynomials} applied to $A_{n+1},$ it suffices to show that the map $\wt{HH}_{2l}(A_{n+1})\to \wt{HH}_{2l}(A_{n})$ is zero for $l\geq n-1.$
This is clear: this map is given by
\begin{equation*}
x^i (d_{dR}\xi_{n+1})^l\mapsto x^i(x d_{dR}\xi_n)^l = x^{i+l} (d_{dR}\xi_n)^l=0\quad\text{for }1\leq i\leq n.
\end{equation*}
Now applying the functor $(-)^{\h S^1}$ we conclude that the map $\wt{\HC}^-(A_{n+1})\to\wt{\HC}^-(A_n)$ factors through $\biggplus[0\leq l\leq n-2] \Q^n[2l+1],$ which is a perfect $u$-torsion $\Q[[u]]$-module. This proves the lemma.
\end{proof}

\begin{proof}[Proof of Proposition \ref{prop:refined_HC^-_affine_line}]
We consider the reduced refined $\HC^-$ of $\Q[x],$ namely we put $\wt{\HC}^{-,\tref}(\Q[x])=\HC^{-,\tref}(\wt{\cU}_{\loc}(\Q[x])/\Q),$ and similarly for $\Q[x^{-1}]/x^{-n}$ etc. As in the proof of Theorem \ref{th:corepresentability_of_TR}, consider the short exact sequence
\begin{equation*}
0\to \Perf_{\{\infty\}}(\PP^1_{\Q})\to \Perf(\PP^1_{\Q})\to\Perf(\Q[x]).
\end{equation*}
As in loc.cit. we have $\wt{\HC}^{-,\tref}(\Q[x])\cong \wt{\HC}^{-,\tref}(\Perf_{\{\infty\}}(\PP^1_{\Q}))[1].$ Arguing as in the proof of Theorem \ref{th:K_homology_smooth_with_compactification} and using Proposition \ref{prop:refined_HC^-_of_nuclear}, we obtain \begin{multline*}
\wt{\HC}^{-,\tref}(\Perf_{\{\infty\}}(\PP^1_{\Q}))\cong \inddlim[n]\wt{\HC}^-(D^b_{\coh}(\Q[x^{-1}]/x^{-n}))\\
\cong \inddlim[n]\Hom_{\Q[[u]]}(\wt{\HC}^-(\Q[x^{-1}]/x^{-n}),\Q[[u]]).
\end{multline*} By Lemma \ref{lem:compact_map_on_HC^-} and by the Hochschild-Konstant-Rosenberg theorem, the latter ind-object is given by $\hat{\cY}(\wt{\HC}^-(\Perf_{\{\infty\}})(\PP^1_{\Q}))\cong \bigooplus[n\geq 1]\Q.$ Hence, we get $\wt{\HC}^{-,\tref}(\Q[x])\cong \bigooplus[n\geq 1]\Q[1],$ which proves the proposition.
\end{proof}

\subsection{Refined $\HC^-$ of $\Q[x,x^{-1}]$ over $\Q[x]$}

We work over $\Q[x],$ and consider the category $\Perf(\Q[x^{\pm 1}]),$ which is an idempotent $\bE_{\infty}$-algebra in $\Cat_{\Q[x]}^{\perf}.$ Since the functor $\HC^{-,\tref}(-/\Q[x])$ is symmetric monoidal, we see that the object $\wt{\HC}^{-,\tref}(\Q[x^{\pm 1}]/\Q[x])\in\Nuc(\Q[x][[u]])$ is naturally an idempotent $\bE_{\infty}$-algebra. It turns out that this algebra has a description in terms of adic geometry.

\begin{prop}\label{prop:refined_HC^-_Laurent_polynomials}
We have an isomorphism
\begin{equation}\label{eq:refined_HC^-_Laurent_polynomials}
\wt{\HC}^{-,\tref}(\Q[x^{\pm 1}]/\Q[x])\cong \cO(\bigcap\limits_{n>0}\{|u|\leq |x|^n\ne 0\})
\end{equation}
-- the algebra of overconvergent functions. More precisely, we have
\begin{equation}\label{eq:explicit_ind_object_overconvergent_Laurent}
\wt{\HC}^{-,\tref}(\Q[x^{\pm 1}]/\Q[x])\cong \inddlim[n] x^{-n}\Q\left[x,\frac{u}{x^n}\right]_u^{\wedge}.
\end{equation}
\end{prop}

We first give a comment on the meaning of the right hand side of \eqref{eq:refined_HC^-_Laurent_polynomials}. Suppose that $u$ was a variable of degree $0.$ Then we have the usual adic space $\Spa(\Q[x][[u]])=\Spa(\Q[x][[u]],\Q[x][[u]]),$ which has an underlying spectral topological space. Consider the spectral space $\Spa(\Q[x][[u]])^{op},$ which has the same underlying set, the same constructible topology, but the specialization order is reversed. Then for any closed subset $Z\subset \Spa(\Q[x][[u]])^{op}$ we have the corresponding idempotent $\bE_{\infty}$-algebra $\cO(Z)\in\Nuc(\Q[x][[u]]),$ which is contained in the original version of the category of nuclear modules $\Nuc^{CS}(\Q[x][[u]])$ \cite{And23}. In particular, this applies to $Z=\bigcap\limits_{n>0}\{|u|\leq |x|^n\ne 0\},$ and the corresponding nuclear module is given by the right hand side of \eqref{eq:explicit_ind_object_overconvergent_Laurent}.

In our case the variable $u$ has cohomological degree $2,$ but the right hand-side of \eqref{eq:refined_HC^-_Laurent_polynomials} still makes sense, namely we can simply define it to be the object from \eqref{eq:explicit_ind_object_overconvergent_Laurent}. 

\begin{proof}[Proof of Proposition \ref{prop:refined_HC^-_Laurent_polynomials}]
We again use the homological grading. Consider the short exact sequence in $\Cat_{\Q[x]}^{\perf}:$
\begin{equation*}
0\to \Perf_{x\hy\tors}(\Q[x])\to\Perf(\Q[x])\to\Perf(\Q[x^{\pm 1}])\to 0.
\end{equation*}
By \cite[Lemma 5.24]{E25} for all $n>0$ the extension of scalars functor $\Perf(\Q[x]/x^{2n})\to\Perf(\Q[x]/x^n)$ is trace-class in the symmetric monoidal category $\Cat_{\Q[x]}^{\perf},$ and so is the restriction of scalars functor $D^b_{\coh}(\Q[x]/x^n)\to D^b_{\coh}(\Q[x]/x^{2n}).$ It follows that in $\Ind((\Mot^{\loc}_{\Q[x]})^{\omega_1})$ we have
\begin{equation*}
\hat{\cY}(\cU_{\loc}(\Perf_{x\hy\tors})(\Q[x]))\cong \inddlim[n]\cU_{\loc}(D^b_{\coh}(\Q[x]/x^n))\cong \inddlim[n]\cU_{\loc}(\Q[x]/x^n)^{\vee}.
\end{equation*} 
It follows that the refined $\HC^-$ of $\Q[x^{\pm 1}]$ over $\Q[x]$ is given by
\begin{multline*}
\HC^{-,\tref}(\Q[x^{\pm 1}]/\Q[x])\\
\cong \Cone(\inddlim[n]\Hom_{\Q[x][[u]]}(\HC^-((\Q[x]/x^n)/\Q[x]),\Q[x][[u]])\to\Q[x][[u]])\\
\cong \inddlim[n]\Hom_{\Q[x][[u]]}(\Fiber(\Q[x][[u]]\to\HC^-((\Q[x]/x^n)/\Q[x])),\Q[x][[u]]).
\end{multline*}
We directly compute the latter ind-object. As in the proof of Proposition \ref{prop:refined_HC^-_affine_line}, we put $A_n=\Q[x]/x^n,$ and take a quasi-isomorphic semi-free commutative dg $\Q[x]$-algebra $B_n=\Q[x,\xi_n],$ where $\deg(\xi_n)=1$ and $d\xi_n=x^n.$ Then the dg module of differentials $\Omega^1_{B_n/\Q[x]}$ is a free $B_n$-module of rank $1,$ generated by the closed element $d_{dR}\xi_n$ of degree $1.$ Hence, for $k\geq 0$ we have (chain-level) isomorphism of $B_n$-modules $\Sym^k(\Omega^1_{B_n/\Q[x]}[1])\cong B_n[2k].$ Now, the standard (reduced) complex $\HC^-(B_n/\Q[x])$ is identified (on the chain level) with the twisted de Rham complex $(\Sym^*(\Omega^1_{B_n/\Q[x]}[1]),d+u d_{dR})$ (here $d$ stands for the differential coming from the differential on $\Omega^1_{B_n/\Q[x]},$ and $d_{dR}$ is the de Rham differential). We are interested in the dg $\Q[x][[u]]$-module $(\HC^-(B_n/\Q[x])/\Q[x][[u]])[-1].$ It is (topologically) semi-free over $\Q[x][[u]],$ and its (topological) basis elements are: $(d_{dR}\xi_n)^k$ of degree $2k-1$ for $k\geq 1,$ and $\xi_n(d_{dR}\xi_n)^k$ of degree $2k$ for $k\geq 0.$ The differential is given by
\begin{equation*}
(d+u d_{dR})(\xi_n(d_{dR}\xi_n)^k)=\begin{cases}
-x^n(d_{dR}\xi_n)^k-u(d_{dR}\xi_n)^{k+1} & \text{for }k>0;\\
-u d_{dR}\xi_n & \text{for }k=0,
\end{cases}
\end{equation*}
\begin{equation*}
(d+u d_{dR})((d_{dR}\xi_n)^k)=0,\quad k>0.
\end{equation*} 
Consider the dual $\Q[x][[u]]$-module $\Hom_{\Q[x][[u]]}((\HC^-(B_n/\Q[x])/\Q[x][[u]])[-1],\Q[x][[u]]).$ It has a dual (topological) basis, consisting of elements $e_k$ of degree $-2k$ for $k\geq 0,$ and $f_k$ of degree $1-2k$ for $k\geq 1.$ The differential is given by
\begin{equation*}
d(f_k)=-u e_{k-1}-x^n e_k,\quad k\geq 1,\quad d(e_k)=0,\quad k\geq 0.
\end{equation*}
Hence, we have a quasi-isomorphism
\begin{equation}\label{eq:dual_complex_quasi_isomorphism}
\Hom_{\Q[x][[u]]}((\HC^-(B_n/\Q[x])/\Q[x][[u]])[-1],\Q[x][[u]])\xto{\sim} x^{-n}\Q\left[x,\frac{u}{x^n}\right]_u^{\wedge},
\end{equation}
given on the basis elements by
\begin{equation*}
e_k\mapsto (-1)^k x^{-n}(\frac{u}{x^n})^k,\quad f_k\mapsto 0.
\end{equation*}
The quasi-isomorphisms \eqref{eq:dual_complex_quasi_isomorphism} for $n>0$ are compatible with the transition maps by construction, since the natural map $B_{n+1}\to B_n$ sends $\xi_{n+1}$ to $x\xi_n.$ Therefore, we obtain isomorphisms of ind-objects:
\begin{multline*}
\inddlim[n]\Hom_{\Q[x][[u]]}(\Fiber(\Q[x][[u]]\to\HC^-((\Q[x]/x^n)/\Q[x])),\Q[x][[u]])\\
\cong \inddlim[n]\Hom_{\Q[x][[u]]}((\HC^-(B_n/\Q[x])/\Q[x][[u]])[-1],\Q[x][[u]])\cong \inddlim[n] x^{-n}\Q\left[x,\frac{u}{x^n}\right]_u^{\wedge}.
\end{multline*}
This proves the proposition.
\end{proof}

\subsection{Refined periodic cyclic homology}
\label{ssec:refined_HP}

As above, we work over a usual commutative ring $\mk,$ which we assume to be a $\Q$-algebra. Again, we denote by $u$ the formal variable of cohomological degree $2,$ and consider the category $\Nuc(\mk[[u]])=(\Mod_u^{\wedge}\hy\mk[[u]])^{\rig}.$ Since the symmetric monoidal category $\Mod_u^{\wedge}\hy\mk[[u]]$ is locally rigid, we have a fully faithful inclusion
\begin{equation*}
\Mod_{u\hy\tors}\hy\mk[[u]]\simeq \Mod_u^{\wedge}\hy\mk[[u]]\xto{\hat{\cY}} (\Mod_u^{\wedge}\hy\mk[[u]])^{\rig}\simeq \Nuc(\mk[[u]]),
\end{equation*}
and its essential image is a smashing ideal. We identify the category $\Mod_{u\hy\tors}\hy\mk[[u]]$ with its image in $\Nuc(\mk[[u]]),$ and define
\begin{equation*}
\Nuc(\mk((u)))=\Nuc(\mk[[u]])/\Mod_{u\hy\tors}\hy\mk[[u]].
\end{equation*}

The following proposition is in particular a definition of refined periodic cyclic homology.

\begin{prop}
The functor $\HC^{-,\tref}(-/\mk):\Mot^{\loc}_{\mk}\to\Nuc(\mk[[u]])$ takes the localizing ideal generated by $\wt{\cU}_{\loc}(\mk[x])$ to the ideal $\Mod_{u\hy\tors}\subset\Nuc(\mk[[u]]).$ In particular, the composition 
\begin{equation*}
\Mot^{\loc}_{\mk}\to\Nuc(\mk[[u]])\to\Nuc(\mk((u)))
\end{equation*}
factors through $\Mot^{\loc,\A^1}_{\mk},$ so we obtain a commutative square in $\CAlg^{\rig}:$
\begin{equation*}
\begin{tikzcd}
\Mot^{\loc}_{\mk}\ar[r, "\HC^{-,\tref}(-/\mk)"]\ar[d] & [2em] \Nuc(\mk[[u]])\ar[d]\\
\Mot^{\loc,\A^1}_{\mk}\ar[r, "\HP^{\tref}(-/\mk)"] & \Nuc(\mk((u)).
\end{tikzcd}
\end{equation*}
We call the lower horizontal functor the refined periodic cyclic homology.
\end{prop}

\begin{proof}
It suffices to prove this for $\mk=\Q.$ In this case by Proposition \ref{prop:refined_HC^-_affine_line} we have $\HC^{-,\tref}(\wt{\cU}_{\loc}(\Q[x])/\Q)\cong \bigooplus[n\geq 1]\Q[1]\in \Mod_{u\hy\tors}\hy\Q[[u]].$ This proves the proposition.
\end{proof}

For a category $\cA\in\Cat_{\mk}^{\perf}$ we put $\HP^{\tref}(\cA/\mk)=\HP^{\tref}(\cU_{\loc}^{\A^1}(\cA)/\mk),$ and for an $\bE_1$-$\mk$-algebra $A$ we put $\HP^{\tref}(A/\mk)=\HP^{\tref}(\Perf(A)/\mk).$

We make a comment on the category $\Nuc(\mk((u))).$ Instead of thinking about it as a quotient, it is convenient to embed it into a compactly generated symmetric monoidal category. Namely, consider the small symmetric monoidal category $(\Mod_u^{\wedge}\hy\mk[[u]])^{\omega_1}[u^{-1}],$ which is obtained by taking the quotient of $(\Mod_u^{\wedge}\hy\mk[[u]])^{\omega_1}$ by the small stable subcategory generated by objects of the form $\Cone(M\xto{u}M).$ Then the functor
\begin{equation*}
\Ind((\Mod_u^{\wedge}\hy\mk[[u]])^{\omega_1})\to \Ind((\Mod_u^{\wedge}\hy\mk[[u]])^{\omega_1}[u^{-1}])
\end{equation*}
is a smashing localization, and its kernel is generated by the coidempotent $\bE_{\infty}$-coalgebra $\hat{\cY}(\mk[[u]]).$ Hence, we have an equivalence
\begin{equation*}
\Ind((\Mod_u^{\wedge}\hy\mk[[u]])^{\omega_1})\tens{\Nuc(\mk[[u]])}\Nuc(\mk((u)))\xto{\sim} \Ind((\Mod_u^{\wedge}\hy\mk[[u]])^{\omega_1}[u^{-1}]),
\end{equation*}
in particular, we have a fully faithful symmetric monoidal strongly continuous functor
\begin{equation}\label{eq:embedding_for_Nuc_over_Laurent}
\Nuc(\mk((u)))\to \Ind((\Mod_u^{\wedge}\hy\mk[[u]])^{\omega_1}[u^{-1}]).
\end{equation}
Below we identify $\Nuc(\mk((u)))$ with the essential image of the functor \eqref{eq:embedding_for_Nuc_over_Laurent}. For an object $M\in (\Mod_u^{\wedge}\hy\mk[[u]])^{\omega_1}$ we will denote by $M[u^{-1}]$ its image in $(\Mod_u^{\wedge}\hy\mk[[u]])^{\omega_1}[u^{-1}].$

\begin{remark}
In fact, the functor \eqref{eq:embedding_for_Nuc_over_Laurent} exposes the source as a rigidification of the target, but we will not need this fact here. We will discuss this in a more general context in \cite{E}.
\end{remark}

\subsection{Refined $\HP$ of $\Q$ over $\Q[x]$}
\label{ssec:refined_HP_of_Q_over_Qx}

In this subsection we do a non-trivial computation of $\HP^{\tref}(\Q/\Q[x]),$ where $\Q$ is considered as a $\Q[x]$-algebra via $x\mapsto 0.$ We first observe that this is actually an idempotent $\bE_{\infty}$-algebra. Throughout this subsection we use only the homological grading. 

\begin{prop}
The object $\cU_{\loc}^{\A^1}(\Q)\in \Mot^{\loc,\A^1}_{\Q[x]}$ is an idempotent $\bE_{\infty}$-algebra. Hence, the object $\HP^{\tref}(\Q/\Q[x])\in\Nuc(\Q[x]((u)))$ is also an idempotent $\bE_{\infty}$-algebra.
\end{prop} 

\begin{proof}
As above, we note that the dg $\Q[x]$-algebra $\Q$ is quasi-isomorphic to the semi-free commutative dg algebra $B=\Q[x,\xi],$ where $\deg(\xi)=1$ and $d\xi=x.$ Hence, in the $\infty$-category of $\bE_1$-$\Q[x]$-algebras we have
\begin{equation*}
\Q\Ltens{\Q[x]}\Q\cong \Q\tens{\Q[x]}B\cong \Q[\xi].
\end{equation*}
Note that $x$ is acting by zero (on the chain level) on $\Q[\xi].$ We have (unique) morphisms of dg $\Q[x]$-algebras $f:\Q\to\Q[\xi],$ $g:\Q[\xi]\to \Q$ and we need to show that they induce mutually inverse maps after in $\Mot^{\loc,\A^1}_{\Q[x]}.$ Clearly, $g\circ f=\id,$ so we need to show that the endomorphism $\cU_{\loc}^{\Q^1}(f\circ g)$ of $\cU_{\loc}^{\Q^1}(\Q[\xi])$ is homotopic to the identity.

Let $t$ be a formal variable of degree $0,$ and consider a morphisms of dg $\Q[x]$-algebras
\begin{equation*}
\varphi:\Q[\xi]\to \Q[x,t]\tens{\Q[x]}\Q[\xi],\quad \xi\mapsto t\otimes\xi,
\end{equation*}
and
\begin{equation*}
i_0,i_1:\Q[x,t]\to\Q[x],\quad i_0(t)=0,\quad i_1(t)=1.
\end{equation*}
Then the maps $\cU_{\loc}^{\Q^1}(i_0)$ and $\cU_{\loc}^{\A^1}(i_1)$ are homotopic (isomorphisms) in $\Mot^{\loc,\A^1}_{\Q[x]}$ and we have (chain-level) equalities of morphisms of dg algebras $(i_0\boxtimes \id)\circ\varphi=f\circ g,$ $(i_1\boxtimes \id)\circ\varphi=\id.$ Hence, $\cU_{\loc}^{\A^1}(f\circ g)$ is homotopic to the identity, as required.
\end{proof}

It turns out that the algebra $\HP^{\tref}(\Q/\Q[x])$ again has a description in terms of adic geometry. Namely, we will prove the following result which is formulated similarly to Proposition \ref{prop:refined_HC^-_Laurent_polynomials}.

\begin{prop}\label{prop:refined_HP_of_Q_over_Qx}
We have an isomorphism
\begin{equation}\label{eq:refined_HP_of_Q_over_Qx_as_overconvergent}
\HP^{\tref}(\Q/\Q[x])\cong\cO(\bigcap_{0<\veps<1}\{|x|\leq |u|^{1-\veps}\})
\end{equation}
-- the algebra of overconvergent functions. More precisely, we have
\begin{equation}\label{eq:refined_HP_of_Q_over_Qx_explicit}
\HP^{\tref}(\Q/\Q[x])\cong \inddlim[n]\left(\Q\left[x,u,\frac{x^{n+1}}{u^n}\right]_u^{\wedge}[u^{-1}]\right).
\end{equation}
\end{prop}

Here the meaning of the right hand side of \eqref{eq:refined_HP_of_Q_over_Qx_as_overconvergent} is similar to \eqref{eq:refined_HC^-_Laurent_polynomials}: if $u$ was a variable of degree $0,$ then we consider $\bigcap_{0<\veps<1}\{|x|\leq |u|^{1-\veps}\}$ as a closed subset of the spectral space $\Spa(\Q[x]((u)),\Q[x][[u]])^{op},$ and the associated idempotent $\bE_{\infty}$-algebra in $\Nuc(\Q[x]((u)))$ is given by the right hand side of \eqref{eq:refined_HP_of_Q_over_Qx_explicit}.

Before proving Proposition \ref{prop:refined_HP_of_Q_over_Qx} we mention some applications.

\begin{cor}
The category $\Mot^{\loc,\A^1}_{\Q[x]}$ is not compactly generated. More precisely, the object $\cU_{\loc}^{\A^1}(\Q)\in \Mot^{\loc,\A^1}_{\Q[x]}$ is not contained in the localizing subcategory generated by compact objects. In particular, the category $\Mot^{\loc}_{\Q[x]}$ is not compactly generated.
\end{cor}

\begin{proof}
The final assertion follows from the first one since the functor $\Mot^{\loc}_{\Q[x]}\to \Mot^{\loc,\A^1}_{\Q[x]}$ is strongly continuous by Theorem \ref{th:A^1_invariant_localizing_motives}. 

To prove the statement about $\cU_{\loc}^{\A^1}(\Q)$ it suffices to prove that the object $\HP^{\tref}(\Q/\Q[x])\in\Nuc(\Q[x]((u)))$ is not contained in the localizing subcategory generated by $\Perf(\Q[x]((u)))\simeq\Nuc(\Q[x]((u)))^{\omega}.$ This follows directly from \eqref{eq:refined_HP_of_Q_over_Qx_explicit}: for any $n\geq 1$ the morphism
\begin{equation*}
\Q\left[x,u,\frac{x^{2}}{u}\right]_u^{\wedge}[u^{-1}]\to \Q\left[x,u,\frac{x^{n+1}}{u^n}\right]_u^{\wedge}[u^{-1}]
\end{equation*}
in $(\Mod_u^{\wedge}\hy\Q[x][[u]])^{\omega_1}[u^{-1}]$ does not factor through an object of $\Perf(\Q[x]((u))).$
\end{proof}

Befor formulating another application, we first recall that by \cite[Theorems IV.2.1 and IV.2.6]{Goo85} for any discrete commutative $\Q$-algebra $\mk$ the usual periodic cyclic homology $\HP(-/\mk)$ is a truncating invariant in the sense of \cite[Definition 3.1]{LT19} (the same in fact holds when $\mk$ is more generally a connective commutative dg $\Q$-algebra). We also recall that by \cite[Corollary 3.5]{LT19} any truncating invariant is also nilinvariant, i.e. it sends nilpotent extensions $A\to B$ of discrete algebras to isomorphisms.

\begin{cor}
The refined periodic cyclic homology $\HP^{\tref}(-/\Q[x])$ is not a truncating invariant. More precisely, the map $\HP^{\tref}((\Q[x]/x^2)/\Q[x])\to \HP^{\tref}(\Q/\Q[x])$ is not an isomorphism. In particular, the functor $\cU_{\loc}^{\A^1}:\Cat_{\Q[x]}^{\perf}\to\Mot^{\loc,\A^1}_{\Q[x]}$ is not a truncating invariant.
\end{cor}

\begin{proof}
It follows from Proposition \ref{prop:refined_HP_of_Q_over_Qx} via base change $x\mapsto x^2$ that we have an isomorphism
\begin{equation*}
	\HP^{\tref}(\Q/\Q[x])\cong \inddlim[n]\left(\Q\left[x,u,\frac{x^{2n+2}}{u^n}\right]_u^{\wedge}[u^{-1}]\right).
\end{equation*}
Thus, it suffices to observe that for any $n\geq 1$ the map
\begin{equation*}
	\Q\left[x,u,\frac{x^{4}}{u}\right]_u^{\wedge}[u^{-1}]\to \Q\left[x,u,\frac{x^{2n+2}}{u^n}\right]_u^{\wedge}[u^{-1}]
\end{equation*}
in $(\Mod_u^{\wedge}\hy\Q[x][[u]])^{\omega_1}[u^{-1}]$ does not factor through $\Q\left[x,u,\frac{x^{2}}{u}\right]_u^{\wedge}[u^{-1}].$
\end{proof} 

Our proof of Proposition \ref{prop:refined_HP_of_Q_over_Qx} is based on the following computation of the ordinary negative cyclic homology.

\begin{lemma}\label{lem:HC^-_difficult_computation}
	For $n>0$ consider the semi-free associative dg $\Q[x]$-algebra $C_n=\Q[x]\la y_1,\dots,y_n\ra,$ where $\deg(y_i)=2i-1$ and the differential is given by
	\begin{equation*}
		d y_1 = x,\quad d y_{i+1} = \sum\limits_{j=1}^i y_j y_{i+1-j}\quad\text{for } 1\leq i\leq n-1.
	\end{equation*}
	Then we have an isomorphism
	\begin{equation*}
		\HC^-(C_n/\Q[x]) \cong \Q\{x^i u^j\mid i\geq 0,\,ni+(n+1)j+n\geq 0\}_u^{\wedge}.
	\end{equation*}
	Here the right hand side is the $u$-completion of the graded $\Q$-vector space spanned by the indicated monomials, considered as a dg $\Q[x,u]$-module with zero differential.
\end{lemma}

To do this computation it is convenient to work with explicit complexes as in \cite{Con94, FT83, Tsy83}. We first recall some terminology and introduce some notation.

We identify the $\infty$-category $\Mod\hy\Q[x] S^1$-modules with the category of the so-called mixed complexes. Namely, a mixed complex over $\Q[x]$ is a triple $(C,b,B),$ where $C$ is a $\Z$-graded $\Q[x]$-module, $b$ resp. $B$ is a $\Q[x]$-linear differential of (homological) degree $-1$ resp. $1,$ and $bB+Bb=0.$ They form a $1$-category in the natural way, and a map $(C,b,B)\to (C',b',B')$ is called a quasi-isomorphism if the map of complexes $(C,b)\to (C',b')$ is a quasi-isomorphism. Inverting quasi-isomorphisms, we obtain the $\infty$-category of mixed complexes which is equivalent to $\Mod\hy \Q[x] S^1.$ If the differential $b$ is clear from the context, we will simply write $(C,B)$ instead of $(C,b,B).$

Given a mixed complex $(C,b,B),$ its associated negative cyclic complex is given by
\begin{equation*}
CC^-(C,b,B)=(C[[u]],b+uB).
\end{equation*} 

Suppose that $A$ is a dg $\Q[x]$-algebra such that $\Q[x]\to A$ is a chain-level monomorphism and the dg $\Q[x]$-module $\bar{A}=A/\Q[x]$ is semi-free. We denote by $C_{\bullet}(A)$ the standard reduced Hochschild complex, considered as a mixed complex via Connes-Tsygan differential. Recall that as a graded $\Q[x]$-module $C_{\bullet}(A)=\biggplus[k\geq 0]A\tens{\Q[x]}(\bar{A}[1])^{\otimes_{\Q[x]}k}.$ We write the decomposable tensors as $(a_0,a_1,\dots a_k).$ The Hochschild differential is given by $b=b_{m_2}+b_{m_1},$ where 
\begin{equation*}
b_{m_2}(a_0[a_1\mid\dots\mid a_k])= \pm (a_ka_0,a_1,\dots, a_{k-1}) + \sum\limits_{i=0}^k \pm (a_0,\dots, a_i a_{i+1},\dots a_n),
\end{equation*}
and
\begin{equation*}
b_{m_1}(a_0,\dots,a_k)=\sum\limits_{i=0}^k \pm (a_0,\dots,d a_i,\dots,a_k).
\end{equation*}
The Connes-Tsygan differential is given by
\begin{equation*}
B(a_0,\dots,a_k)=\sum\limits_{i=0}^n \pm (1,a_i,\dots,a_n,a_0,\dots,a_{i-1}).
\end{equation*}
The associated negative cyclic complex is a chain-level model for negative cyclic homology of $\HC^-(A/\Q[x]).$ We will denote this complex by $CC^-(A/\Q[x]).$

It is convenient to define the ``naive'' version of the Hochschild complex. Namely, put $\Omega_{A/\mQ[x]}=\ker(A\tens{\Q[x]}A\to A)$ -- this is the bimodule of differentials, which is also the $\bE_1$-cotangent complex of $A$ over $\Q[x]$ (under the above assumptions on $A$). We put 
\begin{equation*}
\Omega_{A/\Q[x]}^{\sharp}=\Omega_{A/\Q[x]}/[A,\Omega_{A/\Q[x]}]=\Omega_{A/\Q[x]}\tens{A\tens{\Q[x]}A^{op}}A
\end{equation*}
-- the non-derived tensor product of dg modules. We have a map (of complexes of $\Q[x]$-modules) $\mu:\Omega_{A/\Q[x]}^{\sharp}\to A,$ $\mu(a\otimes b)=ab-(-1)^{|a|\cdot|b|}ba.$ We denote by $d_{dR}:A\to \Omega_{A/\Q[x]}^{\sharp}$ the map given by $d_{dR}(a)=1\otimes a-a\otimes 1.$ Then we have a mixed complex
\begin{equation*}
\Hoch^{\naive}(A/\Q[x])=(\Cone(\Omega_{A/\Q[x]}^{\sharp}\xto{\mu} A),d_{dR}).
\end{equation*}
We have a natural map of mixed complexes
\begin{equation*}
\Hoch(A/\Q[x])\to \Hoch^{\naive}(A/\Q[x]),
\end{equation*}
 given by
\begin{equation}\label{eq:from_Hoch_to_Hoch^naive}
(a_0,\dots,a_k)\mapsto\begin{cases}
a_0 & \text{for }k=0;\\
a_0\otimes a_1-a_0a_1\otimes 1 & \text{for }k=1;\\
0 & \text{for }k\geq 2.
\end{cases}
\end{equation}
If $A$ is semi-free over $\Q[x]$ (as an associative unital dg algebra), then this map is a quasi-isomorphism, since $\Omega_{A/\Q[x]}$ is a semi-free dg $A\tens{\Q[x]}A^{op}$-module.

We will also need a ``curved'' version of the Hochschild complex, i.e. Hochschild complex of the second kind as defined in \cite{PP12}. Recall that a curved dg algebra over $\Q[x]$ is a triple $(A,d,h),$ where $A$ is a $\Z$-graded $\Q[x]$-algebra, $d:A\to A$ is a derivation of (homological) degree $-1,$ and $h\in A_{-2}$ is an element such that $d(h)=0$ and $d^2(a)=ha-ah$ for a homogeneous $a\in A.$ Again, suppose that $\Q[x]\xto{1_A}A$ is a monomorphism of graded $\Q[x]$-modules, and that $\bbar{A}=A/\Q[x]$ is a free $\Q[x]$-module in each degree. We denote by $\Hoch^{II}(A)$ the mixed Hochschild complex of the second kind. As a graded $\Q[x]$-module, it is given by $\prodd[k\geq 0]A\tens{\Q[x]}(\bar{A}[1])^{\otimes_{\Q[x]}k}.$ The Hochschild differential is given by $b=b_{m_2}+b_{m_1}+b_{m_0},$ where $b_{m_2}$ and $b_{m_1}$ are as above, and
\begin{equation*}
b_{m_0}(a_0,\dots,a_k)=\sum\limits_{i=0}^n (-1)^{i+1+\sum\limits_{j=0}^i |a_j|}(a_0,\dots,a_i,h,a_{i+1},\dots,a_n).
\end{equation*}
The Connes-Tsygan differential is given by the same formula as above. We denote the associated negative cyclic complex by $CC^{-,II}(A/\Q[x]).$

Recall from  \cite{HKR62, CFL05}  that for a semi-free commutative dg algebra $A$ over $\Q[x]$ we have a Hochschild-Kostant-Rosenberg quasi-isomorphism of mixed complexes
$\Hoch(A)\to (\biggplus[k\geq 0]\Sym_A^k(\Omega^1_{A/\Q[x]}[1]),d,d_{dR}),$ where $d$ is induced by the differential on $\Omega^1_{A/\Q[x]},$ and $d_{dR}$ is the de Rham differential. Explicitly, the HKR map is given by
\begin{equation}\label{eq:HKR_map}
(a_0,\dots,a_k)\mapsto \frac{1}{k!}a_0 d_{dR}a_1\wedge\dots d_{dR} a_k.
\end{equation}
We will need a version of this map for curved algebras. Namely, suppose that $(A,d)$ is a semi-free commutative dg algebra over $\Q[x].$ Choosing any closed element $h\in A_{-2}$ we obtain a curved dg $\Q[x]$-algebra $(A,d,h).$ Then the formula \eqref{eq:HKR_map} gives a well-defined map of mixed complexes over $\Q[x]:$
\begin{equation}\label{eq:HKR_second_kind_general}
\Hoch^{II}(A,d,h)\to (\prodd[k\geq 0]\Sym^k(\Omega^1_{A/\Q[x]}[1]),d-d_{dR}h,d_{dR}).
\end{equation}
Here ``$d_{dR}h$'' means the operator of left multiplication by this $1$-form. The map \eqref{eq:HKR_second_kind_general} does not have to be a quasi-isomorphism in general, but this will be the case in our particular situation.

\begin{proof}[Proof of Lemma \ref{lem:HC^-_difficult_computation}]
For degree reasons the Hochschild complex $\Hoch(C_n/\Q[x])$ is a pseudo-coherent complex of $\Q[x]$-modules. Therefore, it suffices to compute the mixed complex $\Hom_{\Q[x]}(\Hoch(C_n/\Q[x]),\Q[x])$ over $\Q[x],$ or equivalently the associated negative cyclic complex as a complete $\Q[x][[u]]$-module.

Since $C_n$ is semi-free (as an associative dg algebra over $\Q[x]$), we have a quasi-isomorphism $\Hoch(C_n/\Q[x])\to \Hoch^{\naive}(C_n/\Q[x])$ of mixed complexes over $\Q[x],$ given by \eqref{eq:from_Hoch_to_Hoch^naive}. The graded components of the latter mixed complex are free finitely generated $\Q[x]$-modules. We will compute its dual over $\Q[x].$ To do this, we define the auxiliary commutative dg $\Q[x]$-algebra $A=\Q[x,t]/(t^{n+1}),$ where $\deg(t)=-2$ (homological grading) and the differential is zero. We consider the pair $(A,-xt)$ as a curved dg $\Q[x]$-algebra. It is easy to see that we have a chain-level isomorphism of mixed complexes:
\begin{equation*}
\Hom_{\Q[x]}(\Hoch^{\naive}(C_n/\Q[x]),\Q[x])\cong \Hoch^{II}((A,-xt)/\Q[x]).
\end{equation*}
To compute the latter, we introduce a semi-free commutative dg $\Q[x]$-algebra $\wt{A}=(\Q[x,t,\xi],d),$ where $\deg(t)=-2,$ $\deg(\xi)=-2n-1,$ $dt=0,$ $d\xi=t^{n+1}.$ We have a quasi-isomorphism $\wt{A}\to A,$ given by $t\mapsto t,$ $\xi\mapsto 0.$ We claim that the induced map of mixed complexes
\begin{equation}\label{eq:q-is_for_Hoch^II}
\Hoch^{II}((\wt{A},-xt)/\Q[x])\to \Hoch^{II}((A,-xt)/\Q[x])
\end{equation}
is also a quasi-isomorphism. 

To see this, it is convenient to introduce an additional grading on $A$ and $\wt{A}$ (over $\Q[x]$) or equivalently the $\bG_{m,\Q[x]}$-action. We define the $\bG_m$-weights by $\bw(t)=1,$ $\bw(\xi)=n+1.$ These $\bG_m$-weights are compatible with differentials and multiplication on $A$ and $\wt{A},$ and the map $\wt{A}\to A$ is $\bG_m$-equivariant. Note that the underlying graded $\Q[x]$-modules of the objects in \eqref{eq:q-is_for_Hoch^II} are the same as for the usual Hochschild complexes (for degree reasons), hence they acquire the induced $\bG_m$-action. The Hochschild differentials (of the second kind) don't decrease $\bG_m$-weights, hence we obtain well-defined decreasing filtrations: in both cases $F^k$ is the direct sum of $\bG_m$-homogeneous components of weight $\geq k,$ where $k\geq 0.$ Both filtrations are complete (for degree reasons), and the associated graded complexes are simply the usual Hochschild complexes $\Hoch(\wt{A}/\Q[x])$ resp. $\Hoch(A/\Q[x]).$ It follows that the map \eqref{eq:q-is_for_Hoch^II} induces a quasi-isomorphism on the associated graded complexes, hence it is a quasi-isomorphism (since the filtrations are complete).

Next, we claim that the Hochschild-Kostant-Rosenberg map (of mixed complexes)
\begin{equation}\label{eq:HKR_second_kins_q_is}
\Hoch^{II}((\wt{A},-xt)/\Q[x])\to (\prodd[m\geq 0]\Sym^m(\Omega^1_{\wt{A}/\Q[x]}[1]),d+xd_{dR}t,d_{dR})
\end{equation}
is a quasi-isomorphism. Again, this can be seen using filtrations. Namely, in the target of \eqref{eq:HKR_second_kins_q_is} the product is the same as the direct sum (for degree reasons), hence the underlying graded $\Q[x]$-module acquires the induced $\bG_m$-action. The differential does not decrease the weights, hence we obtain the decreasing filtration $F^{\bullet}$ as above. This filtration is complete, and the map \eqref{eq:HKR_second_kins_q_is} is compatible with filtrations. The induced map on the associated graded is identified with the usual HKR map $\Hoch(\wt{A}/\Q[x])\to \biggplus[m\geq 0]\Sym^m(\Omega^1_{\wt{A}/\Q[x]}[1]),$ hence it is a quasi-isomorphism. Therefore, the map \eqref{eq:HKR_second_kins_q_is} is a quasi-isomorphism.

Our next goal is to compute the homology of the complex of $\Q[x]$-modules $(\prodd[m\geq 0]\Sym^m(\Omega^1_{\wt{A}/\Q[x]}[1]),d-xd_{dR}t).$ We consider the complete decreasing filtration $G^{\bullet},$ where $G^k=\prodd[m\geq k]\Sym^m(\Omega^1_{\wt{A}/\Q[x]}[1]).$ We will compute the non-zero differentials in the associated spectral sequence. The first page (with gradings omitted) is similar to the Hochschild homology of truncated polynomials from the proof of Proposition \ref{prop:refined_HC^-_affine_line}: it is given by
\begin{equation*}
E_1=\Q[x,t]/(t^{n+1})\oplus\biggplus[l\geq 1](t\Q[x,t]/(t^{n+1}))\cdot (d_{dR}\xi)^l\oplus\biggplus[l\geq 0](\Q[x,t]/(t^n))\cdot d_{dR}t\wedge (d_{dR}\xi)^l.
\end{equation*}
The differential $d_1$ on $E_1$ is given by
\begin{equation*}
d_1(t^i(d_{dR}\xi)^l)=xt^i d_{dR}t\wedge (d_{dR}\xi)^l,\quad d_1(t^id_{dR}t\wedge (d_{dR}\xi)^l)=0.
\end{equation*}
Hence, the second page is given by
\begin{equation*}
E_2=\biggplus[l\geq 0]\Q[x]\cdot t^n(d_{dR}\xi)^l\oplus (\Q[t]/t^n)\cdot d_{dR}t\oplus \biggplus[l\geq 1](\Q[x,t]/(xt,t^n))\cdot d_{dR}t\wedge (d_{dR}\xi)^l. 
\end{equation*}
The differential $d_2$ can be computed in a straightforward way: we get
\begin{equation*}
d_2(t^n(d_{dR}\xi)^l)=\frac{-1}{(n+1)(l+1)}x^2 d_{dR}t\wedge(d_{dR}\xi)^{l+1},
\end{equation*}
and
\begin{equation*}
d_2(t^id_{dR}t\wedge (d_{dR}\xi)^l)=0.
\end{equation*}
Hence, the third page is given by
\begin{equation*}
E_3=(\Q[t]/t^n)\cdot d_{dR}t\oplus \biggplus[l\geq 1](\Q[xt]/(x^2,xt,t^n))d_{dR}t\wedge(d_{dR}\xi)^l.
\end{equation*}
It is concentrated in odd degrees, hence the spectral sequence degenerates at $E_3.$ As already noted, the product $\prodd[m\geq 0]\Sym^m(\Omega^1_{\wt{A}/\Q[x]}[1])$ is the same as the direct sum, so from now on we denote it simply by $\Sym(\Omega^1_{\wt{A}/\Q[x]}[1]).$  We conclude that the graded $\Q$-vector space $H_*(\Sym(\Omega^1_{\wt{A}/\Q[x]}[1]),d+xd_{dR}t)$ has a $\Q$-basis given
by the elements represented by the differential forms
\begin{equation}\label{eq:Q-basis_for_HH^II}
t^i d_{dR}t\wedge(d_{dR}\xi)^l,\,0\leq i\leq n-1,\,l\geq 0,\quad x d_{dR}t\wedge(d_{dR}\xi)^l,\quad l\geq 1.
\end{equation}

We now compute the twisted de Rham complex $(\Sym(\Omega^1_{\wt{A}/\Q[x]}[1])[[u]],d+xd_{dR}t+ud_{dR})$ as an object of the category $\Mod_u^{\wedge}\hy\Q[x][[u]].$
Note that the differential forms \eqref{eq:Q-basis_for_HH^II} are $d_{dR}$-closed. Hence, they form a $\Q[u]$-basis in the homology $H_*(\Sym(\Omega^1_{\wt{A}/\Q[x]}[1])[[u]],d+xd_{dR}t+ud_{dR}).$ We now define an explicit semi-free dg $\Q[x][[u]]$-module $\cK$ which is quasi-isomorphic to this twisted de Rham complex. The basis elements of $\cK$ (over $\Q[x][[u]]$) are given by $e_{l,i},f_{l,i},$ where $l\geq 0,$ $0\leq i\leq n-1.$ The degrees are given by $\deg(e_{l,i})=-2nl-2i-1,$ $\deg(f_{l,i})=-2nl-2i.$ The differential is given by
\begin{equation*}
d(f_{l,i})=\begin{cases}
x e_{l,i}+(ln+l+i)u e_{l,i-1} & \text{for }l\geq 0,\,1\leq i\leq n-1;\\
x^2 e_{l,0}+(ln+l)(ln+l-1)u e_{l-1,n-1} & \text{for }l\geq 1,\,i=0;\\
xe_{0,0} & \text{for }l=i=0.
\end{cases}
\end{equation*}
Note that $\cK$ is $u$-complete for degree reasons. We define the map of dg $\Q[x][[u]]$-modules $\varphi:\cK\to (\Sym(\Omega^1_{\wt{A}/\Q[x]}[1])[[u]],d+xd_{dR}t+ud_{dR})$ by the formulas
\begin{equation*}
\varphi(e_{l,i})= t^id_{dR}t\wedge(d_{dR}\xi)^l\quad \text{for }l\geq 0,\,0\leq i\leq n-1.
\end{equation*}
and
\begin{equation*}
\varphi(f_{l,i})=t^i(d_{dR}\xi)^l+l(n+1)t^{i-1}\xi d_{dR}t\wedge (d_{dR}\xi)^{l-1}\quad\text{for }l\geq 1,\,1\leq i\leq n-1;
\end{equation*}
\begin{equation*}
\varphi(f_{l,0})=x(d_{dR}\xi)^l+l(n+1)t^n(d_{dR}\xi)^{l-1}+l(l-1)(n+1)^2 t^{n-1}\xi d_{dR}t\wedge (d_{dR}\xi)^{l-2}\quad\text{for }l\geq 2;
\end{equation*}
\begin{equation*}
\varphi(f_{1,0})=x d_{dR}\xi+(n+1)t^n;
\end{equation*}
\begin{equation*}
\varphi(f_{0,i})=t^i\quad \text{for }0\leq i\leq n-1.
\end{equation*}
It is straightforward to check that $\varphi$ commutes with the differentials. By the above computation of homology, $\varphi$ is a quasi-isomorphism.

It remains to compute the dual dg $\Q[x][[u]]$-module $\cK^{\vee}=\Hom_{\Q[x][[u]]}(\cK,\Q[x][[u]]).$ For degree reasons, $\cK^{\vee}$ is the $u$-completion of a semi-free $\Q[x][[u]]$-module with the dual basis formed by the elements $e_{l,i}^*, f_{l,i}^*,$ where $l\geq 0,$ $0\leq i\leq n-1,$ and $\deg(e_{l,i}^*)=2ln+2i+1,$ $\deg(f_{l,i}^*)=2ln+2i.$ The differential is given by
\begin{equation*}
d(e_{l,i}^*)=\begin{cases}x f_{l,i}^*+(ln+l+i+1)u f_{l,i+1}^* & \text{for }l\geq 0,\,1\leq i\leq n-2;\\
x^2 f_{l,0}^*+(ln+l+1)u f_{l,1}^* & \text{for }l\geq 1,\,i=0;\\
x f_{l,n-1}^*+(l+1)(n+1)((l+1)(n+1)-1)u f_{l+1,0}^* & \text{for }l\geq 0,\, i=n-1;\\
x f_{0,0}^*+u f_{1,0}^* & \text{for }l=i=0,
\end{cases}
\end{equation*}
and
\begin{equation*}
d(f_{l,i}^*)=0\quad\text{for }l\geq 0,\,0\leq i\leq n-1.
\end{equation*}
It follows by a direct computation that we have a quasi-isomorphism of dg $\Q[x][[u]]$-modules
\begin{equation*}
\cK^{\vee}\xto{\sim} \Q\{x^i u^j\mid i\geq 0,\,ni+(n+1)j+n\geq 0\}_u^{\wedge}, 
\end{equation*}
given on the basis elements by
\begin{equation*}
f_{l,i}^*\mapsto\begin{cases}
\frac{(-1)^{ln+i}}{(ln+l+i)!}x^{ln+l+i}u^{-ln-i} & \text{for }l\geq 0,\,1\leq i\leq n-1;\\
\frac{(-1)^{ln}}{(ln+l)!}x^{ln+l-1}u^{-ln} & \text{for }l\geq 1,\,i=0;\\
1 & \text{for }l=i=0.
\end{cases}
\end{equation*}
To finish the proof of the lemma, it remains to observe that in the ($\infty$-category) $\Mod_{u}^{\wedge}\hy\Q[x][[u]]$ we have an isomorphism
\begin{equation*}
\cK^{\vee}\cong \HC^-(C_n/\Q[x]),
\end{equation*}
because the latter object is reflexive.
\end{proof}

\begin{proof}[Proof of Proposition \ref{prop:refined_HP_of_Q_over_Qx}]
Consider the semi-free associative dg $\Q[x]$-algebra $C=\Q[x]\la y_1,y_2,\dots\ra,$ where $\deg(y_i)=2i-1$ for $i\geq 1$ and
\begin{equation*}
d(y_1)=x,\quad d(y_{i+1})=\sum\limits_{j=1}^i y_j y_{i+1-j}\quad\text{for }i\geq 1.
\end{equation*}
Then we have a quasi-isomorphism of dg $\Q[x]$-algebras $C\to\Q,$ given by $x\mapsto 0,$ $y_i\mapsto 0.$ This follows for example from Positselski's construction \cite[Section 9.3]{Pos11} of Quillen equivalence between a certain model category of curved conilpotent dg coalgebras constructed in loc. cit. and the standard model category of dg algebras \cite{Hin97}.

Naturally, we have $C\cong\indlim[n]C_n,$ where $C_n$ is the dg $\Q[x]$-algebra introduced in Lemma \ref{lem:HC^-_difficult_computation} for $n\geq 1.$ Since $\Q$ is proper over $\Q[x],$ by Proposition \ref{prop:refined_HC^-_of_nuclear} we have
\begin{equation*}
\HC^{-,\tref}(\Q/\Q[x])\cong\inddlim[n]\HC^-(C_n/\Q[x]),
\end{equation*}
hence
\begin{equation*}
\HP^{\tref}(\Q/\Q[x])\cong \inddlim[n](\HC^-(C_n/\Q[x])[u^{-1}]).
\end{equation*}
By Lemma \ref{lem:HC^-_difficult_computation}, we have
\begin{equation*}
\HC^-(C_n/\Q[x])\cong \Q\{x^i u^j\mid i\geq 0,\,ni+(n+1)j+n\geq 0\}_u^{\wedge},\quad n\geq 1.
\end{equation*}
Note that we have a (chain-level) inclusion of dg $\Q[x][[u]]$-modules
\begin{equation*}
\Q[x,u,\frac{x^{n+1}}{u^n}]_u^{\wedge}\to \Q\{x^i u^j\mid i\geq 0,\,ni+(n+1)j+n\geq 0\}_u^{\wedge}.
\end{equation*}
Its cokernel is annihilated by $u^n,$ hence we have an isomorphism
\begin{equation*}
\Q[x,u,\frac{x^{n+1}}{u^n}]_u^{\wedge}[u^{-1}]\xto{\sim}\Q\{x^i u^j\mid i\geq 0,\,ni+(n+1)j+n\geq 0\}_u^{\wedge}[u^{-1}]
\end{equation*}
in the category $(\Mod_u^{\wedge}\hy\Q[x][[u]])^{\omega_1}[u^{-1}].$

To finish the proof, it suffices to check that for $n\geq 1$ the following square commutes in the homotopy category of $\bE_0$-algebras $\h((\Mod_u^{\wedge}\Q[x][[u]])_{\Q[x][[u]]/}):$
\begin{equation*}
\begin{tikzcd}
\HC^-(C_n/\Q[x]) \ar[r, "\sim"]\ar[d] & \Q\{x^i u^j\mid i\geq 0,\,ni+(n+1)j+n\geq 0\}_u^{\wedge}\ar[d]\\
\HC^-(C_{n+1}/\Q[x]) \ar[r, "\sim"] & \Q\{x^i u^j\mid i\geq 0,\,(n+1)i+(n+2)j+n+1\geq 0\}_u^{\wedge}. 
\end{tikzcd}
\end{equation*}
This is in fact automatic: there is only one morphism from the top right object to the bottom right object in the category $\h((\Mod_u^{\wedge}\Q[x][[u]])_{\Q[x][[u]]/}).$ This proves the proposition.
\end{proof}




\begin{thebibliography}{AGKRV20}

\bibitem[Alm74]{Alm74}C.~Almkvist, ``The Grothendieck ring of the category of endomorphisms''. J. Algebra 28 (1974), 365–388.

\bibitem[AJS23]{AJS23} L.~Alonso, A.~Jeremias, and F.~Sancho, ``Relative perfect complexes''. Math. Z. 304 (2023), no. 3, Paper No. 42.

\bibitem[And21]{And21}G.~Andreychev, ``Pseudocoherent and Perfect Complexes and Vector Bundles on Analytic Adic Spaces''. arXiv:2105.12591 (preprint).

\bibitem[And23]{And23}G.~Andreychev, ``$K$-theorie adischer R\"aume''. arXiv:2311.04394 (preprint).


\bibitem[AGH19]{AGH19} B.~Antieau, D.~Gepner and J.~Heller, ``$K$-theoretic obstructions to bounded $t$-structures''. Invent. math. 216, 241-300 (2019).

\bibitem[Aok23]{Aok23}K.~Aoki, ``The sheaves-spectrum adjunction''. arXiv:2302.04069 (preprint).

\bibitem[AFR00]{AFR00} C.~S.~Aravinda, F.~T.~Farrell, and S.~K.~Roushon, ``Algebraic $K$-theory of pure braid groups''.
Asian J. Math., 4(2):337-343, 2000.

\bibitem[AGKRV20]{AGKRV20} D.~Arinkin, D.~Gaitsgory, D.~Kazhdan, S.~Raskin, N.~Rozenblyum, and Y.~Varshavsky. ``The stack of local systems with restricted variation and geometric Langlands theory with nilpotent singular support''. arXiv:2010.01906 (preprint).



\bibitem[BB19]{BB19} A.~Bartels and M.~Bestvina, ``The Farrell-Jones conjecture for mapping class groups''. Invent.
Math., 215(2):651-712, 2019.

\bibitem[BFL14]{BFL14} A.~Bartels, F.~T.~Farrell, and W.~L\"uck, ``The Farrell-Jones Conjecture for cocompact lattices
in virtually connected Lie groups''. J. Amer. Math. Soc., 27(2):339-388, 2014. 

\bibitem[BL12]{BL12} A.~Bartels and W.~L\"uck, ``The Borel conjecture for hyperbolic and $\operatorname{CAT}(0)$-groups''. Ann. of
Math. (2), 175:631–689, 2012.

\bibitem[BLR08]{BLR08} A.~Bartels, W.~L\"uck, and H. Reich, ``The $K$-theoretic Farrell-Jones conjecture for hyperbolic
groups''. Invent. Math., 172(1):29-70, 2008.

\bibitem[Bar15]{Bar15} C.~Barwick, ``On exact $\infty$-categories and the Theorem of the Heart''. Compos. Math. 151(11), 2160–2186 (2015).


\bibitem[BGMN21]{BGMN21} C.~Barwick, S.~Glasman, A.~Mathew, and T. Nikolaus, ``K-theory and polynomial functors''. arXiv:2102.00936 (preprint).

\bibitem[Bas68]{Bas68} Hyman Bass, ``Algebraic K-theory''. W. A. Benjamin (1968).

\bibitem[BHS64]{BHS64} H.~Bass, A.~Heller, and R.~G.~Swan, ``The Whitehead group of a polynomial extension''. Publ. Math.
IHES, 22 (1964), 61-97.

\bibitem[BGT13]{BGT} A.~Blumberg, D.~Gepner, and G.~Tabuada, ``A universal characterization of higher algebraic $K$-theory''. Geom. Topol.
17 (2013), no. 2, pp. 733-838.

\bibitem[BM16]{BM16} A.~Blumberg and M.~Mandell. ``The homotopy theory of cyclotomic spectra''. Geometry and Topology, 19(6):3105-
3147, 2016.


\bibitem[Bei78]{Bei78} A. A. Beilinson. “Coherent sheaves on $\PP^n$ and problems of linear algebra”. In: Functional Analysis
and Its Applications 12.3 (July 1978), pp. 214-216. 



\bibitem[BBD82]{BBD82} A.~Beilinson, J.~Bernstein, and P.~Deligne, ``Faisceaux pervers'', Asterisque 100 (1982) 5-171.



\bibitem[BHM93]{BHM93} M.~B\"okstedt, W.~Hsiang, and I.~Madsen. ``The cyclotomic trace and algebraic K-theory of spaces''. Inventiones
mathematicae, 111(1):465-539, 1993.

\bibitem[BVdB03]{BVdB03} A.~Bondal, M.~Van den Bergh, ``Generators and representability of functors in commutative and
noncommutative geometry''. Mosc. Math. J. 3 (2003), no. 1, 1-36, 258.

\bibitem[Bon10]{Bon} M.~V.~Bondarko, ``Weight structures vs. $t$-structures; weight filtrations, spectral sequences, and complexes (for motives and in general)''. J. K-Theory, 6(3):387-504, 2010.

\bibitem[BKW21]{BKW21} U.~Bunke, D.~Kasprowski, and C.~Winges, ``On the Farrell-Jones conjecture for localising
invariants''. arXiv:2111.02490 (preprint).

\bibitem[Bur22]{Bur22} R.~Burklund, ``Multiplicative structures on Moore spectra''. arXiv:2203.14787 (preprint).

\bibitem[CFL05]{CFL05} A.~S.~Cattaneo, D.~Fiorenza, and R.~Longoni, ``On the Hochschild-Kostant-Rosenberg map for graded manifolds''. Int. Math. Res. Not. 2005, no. 62, 3899-3918.



\bibitem[Con94]{Con94} A.~Connes, ``Non-commutative Geometry''. Academic Press, Cambridge (1994).

\bibitem[CS20]{CS20}D.~Clausen and P.~Scholze, ``Lectures on Analytic Geometry''. Available at: https://www.math.uni-bonn.de/people/scholze/Analytic.pdf,
2020.




\bibitem[DGM13]{DGM13} B.~I.~Dundas, T.~G.~Goodwillie, and R.~McCarthy, ``The local structure of algebraic K-theory''. Algebra
and Applications, vol. 18, Springer-Verlag London, Ltd., London, 2013.

\bibitem[E24a]{E24}A.~I.~Efimov, ``$K$-theory and localizing invariants of large categories''.  arXiv:2405.12169 (preprint).

\bibitem[E24b]{E24b}A.~I.~Efimov, ``Refined variants of (topological) Hochschild homology''. Workshop Report No. 33/2024: Arithmetic Geometry. Oberwolfach Research Institute for Mathematics, 14 July - 19 July 2024, pp. 1903-1906. Available at: https://publications.mfo.de/handle/mfo/4247.

\bibitem[E25]{E25}A.~I.~Efimov, ``Localizing invariants of inverse limits''. arXiv:2502.04123 (preprint).

\bibitem[E]{E} A.~I.~Efimov, ``Pro-cdh descent for general localizing invariants'', in preparation.

\bibitem[EP15]{EP15}A.~I.~Efimov and L.~Positselski, ``Coherent analogues of matrix factorizations and relative singularity categories''. Algebra Number Theory 9 (2015), no. 5, 1159-1292.

\bibitem[FJ93]{FJ93} F.~T.~Farrell and L.~E.~Jones, ``Isomorphism conjectures in algebraic $K$-theory''. Journal of the American Mathematical Society, v. 6, pp. 249-297, 1993.

\bibitem[FT83]{FT83} B.~L.~Feigin and B.~L.~Tsygan, ``Cohomologies of Lie algebras of generalized Jacobi matrices''.
Funct. Anal. Appl., 1983, Volume 17, Issue 2, Pages 153–155.



\bibitem[GR17]{GaRo17} D.~Gaitsgory and N.~Rozenblyum. ``A study in derived algebraic geometry. Volume I: Correspondences and
duality''. Volume 221 of Math. Surv. Monogr. Providence, RI: American Mathematical Society (AMS), 2017.

\bibitem[GMR15]{GMR15} G.~Gandini, S.~Meinert, and H.~R\"uping, ``The Farrell-Jones conjecture for fundamental groups
of graphs of abelian groups''. Groups Geom. Dyn., 9(3):783-792, 2015.

\bibitem[Goo85]{Goo85} T.~G.~Goodwillie, ``Cyclic homology, derivations, and the free loopspace''. Topology 24 (1985), no. 2, 187–215.

\bibitem[Goo86]{Goo86} T.~G.~Goodwillie, ``Relative algebraic K-theory and cyclic homology''. Ann. of Math. (2) 124 (1986), no. 2,
347-402.





\bibitem[Gro57]{Gro} A.~Grothendieck, ``Sur quelques points d'alg\`ebre homologique''.
Tohoku Math. J. (2) 9 (2), 1957, pp. 119-221.

\bibitem[Har66]{Har} A.~Grothendieck, ``Local cohomology'', notes by R. Hartshorne, Lecture Notes in Math. 20
(1966).



\bibitem[HM97]{HM97} L.~Hesselholt, and I.~Madsen, ``On the $K$-theory of finite algebras over Witt vectors of perfect fields''. Topology, 36(1):29-101, 1997.

\bibitem[Hin97]{Hin97} V.~Hinich, ``Homological algebra of homotopy algebras''. Comm. in Algebra vol. 25, issue 10, p. 3291–3323, 1997. Erratum: arXiv:math/0309453.

\bibitem[HKR62]{HKR62} G.~Hochschild, B.~Kostant, A.~Rosenberg, ``Differential forms on regular affine algebras''. Transactions
AMS 102 (1962), No.3, 383-408 (reprinted in G. P. Hochschild, Collected Papers. Volume I 1955-1966,
Springer 2009, 265-290).

\bibitem[KLR16]{KLR16} H.~Kammeyer, W.~L\"uck, and H. R\"uping, ``The Farrell–Jones conjecture for arbitrary lattices
in virtually connected Lie groups''. Geom. Topol., 20(3):1275-1287, 2016.






\bibitem[KW19]{KW19} D.~Kasprowski, C.~Winges, (2019), "Algebraic $K$-theory of stable $\infty$-categories via binary complexes". Journal of Topology, 12: 442-462.


\bibitem[KS09]{KonSob09} M.~Kontsevich, Y.~Soibelman, ``Notes on $A_{\infty}$-algebras, $A_{\infty}$-categories and non-commutative geometry''. In: Homological Mirror Symmetry, Lecture Notes in Phys. 757,
Springer, Berlin, 153-219 (2009).

\bibitem[KNP]{KNP}A.~Krause, T.~Nikolaus, and P.~P\"utzst\"uck, ``Sheaves on manifolds''.
\href{https://www.uni-muenster.de/IVV5WS/WebHop/user/nikolaus/Papers/sheaves-on-manifolds.pdf}{Available at author’s webpage.}

\bibitem[LT19]{LT19} M.~Land and G.~Tamme, ``On the K-theory of pullbacks''. Ann. of Math. (2) 190 (2019), no. 3, 877–930.

\bibitem[Lb06]{Lb06}M.~Lieblich, ``Moduli of complexes on a proper morphism''. J. Algebraic Geom. 15 (2006), no. 1, 175-206. 

\bibitem[L\"uc25]{Luc25}W.~L\"uck, ``Isomorphism Conjectures in $K$- and $L$-Theory''. A Series of Modern Surveys in Mathematics. Springer Cham (2025).




\bibitem[Lur09]{Lur09} J. Lurie, ``Higher Topos Theory''. Annals of Mathematics Studies, vol. 170, Princeton
University Press, Princeton, NJ, 2009.


\bibitem[Lur17]{Lur17} J.~Lurie, ``Higher Algebra''. 2017, www.math.ias.edu/~lurie/papers/HA.pdf

\bibitem[Lur18]{Lur18} J.~Lurie, ``Spectral Algebraic Geometry''. 2018, www.math.ias.edu/~lurie/papers/SAG-rootfile.pdf


\bibitem[McC23]{McC23} J.~McCandless, ``Curves in $K$-theory and $\TR$''. J. Eur. Math. Soc. 26 (2024), no. 11, pp. 4315-4373.

\bibitem[MW24]{MW24} S.~Meyer, F.~Wagner, ``$q$-Hodge complexes and refined $\TC^-$''. arXiv:2410.23115 (preprint).

\bibitem[MV99]{MV99} F.~Morel, V.~Voevodsky, ``$\A^1$-homotopy theory of schemes''. Publ. Math. IHES 90, 45–143 (1999).


\bibitem[NS18]{NS18} T.~Nikolaus and P.~Scholze. ``On topological cyclic homology''. Acta Mathematica, 221(2):203-409, 2018.












\bibitem[Or04]{Or04} D.~Orlov, ``Triangulated categories of singularities and D-branes in Landau–Ginzburg models''.
Proc. Steklov Math. Inst. 246 (2004), pp. 227-248.

\bibitem[PP12]{PP12} A.~Polishchuk and L.~Positselski, ``Hochschild (co)homology of the second kind I''. Transactions of the American Mathematical Society 364, no. 10 (2012): 5311–68.

\bibitem[Pos11]{Pos11} L.~Positselski, ``Two kinds of derived categories, Koszul duality, and comodule-contramodule correspondence''. Memoirs of the Amer. Math. Soc. 212, Vol. 996, 2011, vi+133 pp.

\bibitem[Ram24a]{Ram24a} M.~Ramzi, ``Dualizable presentable $\infty$-categories''. arXiv:2410.21537 (preprint).

\bibitem[Ram24]{Ram24b}M.~Ramzi, ``Locally rigid $\infty$-categories''. arXiv:2410.21524 (preprint).

\bibitem[RSW25]{RSW25} M.~Ramzi, V.~Sosnilo and C.~Winges, ``Every motive is the motive of a stable $\infty$-category''. arXiv:2503.11338 (preprint).



\bibitem[R\"up16]{Rup16} H.~R¨uping, ``The Farrell–Jones conjecture for $S$-arithmetic groups''. J. Topol., 9(1):51-90, 2016.

\bibitem[SW25]{SW25} V.~Saunier, C.~Winges, ``On exact categories and their stable envelopes''. arXiv:2502.03408 (preprint).

\bibitem[Schl06]{Sch06} M.~Schlichting, ``Negative K-theory of derived categories''. Math. Z. 253 (2006), no. 1, 97-134. 

\bibitem[Sch17]{Sch17}P.~Scholze, ``Canonical $q$-deformations in arithmetic geometry''. Ann. Fac. Sci. Toulouse
Math. (6), vol. 26, no. 5, pp. 1163-1192, 2017.

\bibitem[Sch22]{Sch22}P.~Scholze, ``Six-Functor Formalisms''. Available at: https://people.mpim-bonn.mpg.de/scholze/SixFunctors.pdf.

\bibitem[Sch25]{Sch25}P.~Scholze, ``Habiro Cohomology''. Lecture course held at MPIM Bonn, Summer term 2025,
Recordings available at: https://archive.mpim-bonn.mpg.de/id/eprint/5155/.



208 (1995), no. Teor. Chisel, Algebra i Algebr. Geom., 290-317.

\bibitem[Tab14]{Tab14}G.~Tabuada, ``Witt vectors and K-theory of automorphisms via noncommutative motives''. Manuscripta Math. 143, 473-482 (2014).

\bibitem[Tab15]{Tab15} G.~Tabuada, ``$\A^1$-homotopy theory of noncommutative motives''. J. Noncommut. Geom. 9 (2015), no. 3,
851–875.



\bibitem[TT90]{TT90} R.~Thomason, T.~Trobaugh. ``Higher algebraic K-theory of schemes and of derived categories''.
The Grothendieck Festschrift vol. 3, p. 247–435, Birkh\"auser, 1990.

\bibitem[TV07]{TV07} B. T\"oen, M. Vaqui\'e, ``Moduli of objects in dg-categories''. Annales scientifiques de l'Ecole Normale
Superieure 40.3 (2007): 387-444.

\bibitem[Tsy83]{Tsy83} B.~L.~Tsygan, ``The homology of matrix Lie algebras over rings and the Hochschild homology''. Russian Math. Surveys, 38:2 (1983), 198–199.


\bibitem[Wag25a]{Wag25a} F.~Wagner, ``$q$-Hodge complexes over the Habiro ring''. arXiv:2510.04782 (preprint).

\bibitem[Wag25b]{Wag25b} F.~Wagner, ``$q$-de Rham cohomology and topological Hochschild homology over $\operatorname{ku}$''. arXiv:2510.06057 (preprint).

\bibitem[Weg12]{Weg12} C.~Wegner, ``The $K$-theoretic Farrell-Jones conjecture for $\operatorname{CAT}(0)$-groups. Proc. Amer. Math.
Soc., 140(3):779–793, 2012.

\bibitem[Weg15]{Weg15} C.~Wegner, ``The Farrell-Jones conjecture for virtually solvable groups''. J. Topol., 8(4):975-1016, 2015.

\bibitem[Wei89]{Wei89} C.~Weibel, ``Homotopy algebraic $K$-theory''. Algebraic K-theory and algebraic number theory,
Proc. Semin., Honolulu/Hawaii 1987, Contemp. Math. 83, 461-488, 1989.

\bibitem[Wei13]{Wei13} C.~Weibel, ``The K-book:An introduction to algebraic K-theory''.  Graduate Studies in Mathematics, vol. 145, American Mathematical Society, Providence, RI, 2013.

\bibitem[Wu16]{Wu16} X.~Wu, ``Farrell-Jones conjecture for fundamental groups of graphs of virtually cyclic groups''.
Topology Appl., 206:185–189, 2016.

\end{thebibliography}
\end{document}